\documentclass[11pt,a4paper]{article}
\usepackage[T1]{fontenc}
\usepackage{extarrows}
\usepackage{iftex}
\ifPDFTeX
  \usepackage[utf8]{inputenc}
\fi
\usepackage[dvipsnames]{xcolor}
\usepackage{amsmath,amssymb,amsbsy,amsthm,calc,enumerate,url,framed,fancyhdr,latexsym,tikz,bbm,enumitem}
\usepackage{mathrsfs,comment}
\usepackage{mathtools}
\usepackage[mathscr]{euscript}
\usepackage{graphicx}
\usepackage{algorithm,algpseudocode,aliascnt}
\usepackage[
  pageanchor,
  colorlinks,
  citecolor=blue,
  urlcolor=blue
]{hyperref}
\newcommand{\safeincludegraphics}[2][]{\includegraphics[#1]{#2}}
\usepackage[normalem]{ulem}
\usepackage[all,arc,2cell]{xy}
\UseAllTwocells
\usepackage[
  backend=biber,
  bibstyle=mla,
  citestyle=authoryear,
  sorting=nyt,
  doi=false,
  url=false
]{biblatex}
\IfFileExists{ref.bib}{\addbibresource{ref.bib}}{\addbibresource{./ref.bib}}

\DeclareSourcemap{
  \maps[datatype=bibtex]{
    \map{\step[fieldsource=month,match=\regexp{(?i)^(jan|january)$},replace={1}]}
    \map{\step[fieldsource=month,match=\regexp{(?i)^(apr|april)$},replace={4}]}
  }
}

\usepackage[font=small,skip=0pt]{caption}
\usepackage[all]{hypcap}

\usepackage{amsfonts,bbm,bm}
\usepackage{amssymb}
\RequirePackage[colorlinks,citecolor=blue,urlcolor=blue]{hyperref}
\RequirePackage[capitalize]{cleveref}

\newcommand{\Bd}{\mathbf{d}}

\newcommand{\Bp}{\mathbf{p}}
\newcommand{\Bq}{\mathbf{q}}
\newcommand{\Br}{\mathbf{r}}

\newcommand{\Bu}{\mathbf{u}}
\newcommand{\Bv}{\mathbf{v}}
\newcommand{\Bw}{\mathbf{w}}
\newcommand{\Bx}{\mathbf{x}}

\newcommand{\abs}[1]{\lvert #1 \rvert}
\newcommand{\norm}[1]{\lVert #1 \rVert}

\newcommand{\BS}{\mathbf{S}}

\newcommand{\BX}{\mathbf{X}}
\newcommand{\BY}{\mathbf{Y}}

\newcommand{\cC}{\mathcal{C}}

\newcommand{\cE}{\mathcal{E}}

\newcommand{\cP}{\mathcal{P}}

\newcommand{\PP}{\mathbb{P}}

\newcommand{\R}{\mathbb{R}} 
\newcommand{\N}{\mathbb{N}} 

\newcommand{\vecone}{\mathbf{1}} 
\newcommand{\diag}{\mathrm{diag}} 
\newcommand{\VAR}{\mathrm{Var}} 
\newcommand{\dom}{\mathrm{dom}} 
\newcommand{\tr}{\mathrm{tr}} 

\newcommand{\Id}{\mathrm{Id}}

\newcommand{\rom}[1]{%
  \textup{\uppercase\expandafter{\romannumeral#1}}%
}

\newcommand{\inner}[2]{\left\langle #1,#2\right\rangle}

\DeclareMathOperator*{\argmin}{arg\,min} 

\usepackage{amsthm}

\theoremstyle{plain}

\newtheorem{theorem}{Theorem}[section]

\newaliascnt{proposition}{theorem}
\newtheorem{proposition}[proposition]{Proposition}
\aliascntresetthe{proposition}

\newaliascnt{lemma}{theorem}
\newtheorem{lemma}[lemma]{Lemma}
\aliascntresetthe{lemma}

\newaliascnt{corollary}{theorem}
\newtheorem{corollary}[corollary]{Corollary}
\aliascntresetthe{corollary}

\newaliascnt{claim}{theorem}

\aliascntresetthe{claim}

\theoremstyle{definition}

\newaliascnt{definition}{theorem}
\newtheorem{definition}[definition]{Definition}
\aliascntresetthe{definition}

\newaliascnt{assumption}{theorem}
\newtheorem{assumption}[assumption]{Assumption}
\aliascntresetthe{assumption}

\theoremstyle{definition}

\newaliascnt{remark}{theorem}

\aliascntresetthe{remark}

\newaliascnt{example}{theorem}

\aliascntresetthe{example}

\newcommand{\E}{\mathbb{E}}

\DeclareMathOperator{\cvx}{cvx}

\DeclareMathOperator{\prox}{Prox}

\DeclareMathOperator{\mmse}{mmse}

\DeclareMathOperator{\dive}{div}

\newcommand*{\rd}{\mathrm{d}}
\newcommand*{\dd}{\, \rd}

\newcommand{\bgs}{\bar\gamma_{1\star}}

\definecolor{darkred}{RGB}{139,0,0}
\definecolor{darkyellow}{RGB}{204,160,0}
\definecolor{darkgreen}{RGB}{0,100,0} 

\hypersetup{
  colorlinks,
  linkcolor={red!70!black},
  citecolor={blue!65!black},
  urlcolor={blue!70!black}
}

\DeclareDelimFormat{nameyeardelim}{\addcomma\space}

\DeclareFieldFormat{fullcitehyperref}{%
  \DeclareFieldAlias{bibhyperref}{noformat}%
  \bibhyperref{#1}%
}

\DeclareCiteCommand{\parencite}[\mkbibparens]
  {\usebibmacro{prenote}}
  {%
    \usebibmacro{citeindex}%
    \printtext[fullcitehyperref]{%
      \usebibmacro{cite}%
      \iflastcitekey
        {\usebibmacro{postnote}}
        {}%
    }%
  }
  {\multicitedelim}
  {}

\DeclareCiteCommand{\textcite}
  {\boolfalse{cbx:parens}}
  {%
    \usebibmacro{citeindex}%
    \iffirstcitekey
      {\setcounter{textcitetotal}{1}}
      {\stepcounter{textcitetotal}\textcitedelim}%
    \printtext[fullcitehyperref]{%
      \usebibmacro{textcite}%
      \iflastcitekey
        {%
          \usebibmacro{textcite:postnote}%
          \global\boolfalse{cbx:parens}%
        }
        {}%
    }%
  }
  {%
    \ifbool{cbx:parens}
      {\bibcloseparen\global\boolfalse{cbx:parens}}
      {}%
  }
  {}

\usepackage{bm}

\newcommand{\bgamma}{\bar{\gamma}}
\newcommand{\pliminf}{\mathop{\mathrm{p\text{-}liminf}}}

\theoremstyle{plain}

\theoremstyle{definition}

\newtheorem{rem}{Remark}[section]

\crefname{thm}{Theorem}{Theorems}
\Crefname{thm}{Theorem}{Theorems}
\crefname{cor}{Corollary}{Corollaries}
\Crefname{cor}{Corollary}{Corollaries}
\crefname{prop}{Proposition}{Propositions}
\Crefname{prop}{Proposition}{Propositions}
\crefname{lem}{Lemma}{Lemmas}
\Crefname{lem}{Lemma}{Lemmas}
\crefname{ass}{Assumption}{Assumptions}
\Crefname{ass}{Assumption}{Assumptions}
\crefname{defn}{Definition}{Definitions}
\Crefname{defn}{Definition}{Definitions}
\crefname{rem}{Remark}{Remarks}
\Crefname{rem}{Remark}{Remarks}

\makeatletter
\let\c@obs=\c@thm
\let\c@cor=\c@thm
\let\c@prop=\c@thm
\let\c@lem=\c@thm
\let\c@prob=\c@thm
\let\c@con=\c@thm
\let\c@conj=\c@thm
\let\c@defn=\c@thm
\let\c@notn=\c@thm
\let\c@notns=\c@thm
\let\c@exmp=\c@thm
\let\c@ax=\c@thm
\let\c@pro=\c@thm
\let\c@ass=\c@thm
\let\c@warn=\c@thm
\let\c@rem=\c@thm
\let\c@sch=\c@thm
\let\c@equation\c@thm
\numberwithin{equation}{section}
\makeatother

\title{Spectral Dependence of Convex Regularization:
Fundamental Limits under Right-Rotationally Invariant Designs}

\date{}

\author{
Baichen Tan \qquad Audrey Yang \qquad Cynthia Rush\\[0.5em]
\small Department of Statistics, Columbia University
}

\crefname{ALG@line}{Line}{Lines}
\Crefname{ALG@line}{Line}{Lines}
\crefname{theorem}{Theorem}{Theorems}
\Crefname{theorem}{Theorem}{Theorems}
\crefname{proposition}{Proposition}{Propositions}
\Crefname{proposition}{Proposition}{Propositions}
\crefname{lemma}{Lemma}{Lemmas}
\Crefname{lemma}{Lemma}{Lemmas}
\crefname{corollary}{Corollary}{Corollaries}
\Crefname{corollary}{Corollary}{Corollaries}
\crefname{claim}{Claim}{Claims}
\Crefname{claim}{Claim}{Claims}
\crefname{definition}{Definition}{Definitions}
\Crefname{definition}{Definition}{Definitions}
\crefname{assumption}{Assumption}{Assumptions}
\Crefname{assumption}{Assumption}{Assumptions}
\crefname{remark}{Remark}{Remarks}
\Crefname{remark}{Remark}{Remarks}
\crefname{example}{Example}{Examples}
\Crefname{example}{Example}{Examples}

\makeatletter
\def\theHALG@line{\thealgorithm.\arabic{ALG@line}}
\patchcmd{\ALG@step}{\addtocounter{ALG@line}{1}}{\refstepcounter{ALG@line}}
  {}{\PackageError{arxiv-main}{Could not enable algorithm line references}{}}
\newcommand{\labelalgorithmline}[1]{%
  \begingroup
  \def\@currentcounter{ALG@line}%
  \protected@edef\@currentlabel{\arabic{ALG@line}}%
  \protected@edef\cref@currentlabel{[ALG@line][\arabic{ALG@line}][]\arabic{ALG@line}}%
  \protected@edef\@currentHref{ALG@line.\theHALG@line}%
  \cref@label{#1}%
  \endgroup}
\makeatother
\begin{document}
\maketitle

\begin{abstract}
We study the fundamental limits of convex-regularized estimation in high-dimensional linear regression with right-rotationally invariant design matrices. We show that the asymptotic risks of convex-penalized least squares estimators are lower bounded by a spectrum-dependent value that is the risk of an approximate message passing algorithm known as Bayes VAMP, and we further characterize when the lower bound is attainable. 
As a technical ingredient in the proof of our lower bound theorem, we characterize the asymptotic performance of the \(\ell_2\)-perturbed convex estimator for every fixed perturbation strength \(\lambda>0\). This closes a gap in the literature, in which \(\lambda\) was required to be sufficiently large or restrictions were imposed on the class of convex estimators.

The benefit of our approach is that we can conduct a direct analysis of the spectrum's impact on the lower bound in the high-dimensional limit. In particular, we show that the lower bound is monotone in an ordering on the limiting spectral distributions of the design matrix. These results isolate how the full singular-value distribution of the design matrix, not merely the aspect ratio or average measurement strength, governs the limitations of convex regularization.
\end{abstract}

\pagestyle{plain}
{\setlength{\parskip}{0pt}
\tableofcontents}
\clearpage

\section{Introduction}
\label{sec:introduction}
Convex regularization is a central tool used in statistics and machine learning for constructing stable and computationally tractable estimators.
In linear regression problems, it augments the loss function with a penalty on the coefficients. Foundational and commonly used examples include ridge regression, which imposes an \(\ell_2\) penalty to stabilize estimation with correlated covariates \parencite{Hoerl1970ridge}, the Lasso, which uses an \(\ell_1\) penalty to combine shrinkage with variable selection \parencite{Tibshirani1996lasso}, and the elastic net, which combines these two forms of regularization \parencite{Zou2005elastic}. The performance of these methods, however, can depend not only on the choice of the penalty but also on aspects of the model including the structure of the design matrix and assumptions on the signal elements.

Indeed, a design matrix does more than determine how many observations are available: its squared singular values describe the strength of measurement in different directions in parameter space. Their sum is the total measurement strength, and their average is the normalized squared Frobenius norm of the design. Two designs can have the same number of observations per parameter, \(n/p\), and the same total measurement strength while allocating that strength very differently. One may measure all directions approximately equally, whereas another may compensate for many weakly observed directions with a smaller number of strongly observed directions. The same convex regularizer can perform differently under these two allocations.

This paper asks how that spectral allocation affects the fundamental limitations of convex estimation in high-dimensional linear regression. Our main conclusion is that the limiting spectral law, not only the aspect ratio or the average squared singular value, governs a family of lower bounds for the risk of convex regularization. This permits comparisons that are invisible in the canonical iid Gaussian model, where the limiting spectral shape is fixed once the aspect ratio $n/p$ is specified.

\textbf{Statistical model.}
Throughout, we assume the data is drawn from the linear model
\begin{equation}
    y=X\beta^\star+\varepsilon,
    \label{eq:linear-model}
\end{equation}
where we have a design matrix $X \in \mathbb R^{n\times p}$, signal vector $\beta^\star \in \mathbb R^p$, and noise vector $\varepsilon \in \mathbb R^n$. We will study estimation performance asymptotically, under the proportional high-dimensional asymptotic regime where both $n, p \rightarrow \infty$ with $n/p \rightarrow \delta \in (0, \infty)$. We observe the outcome vector $y$ and the design matrix $X$, and we seek to estimate the unobserved signal vector $\beta^\star$, whose limiting empirical coordinate law is assumed known. This assumption allows deterministic signal sequences, while an iid Bayesian signal is a special case.

Our assumption on the design matrix is that it is \textit{right-rotationally invariant} (RRI), or in other words, writing \(X=U\Sigma V^T\), the matrix \(V\) is Haar distributed and independent of the remaining model components. We assume that the empirical law of the eigenvalues of \(X^T X\) converges to a compactly supported probability law \(\mu\) on \([0,\infty)\), which we refer to as the limiting Gram spectral law. Unlike in an iid Gaussian model, this law is not determined by the aspect ratio, allowing us to treat the shape of the spectrum as a statistical parameter.

We focus on the widely used and computationally tractable class of separable convex-penalized least-squares estimators defined as
\begin{equation}
    \widehat\beta_{\cvx}^{h}
    \in
    \argmin_{\beta\in\R^p}
    \Big\{
        \frac12\norm{y-X\beta}_2^2
        +\sum_{j=1}^p h(\beta_j)
    \Big\},
    \label{eq: penalized-OLS}
\end{equation}
for scalar penalty functions \(h:\R\to\R\cup\{+\infty\}\) that are proper, lower semicontinuous (lsc), and convex, a canonical class in convex analysis \parencite{Rockafellar1970}. For example, all separable norms on vector spaces satisfy these properties. This class is also natural for our analysis: the assumptions ensure that the scalar proximal map of each penalty is well defined.

This work studies the question: 
\emph{within this class of convex regularizers, how does the limiting Gram spectral law affect the best possible asymptotic performance?} For a given spectral distribution and asymptotic empirical distribution on the signal, we will call the optimal achievable asymptotic performance by the class of estimators defined in \cref{eq: penalized-OLS} the \textit{convex barrier}, presented more precisely in \cref{thm:Bayes VAMP-lower-bound}.

\textbf{A motivating example.}
Let us first show a simple example to see how design matrix properties can change the performance of convex estimation. Consider an orthogonal matrix $V \in \mathbb R^{p \times p}$ that is drawn from the uniform distribution on the orthogonal group $\mathbb O(p)$. 

The first design matrix we consider is $X_{0}=V^T$. Notice that $X_{0}^T X_{0}=\mathbf I_p$, so that all of the squared singular values of $X_0$ are equal to one. We will refer to this as a \emph{flat} design. For the second design, using the same $V$,
fix \(a\in(0,1)\), and set $X_{a}=\Sigma_{a}V^T,$ where 
\[
    \Sigma_{a}^2 = \diag(\underbrace{1-a,\ldots,1-a}_{p/2}, \underbrace{1+a,\ldots,1+a}_{p/2}).
\]
Now, half of the squared singular values of $X_a$ are equal to \(1-a\) and half are equal to \(1+a\). Thus,
both of the design matrices considered use the same number of observations per parameter and have the
same total squared singular-value mass: $p^{-1}\tr(X_{0}^T X_{0}) = p^{-1}\tr(X_{a}^T X_{a}) = 1$. The only difference is the allocation of the squared singular-value mass; the flat design measures every direction equally, whereas \(X_{a}\) trades weaker measurements in half of the directions for stronger measurements in the other half.

\begin{figure}[t]
    \centering
    \safeincludegraphics[width=0.96\textwidth]{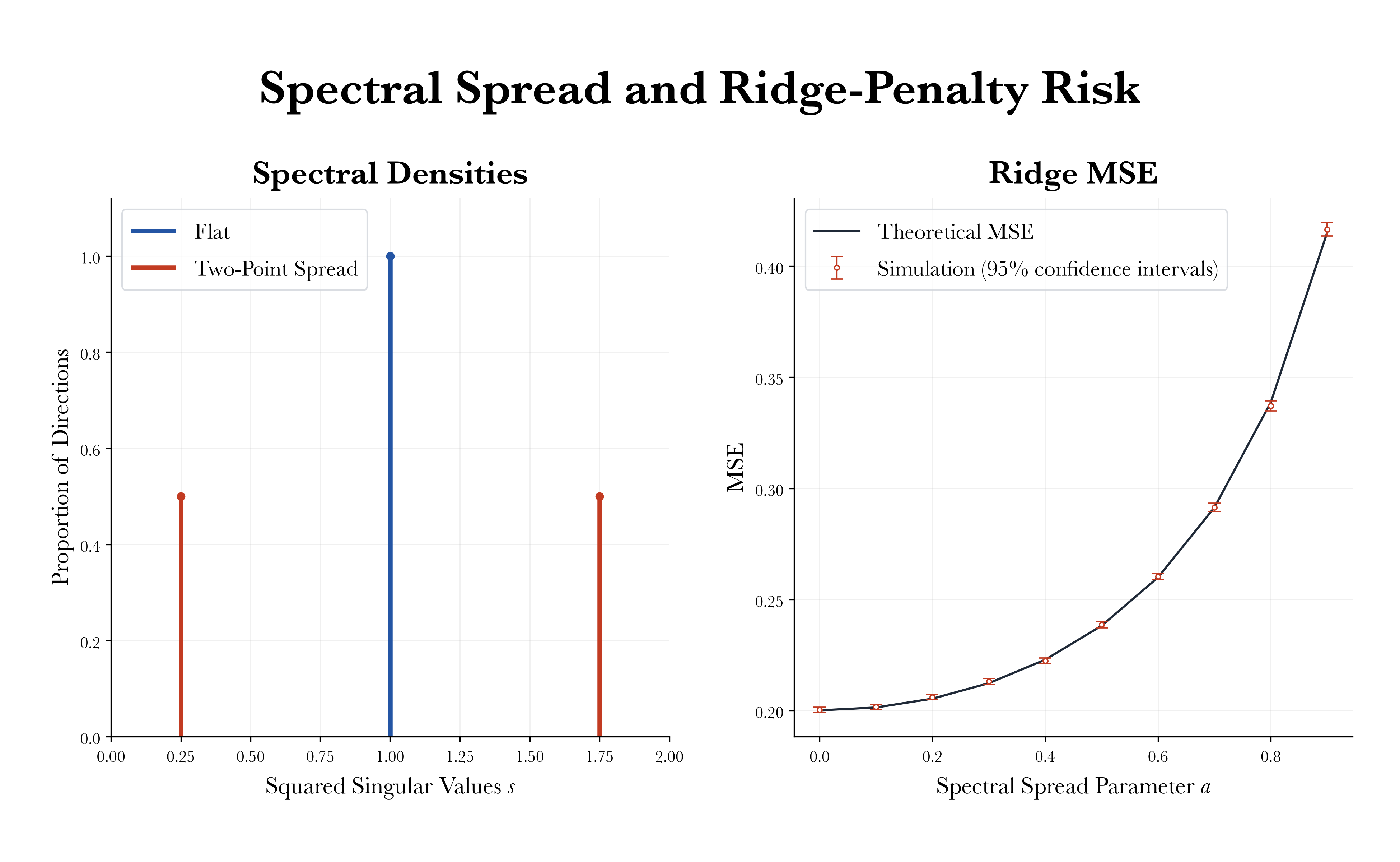}
    \caption{Spectral allocation affects estimation even when \(\tr(X^T X)=p\) is held fixed. We assume the data are drawn from the linear model in \cref{eq:linear-model} using \(n = p=800\),
    \(\beta^\star\sim N(0,\mathbf I_p)\), and
    \(\varepsilon\sim N(0,0.25\mathbf I_p)\). We use the ridge regression estimator defined in \cref{eq: penalized-OLS} where the penalty $h(b) = 0.125 b^2$.
    The figure on the left shows the spectral distributions for two design matrices with \(a=0\) and \(a=0.75\). The right panel shows the MSE of the ridge estimator as \(a\) increases. The 300-repetition averages are plotted along with 95\% Monte Carlo error bars, with exact expected MSE in black.}
    \label{fig:intro-spectral-spread}
\end{figure}

\cref{fig:intro-spectral-spread} shows that this redistribution in measurement allocation alone is sufficient to substantially change the performance of convex estimation. The left plot in 
\cref{fig:intro-spectral-spread} visualizes the spectral distributions of the flat design (i.e.\ $a=0$) and the design with $a= 0.75$. The right plot shows both the exact expected MSE and empirical performance measured by the mean squared error (MSE) of ridge regression under a linear model with Gaussian noise, a Gaussian prior on the unknown signal, and $p = 800$. We can see that as the spectral spread increases, meaning as $a$ grows, the performance of ridge regression decays. In other words, in terms of ridge performance, it appears that strengthening half of the directions does not compensate for weakening the other half.
We provide a rigorous result in \cref{sec:main-spectral-order} that captures this phenomenon.

While this example concerns one particular penalty and is not itself a statement
about the best convex estimator, its purpose is to expose the statistical
question addressed in this paper: does the same spectral-allocation principle
govern a lower bound over an entire class of convex
procedures? \cref{prop:spectral-monotonicity,cor:flat-spectrum-best} show that
it does. The aspect ratio and average squared singular value are not sufficient
to determine this lower bound; the shape of the full squared-singular-value law
matters.

\textbf{From an iid Gaussian lower bound to spectrum-dependent lower bounds.} When the design matrix $X$ has entries that are iid $N(0, 1/n)$, the work in \textcite{Celentano2022barrier} established a fundamental
lower bound on the asymptotic risk of convex-penalized regression. Their lower bound
is given by the asymptotic MSE of Bayes-AMP, a message-passing procedure whose
denoising step uses the signal prior information \parencite{Bayati2011AMP, Donoho2009MessagePassing}. They also characterized when a convex
estimator can attain that lower bound and when a strict gap remains.

The iid Gaussian setting studied in \textcite{Celentano2022barrier} cannot, however, reveal how the shape of the spectrum
affects this barrier. Once the aspect ratio is fixed for iid Gaussian matrices, the limiting Gram
spectrum is the Marchenko--Pastur law, derived in \textcite{MarchenkoPastur}. The aspect ratio records the number of
measurements $n$ relative to the ambient dimension $p$, but it does not separately
describe whether the measurements cover the parameter space evenly, whether
some directions are nearly invisible, or whether a small number of
directions receive a disproportionate share of the measurement strength.
 These aspects of a design's measurement are instead reflected in its spectrum.
The motivating example just discussed isolates this feature in ridge regression and demonstrates that the limiting spectral distribution greatly affects the estimator's performance.
This limitation is also practically significant:
in many statistical settings where the linear model in \cref{eq:linear-model} is used, the design matrix is in fact
ill-conditioned or otherwise far from iid Gaussian, so it is worthwhile to better understand the limits of convex estimation in such settings.

We therefore study the fundamental limits of convex-penalized least squares for a more general class of design matrices; in particular, we will consider the class of RRI design matrices introduced above. This class, which will be defined concretely in the next section and is studied often in the literature related to high-dimensional asymptotics, encapsulates a much wider range of covariate matrices: the class allows dependence between the columns and a broad range of bounded spectral laws under the assumptions stated below.
When using the RRI class to model the design, it therefore becomes possible to ask not only what the fundamental lower bounds on the MSE for convex-penalized least squares under different signal prior distributions are, but also how the design spectrum changes these bounds. Indeed, the central question explored in this paper is:
\begin{quote}
\emph{For an RRI design matrix, how well can we perform estimation of the signal $\beta^\star$ when limited to convex-penalized estimators, and how does the distribution of the design's singular values, along with the signal prior distribution, change this barrier to convex-penalized estimation?}
\end{quote}

For RRI designs, the corresponding
message-passing procedure is Bayes-VAMP, where VAMP is shorthand for \textit{vector approximate message passing}; see \textcite{Rangan2018vamp, Fletcher2018Plugin, SchniterRF16a}. Its asymptotic MSE depends on the
signal law, the noise level, and the full limiting spectral law $\mu$. We prove in \cref{thm:Bayes VAMP-lower-bound}
that this MSE lower-bounds the asymptotic MSE of every estimator in our admissible convex class. As the spectral law changes, this produces a family of lower bounds on the convex barrier rather than a single extension of the iid Gaussian formula.

\subsection{Main contributions}
\label{sec:intro-contributions}

The paper contains three main contributions:

\begin{enumerate}[label=\textup{(\roman*)},leftmargin=2.4em]

    \item \textbf{Lower bounds indexed by the spectrum.} We prove that, for every limiting spectral law satisfying certain assumptions, the Bayes-VAMP asymptotic MSE lower bounds the convex barrier; in other words, it is a lower bound for the asymptotic MSE of every estimator obtained from a coordinate-separable convex penalty in the admissible class.
    
    Our analysis shows that different spectral laws generally yield different lower bounds. We prove a spectral monotonicity result comparing these bounds under specific orderings of the spectral laws. In particular, if one spectrum dominates another in Stieltjes-transform order, then the dominant spectrum has a larger Bayes-VAMP lower bound. Convex order, which describes a mean-preserving spread of spectral mass, implies this Stieltjes-transform ordering, and we use the convex ordering to show that, among spectral laws with fixed average measurement strength and common bounded support, the flat spectrum minimizes the lower bound and an endpoint-supported spectrum maximizes it. Additional results characterize how spectral moments, mass or density around zero, and product-matrix structure affect the lower bound in different regimes.

    \item \textbf{An exact criterion for attainability by convex regularization.} 
    The Bayes-VAMP MSE is a lower bound for the convex barrier in all regimes (under our stated assumptions), but it need not equal the smallest risk attainable by the class of admissible penalties. We present an exact criterion for when they coincide: equality holds precisely when the signal law, after convolution with a Gaussian distribution whose variance is determined by the signal law, noise level, and spectral law, is log-concave. In that case we construct an explicit optimal penalty, and every measurable minimizer for that penalty attains the lower bound.

    The spectrum therefore has two roles. It changes the numerical value of the Bayes-VAMP lower bound, and it can determine whether that lower bound is attainable by convex regularization. Whenever the equality criterion holds for the spectra under comparison, e.g., for a log-concave signal prior, the spectral ordering also orders the actual infimal convex risk. Outside that regime, it orders the lower bound, while the convex barrier may be strictly larger.
    
    \item \textbf{MSE characterization of $\ell_2$-perturbed penalties.} 
    The proof of our lower bound theorem relies critically on a technical result that adds a further $\ell_2$-perturbation to the convex objective in  \cref{eq: penalized-OLS} to make the resulting estimator's asymptotic MSE exactly characterizable.
    We believe this technical result will be independently useful for the AMP community.
    We first prove that, at every regularization strength for the perturbed penalty, the associated convex-VAMP algorithm is described by a state evolution whose equations admit a solution. We can then use this solution to characterize the asymptotic MSE of the optimizer of the perturbed objective.
    \textcite{Gerbelot2020vamp} previously prove this characterization for sufficiently large perturbation strength and conjecture that the result should also hold as the regularization vanishes. \textcite{Li2026debiasing} prove this for all perturbation strengths with additional assumptions on the convex penalty. Our result builds on these two works and proves the result for the entire admissible penalty class, hence closing the gap in the literature.
\end{enumerate}

\subsection{Relation to prior work}
\label{sec:related-work}
The closest conceptual predecessor is  Celentano and Montanari \parencite{Celentano2022barrier}, which relies on techniques from approximate message passing (AMP) to establish a similar lower bound in the iid Gaussian design setting. More broadly, AMP and Gaussian comparison methods give precise high-dimensional risk characterizations for many classes of estimators \parencite{Bayati2012LassoRisk, donoho2016high, bu2020algorithmic, sur2019modern, Thrampoulidis2018Precise,Miolane2021Lasso}. For a survey on AMP and related algorithms, see \textcite{Feng2021Tutorial}. 
Celentano and Montanari combine these methods with an oracle-perturbation argument to lower-bound the convex barrier by the Bayes-AMP risk, and characterize attainability through log-concavity. Our work considers the broader class of RRI matrices and can therefore compare the resulting lower bounds across different limiting spectral laws, but our penalty class is narrower: \textcite{Celentano2022barrier} allow dimension-dependent sequences of certain nonseparable symmetric penalties, whereas we restrict our main theorem to fixed coordinate-separable penalties.

RRI matrices are of direct interest in the message-passing literature: orthogonal AMP and VAMP \parencite{Ma2017OAMP,Rangan2018vamp, Fletcher2018Plugin,Takeuchi2020EP,Fan2022RotInvAMP} extend the state evolution -- the scalar recursion that determines the asymptotic behavior of the AMP iterates -- beyond iid Gaussian matrices. 
Works such as \textcite{Liu2024UnifyingRRI} give a unified construction of AMP algorithms for rotationally invariant models and make explicit the role of the limiting spectral law in their state evolution.
RRI models also arise in spectral-universality results, which show that asymptotic behavior may depend primarily on the squared singular-value law even for more structured designs \parencite{Dudeja2022SpectralUniversality,Wang2024AMPUniversality}. 
These results suggest that limiting spectra may govern asymptotic behavior beyond exact rotational invariance. Our results, however, assume an exactly Haar right singular-vector matrix; extending the lower bound for convex-penalized estimation through universality is a separate problem that we do not pursue here. Related Bayesian work characterizes mutual information and posterior risk for important rotationally invariant models \parencite{BarbierEtAl2018MutualInformation,LiFanSenWu2024RotInv}.

Within the VAMP literature, Gerbelot and collaborators \parencite{Gerbelot2020vamp} studied the connection between VAMP and convex-penalized linear regression under sufficiently large additional \(\ell_2\) regularization. Related work in \textcite{Li2026debiasing} gives conditions on the penalty for which the asymptotic MSE of the associated convex estimator can be characterized through VAMP. Our analysis builds on this framework by extending their fixed-point results to every fixed positive perturbation strength, including arbitrarily small values, without imposing additional assumptions on our admissible penalty class. 

Finally, random-matrix analyses of ridge and minimum-norm regression already show that estimation risk can depend on the full spectral distribution \parencite{DobribanWager2018Ridge,HastieEtAl2022Ridgeless}. Our contribution is different: rather than evaluating a penalized high-dimensional linear regression problem with a fixed convex penalty $h$ and spectrum, we compare the lower bound across spectral laws. 
State evolution recursions with spectrum-based updates enable detailed analysis on the impact of spectral geometry. For example, \textcite{Ma2023Spikiness} investigate how the ``spikiness'' of a bi-rotationally invariant matrix, where both singular matrices are Haar, affects recovery in generalized linear models, also emphasizing the role of the design spectrum in determining statistical recovery performance.

\subsection{Organization of the paper}
The paper is organized as follows. \cref{sec:model} introduces the model, admissible penalty class, and spectrum-dependent lower bound. \cref{sec:main-results} states the main results, and \cref{sec:spectral-consequences} develops spectral examples and consequences. \cref{sec: examples-penalties} explains the width condition and gives sufficient criteria for standard penalties. \cref{sec:simulations} presents numerical illustrations. The longer technical proofs are collected in the appendices.

\section{Model and the spectrum-dependent lower bound}
\label{sec:model}

Our lower bound depends on two limiting objects: the empirical law of the
signal coordinates and the eigenvalue law of the Gram matrix. We first state
the assumptions on these objects and on the convex estimator. We then
construct the scalar Bayes--VAMP recursion, whose greatest fixed point
determines the risk benchmark, and specify the penalty class for which the
lower bound holds.

\subsection{Model and penalties}

\textbf{Notation and empirical convergence.}
Let \(\mathcal P(E)\) denote the probability laws on a space \(E\), and let
\(\mathcal P_2(\R)\) denote those on \(\R\) with finite second moment.
We write \(W_2\) for the Wasserstein-2 distance and \(\delta_x\) for the
point mass at \(x\). Subscripts on \(\E_\nu\) and \(\PP_\nu\) indicate
the law under which an expectation or probability is taken. Unless another
limit is specified, \(p\to\infty\) with \(n=n(p)\), and
\(n/p\to\delta\in(0,\infty)\). 
We write \([p]:=\{1,\ldots,p\}\).
For a differentiable map \(F:\R^d\to\R^m\), let \(DF(x)\) denote its
Jacobian and when \(m=d\), let \(\dive F(x):=\tr DF(x)\) denote its divergence.

Both the design and the signal assumptions are stated through empirical
laws. We use Wasserstein-2 convergence so that convergence of the law is
accompanied by convergence of its second moment.

\begin{definition}[Convergence in empirical distribution]
Let \(\Bx_p=(x_{1},\ldots,x_{p})\in\R^p\) be a deterministic or
random vector sequence, with empirical law
\(\nu_p:=p^{-1}\sum_{j=1}^p\delta_{x_{j}}\).
We write \(\nu_p\xrightarrow{W_2}_{\mathrm{a.s.}}\nu\) if
\(W_2(\nu_p,\nu)\to0\) almost surely, and
\(\nu_p\xrightarrow{W_2}_{\PP}\nu\) if
\(\PP(W_2(\nu_p,\nu)>\epsilon)\to0\) for every \(\epsilon>0\).
\end{definition}

\textbf{Standing assumptions.}
For each \(p\), consider the linear model in \cref{eq:linear-model}
and the estimator in \cref{eq: penalized-OLS}. The following assumptions
are maintained throughout the main results.

\begin{assumption}[RRI design]
\label{ass1}
The design admits an SVD given by \(X=U\Sigma V^T\) in which
\(V\in\R^{p\times p}\) is Haar distributed on the orthogonal group
\(\mathbb O(p)\) and independent of \((U,\Sigma,\beta^\star)\).
The pair \((U,\Sigma)\) may be deterministic or random, provided that
they are jointly independent of \(V\).
For the eigenvalues \(s_{1,p},\ldots,s_{p,p}\) of \(X^TX\), we assume $\mu_p:=\frac1p\sum_{i=1}^p\delta_{s_{i,p}}
    \xrightarrow{W_2}_{\mathrm{a.s.}} \mu,$
where \(\mu\) is compactly supported on \([0,\infty)\) and
\(\E_\mu[S]>0\) for \(S\sim\mu\). We also assume $\limsup_{p \rightarrow \infty} \max_{i \in [p]} s_{i, p} < \infty$ almost surely
or equivalently \(\limsup_{p\to\infty}\|X\|_{\rm op}<\infty\)
almost surely.
\end{assumption}

The Haar condition specifies the random orientation of the right-singular
directions and the law \(\mu\) describes the measurement strengths along
them.  The spectral convergence and operator-norm bound are additional
restrictions on an RRI design.
In particular, the latter excludes unbounded
spectral outliers that empirical convergence alone need not detect. We call $\mu$ the limiting spectral law of $X$ and we define $q:=\PP_\mu(S>0)=1-\PP_\mu(S=0)$
as the positive spectral mass and $s_+ := \sup \mathrm{supp}(\mu)$ as the upper endpoint of the limiting support.

\begin{assumption}[Gaussian noise]
\label{ass2}
For a fixed \(\sigma^2>0\), the noise  is independent of
\((U,\Sigma,V,\beta^\star)\) and satisfies
\(\varepsilon\sim N(0,\sigma^2 I_n)\).
\end{assumption}

\begin{assumption}[Empirical signal prior law]
\label{ass3}
For some law \(\pi\in\cP_2(\R)\), the signal satisfies
$\frac1p\sum_{j=1}^p\delta_{\beta_j^\star}
    \xrightarrow{W_2}_{\mathrm{a.s.}}\pi.$
We write \(B^\star\sim\pi\).
\end{assumption}

The signal sequence may be deterministic or random. Thus \(\pi\) is a
limiting empirical coordinate law, and \cref{ass3} does not require iid
signal coordinates. The scalar variable \(B^\star\) will represent this
law in the denoising calculations below.

\begin{assumption}[Convex penalty]
\label{ass4}
The scalar penalty \(h:\R\to\R\cup\{+\infty\}\) is proper, lsc, nonconstant, and convex. We denote this class of
functions by \(\cC\). The regression objective in
\cref{eq: penalized-OLS} has a nonempty minimizer set almost surely for
every \(p\in\N\).
\end{assumption}

As in \cref{eq: penalized-OLS}, the same \(h\) is applied to every
coordinate and is fixed as \(p\to\infty\). Its choice may depend on the
limiting model parameters. The connection between this penalty and the
scalar denoising analysis is its proximal map.

\begin{definition}\label{defn:proximal-map}
For \(h\in\cC\) and \(a>0\), the scalar proximal map is
\[
    \prox_{a h}(y)
    :=\argmin_{x\in\R}
    \Big\{\frac12(x-y)^2+a h(x)\Big\}.
\]
\end{definition}

The proximal map is nondecreasing and one-Lipschitz. We denote its
almost-everywhere derivative by \(\prox_{ah}'\), and write \(\partial h\)
for the convex subdifferential and \(M:=\argmin_{x\in\R}h(x)\) for the
minimizer set of the scalar penalty, which may be empty.

\subsection{The scalar benchmark}
\label{sec:vamp-intro}

The lower bound is constructed by alternating a scalar denoising step, which
uses the signal prior \(\pi\), and a linear estimation step, which uses the spectral law \(\mu\) and noise variance
\(\sigma^2\). This separation is the mechanism that later lets
us compare performance under different design spectra while holding the signal prior fixed. Each step receives an effective noise variance and returns
the variance passed to the other step. We define these maps directly.
Their relation to the finite-dimensional Bayes--VAMP algorithm, including
the regularity conditions for its state evolution, is given in
\cref{sec:Bayes-VAMP-intro}.

\textbf{The prior stage.}
Let \(B^\star\sim\pi\) and \(Z\sim N(0,1)\) be independent.
For an effective noise variance \(\tau>0\), consider estimating
\(B^\star\) from \(B^\star+\sqrt\tau Z\). The posterior mean has
minimum mean squared error
\begin{equation}
\label{eq:mmse}
m_\pi(\tau) :=
    \mmse_\pi(\tau)
:=
\mathbb E\left[
    \left(
        \mathbb E[B^\star\mid B^\star+\sqrt{\tau}Z]
        -
        B^\star
    \right)^2
\right].
\end{equation}
VAMP passes an extrinsic variance, rather than the posterior MSE itself, from
the prior stage to the spectrum stage.
This variance passed to the linear stage includes VAMP's extrinsic
correction, which accounts for the information already present in the
incoming observation. For the posterior-mean denoiser, it is
\begin{equation}
\mathcal E_\pi(\tau)
:=
\frac{m_\pi(\tau)}
     {1-m_\pi(\tau)/\tau}.
    \label{eq:extrinsic}
\end{equation}
When \(m_\pi(\tau)>0\), this is equivalently
\(\mathcal E_\pi(\tau)^{-1}=m_\pi(\tau)^{-1}-\tau^{-1}\):
the incoming precision is subtracted from the inverse posterior MSE.
The scalar comparison in \cref{lem:scalarcomparison} shows that the
posterior mean minimizes this outgoing variance among the denoisers used
in our proof. This is the step that makes it a benchmark for convex denoising.

\textbf{The spectrum stage.}
Let \(S\sim\mu\). Given an incoming variance \(\omega>0\), the linear
stage returns
\begin{equation}
L_\mu(\omega)
:=
\frac{
    \mathbb E_\mu\left[
        \dfrac{\sigma^2\omega}{\sigma^2+\omega S}
    \right]
}{
    \mathbb E_\mu\left[
        \dfrac{\omega S}{\sigma^2+\omega S}
    \right]
}.
    \label{eq:spectrum-mod}
\end{equation}
The numerator is the mean squared error of the linear update. If this
error is denoted by \(v\), the denominator is \(1-v/\omega\), so the
ratio applies the same extrinsic correction as \cref{eq:extrinsic}.
The denominator is positive because \(\E_\mu[S]>0\). We use the
continuous extension
$L_\mu(0):=\lim_{\omega\downarrow0}L_\mu(\omega)
    ={\sigma^2}/{\E_\mu[S]}.$
The averaging over \(S\) in \cref{eq:spectrum-mod} is the sole entry
point of the design spectrum into the scalar recursion.

\textbf{The selected fixed point.}
One full step of the scalar recursion determining our lower bound first applies the prior map and then the spectrum map:
\begin{equation}
    \tau_{t+1}^{B}
    =
    T_{\mu,\pi}(\tau_t^B)
    :=
    L_\mu\bigl(\cE_\pi(\tau_t^B)\bigr).
    \label{eq:Bayes-VAMP-recursion}
\end{equation}
Write \(V_\pi:=\operatorname{Var}(B^\star)\). For \(V_\pi>0\),
\cref{lem:mmsecalculus,lem:Lcalculus} show that \(m_\pi\),
\(\mathcal E_\pi\), and \(L_\mu\) are continuous and nondecreasing,
and that \(0\leq\mathcal E_\pi(\tau)\leq V_\pi\).
If \(V_\pi=0\), then \(m_\pi=\mathcal E_\pi=0\), so these
properties hold directly. In either case,
$0<L_\mu(0)\leq T_{\mu,\pi}(\tau)\leq L_\mu(V_\pi)<\infty.$
Thus \(T_{\mu,\pi}\) is a continuous, nondecreasing self-map of
\([L_\mu(0),L_\mu(V_\pi)]\), and its fixed-point set is nonempty
and compact. We define its greatest fixed point by
\begin{equation}
\tau_B^+
:=
\max\left\{
    \tau>0:
    T_{\mu,\pi}(\tau)=\tau
\right\}.
    \label{eq:tau-B-definition}
\end{equation}
This selection agrees with the recursion initialized at
\(\tau_0^B=L_\mu(V_\pi)\): its iterates decrease to \(\tau_B^+\),
as shown in \cref{lem: Bayes-fpt-limit}. The resulting risk benchmark is
\(m_\pi(\tau_B^+)\). Under the standing assumptions,
\cref{thm:Bayes VAMP-lower-bound} proves that it lower-bounds the
asymptotic MSE for every fixed penalty in the admissible class defined next.
The prior affects this lower bound through \(m_\pi\) and
\(\mathcal E_\pi\), while the design affects it through \(L_\mu\).

\subsection{The admissible penalty class}

The proof of the lower-bound theorem, \cref{thm:Bayes VAMP-lower-bound}, requires an additional condition on \(h\) only
when the limiting spectrum has an atom at zero and a finite inverse
moment on its positive part. To describe this regime, write
\[
    p_0:=\PP_\mu(S=0)=1-q,
    \qquad
    I_+:=\int_{(0,\infty)}\frac1s\,\mu(ds) = \E\left[\frac{\mathbf 1_{\{S > 0\}}}{S}\right].
\]
The restriction to \((0,\infty)\) is essential: \(I_+\) can be finite
even when \(\E[1/S]=\infty\) because \(p_0>0\). An atom at zero may
arise from exact zero eigenvalues or from positive eigenvalues that vanish
asymptotically. It need not represent the finite-dimensional null space.
Also, \(q>0\) by \cref{ass1}.

The additional condition controls the average slope of the proximal
denoiser when its penalty scale becomes large. For \(h\in\cC\) and
\(\kappa,\nu>0\), define
\begin{equation}
\label{eq:Ah}
    A_h(\kappa,\nu)
    :=
    \frac1{\sqrt\nu} \,
    \E\left[
        Z\prox_{\kappa h}\left(B^\star+\sqrt\nu Z\right)
    \right] = \E\left[
        (\prox_{\kappa h})'
        (B^\star+\sqrt\nu Z)
    \right],
\end{equation}
where \(B^\star\sim\pi\) and \(Z\sim N(0,1)\) are independent,
and the equality follows from Stein's lemma. Since a scalar proximal map
is nondecreasing and one-Lipschitz, \(A_h(\kappa,\nu)\in[0,1]\).

\begin{definition}[\(q\)-bounded width]\label{defn:q-width-bounded-b}
A penalty \(h\in\cC\) satisfies the
\emph{\(q\)-bounded width condition} if there are
\(\epsilon_0\in(0,\tau_B^+)\) and \(\delta_0>0\) such that
\begin{equation}
    \limsup_{\kappa\to\infty}
    \sup_{\nu\in[\tau_B^+ -\epsilon_0,\tau_B^+]}A_h(\kappa,\nu)
    \le q-\delta_0.
    \label{eq:q-width-condition}
\end{equation}
We write \(\cC_{\rm WB}(q;\tau_B^+)\) for the penalties satisfying
\cref{eq:q-width-condition}.
\end{definition}

The condition keeps the average slope strictly below the positive spectral
mass \(q\), uniformly in a left neighborhood of \(\tau_B^+\), as
\(\kappa\to\infty\). When \(h\) attains its minimum,
\cref{prop:q-bounded-minimizer-set} identifies the limiting slope with
\(\PP(B^\star+\sqrt\nu Z\in\operatorname{int}(M))\). In that case,
bounded width limits the probability captured by the flat minimum of the
penalty. Its role in the proof is to exclude fixed-point sequences with
vanishing denoising precision. \cref{sec: examples-penalties} explains this
argument and the geometric criterion.

With the minimizer-existence requirement in \cref{ass4} understood, the
admissible penalty class for \cref{thm:Bayes VAMP-lower-bound} is
\begin{equation}
    \cC_{\rm LB}(\mu,\pi,\sigma^2)
    :=
    \begin{cases}
        \cC,
        & p_0=0\ \text{or}\ I_+=\infty,\\[1mm]
        \cC_{\rm WB}(q;\tau_B^+),
        & p_0>0\ \text{and}\ I_+<\infty.
    \end{cases}
    \label{eq:admissible-penalty-class}
\end{equation}
Thus width is a sufficient condition needed by the present proof in one
spectral regime, and we do not claim that it is statistically necessary.
Even in that regime, every norm on \(\R\) and every strongly convex
penalty satisfies the condition; see \cref{sec: examples-penalties}.
The relation to the bounded-width condition of
\textcite[Definition 2.1]{Celentano2022barrier} for iid Gaussian designs
is discussed in \cref{sec: comparison-celentano-delta}.

\section{Main results}
\label{sec:main-results}
The main lower bound follows by adding a quadratic penalty centered at the
true signal. This perturbation makes the estimator's asymptotic risk
accessible through VAMP without increasing its distance to the signal.
We first establish the perturbed risk characterization for a general
admissible center (see \cref{thm:perturbed-fixed-point,thm:perturbed-convex-mse}), then use the true signal-centered specialization to prove
the lower bound in \cref{thm:Bayes VAMP-lower-bound}. The remaining results compare that bound across spectra in~\cref{prop:spectral-monotonicity}
and characterize when a fixed convex penalty attains it in ~\cref{prop: strict-equality}.

\subsection{Perturbed estimators and their asymptotic risk}
\label{sec:main-oracle-fixed-point}

The quadratic perturbation serves two purposes: it makes the regression
objective strongly convex, and, for each fixed center, it makes the scalar
denoising map a strict contraction in its incoming observation.
These properties allow the state-evolution calculation to be
connected to the perturbed-convex estimator. For a fixed \(\lambda>0\) and a
reference vector \(b\in\R^p\), define
\begin{equation}
\label{eq:perturbed_beta}
    \widehat\beta_{\cvx}^{h,\lambda,b}
    :=
    \arg\min_{\beta\in\mathbb R^p}
    \left\{
        \frac12\|y-X\beta\|_2^2
        +h^{(\lambda,b)}(\beta)
    \right\}, \qquad
    h^{(\lambda,b)}(\beta)
    := \sum_{j=1}^p h(\beta_j) + \frac{\lambda}{2}\|\beta-b\|_2^2.
\end{equation}
The objective is \(\lambda\)-strongly convex and coercive, so the
minimizer exists and is unique. The lower-bound argument will use the
oracle center \(b=\beta^\star\), for which we abbreviate
$\widehat\beta_{\cvx}^{h,\lambda}
    :=\widehat\beta_{\cvx}^{h,\lambda,\beta^\star},$ and
    $h^{(\lambda)}:=h^{(\lambda,\beta^\star)}.$
This center uses the unknown signal and serves only as a proof device, not an implementable estimator.

To state the general perturbed results, we require a non-singleton
effective domain and a joint empirical limit for the signal and reference
vector.

\begin{assumption}
\label{ass5}
Let $h$ satisfy \cref{ass4}, and assume that $\operatorname{dom}\,h$ contains at least two distinct points.
\end{assumption}

\begin{assumption}
\label{ass:general-reference-vector}
For each $p$, let $b=b^{(p)}\in\R^p$ be the reference vector for the $\ell_2$ perturbation, which can be either deterministic or random. Impose the additional independence assumptions $V \perp (U,\Sigma,\beta^\star,b)$ and $\varepsilon\perp (U,\Sigma,V,\beta^\star,b)$. For some deterministic law $\nu\in\mathcal P_2(\R^2)$, we have $\frac1p\sum_{j=1}^p\delta_{(\beta_j^\star,b_j)}
    \xrightarrow{W_2}_{\mathrm{a.s.}}\nu.$
We write $(B^\star,B)\sim\nu$, and take $Z\sim N(0,1)$ independently of $(B^\star,B)$.
\end{assumption}

The reference condition allows dependence between \(b\) and
\(\beta^\star\), while retaining the independence needed for the Haar
and Gaussian arguments. In particular, the oracle choice satisfies it
with \(B=B^\star\) almost surely. The domain restriction in
\cref{ass5} is needed only for the perturbed VAMP analysis, as penalties supported
at a single point are handled directly in the lower-bound proof.

The iteration in \cref{alg:perturbed-convex-vamp} associates a
state-evolution system with \cref{eq:perturbed_beta}. Before that system
can describe the optimizer's risk, its stationary equations must have
a solution.

\begin{lemma}[Perturbed convex VAMP fixed point]
\label{thm:perturbed-fixed-point}
Let \cref{ass1,ass2,ass3,ass4}, \cref{ass5}, and \cref{ass:general-reference-vector} hold. Then, for every \(\lambda>0\), the associated fixed-point equations in \cref{eq: fpt-oracle-a,eq: fpt-oracle-b,eq: fpt-oracle-c,eq: fpt-oracle-d} for the $\ell_2$-perturbed penalty $h^{(\lambda, b)}$ admit at least one solution.
\end{lemma}

The fixed-point equations are given in \cref{sec:perturbed-convex-vamp-fp}.
The proof of \cref{thm:perturbed-fixed-point}, in
\cref{sec:oracle-fp-existence}, reduces them to two variables and applies
the Poincar\'e--Miranda theorem
\parencite[Eq.~(1)]{Frankowska2018PM}. The conclusion holds for each fixed
\(\lambda>0\), including arbitrarily small positive strengths.

\begin{rem}[Relation to earlier perturbed VAMP results]
The conclusion of
\cref{thm:perturbed-fixed-point}, together with the risk characterization
below, holds for every \(\lambda>0\) and addresses a restriction in
earlier VAMP analyses.
\textcite[p.~4, Eq.~(6)]{Gerbelot2020vamp} proved the distributional
characterization for sufficiently strong quadratic regularization strength $\lambda$ and
conjectured its validity for every positive strength $\lambda > 0$ for the reference vector \(b=0\).
\textcite{Gerbelot2023glm} pursued an analytic-continuation argument to remove the restriction on $\lambda$, but as pointed out in \textcite[Remark C.1]{Li2026debiasing}, that argument does not justify a required exchange of limits. Li and Sur establish fixed-point existence
under additional structural conditions on the penalty. Here the result
covers every \(h\in\cC\) with non-singleton effective domain satisfying
the stated assumptions, every admissible reference sequence, and every
fixed \(\lambda>0\).
\end{rem}

The fixed-point existence of \cref{thm:perturbed-fixed-point} supplies the parameters for a scalar risk formula and the next theorem identifies this formula with the asymptotic MSE of the
unique perturbed optimizer $\widehat\beta_{\cvx}^{h,\lambda,b}$. This requires a separate argument linking
state evolution, convergence across iterations, and the regression
objective in \cref{eq:perturbed_beta}.

\begin{theorem}[Perturbed convex VAMP MSE]
\label{thm:perturbed-convex-mse}
Let \cref{ass1,ass2,ass3,ass4}, \cref{ass5}, and \cref{ass:general-reference-vector} hold. For each $\lambda>0$, let
$(\bar\gamma_{1,\lambda},\eta_\lambda,
      \tau_{1,\lambda},\tau_{2,\lambda})$
be any solution of fixed-point equations in \cref{eq: fpt-oracle-a,eq: fpt-oracle-b,eq: fpt-oracle-c,eq: fpt-oracle-d}, whose existence is guaranteed by \cref{thm:perturbed-fixed-point}. Then
\begin{equation}
    \frac1p
    \|\widehat\beta_{\cvx}^{h,\lambda,b}-\beta^\star\|_2^2
    \xrightarrow{\PP}
    \E\left[
        \left(
            \prox_{\alpha_\lambda h}
            \left(
                \frac{\bar\gamma_{1,\lambda}}
                     {\bar\gamma_{1,\lambda}+\lambda}B^\star
                +\frac{\lambda}
                      {\bar\gamma_{1,\lambda}+\lambda}B
                +\sqrt{\nu_\lambda}Z
            \right)-B^\star
        \right)^2
    \right].
    \label{eq:general-signal-risk}
\end{equation}
In particular, if $b=\beta^\star$, then $B=B^\star$ almost surely,
and the preceding limit reduces to
\begin{equation}
    \frac1p
    \|\widehat\beta_{\cvx}^{h,\lambda}-\beta^\star\|_2^2
    \xrightarrow{\PP}
    \E\left[
        \left(
            \prox_{\alpha_\lambda h}
            \left(
                B^\star+\sqrt{\nu_\lambda}Z
            \right)-B^\star
        \right)^2
    \right],
    \label{eq:oracle-reference-risk}
\end{equation}
where
\begin{equation}
\label{eq:alpha_and_nu}
    \alpha_\lambda
    =
    \frac{1}
         {\bar\gamma_{1,\lambda}+\lambda}
\qquad \text { and } \qquad
    \nu_\lambda
    =
    \left(
        \frac{\bar\gamma_{1,\lambda}}
             {\bar\gamma_{1,\lambda}+\lambda}
    \right)^2
    \tau_{1,\lambda}.
\end{equation}
\end{theorem}

In \cref{eq:general-signal-risk}, the scalar observation combines
\(B^\star\) and the reference \(B\) with weights summing to one, then
adds Gaussian noise of variance \(\nu_\lambda\). The proximal map uses
penalty scale \(\alpha_\lambda\). At the oracle center, the weighted
combination is exactly \(B^\star\), giving
\cref{eq:oracle-reference-risk} and permitting comparison with the
scalar posterior mean.

The proof in \cref{sec: oracle-convex-vamp-2} first establishes a state
evolution characterization at each fixed VAMP iteration, then controls the iterates as their
number grows after the dimensional limit has been taken, and finally identifies the
limit with the unique optimizer in \cref{eq:perturbed_beta}. Throughout, \(\lambda>0\) is fixed
while \(p\to\infty\).

\begin{rem}[Nonuniqueness of the fixed points]
The scalar fixed point need not be unique. However, this does not create an ambiguity in the joint-law or risk characterization of
\cref{thm:perturbed-convex-mse}.  Indeed, the objective in
\cref{eq:perturbed_beta} is \(\lambda\)-strongly convex and
therefore has a unique minimizer
\(\widehat\beta_h^{h, \lambda, b}\).
For any selected fixed point,
the proof initializes the VAMP algorithm at that solution and shows that, in the relevant iterated limit, the iterates of this VAMP algorithm converge to the unique minimizer.
Consequently,
all solutions covered by \cref{thm:perturbed-convex-mse} give the same
limiting risk; see also \textcite[Remark~C.2]{Li2026debiasing}.
This argument does not assert convergence from arbitrary initialization.
The lower-bound proof uses the particular oracle fixed point selected
in \cref{prop:perturbed-fp-var}.
\end{rem}

\subsection{A lower bound for convex estimation}
\label{sec:main-convex-barrier}

We now set \(b=\beta^\star\) and write \(r^{h,\lambda}\) for the
scalar risk on the right-hand side of \cref{eq:oracle-reference-risk}.
The reason for choosing the oracle center is because its estimator outperforms
$\widehat\beta_{\mathrm{cvx}}^{h}$ in the sense that 
\(\|\widehat\beta_{\mathrm{cvx}}^{h,\lambda}-\beta^\star\|_2^2
\leq\|\widehat\beta_{\mathrm{cvx}}^h-\beta^\star\|_2^2\),
valid for every minimizer of the unperturbed objective; see
\cref{lem:oraclemonotone}. Thus a lower bound for the perturbed risk
also bounds the original risk from below.

For each desired accuracy, the proof selects a sufficiently small
positive \(\lambda\), holds it fixed while \(p\to\infty\), and compares
\(r^{h,\lambda}\) with \(m_\pi(\tau_B^+)\). This is why the
characterization for every fixed \(\lambda>0\) is needed.  Since the unperturbed risk need not converge, we express the
conclusion using the probability-liminf.

\begin{definition}
For any sequence of real-valued random variables $\{X_p\}_{p \in \N}$, not necessarily defined on the same probability space, we denote
\[
\overset{\PP}{\liminf_{p \rightarrow \infty}} X_p = \sup \left\{t \in \R \mid \lim\sup_{p \rightarrow \infty} \PP(X_p < t) = 0\right\},
\]
which is called the probability-liminf of $\{X_p\}_{p \in \N}$. As defined, the probability-liminf can take either extended-real value but in our case  for normalized squared error, it belongs to $[0,+\infty]$. For inline equations, we also use the notation $\mathrm{p}\liminf_{p \rightarrow \infty} X_p$.
\end{definition}

\begin{theorem}[Bayes VAMP MSE lower bound for convex penalties]
\label{thm:Bayes VAMP-lower-bound}
Suppose that \cref{ass1,ass2,ass3,ass4} hold and let
\(\tau_B^+\) be defined as in
\cref{eq:tau-B-definition}. Let \(\widehat\beta_{\cvx}^{h}\) be any measurable minimizer of the convex problem \cref{eq: penalized-OLS} for \(h\in\cC_{\rm LB}(\mu,\pi,\sigma^2)\) defined in \cref{eq:admissible-penalty-class}. Then, for $m_{\pi}$ defined in \cref{eq:mmse},
\begin{equation*}
\begin{aligned}
   \mathrm{convex \ barrier } \ &:= \ \inf_{h\in\cC_{\rm LB}(\mu,\pi,\sigma^2)}
    \overset{\PP}{\liminf_{p \rightarrow \infty}} \frac1p
        \left\|
            \widehat\beta_{\cvx}^{h}-\beta^\star
        \right\|_2^2
    \\
    &\ge m_{\pi}(\tau_B^+) \ =: \ \mathrm{Bayes\text{-}VAMP \ lower \ bound}.
\end{aligned}
\end{equation*}
\end{theorem}

For each fixed admissible \(h\) and each \(\epsilon>0\), the theorem
says that the probability of normalized squared error below
\(m_\pi(\tau_B^+)-\epsilon\) tends to zero. It does not require that
error to have a limit.
The proof in \cref{sec:main_result2_proof} selects oracle convex-VAMP
fixed points whose scalar risks can be compared with the greatest
Bayes--VAMP fixed point, then transfers the comparison to the original risk through the
inequality discussed above.

\begin{rem}
The lower bound \(m_\pi(\tau_B^+)\) is defined by the scalar recursion
under our assumptions. Under the additional regularity
conditions in \cref{prop:bayes-vamp-risk}, it is also the limiting risk
of the Bayes--VAMP iterates of \cref{alg: Bayes VAMP} in the specified iterated limit.  See \cref{sec:bayes-vamp-risk-selection}.
\end{rem}

For iid Gaussian designs, the scalar expression agrees with the AMP
benchmark.
\begin{corollary}
\label{cor:gaussian-specialization}
If \(X_{ij}\stackrel{\mathrm{iid}}{\sim}N(0,1/n)\), then \(\mu\) is the corresponding
Marchenko--Pastur law. Therefore, the VAMP and AMP scalar fixed points agree
\parencite[Theorem 1]{Zhang2020identical}, and \cref{thm:Bayes VAMP-lower-bound} has the same scalar lower-bound expression as \textcite[Corollary 2.3]
{Celentano2022barrier}, restricted here to our admissible fixed separable penalties.
\end{corollary}

\textbf{Comparison with the Bayes risk.} If, in addition, \(\beta_j^\star\stackrel{\rm iid}{\sim}\pi\)
independently of the design and noise, the posterior mean
\(\E[\beta^\star\mid X,y]\) minimizes finite-dimensional expected
squared error. We call this the true Bayes MSE. \cref{thm:Bayes VAMP-lower-bound} locates the convex barrier within the broader hierarchy of
estimation risks:
$
\text{true Bayes MSE}
\ \le\ 
\text{Bayes-VAMP MSE}
\ \le\ 
\text{convex barrier}.
$
The second inequality follows from~\cref{thm:Bayes VAMP-lower-bound} and, in \cref{sec:main-strict-gap}, we will present an exact characterization of when the inequality is strict. We recall that the true Bayes estimator $\hat\beta_{\mathrm{Bayes}}(X, y) := \E[\beta^\star | X, y]$ is not generally computable in polynomial time for high-dimensional problems, and while the true Bayes risk is not understood in full generality for all RRI matrices, it is known for some specific spectral regimes \parencite{LiFanSenWu2024RotInv}.

\cref{fig:three-curve-gaps} illustrates the distinction between the Bayes
risk, the scalar Bayes--VAMP benchmark, and optimal convex risk. Its
posterior curve is a replica prediction and its convex curve a numerical
approximation, so the apparent gaps are not established by this figure.
Rigorous Bayes risk results for certain RRI regimes are given in
\textcite{LiFanSenWu2024RotInv}. However, we note that the gaps between the Bayes--VAMP MSE and the true Bayes MSE are based on predictions from the replica method and have not been rigorously proved.

We visualize the three MSE curves in \cref{fig:three-curve-gaps} for three aspect-ratio regimes: in the left plot, the convex barrier curve is strictly above the Bayes VAMP curve, which equals the Bayes curve; in the right plot, both the convex barrier curve and the Bayes VAMP curve are strictly above the true Bayes curve; and in the center plot, both gaps are visible. 

\begin{figure}[t]
    \centering
    \safeincludegraphics[width=\textwidth]{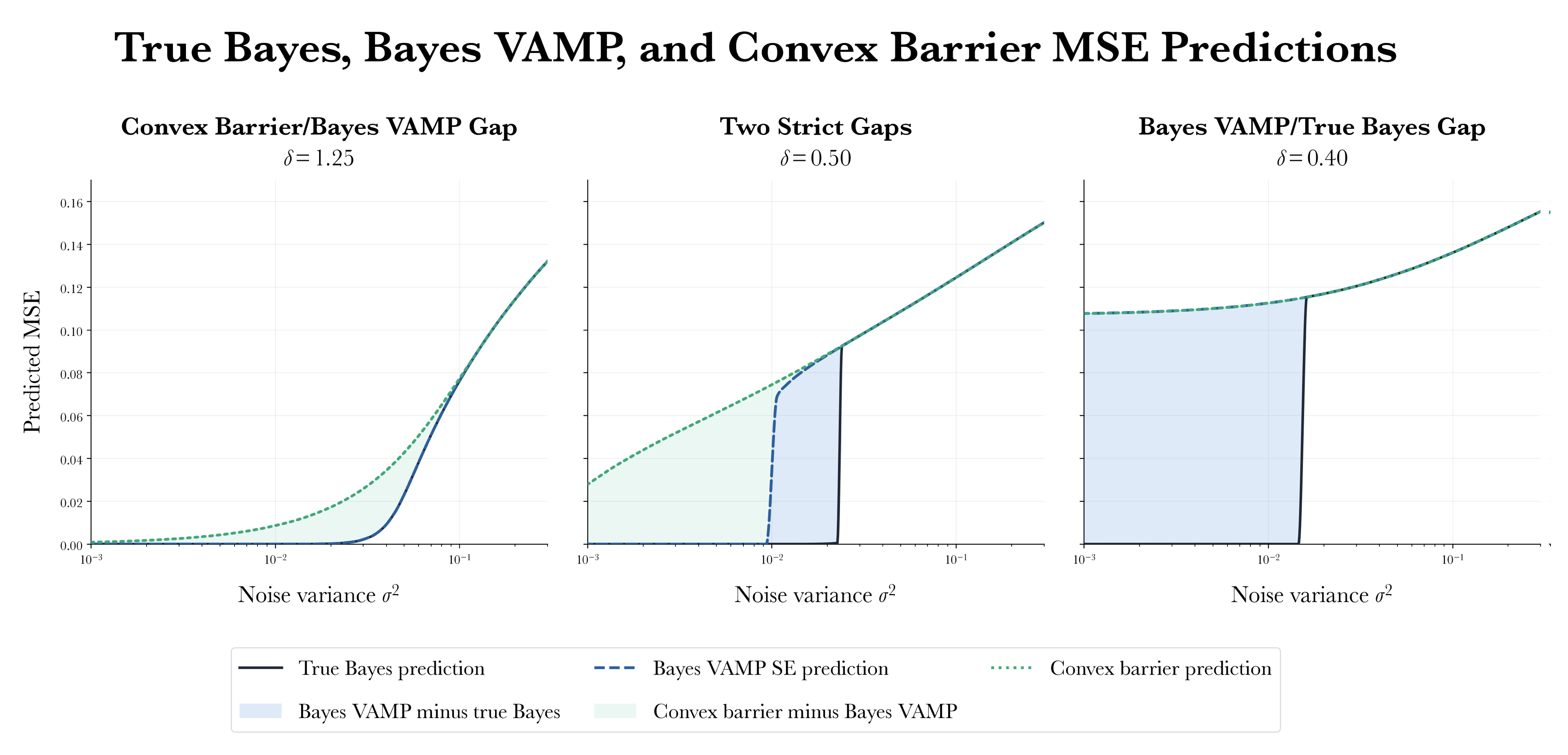}
    \caption{
    We assume a sparse Rademacher prior with $\mathbb P(\beta_j = 0) = 0.8$ and a deformed Marchenko--Pastur law generated by $X = DG$, where $G$ is iid Gaussian with entries $G_{ij} \sim N(0, 1/p)$ and $D$ is diagonal with $\mathbb P(D^2_{ii} = (1 - 0.4)/\delta) = \mathbb P(D^2_{ii} = (1 + 0.4)/\delta) = \frac{1}{2}$. 
    We plot a replica prediction of the true Bayes MSE \parencite{LiFanSenWu2024RotInv}, the scalar Bayes--VAMP benchmark $m_{\pi}(\tau_B^+)$, and a numerical approximation to optimal convex risk. 
    Over three aspect-ratio regimes $\delta \in \{1.25, 0.5, 0.4\}$, we see that the curves demonstrate different behaviors: sometimes having strict gaps and sometimes coinciding.
    }
    \label{fig:three-curve-gaps}
\end{figure}

\begin{rem}[Bayes--VAMP optimality]
We mention that Bayes--VAMP is conjectured to be optimal among all
polynomial-time estimators (not just amongst convex regularizers) in our linear regression setting \parencite{Zhang2026optimality}. Specifically, it is believed that, in the high-dimensional limit, no polynomial-time estimator can have an asymptotic MSE smaller than that of the Bayes--VAMP algorithm.
Under that conjecture, its risk serves more generally as a computational barrier. 
\end{rem}

\subsection{Spectral ordering of the lower bound}
\label{sec:main-spectral-order}

For a fixed prior $\pi$ and noise variance $\sigma^2$, the design spectrum enters the
Bayes--VAMP lower bound only through the spectrum-stage map \(L_\mu\) defined in \cref{eq:spectrum-mod}.
We therefore first compare these maps, then pass the comparison through
the scalar recursion to its greatest fixed point.

\textbf{Stieltjes transform order.}
The useful spectral quantity is the Stieltjes transform: for \(c>0\),
\begin{equation}
    G_\mu(c)
    :=
    \E_\mu\left[\frac1{S+c}\right].
    \label{eq:stieltjes-transform}
\end{equation}
Indeed, \(G_\mu(c)\) is the standard Stieltjes transform
\(\int(s-z)^{-1}\mu(ds)\) evaluated at the negative real argument \(z=-c\).
Setting the noise-to-extrinsic-uncertainty ratio as \(c=\sigma^2/\omega\) in \cref{eq:spectrum-mod} gives
\begin{equation}
    L_\mu(\omega)
    =
    \frac{\sigma^2G_\mu(c)}{1-cG_\mu(c)}.
    \label{eq:stieltjes-L-preview}
\end{equation}
The denominator is positive because
\(1-cG_\mu(c)=\E_\mu[S/(S+c)]>0\) by \cref{ass1}.
For fixed \(c\), the function \(u\mapsto\sigma^2u/(1-cu)\) is strictly
increasing on \(0<u<1/c\). Comparing \(G_\mu(c)\) at every \(c>0\)
therefore compares \(L_\mu(\omega)\) at every \(\omega>0\).

The kernel \((s+c)^{-1}\) gives greater weight to small eigenvalues,
with \(c=\sigma^2/\omega\) setting the scale of this comparison.
When \(\omega\downarrow0\), an expansion in inverse powers of \(c\)
expresses the leading behavior of \(L_\mu\) in terms of the first two moments of $\mu$. When \(\omega\uparrow\infty\), the behavior depends
on the inverse moment and on the density of
\(\mu\) near or at zero. The precise regimes are given in
\cref{prop:moment-expansion,prop:hard-edge}.

Following \textcite{Alenazi2026Stieltjes}, we now define the transform ordering on the spectral laws. For probability laws \(\mu_1,\mu_2\) on
\([0,\infty)\), we say that $\mu_2$ dominates $\mu_1$ in Stieltjes transform order, denoted $\mu_2\succeq_{\mathrm{St}}\mu_1,$ if
$G_{\mu_2}(c)\ge G_{\mu_1}(c)$ for every $c>0.$
We emphasize this ordering compares the two transforms at every scale: a comparison of the mass in a single spectral interval does not by itself establish it.

\textbf{Convex order.} A stronger but more familiar comparison is convex order. For laws with finite
first moments, we say that \(\mu_2\) dominates \(\mu_1\) in convex order, and write \(\mu_2\succeq_{\mathrm{cx}}\mu_1\), if
$\E_{\mu_2}[\varphi(S)]
\ge
\mathbb E_{\mu_1}[\varphi(S)]
$
for every integrable convex function
\(\varphi:[0,\infty)\to\mathbb R\). This is the usual
mean-preserving-spread comparison. Since \(s\mapsto(s+c)^{-1}\) is convex
for every \(c>0\), convex order implies Stieltjes order; for a reference, see \textcite{convex_order}.

The next proposition considers the spectral law ordering and the consequences for the Bayes VAMP lower bound found in~\cref{thm:Bayes VAMP-lower-bound}.

\begin{proposition}[Spectral monotonicity and comparison]
\label{prop:spectral-monotonicity}
For the spectrum-stage map defined in \cref{eq:spectrum-mod}, the following hold:
\begin{enumerate}[label=\textup{(\roman*)},leftmargin=2.2em]
\item For every spectral law \(\mu\), the map \(\omega\mapsto L_\mu(\omega)\)
is nondecreasing.

\item For two spectral laws $\mu_1$ and $\mu_2$, we have 
$\mu_2\succeq_{\mathrm{St}}\mu_1
    \Longleftrightarrow
    L_{\mu_2}(\omega)\ge L_{\mu_1}(\omega)$
    for every $\omega>0.$

\item Consequently, we also have
$\mu_2\succeq_{\mathrm{cx}}\mu_1 \implies L_{\mu_2}(\omega)\ge L_{\mu_1}(\omega)$ for every $\omega>0.$

\item Fix a prior \(\pi\) and noise level \(\sigma^2\). Initialize Bayes VAMP
with $\tau_{0,+}^B(\mu)
    :=
    L_\mu(\VAR(B^\star))$. Then, if \(\mu_2\succeq_{\mathrm{St}}\mu_1\), we have that
$\tau_B^+(\mu_2)\ge\tau_B^+(\mu_1),$ and
    $m_\pi(\tau_B^+(\mu_2))
    \ge
    m_\pi(\tau_B^+(\mu_1)).$
\end{enumerate}
\end{proposition}
The proof is in \cref{sec:spectral_proofs}. The proposition separates
two monotonicity statements: the spectrum-stage output is nondecreasing
in the input uncertainty, while Stieltjes order compares the entire
maps pointwise. The latter comparison orders the
greatest fixed points and hence the lower bounds on convex risk.
Convex order provides a sufficient condition in terms of a
mean-preserving spread. These conclusions compare the scalar lower
bounds, while \cref{cor:spectral-attainment-risk} gives conditions under
which they also compare the infimal convex risks.

As a simple corollary: at fixed mean, the extremal laws among those supported in a common
bounded interval are a point mass at the mean and a two-point law at
the endpoints.

\begin{corollary}
\label{cor:flat-spectrum-best}
Let \(\bar s := \E_\mu[S]\). Fix \(0\le a<\bar s<b<\infty\) and let \(\mu\) be supported on \([a,b]\). Define
\[
    \mu_{\bar s}:=\delta_{\bar s},
    \qquad \text{ and } \qquad 
    \mu_{a,b,\bar s}
    :=
    \frac{b-\bar s}{b-a}\delta_a
    +
    \frac{\bar s-a}{b-a}\delta_b.
\]
Then, for every \(\omega>0\), we have $L_{\mu_{\bar s}}(\omega)
    \le
    L_\mu(\omega)
    \le
    L_{\mu_{a,b,\bar s}}(\omega).$
Hence, the same ordering holds for the greatest Bayes-VAMP fixed points and their
associated lower bounds.
\end{corollary}

The proof is in \cref{sec:spectral_proofs}. The flat law
\(\mu_{\bar s}\) concentrates the limiting spectrum at its mean and
minimizes the lower bound. The endpoint law
\(\mu_{a,b,\bar s}\) is the largest mean-preserving spread allowed by
the interval and maximizes the lower bound.
Here the common constraint is containment in \([a,b]\), not equality
of the supports. These are comparisons of spectral laws: additional
constraints on a design may exclude an extremizer. In particular,
when \(\delta<1\), rank deficiency rules out a positive full flat
spectrum. The corresponding comparison with the mass of the positive
spectrum fixed is given in \cref{sec:tightframe}.

\subsection{Attainment, strict gaps, and optimal penalties}
\label{sec:main-strict-gap}

The lower bound leaves two obvious questions, whether an admissible convex
penalty attains \(m_\pi(\tau_B^+)\) and whether a failure of attainment
leaves a positive gap after optimizing over penalties. Both are
determined by log-concavity of the prior after Gaussian smoothing at
the effective variance \(\tau_B^+\).

\subsubsection{The attainment criterion}
It turns out that the scalar comparison that produces the lower bound is tight exactly when a
proximal map agrees with the posterior mean, in which case the associated convex--VAMP is an instance of Bayes--VAMP. A scalar posterior mean
is the proximal map of a convex penalty precisely when the
Gaussian-smoothed prior is log-concave; see
\cref{lem:posterior-prox-characterization}.
We first construct the candidate penalty by convex conjugation, then
state the attainment result.

Let \(\tau_B^+\) be defined by \cref{eq:tau-B-definition}, and set
$\gamma_B:={\sigma^2}/{\tau_B^+}.$
To describe an explicit penalty attaining the Bayes--VAMP bound,
define the Gaussian-tilted log-partition function
\begin{equation}
    \Lambda_B(s)
    :=
    \log\E\left[
        \exp\left\{
            sB^\star-\frac{(B^\star)^2}{2\tau_B^+}
        \right\}
    \right],
    \qquad s\in\R,
    \label{eq:main-Bayes-log-partition}
\end{equation}
where \(B^\star\sim\pi\). Write
$\Lambda_B^\star(x)
    :=\sup_{s\in\R}\{sx-\Lambda_B(s)\}$
for its convex conjugate. The derivative of \(\Lambda_B\) parametrizes
the posterior mean and conjugation inverts this relation. Subtracting
the quadratic term gives the candidate penalty
\begin{equation}
    h_B^\star(x)
    :=
    \sigma^2\Lambda_B^\star(x)
    -\frac{\sigma^2}{2\tau_B^+}x^2,
    \qquad x\in\R.
    \label{eq:main-explicit-Bayes-penalty}
\end{equation}
Although \(\Lambda_B^\star\) is convex, the subtraction in
\cref{eq:main-explicit-Bayes-penalty} need not preserve convexity.
The next theorem identifies log-concavity of the smoothed prior as
the exact condition for convex attainment and gives a uniform
positive gap when that condition fails.

\begin{theorem}
\label{prop: strict-equality}
Suppose that \cref{ass1,ass2,ass3,ass4} hold.
\begin{enumerate}
    \item[\textup{(i)}] \textbf{Attainment.}
    If \(\pi*N(0,\tau_B^+)\) is log-concave, then
    \(h_B^\star\) is proper, lsc, nonconstant,
    and convex, and $ h_B^\star\in\cC_{\rm LB}(\mu,\pi,\sigma^2).$
    Every measurable selection
    \(\widehat\beta_{\cvx}^{h_B^\star}\) from the corresponding
    minimizer set satisfies
    \begin{equation}
        \frac1p
        \left\|
            \widehat\beta_{\cvx}^{h_B^\star}-\beta^\star
        \right\|_2^2
        \xrightarrow{\PP}
        m_\pi(\tau_B^+).
        \label{eq:main-explicit-penalty-attainment}
    \end{equation}
    Moreover, its scalar proximal map equals the posterior mean: for \(Z\sim N(0,1)\) independent of \(B^\star\sim\pi\),
    \begin{equation}
        \prox_{h_B^\star/\gamma_B}(y)
        =
        \E\left[B^\star \big \lvert B^\star+\sqrt{\tau_B^+}Z=y\right],
        \qquad y\in\R.
        \label{eq:main-explicit-penalty-prox}
    \end{equation}

    \item[\textup{(ii)}] \textbf{Uniform strict gap.}
    If \(\pi*N(0,\tau_B^+)\) is not log-concave, then there exists
    \(\Delta_{\rm cvx}>0\), depending only on
    \((\mu,\pi,\sigma^2)\), such that every
    \(h\in\cC_{\rm LB}(\mu,\pi,\sigma^2)\) and every corresponding
    measurable minimizer satisfy
    \begin{equation}
        \overset{\PP}{\liminf_{p\to\infty}}
        \frac1p
        \left\|
            \widehat\beta_{\cvx}^{h}-\beta^\star
        \right\|_2^2
        \geq
        m_\pi(\tau_B^+)+\Delta_{\rm cvx}.
        \label{eq:main-uniform-strict-convex-gap}
    \end{equation}
\end{enumerate}
Consequently,
\begin{equation}
    \begin{gathered}
        \inf_{h\in\cC_{\rm LB}(\mu,\pi,\sigma^2)}
        \overset{\PP}{\liminf_{p\to\infty}}
        \frac1p
        \left\|
            \widehat\beta_{\cvx}^{h}-\beta^\star
        \right\|_2^2
        =m_\pi(\tau_B^+) \quad 
        \Longleftrightarrow
        \quad
        \pi*N(0,\tau_B^+)\text{ is log-concave}.
    \end{gathered}
    \label{eq: strict-equality}
\end{equation}
\end{theorem}

The proof is in \cref{sec: proof-strict-equality}. Part (i) concerns
the fixed penalty \(h_B^\star\) itself, including when it
is not strongly convex. In particular, it is not only a limit of risks for
the quadratically-perturbed penalties considered before. In part (ii), the same
\(\Delta_{\rm cvx}\) works for every admissible \(h\). Thus failure
of log-concavity prevents even the infimum over fixed penalties
from reaching the lower bound.

\begin{rem}
\label{rem:scalar-construction-and-attainment}
The conjugate construction underlying \cref{eq:main-explicit-Bayes-penalty,eq:main-explicit-penalty-prox} has precedents in the literature.
In their study of optimal convex losses for high-dimensional
\(M\)-estimation with Gaussian designs,
\textcite{Bean2013OptimalMEstimation} construct an optimal loss by
conjugating a quadratic plus the logarithm of a Gaussian-smoothed
density. 
Their loss acts on regression residuals, whereas
\(h_B^\star\) penalizes coefficients. The work in \textcite[Proposition~3 and Remark~4]{GribonvalNikolova2021} provides the representation of Gaussian posterior means as proximal maps and they also characterize convexity
of the associated penalty by log-concavity of the smoothed marginal.
For us, the effective variance \(\tau_B^+\) calibrates this construction
to \((\mu,\pi,\sigma^2)\). The attainment assertion in
\cref{prop: strict-equality} establishes admissibility and that the
asymptotic risk of every measurable minimizer is \(m_\pi(\tau_B^+)\). This conclusion concerns the fixed penalty \(h_B^\star\)
itself, including cases in which it is not strongly convex.
\end{rem}

\subsubsection{Consequences for spectral comparison and selection}

Moving upward in Stieltjes order cannot decrease the greatest fixed point.
If the smoothed
prior is already log-concave at the smaller variance, further
Gaussian smoothing preserves log-concavity, so both models lie
in the attainment regime. Their scalar comparison then becomes a
comparison of infimal convex risks, even though their admissible
penalty classes may differ.

\begin{corollary}[Spectral order, attainment, and optimal convex risk]
\label{cor:spectral-attainment-risk}
Consider two model sequences satisfying \cref{ass1,ass2,ass3}, with
the same prior \(\pi\) and noise variance \(\sigma^2>0\), and with
limiting spectral laws \(\mu_1,\mu_2\). For \(i\in\{1,2\}\), write
$\tau_i:=\tau_B^+(\mu_i,\pi,\sigma^2),$ and $\cC_i:=\cC_{\rm LB}(\mu_i,\pi,\sigma^2),$
and restrict to fixed scalar penalties satisfying \cref{ass4} in
the corresponding model. For any measurable choices of their
minimizers, define $\mathcal R_i
 :=\inf_{h\in\cC_i}
    \overset{\PP}{\liminf_{p\to\infty}}
    \frac1p
    \|\widehat\beta_{\cvx,i}^{h}-\beta_i^\star\|_2^2.$
Suppose that \(\mu_2\succeq_{\rm St}\mu_1\) and that
\(\pi*N(0,\tau_1)\) is log-concave. Then
\(\tau_2\geq\tau_1\), the law \(\pi*N(0,\tau_2)\) is log-concave,
and
\begin{equation}
 \mathcal R_1=m_\pi(\tau_1)
 \leq m_\pi(\tau_2)=\mathcal R_2.
 \label{eq:spectral-attainment-risk-order}
\end{equation}
Each infimum is attained by the penalty in
\cref{eq:main-explicit-Bayes-penalty}, evaluated at its own
effective variance \(\tau_i\). The classes \(\cC_1\) and
\(\cC_2\) need not coincide.

If \(V_\pi:=\operatorname{Var}(B^\star)>0\) and
$G_{\mu_2}(c_1)>G_{\mu_1}(c_1),$ where $c_1:={\sigma^2}/{\cE_\pi(\tau_1)},$
then \(\tau_2>\tau_1\) and \(\mathcal R_2>\mathcal R_1\).
In particular, the strict transform inequality at every \(c>0\)
suffices for this strict conclusion.
\end{corollary}


\subsubsection{Examples of attaining penalties}

The explicit construction of the penalty in \cref{prop: strict-equality} recovers ridge regularization for a Gaussian prior,
but other priors produce penalties whose shape depends on the
spectrum through \(\tau_B^+\). The examples below show how this
dependence enters and how the log-concavity criterion determines
whether the candidate is convex. Throughout, we impose the
hypotheses of \cref{prop: strict-equality} and identify penalties
up to additive constants.

\paragraph{Gaussian prior: ridge regularization}
For \(\pi=N(m_0,v)\), with \(v>0\), evaluating
\cref{eq:main-explicit-Bayes-penalty} gives
$h_B^\star(x) = \frac{\sigma^2}{2v}(x-m_0)^2$ up to an additive constant.
Thus, the construction recovers ridge regularization centered
at the prior mean. The Gaussian-smoothed prior is log-concave
for every \(\tau_B^+ > 0\), so this penalty attains the Bayes--VAMP
lower bound. In particular, although the spectrum affects
\(\tau_B^+\) and the attained risk, it does not affect the
optimal penalty in this example.

\paragraph{Binary prior: an entropy penalty with a
spectrum-dependent quadratic term}
Let \(\pi=(\delta_{-a}+\delta_a)/2\), where \(a>0\), and define
$ I(u)
    =
    \frac{1}{2}(1+u)\log(1+u)
    +
    \frac{1}{2}(1-u)\log(1-u),$ for $|u|\leq1,$
with \(0\log0=0\).
Here
\(\Lambda_B(s)=-a^2/(2\tau_B^+)+\log\cosh(as)\), so
\cref{eq:main-explicit-Bayes-penalty} yields
\begin{equation}
    h_B^\star(x)-h_B^\star(0)
    =
    \begin{cases}
        \displaystyle
        \sigma^2\left(
            I(x/a)-\frac{1}{2}(x^2/\tau_B^+)
        \right),
        & |x|\leq a,\\[5pt]
        +\infty,
        & |x|>a.
    \end{cases}
    \label{eq:binary-explicit-optimal-penalty}
\end{equation}
For \(|x|<a\), we have
$(h_B^\star)''(x)
    =
    \sigma^2(
        {1}/{(a^2-x^2)}-{1}/{\tau_B^+}).$
Consequently, this candidate is convex exactly when
\(\tau_B^+\geq a^2\), which is also the condition for log-concavity
of the Gaussian-smoothed binary prior. In this regime,
\cref{prop: strict-equality} gives attainment, and $ \prox_{h_B^\star/\gamma_B}(y)
    =
    a\tanh(ay/\tau_B^+).$
For \(\tau_B^+<a^2\), the displayed candidate is nonconvex,
and \cref{prop: strict-equality} instead gives a uniform strict gap
over the admissible convex penalties.

This example makes the spectral dependence particularly
transparent: at fixed prior and noise variance, the spectrum
changes the quadratic term in
\cref{eq:binary-explicit-optimal-penalty} through $\tau_B^+$, while the entropy
term remains unchanged. Thus the spectrum can change the
shape of the optimal penalty, rather than merely its overall
scale.

At the boundary \(1+u=a^2\), the expansion
$h_B^\star(x)-h_B^\star(0)
    =
    \frac{\sigma^2}{12a^4}x^4+O(x^6)$
also shows that attainment does not require strong convexity.
The penalty is finite at \(\pm a\), although its slopes
diverge at these endpoints.

\paragraph{Sparse three-point prior: optimal estimation without
thresholding}
Consider for $0<\varepsilon<1$ and $a>0$ the prior
$\pi
    =
    (1-\varepsilon)\delta_0
    +\frac{\varepsilon}{2}
        (\delta_{-a}+\delta_a).$
Despite the atom at zero, the optimal penalty that matches the
posterior mean has a smooth interior and does not produce a
thresholding interval. Write \(\tau=\tau_B^+\)
and set
$ k_\tau
    =
    \frac{1-\varepsilon}{\varepsilon}
    \exp(\frac{a^2}{2\tau}).$
The log-partition function and posterior mean are
\[
\begin{aligned}
    \Lambda_B(s)
    &=
    \log\left(
        1-\varepsilon
        +\varepsilon e^{-a^2/(2\tau)}\cosh(as)
    \right),\quad \text{ and } \quad f_\tau(y)
    =
    \frac{a\sinh(ay/\tau)}
           {k_\tau+\cosh(ay/\tau)}.
\end{aligned}
\]
The posterior mean is strictly increasing, with range
\((-a,a)\), and its inverse can be written in elementary
functions. 

For \(k>0\), \(|z|<1\), let
\(
    t_k(z)
    =
    \log(({kz+\sqrt{1+(k^2-1)z^2}}\big)/\big({1-z})),
\)
so that $f_\tau^{-1}(x)
    =
    \frac{\tau}{a}\,t_{k_\tau}(x/a),$ for $|x|<a$ and define
\(
    J_k(z)
    =
    z\,t_k(z)
    -
    \log(
        ({k+\cosh(t_k(z))})/({k+1})
    ),
\)
and extend it continuously to the endpoints by
\(J_k(-1)=J_k(1)=\log(2(k+1))\).
Evaluating the conjugate in
\cref{eq:main-explicit-Bayes-penalty} gives
\begin{equation}
    h_B^\star(x)-h_B^\star(0)
    =
    \begin{cases}
        \displaystyle
        \sigma^2\Big(
            J_{k_\tau}(x/a)-\frac{x^2}{2\tau}
        \Big),
        & |x|\leq a,\\[5pt]
        +\infty,
        & |x|>a.
    \end{cases}
    \label{eq:sparse-three-point-optimal-penalty}
\end{equation}
In particular, on \((-a,a)\), we have $(h_B^\star)'(x)
    =
    \frac{\sigma^2}{a}\,t_{k_\tau}(x/a)
    -\frac{\sigma^2x}{\tau}.$
The penalty has finite values at \(\pm a\), while its slopes
diverge at these endpoints.

The condition under which this candidate is convex is also
explicit. Differentiating the posterior mean gives
\(
    f_\tau'(y)
    =
    {a^2}
    ({1+k_\tau\cosh(ay/\tau)})/(\tau
         {(k_\tau+\cosh(ay/\tau))^2}).
\)
Maximizing over \(\cosh(ay/\tau)\geq1\), we obtain
\begin{equation}
    \sup_{y\in\R}f_\tau'(y)
    =
    \frac{a^2}{\tau}
    \begin{cases}
        \displaystyle\frac{1}{1+k_\tau},
        & k_\tau\leq2,\\[6pt]
        \displaystyle\frac{k_\tau^2}{4(k_\tau^2-1)},
        & k_\tau>2.
    \end{cases}
    \label{eq:sparse-three-point-convexity}
\end{equation}
By \cref{lem:posterior-prox-characterization,eq:log-concavity-prox-derivative}, the smoothed prior
is log-concave exactly when the right-hand side of
\cref{eq:sparse-three-point-convexity} is at most one.
In this regime, the penalty in
\cref{eq:sparse-three-point-optimal-penalty} is convex and
attains the Bayes--VAMP lower bound by
\cref{prop: strict-equality}. Otherwise, the candidate is
nonconvex and the proposition gives a uniform strict gap
over the admissible convex penalties.

This example distinguishes sparsity of the prior from
thresholding of the estimator. Although the prior assigns
positive probability to zero, its posterior mean satisfies
\(f_\tau(y)=0\) only when \(y=0\). Thus, the attaining
proximal map has no interval of observations mapped to zero.
For scalar squared-error estimation, retaining posterior
uncertainty about whether a coefficient is zero leads to
smooth shrinkage rather than a hard decision about its
support.

The example also exhibits a richer dependence on the spectrum
than the binary prior. At fixed \((\varepsilon,a,\sigma^2)\),
changing the spectral law $\mu$ changes \(\tau_B^+\), and hence both
the quadratic term and the parameter \(k_\tau\) in
\cref{eq:sparse-three-point-optimal-penalty}.
The spectrum therefore changes the nonquadratic part of the
penalty as well. In the limit \(\varepsilon\uparrow1\), we have 
\(k_\tau\to0\) and \(J_{k_\tau}\to I\), recovering the binary
penalty in \cref{eq:binary-explicit-optimal-penalty}.

\subsubsection{A smoothing threshold for attainment}

The preceding examples test log-concavity at \(\tau_B^+\).
A complementary formulation separates the amount of smoothing
required by the prior from the effective variance supplied by
the model. The next corollary defines this prior-dependent
threshold and gives an equivalent test using a single evaluation
of the Stieltjes transform, without solving the fixed-point
equations.

\begin{corollary}\label{cor:strict-gap-criterion}
Define
\(
  \tau_{\rm lc}(\pi)
  :=
  \inf\left\{
    \tau\ge0:
    \pi*N(0,\tau)\text{ is log-concave}
  \right\},
\)
with the convention $\inf\varnothing=+\infty$ and suppose that \(\tau_{\rm lc}(\pi)\in(0,\infty)\). Let
\(\tau_B^+\) be defined in \cref{eq:tau-B-definition}. Then equality
\cref{eq: strict-equality} holds if and only if
\(
  \tau_B^+\ge\tau_{\rm lc}(\pi).
\)
Equivalently, define
\(
  c_{\rm lc}
  :=
 {\sigma^2}/{
    \mathcal E_\pi(\tau_{\rm lc}(\pi))
  }.
\)
Then equality \cref{eq: strict-equality} holds if and only if
\(
  c_{\rm lc}G_\mu(c_{\rm lc})
  \ge
  1-{
    m_\pi(\tau_{\rm lc}(\pi))
  }/{
    \tau_{\rm lc}(\pi)
  }.
\)
\end{corollary}
The proof is in \cref{sec:strict-gap-criterion}. The first criterion
asks whether the greatest scalar fixed point reaches the smoothing
threshold of the prior. The second evaluates the spectrum-stage
map at the corresponding prior-stage uncertainty. The threshold
\(\tau_{\rm lc}(\pi)\) depends only on the prior, while the attainment
test also depends on the noise variance and spectrum.

An increase in the lower bound alone does not establish an
increase in the infimal convex risk outside the attainment
regime: the gap between the two may also change.
When equality holds, the best admissible convex risk equals
the scalar Bayes--VAMP benchmark. This does not by itself
identify that benchmark with the true Bayes risk.

If the prior is log-concave, convolution preserves log-concavity,
so attainment follows directly from \cref{prop: strict-equality}
for every spectral law covered by its hypotheses. 
The positive,
finite threshold condition in \cref{cor:strict-gap-criterion}
is needed only for its threshold formulation.
For a binary prior, the threshold is positive and explicit.
The following specialization of the earlier binary example
shows how a flat spectrum makes the test immediate.

\textbf{Example: the Rademacher prior.}
Let $\beta^\star$ have iid Rademacher entries, and let $\mu = \delta_{s}$ be the flat spectrum for some $s > 0$. The density of $Y_\tau=B^\star+\sqrt\tau Z$ satisfies
\(
    f_\tau(y)\propto \exp(-{y^2}/({2\tau})) \cosh({y}/{\tau})
\)
and we therefore have
\(
    \frac{\mathrm d^2}{\mathrm d y^2}\log f_\tau(y) = \operatorname{sech}^2(y/\tau)/{\tau^2} - 1/\tau.
\)
Because \(\operatorname{sech}^2(y/\tau)\le1\), with equality at \(y=0\),
the density is log-concave exactly when \(\tau\ge1\). Thus
\(\tau_{\mathrm{lc}}(\pi)=1\). For the flat spectrum, \(G_\mu(c)=(s+c)^{-1}\) and
\(L_\mu(\omega)=\sigma^2/s\) for every \(\omega>0\).
Hence \(\tau_B^+=\sigma^2/s\), and the necessary and sufficient
condition for attainment is \(s\leq\sigma^2\), equivalently
\(\tau_B^+\geq1\). Substitution in the transform criterion of
\cref{cor:strict-gap-criterion} gives the same condition,
\(\sigma^2G_\mu(\sigma^2)\geq1/2\).

\section{Design spectra: examples and extremal regimes}
\label{sec:spectral-consequences}

We now apply the spectral ordering in \cref{prop:spectral-monotonicity}
to concrete design families. The first examples compare allocations of
measurement strength at fixed mean, with the positive spectral mass
also fixed when rank constraints matter. Gaussian products then show
how depth and hidden widths affect the lower bound. The final example
shows how concentration of measurement strength can make the bound
approach the risk of ignoring the data. Throughout these comparisons,
the prior and noise variance are held fixed. Numerical illustrations
appear in \cref{sec:simulations}.

\textbf{Measurement strength and the spectral map.}
Let \(s_{1,p},\ldots,s_{p,p}\) be the eigenvalues of \(X_p^T X_p\).
For a corresponding unit eigenvector \(v_{j,p}\),
\(\|X_pv_{j,p}\|_2^2=s_{j,p}\): the eigenvalue measures the response
to a unit perturbation along that direction. An atom at zero in the
limiting law can arise from exact null directions or from positive
eigenvalues that vanish asymptotically.

By \cref{ass1}, the empirical law
\(\mu_p=p^{-1}\sum_{j=1}^p\delta_{s_{j,p}}\) converges almost surely
to \(\mu\) in \(W_2\). For a nonzero design and \(\omega>0\),
its empirical spectrum-stage map is
\[
    L_{\mu_p}(\omega)
    =
    \Big(\dfrac1p\sum_{j=1}^p
        \dfrac{\sigma^2\omega}{\sigma^2+\omega s_{j,p}}\Big) \Big(\dfrac1p\sum_{j=1}^p
        \dfrac{\omega s_{j,p}}{\sigma^2+\omega s_{j,p}}\Big)^{-1}.
\]
For each fixed \(\omega>0\), this expression converges almost surely
to \(L_\mu(\omega)\) and the limiting denominator is positive by
\cref{ass1}. Thus, the empirical spectrum gives a direct way to
evaluate the scalar map. The risk conclusion in
\cref{thm:Bayes VAMP-lower-bound} concerns the high-dimensional
limit, rather than a finite-sample guarantee from this substitution.

In the explicit constructions below, to construct a right-rotationally invariant design with
the desired spectrum, we choose the
specified singular values and an independent Haar right singular
factor, as required by \cref{ass1}. The left singular vectors may
be fixed.

\subsection{Uniform measurement strength: orthogonal and tight-frame designs}
\label{sec:tightframe}

For the flat law \(\mu=\delta_{\bar s}\), where
\(\bar s=\E_\mu[S]>0\), the spectrum stage has no dependence on its
incoming uncertainty:
\(L_\mu(\omega)=\sigma^2/\bar s\) for every \(\omega\geq0\).
Hence \(\tau_B^+=\sigma^2/\bar s\), and
\cref{thm:Bayes VAMP-lower-bound} gives, for every admissible \(h\),
\(\mathop{\mathbb P\text{-}\liminf}_{p\to\infty}
p^{-1}\|\widehat\beta_{\mathrm{cvx}}^h-\beta^\star\|_2^2
\geq m_\pi(\sigma^2/\bar s)\). By \cref{cor:flat-spectrum-best}, this is the smallest spectral
lower bound at mean \(\bar s\). For dimensions with \(n\geq p\),
the flat law is realized by \(X_p^T X_p=\bar s I_p\), so the
columns are orthogonal and have equal norm.

When the positive spectral mass \(q=\PP_\mu(S>0)\) is less than
one, uniform allocation on the positive part instead gives
\(\mu_{q,\mathrm{flat}}=(1-q)\delta_0+q\delta_{\bar s/q}\).
Substitution yields
$
L_{\mu_{q,\mathrm{flat}}}(\omega)
=
{\sigma^2}/{\bar s}
+({1-q})\omega/{q}.
$
The first term is the flat-spectrum noise contribution. The
second carries forward incoming uncertainty through the spectral
mass at zero and disappears when \(q=1\).

This law minimizes the lower bound when both \(q\) and \(\bar s\)
are fixed. Indeed, any other such law has the form
\(\mu=(1-q)\delta_0+q\nu_+\), with
\(\E_{\nu_+}[S]=\bar s/q\). The point mass at that mean is
minimal in convex order, so
\(\mu\succeq_{\mathrm{cx}}\mu_{q,\mathrm{flat}}\).
It follows from \cref{prop:spectral-monotonicity} that
\(L_\mu(\omega)\geq L_{\mu_{q,\mathrm{flat}}}(\omega)\)
for every \(\omega>0\), with the corresponding ordering of
greatest fixed points and lower bounds. Fixing \(q\) separates
the effect of zero spectral mass from that of unequal positive
measurement strengths.

A projection construction makes this separation concrete:
take \(X_p^T X_p=(\bar s/q)P_p\), where \(P_p\) is a
Haar-oriented projection of rank \(r_p\), with \(r_p/p\to q\).
The dimensions must allow \(r_p\leq\min\{n,p\}\), hence
\(q\leq\min\{1,\delta\}\). This design measures its
\(r_p\)-dimensional row space uniformly and vanishes on the
orthogonal complement. When \(r_p=n<p\) and \(n/p\to q=\delta\),
it is a scaled row-orthogonal, or tight-frame, sensing design
\parencite{tightframe}. In this realization the mass at zero
comes from an exact null space, while for a general sequence it can
also include vanishing positive eigenvalues.

\subsection{Canonical iid Gaussian design: the Marchenko--Pastur spectrum}
Even when the expected Gram matrix is the identity, an iid
Gaussian design has nonuniform measurement strengths in the
proportional asymptotic regime. For
\((X_p)_{ij}\stackrel{\mathrm{iid}}{\sim}N(0,1/n)\) and
\(n/p\to\delta\in(0,\infty)\), the limiting Gram law is the
mean-one Marchenko--Pastur law. Set
\(s_\pm=(1\pm\delta^{-1/2})^2\). The law is
$\mu_{\mathrm{MP},\delta}
    :=
    (1-\delta)_+\delta_0
    +f_\delta(s)\,\rd s,$ with 
    $f_\delta(s)
    =
    \frac{\delta}{2\pi s}
    \sqrt{(s_+-s)(s-s_-)},$ for $s\in[s_-,s_+].$
Consequently,
\(\mathbb E[S]=1\) and \(q=\mathbb P_\mu(S>0)=\min\{1,\delta\}\).

The aspect ratio determines how weak directions arise.
For \(\delta<1\), the exact null space occupies an asymptotic
fraction \(1-\delta\) of parameter directions. At \(\delta=1\),
there is no limiting atom at zero, but the continuous spectrum
reaches zero. This produces weakly measured directions even in
the square case \(n=p\), where the matrix is almost surely full
rank. For \(\delta>1\), the matrix is eventually full column rank
and the limiting spectrum is bounded away from zero, although
it remains spread over a nontrivial interval.

Within this family, letting \(\delta\to\infty\) concentrates the
law at one and recovers the flat-spectrum benchmark. At each
fixed finite \(\delta\), the spread reflects high-dimensional
Gaussian sampling. Substituting this law into the VAMP equations
recovers the AMP scalar lower-bound expression in
\cref{cor:gaussian-specialization}.

\subsection{Two-tier spectra: latent directions measured at two strengths}

A two-point law isolates the effect of unequal measurement
strengths:
\(\mu=\theta\delta_{s_-}+(1-\theta)\delta_{s_+}\), where
\(0\leq s_-<s_+\) and \(\theta\in(0,1)\).
Hold \(\theta\) and the mean
\(\bar s=\theta s_-+(1-\theta)s_+\) fixed. Increasing the
separation between \(s_-\) and \(s_+\), while keeping \(s_-\geq0\),
gives a mean-preserving spread. By
\cref{prop:spectral-monotonicity}, the spectrum-stage map,
greatest fixed point, and Bayes--VAMP lower bound cannot
decrease. Thus transferring measurement strength from the
weaker tier to the stronger tier does not improve the bound.

A finite realization is
\(X_p^T X_p=
V_p\operatorname{diag}(s_-I_{r_p},s_+I_{p-r_p})V_p^T\),
where \(r_p/p\to\theta\) and \(V_p\) is Haar distributed.
If \(s_->0\), this construction requires \(n\geq p\) and 
if \(s_-=0\), it requires \(n\geq p-r_p\).
The two tiers describe orthogonal latent directions randomized
by \(V_p\), rather than predetermined groups of observed columns.
The law \(\mu_{q,\mathrm{flat}}\) in \cref{sec:tightframe}
is the special case \(s_-=0\) and
\(s_+=\bar s/q\) with \(\theta=1-q\).

\subsection{Products of Gaussian matrices: multilayer random features}
\label{sec:lnn-model}

Gaussian products provide a family in which the spectrum depends
on both depth and intermediate dimensions. Following the linear
neural network (LNN) model in \textcite{Li2026debiasing}, consider
\(X_p^{(k)}=G_kG_{k-1}\cdots G_1\), where the factors are
mutually independent,
\(G_\ell\in\R^{n_\ell\times n_{\ell-1}}\) where
\((G_\ell)_{ij}\stackrel{\mathrm{iid}}{\sim}N(0,1/n_\ell)\),
and \(n_0=p\). The depth \(k\) is fixed as \(p\to\infty\),
and \(n_\ell/p\to\delta_\ell\in(0,\infty)\) for every layer. The product is the observed regression design and the unknown
parameter remains \(\beta^\star\).

\textbf{Adding a layer.}
The normalization gives
\(\E[G_\ell^T G_\ell]=I_{n_{\ell-1}}\). Consequently, we have 
\[
\E[(X_p^{(k+1)})^T X_p^{(k+1)}\mid X_p^{(k)}]
=(X_p^{(k)})^T X_p^{(k)}:
\]
appending a layer preserves the Gram matrix in conditional
expectation, and each limiting spectral law has mean one.
The next proposition shows that this preservation of average
strength nevertheless permits a strict increase in the
spectral lower bound.

\begin{proposition}[Products of iid Gaussian matrices]
\label{prop: product-gaussian}
For any \(k\ge1\), let $\mu_k$ denote the limiting Gram law of an LNN design $X_p^{(k)}$, and let $\mu_{k + 1}$ be the limiting law of an LNN design resulting from appending an additional conformable Gaussian factor. We show that $ \mu_{k+1}\succeq_{\mathrm{St}}\mu_k$ and that $L_{\mu_{k+1}}(\omega)
    >
    L_{\mu_k}(\omega)$ for every $\omega>0.$
Consequently,
adding an independent Gaussian layer strictly increases the
Bayes--VAMP lower bound for every nondegenerate prior. For a point-mass prior both risks are zero.
\end{proposition}
The proof in \cref{sec:spectral-examples} establishes strict
Stieltjes transform inequalities at every positive argument.
The comparison holds for each appended layer with a finite,
positive limiting width ratio; the final aspect ratio can
change when the layer is added. The conclusion concerns the
Bayes--VAMP lower bound, while its transfer to infimal convex risks
is covered by the attainment conditions in
\cref{cor:spectral-attainment-risk}.

\textbf{Allocating hidden widths.}
We now fix depth \(k=3\) and final aspect ratio
\(n_3/p\to\delta\). The two hidden ratios \(\delta_1,\delta_2\)
can be varied subject to a common entry budget: define
\(C_\delta(\delta_1,\delta_2)
=\delta_1+\delta_1\delta_2+\delta\delta_2\).
This is the limiting number of generated Gaussian entries per
\(p^2\), since
\((n_1p+n_2n_1+n_3n_2)/p^2\to C_\delta(\delta_1,\delta_2)\).
Write \(\mu_{\delta_1,\delta_2,\delta}\) for the induced limiting
Gram law.

The profile class below requires
\(\delta_1,\delta_2\geq\bar\delta:=\min\{1,\delta\}\).
These constraints prevent an additional rank bottleneck in a
hidden layer: every compared law has positive spectral mass
\(\bar\delta\). Thus the comparison fixes the mean, positive
spectral mass, final aspect ratio, and entry budget, while
varying the allocation between the two hidden widths.

\begin{proposition}
\label{prop:lnn-fix-three-layers}

Let $k = 3$. Let $\delta>0$ and $\bar \delta := \min\{1, \delta\}$. Define,  for any $C > 0$,
$$
    \mathcal G_{C, \delta} := \bigg\{(\delta_1, \delta_2) \in [\bar\delta, \infty)^2 : C_\delta(\delta_1, \delta_2) = C\bigg\}.
$$
The profile set is nonempty exactly when $C\geq\bar\delta+\bar\delta^2+\delta\bar\delta$. Then, for any nondegenerate prior,
\begin{enumerate}[label=\roman*)]
    \item If \(\delta>1\), for any profile 
    \(
        (\delta_1,\delta_2)\in\mathcal G_{C,\delta}
    \) 
    with \(\delta_1<\delta_2\), there exists another profile 
    \(
    (\tilde\delta_1,\tilde\delta_2)\in\mathcal G_{C,\delta}
    \) 
    such that the induced spectral law of the former profile strictly dominates that of the latter in Stieltjes transform, and hence the lower bound of $(\delta_1, \delta_2)$ is strictly larger than that of $(\tilde\delta_1, \tilde\delta_2)$.
    

    \item If $0<\delta < 1$, for the spectral law induced by any profile $(\delta_1, \delta_2) \in \mathcal G_{C, \delta}$ with $\delta_1>\delta_2$, there exists another profile $(\tilde\delta_1, \tilde\delta_2) \in \mathcal G_{C, \delta}$ whose spectral law it strictly dominates.

    \item Suppose now that $\delta = 1$. Further assume that $\delta_1, \delta_2 \geq 1$ and $C \geq 3$. Then the induced spectral law of the profile $\delta_1 = \delta_2 = \sqrt{C + 1} - 1$
    uniquely minimizes the Stieltjes transform 
    over all spectral laws 
    for $(\delta_1, \delta_2) \in \mathcal G_{C, 1}$,
    and therefore also uniquely minimizes the Bayes-VAMP lower bound.
\end{enumerate}

\end{proposition}

The proof is in \cref{sec:spectral-examples}.
For \(\delta>1\), any profile with \(\delta_1<\delta_2\)
can be improved at the same budget; for \(\delta<1\),
the same holds when \(\delta_1>\delta_2\).
Thus an optimal profile must satisfy \(\delta_1\geq\delta_2\)
when \(\delta>1\), and \(\delta_1\leq\delta_2\) when \(\delta<1\).
These necessary conditions compare the input-adjacent width
\(n_1\) with the output-adjacent width \(n_2\).
The square case gives a complete answer within the stated
class: equal hidden ratios
\(\delta_1=\delta_2=\sqrt{C+1}-1\) uniquely minimize the
transform and, for a nondegenerate prior, the lower bound.

The distinction between square and rectangular designs comes
from the budget. Exchanging the two hidden ratios leaves the
limiting spectral law unchanged, but generally changes
\(C_\delta\). In the unfavorable ordering, exchanging them
saves entries, while increasing a hidden width with those savings strictly improves the transform. When \(\delta=1\), the budget is also symmetric, and balancing gives the unique optimum.


\subsection{Extreme energy concentration: many weak directions and a few dominant ones}
\label{sec:extreme}

The preceding comparisons constrain more than average
measurement strength. To see why the mean alone is insufficient,
we construct mean-one spectra with no atom at zero for which the
lower bound approaches \(\operatorname{Var}(B^\star)\), the risk
of the constant estimate \(\E[B^\star]\).

Let \(T\) be almost surely finite and positive, with
\(\E[T]=\infty\). For a fixed \(M>0\), set
\(S_M=(T\wedge M)/\E[T\wedge M]\), and let \(\mu_M\)
be its law. Each \(\mu_M\) is compactly supported, has mean
one, and assigns no mass to zero. As \(M\to\infty\), however,
\(S_M\to0\) almost surely. For each fixed \(\omega>0\),
bounded convergence in \cref{eq:spectrum-mod} sends the
numerator to \(\omega\) and the denominator to zero, so
\(L_{\mu_M}(\omega)\to\infty\).

For a nondegenerate prior, the mean-one normalization gives
\(L_{\mu_M}(0)=\sigma^2\), while
\(\cE_\pi(\sigma^2)>0\).
The fixed-point identity and monotonicity then give
\(\tau_B^+(\mu_M)\geq
L_{\mu_M}(\cE_\pi(\sigma^2))\to\infty\).
Consequently,
\(m_\pi(\tau_B^+(\mu_M))\to\operatorname{Var}(B^\star)\).
This is a comparison across limiting models: first take the
high-dimensional limit at each fixed \(M\), then let
\(M\to\infty\). The spectral bound in \cref{ass1} holds
separately for every fixed \(M\), with no uniform bound in
\(M\) required.

For dimensions with \(n\geq p\), an explicit full-rank
realization uses deterministic quantiles
\(s_{j,p}^{(M)}\) of \(\mu_M\) at levels \((j-\tfrac12)/p\) and
\((X_p^{(M)})^T X_p^{(M)}
=V_p\operatorname{diag}(s_{1,p}^{(M)},\ldots,s_{p,p}^{(M)})V_p^T\).
These quantiles are positive. As \(M\) grows, most singular
directions become weakly measured, while a vanishing
fraction carries the energy needed to keep the mean at one.
The heavy-tailed variable \(T\) constructs a family of bounded
spectral laws; it does not specify an iid heavy-tailed
matrix-entry ensemble.

These examples show how the distribution of measurement strength
affects the spectral lower bound even when its mean is fixed.
The next section returns to the admissible penalty class and
explains why mass at zero sometimes requires an additional
width condition.

\section{The role of \texorpdfstring{$q$}{q}-bounded width}
\label{sec: examples-penalties}

The $q$-bounded width condition in \cref{defn:q-width-bounded-b}
serves a specific step in the proof of \cref{thm:Bayes VAMP-lower-bound}:
it rules out degenerating sequences that would obstruct the lower-bound
comparison as the oracle perturbation is removed. The proof needs this
additional penalty condition only in a particular regime of spectral
mass near zero. We first explain that regime, then show how the
minimizer set of a scalar penalty often makes the condition easy to check.

Recall that $ p_0:=\PP_\mu(S=0)=1-q$ and 
    $I_+
    :=
    \E_\mu[{\mathbf 1_{\{S>0\}}}/{S}].$
The admissible penalty class in \cref{eq:admissible-penalty-class} is
\[
    \cC_{\rm LB}(\mu,\pi,\sigma^2)
    =
    \begin{cases}
        \cC,
        & p_0=0\ \text{or}\ I_+=\infty,\\[1mm]
        \cC_{\rm WB}(q;\tau_B^+),
        & p_0>0\ \text{and}\ I_+<\infty.
    \end{cases}
\]
Thus the width condition is imposed only when the limiting law has an atom at zero \(p_0>0\), but the reciprocal
moment of its positive part is finite \(I_+<\infty\). Here \(q\) is the positive mass
of the limiting spectral law. Eigenvalues that are positive in every
finite dimension may still converge to zero, so this mass need not equal
the limiting fraction of nonzero finite-dimensional eigenvalues.
If \(p_0=0\) or \(I_+=\infty\), the lower bound applies to every
\(h\in\cC\) whose finite-dimensional argmin is nonempty, as required
by \cref{ass4}.

\textbf{The endpoint comparison in the lower-bound proof.}
The condition compares the mean slope of the scalar proximal denoiser
with the positive spectral mass. Specifically,
\cref{defn:q-width-bounded-b} requires
\(\epsilon_0\in(0,\tau_B^+)\) and \(\delta_0>0\) such that
\begin{equation*}
\begin{aligned}
    \limsup_{\kappa\to\infty}
    \sup_{\nu\in[\tau_B^+ -\epsilon_0,\tau_B^+]}A_h(\kappa,\nu)
    &\le q-\delta_0,
\end{aligned}
\end{equation*}
where $A_h(\kappa,\nu) =
    \E\left[
        (\prox_{\kappa h})'
        (B^\star+\sqrt\nu Z)
    \right]
    \in[0,1]$. 
Increasing \(\kappa\) strengthens the penalty in the scalar proximal
problem, or equivalently decreases the precision of its quadratic
term. The condition requires the average response to the scalar
observation to remain strictly below \(q\) in this limit, uniformly
for effective variances just below \(\tau_B^+\).

To see its role, consider the comparison used to establish the lower
bound as the perturbed $\ell_2$ oracle regularization is sent to zero.  In this limit, the spectrum stage can become increasingly sensitive to
these zero or vanishing spectral values.
That argument must rule out certain candidate sequences whose effective
denoising precision tends to zero. The difficult regime has \(p_0>0\) and
\(
I_+
=
\E[\frac{1}{S}\mathbf 1\{S>0\}]
<\infty
\).
For an obstructing candidate sequence to cause problems with the proof, the scalar state evolution requires that
\(\liminf_{j\to\infty}A_h(\kappa_j,\nu_j)\geq q\), where
\(\kappa_j\to\infty\) and
the effective noise level \(\nu\) is near \(\tau_B^+\).
The \(q\)-bounded width condition excludes this possibility by requiring, uniformly
over that window,
\(\limsup_{j\to\infty}A_h(\kappa_j,\nu_j)\leq q-\delta_0\), for some \(\delta_0>0\),
which supplies the needed contradiction. When \(p_0=0\) or
\(I_+=\infty\), the spectral endpoint estimates already give the
contradiction required for this comparison, and the proof needs no
width condition.

\begin{rem}
Failure of the \(q\)-bounded width does not show that a degenerate fixed-point scenario
exists; we use
\(q\)-bounded width as a sufficient condition for the present proof and do not
claim that it is statistically necessary.
\end{rem}

\subsection{A minimizer-set characterization}

The definition involves a limit of proximal derivatives. When \(h\)
attains its minimum, this limit is governed by a simpler object: the
scalar minimizer set of \(h\). As \(\kappa\) grows, the proximal map approaches
Euclidean projection onto that set. Since projection leaves observations in
its interior unchanged and maps observations outside it to a boundary
point, the  slope of the proximal map is one in the interior and zero outside. Therefore,
the Gaussian convolution gives the boundary probability zero. The next
proposition makes this interpretation uniform over compact ranges of
positive noise variances.

\begin{proposition}
\label{prop:q-bounded-minimizer-set}
Let \(h\in\cC\), and suppose that
\(M:=\argmin_{x\in\R}h(x)\) is nonempty. For every nonempty compact
\(K\subset(0,\infty)\),
\begin{equation}
    \sup_{\nu\in K}
    \left|
        A_h(\kappa,\nu)
        -
        \PP\left(
            B^\star+\sqrt\nu Z\in\operatorname{int}(M)
        \right)
    \right|
    \longrightarrow0, \qquad \text{as }
    \kappa\to\infty.
    \label{eq:Ah-minimizer-limit}
\end{equation}
Consequently, \(h\in\cC_{\rm WB}(q;\tau_B^+)\) if and only if $\PP(
        B^\star+\sqrt{\tau_B^+}Z
        \in\operatorname{int}(M)
    )<q.$
\end{proposition}

The proof can be found in \cref{sec:minimizer-set-proof}. Only the flat region where \(h\) is minimized contributes to the
limiting average slope. Bounded width asks that this region contain
strictly less than a fraction \(q\) of the effective scalar observations
at variance \(\tau_B^+\). The uniform convergence in
\cref{eq:Ah-minimizer-limit} turns this geometric test into the
variance-window condition in \cref{eq:q-width-condition}.
If \(h\) does not attain its minimum, the proposition does not apply,
and \cref{eq:q-width-condition} remains the relevant assumption.

\subsection{Relevant penalty classes}\label{sec:relevant penalties}

\textbf{Penalties with a unique scalar minimizer.}
If \(M=\{x_0\}\), then \(\operatorname{int}(M)=\varnothing\), and $\PP(
        B^\star+\sqrt{\tau_B^+}Z
        \in\operatorname{int}(M)
    ) = 0$. Hence every penalty
with a unique scalar minimizer belongs to
\(\cC_{\rm WB}(q;\tau_B^+)\) for every \(q>0\). This includes every norm on
\(\R\), shifted absolute-value penalties \(h(x)=c|x-x_0|\), ridge and
elastic-net penalties, and every proper, lsc, strongly convex scalar penalty.

A unique scalar minimizer also ensures the separate
finite-dimensional existence requirement in \cref{ass4}.
Indeed, a proper, lsc, convex function on \(\R\) with a unique
minimizer is coercive. On an unbounded side of its effective domain,
a secant through the minimizer and any other finite point has strictly
positive outward slope; convexity then gives at least linear growth
along that tail. On a bounded side, the extended-valued function is
eventually \(+\infty\). Thus \(h(x)\to+\infty\) as \(|x|\to\infty\), and
\(\sum_{j=1}^p h(\beta_j)\) is coercive on \(\R^p\). The penalized
least-squares objective consequently attains its minimum.

\textbf{Flat-bottom penalties.}
When the scalar minimizer set is a nontrivial interval, its interior
can contain effective observations with positive probability.
The quantity $\PP(
        B^\star+\sqrt{\tau_B^+}Z
        \in\operatorname{int}(M)
    )$
then determines whether the interval is small enough for bounded width.
For the dead-zone penalty \(h(x)=(|x|-a)_+\), the minimizer set is
\(M=[-a,a]\), and \cref{prop:q-bounded-minimizer-set} gives the criterion
$\PP(
        |B^\star+\sqrt{\tau_B^+}Z|<a
    )<q$.
For the hinge penalty \(h(x)=(x-b)_+\), the minimizer set is
\(M=(-\infty,b]\), and the criterion becomes
\(\PP(B^\star+\sqrt{\tau_B^+} Z<b) < q\).
These tests establish the width condition, while finite-dimensional
minimizer existence remains a separate requirement of \cref{ass4}.

\section{Numerical Experiments}
\label{sec:simulations}

We illustrate two ways in which measurement geometry affects the scalar
Bayes--VAMP lower bound: redistributing spectral mass at fixed mean, and
changing the depth or hidden widths of a Gaussian product. These are
numerical evaluations of the scalar benchmark. Their interpretation as
infimal convex risks uses the attainment criterion in
\cref{prop: strict-equality}.

\textbf{Spectral shape at fixed mean.}
\cref{fig:quantiles_barriers} compares six spectral laws: a flat law,
two two-point laws, and three laws based on the ``MatrixNormal,'' ``Multi-t,''
and ``Spiked'' constructions in \textcite{Li2026debiasing}.
The matrix-normal construction uses Toeplitz and inverse-Wishart covariance
factors, multivariate-$t$ has independent rows with identity
scale and three degrees of freedom, and the spiked construction adds a
low-rank component to Gaussian noise. The generation and normalization
procedures are described in \cref{sec:li-sur-matrices}.
For these matrix-based examples, we use the finite-sample empirical Gram law as the limiting spectrum. Notably, this does not require the original random-matrix
ensembles to satisfy our asymptotic spectral-boundedness assumptions.

The left panels display power-scaled Gram eigenvalue quantiles, and the
right panel displays the corresponding scalar risk curves for the
Rademacher prior \(B^\star\in\{-1,1\}\) with equal probabilities.
The spectra are normalized to the common mean \(\bar s=0.019\).
For the flat law, \(\tau_B^+=\sigma^2/\bar s\); hence
\(\sigma^2\geq\bar s\) implies \(\tau_B^+\geq1\).
The spectral comparison then places every law of this mean in the
Rademacher log-concavity regime. Thus, on this noise range, the scalar
curves also give the admissible infimal convex risks for the constructed
RRI sequences.

First, we remark that the illustrated spectra cannot be completely ordered using our transform or convex ordering, so the lower bound curves intersect. 
The flat law has the smallest curve, and increasing the separation of
the two-point laws increases the lower bound, as predicted by
\cref{cor:flat-spectrum-best} and the convex-order comparison.
The remaining curves show that spectral shape matters beyond a single
measure of spread. The spiked law places much of its mass near zero and
has a large risk at the smaller noise levels, while its curve crosses
that of the matrix-normal example as the noise increases.

\begin{figure}[t]
    \centering
    \safeincludegraphics[width=0.9\textwidth]{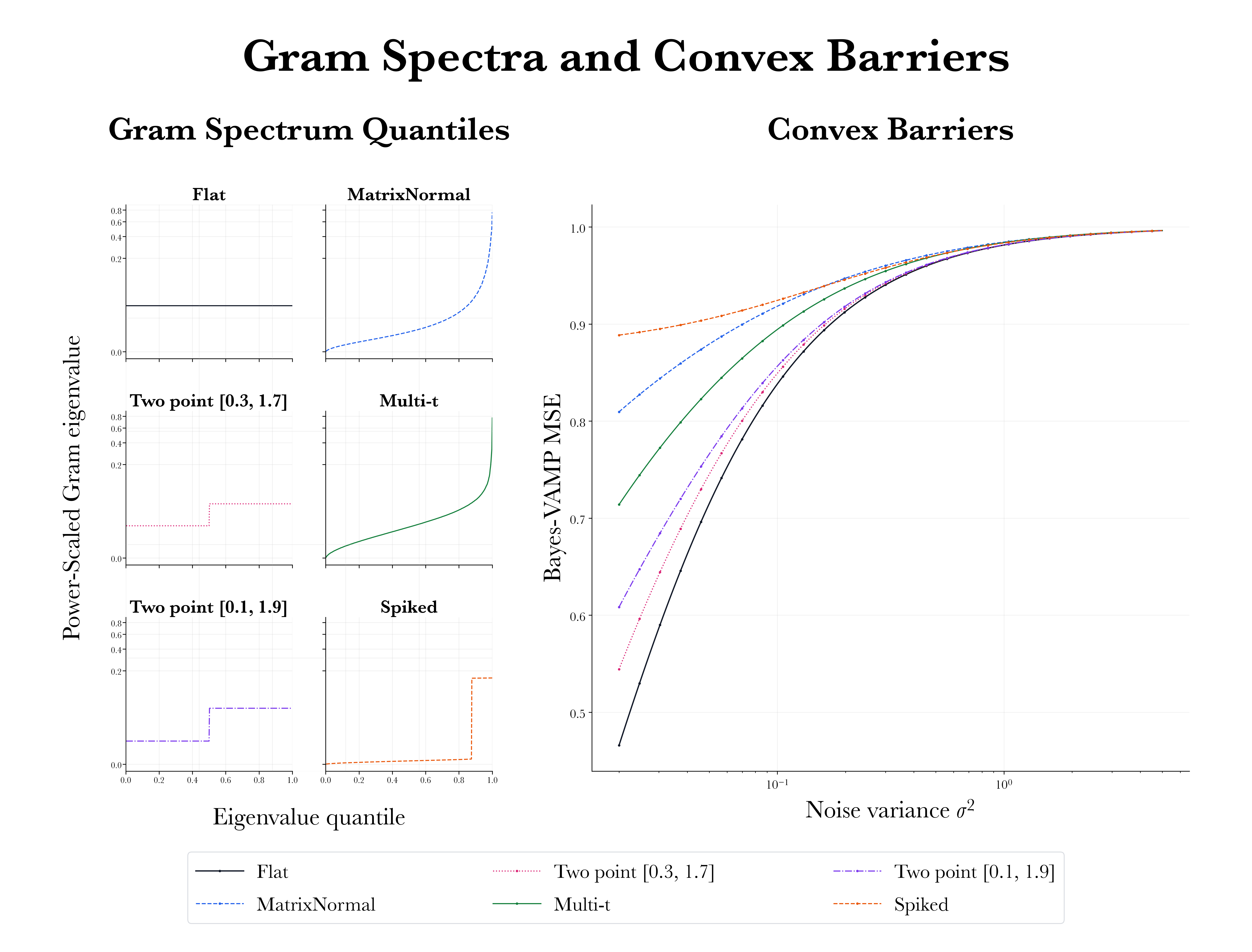}
    \caption{
    Six spectral laws with common mean \(\bar s=0.019\).
    Left: power-scaled Gram eigenvalue quantiles after normalization;
    design names and two-point labels refer to the unnormalized spectra.
    Right: scalar Bayes--VAMP risk curves for the Rademacher prior.
    For \(\sigma^2\geq\bar s\), the log-concavity attainment condition
    holds, so these curves equal the admissible infimal convex risks
    for the constructed RRI sequences.
    }
    \label{fig:quantiles_barriers}
\end{figure}

\textbf{Depth and hidden-width allocation.}
\cref{fig:lnn-barriers} compares Gaussian-product spectra at final
aspect ratio \(\delta=25\). Here the signal prior is Gaussian, so
\cref{prop: strict-equality} identifies each scalar lower bound with the
corresponding admissible infimal convex risk. The left panel shows Gram
eigenvalue quantiles on a logarithmic scale and the right panel shows the
risk as the noise variance changes.

The three colored curves have depth \(k=3\) and Gaussian-entry
budget \(C_{25}=41.3\), as defined in \cref{sec:lnn-model}.
Their hidden-width profiles are \((1,1.55)\) in red,
\((1.52,1.5)\) in green, and \((8.15,1)\) in blue.
The red profile has \(\delta_1<\delta_2\), so
\cref{prop:lnn-fix-three-layers} guarantees that another profile with the
same budget improves its lower bound at every noise level. In the figure,
both selected alternatives improve on the red curve. The green and blue
curves cross, so neither gives the smaller risk throughout the displayed
noise range. Thus the preferred allocation between these two profiles
depends on the noise level, even though the entry budget is fixed.

The black curves, labeled \(k=1\) and \(k=6\), provide depth references
and bracket the displayed colored curves. Their labels ``lower bound''
and ``upper bound'' refer only to this displayed comparison.
The general depth result in \cref{prop: product-gaussian} compares a
Gaussian product with the product obtained by appending an independent
conformable Gaussian layer. Depth labels alone do not establish that
relation between arbitrary width profiles.

\begin{figure}[t]
    \centering
    \safeincludegraphics[width=0.90\textwidth]{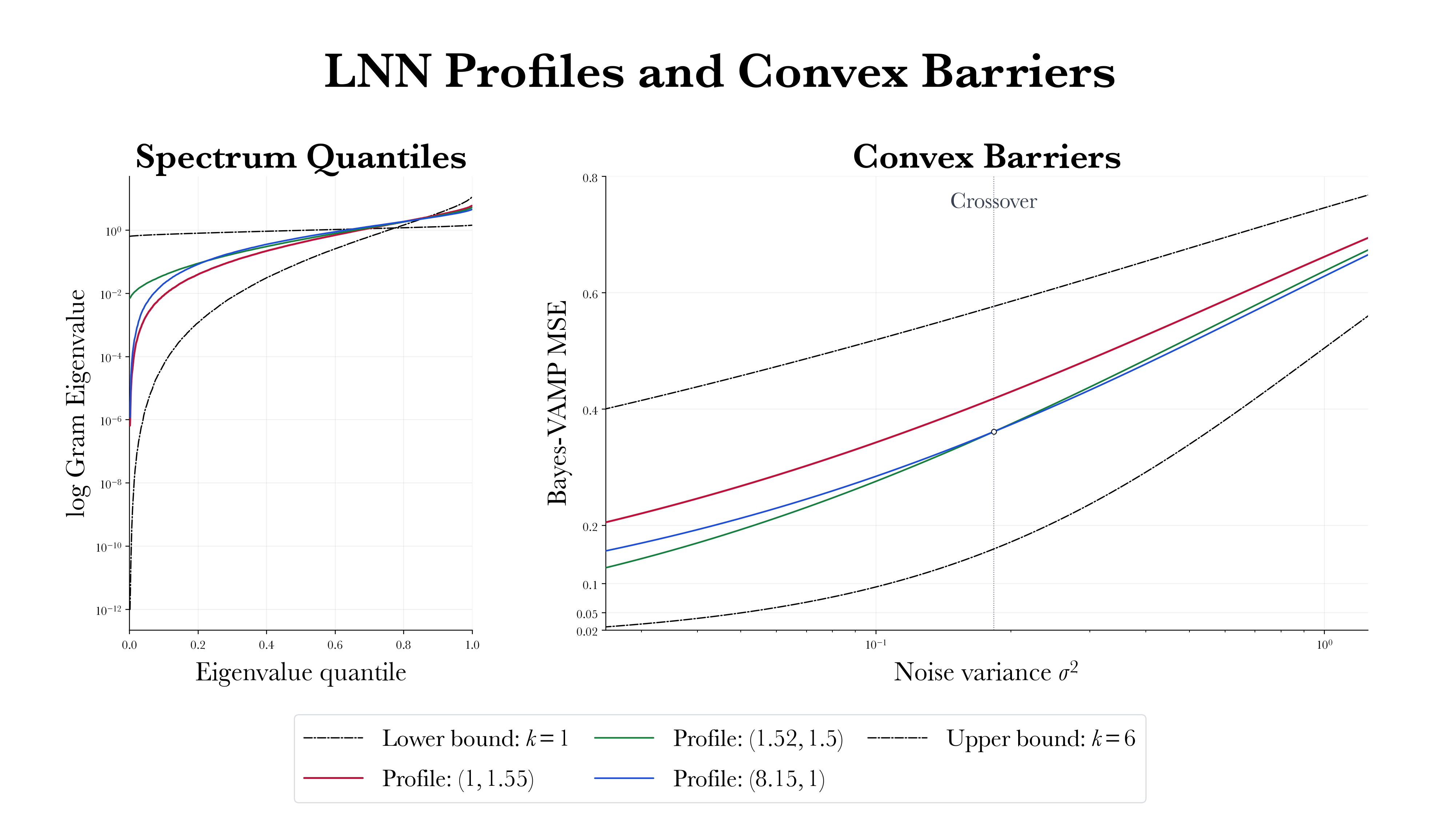}
    \caption{
Gaussian-product spectra and risks at \(\delta=25\).
    Left: Gram eigenvalue quantiles on a logarithmic scale.
    Right: scalar Bayes--VAMP risk curves for a Gaussian prior;
    attainment identifies these with admissible infimal convex risks.
    The three colored depth-three profiles have common budget
    \(C_{25}=41.3\). The red curve lies above both selected alternatives,
    while the green and blue curves cross.
    The black depth-one and depth-six curves are references;
    their ``lower bound'' and ``upper bound'' labels describe the
    envelopes of the displayed curves, not universal extrema over
    all widths or depths.
    }
    \label{fig:lnn-barriers}
\end{figure}

\section*{Acknowledgements}
\addcontentsline{toc}{section}{Acknowledgements}

This material is based upon work supported by the National Science Foundation under Award No.\ DMS-2413828. The authors used ChatGPT to explore possible proof approaches and to assist with editing and organization of the manuscript. The authors take full responsibility for the mathematical arguments and final text.

\clearpage
\appendix
\section{Bayes VAMP: matched state evolution and greatest-fixed-point selection}\label{sec:Bayes-VAMP-intro}

This appendix establishes the algorithmic interpretation of the quantity
\(m_\pi(\tau_B^+)\) appearing in the main results.  We first state a
finite-dimensional Bayes VAMP iteration with a spectrum-based
initialization.  We then invoke the VAMP state-evolution theorem and show
that this initialization is well calibrated, causing the full state evolution to
collapse to the scalar recursion
$\tau_{t+1}^B=T_{\mu,\pi}(\tau_t^B).$
The initialization starts this recursion above every fixed point, so its
iterates decrease to the greatest fixed point \(\tau_B^+\).  Consequently,
the asymptotic MSE of Bayes VAMP converges, in the iterated
high-dimensional and large-iteration limit, to \(m_\pi(\tau_B^+)\).  We
conclude by recording the equivalent four-parameter fixed-point
representation used in later proofs. 

\subsection{The finite-dimensional algorithm and its error-coordinate form}

The ordinary two-stage Bayes VAMP algorithm was introduced in
\textcite{Rangan2018vamp} and
consists of an initialization followed by updates that start with the denoising stage (\crefrange{algline:bayes-update-start}{algline:bayes-update-end}). We present it below as \cref{alg: Bayes VAMP}, using a spectrum-based initialization in \crefrange{algline:bayes-init-start}{algline:bayes-init-end}.

\begin{algorithm}
\caption{Bayes VAMP with spectrum-based initialization}
\label{alg: Bayes VAMP}
\begin{algorithmic}[1]

\State \textbf{Input:} \(X,y,\sigma^2\), iteration count \(T\), and known prior \(\pi\) with \(b_\pi=\mathbb E_\pi[B^\star]\) and \(V_\pi=\operatorname{Var}_\pi(B^\star)>0\)

\State \(\Br_{2,-1}=b_\pi\vecone_p,\quad \gamma_{2,-1}=V_\pi^{-1}\) \labelalgorithmline{algline:bayes-init-start}

\State \(\widehat\beta_{2,-1}= \tilde{g}_t(\Br_{2,-1},\gamma_{2, -1}) \) \labelalgorithmline{algline:bayes-initial-linear-update}

\State \(c_{-1}=p^{-1}\operatorname{tr}\!\left[\gamma_{2,-1}(\sigma^{-2}X^T X+\gamma_{2,-1}I_p)^{-1}\right]\)

\State \textbf{if} \(c_{-1}\notin(0,1)\), \textbf{return} \(b_\pi\vecone_p\) and stop \labelalgorithmline{algline:bayes-initial-safeguard}

\State \(\Br_1^0=(\widehat\beta_{2,-1}-c_{-1}\Br_{2,-1})/(1-c_{-1}),\quad
\gamma_{10}=\gamma_{2,-1}(c_{-1}^{-1}-1)\) \labelalgorithmline{algline:bayes-init-end}

\For{\(t=0,\ldots,T\)} \labelalgorithmline{algline:bayes-update-start}

\State \(\widehat\beta_{1,t}=\tilde f_t(\Br_1^t,\gamma_{1t})\) \labelalgorithmline{eq:f_func} \labelalgorithmline{algline:bayes-denoiser}

\State \textbf{if} \(t=T\), \textbf{return} \(\widehat\beta_{1,T}\) and stop

\State \(b_t=p^{-1}\sum_{j=1}^p\tilde f_t'(r_{1j}^t,\gamma_{1t})\)

\State \textbf{if} \(b_t\notin(0,1)\) or \(\gamma_{1t}\leq0\), \textbf{return} \(b_\pi\vecone_p\) and stop \labelalgorithmline{algline:bayes-denoiser-safeguard}

\State \(\eta_1^t=\gamma_{1t}/b_t\)

\State \(\Br_2^t=\gamma_{2t}^{-1}(\eta_1^t\widehat\beta_{1,t}-\gamma_{1t}\Br_1^t), \quad \gamma_{2t}=\eta_1^t-\gamma_{1t}\)

\State \(\widehat\beta_{2,t}=\tilde{g}_t(\Br_2^t,\gamma_{2t})\) \labelalgorithmline{algline:bayes-linear-update}

\State \(c_t=p^{-1}\operatorname{tr}\!\left[\gamma_{2t}(\sigma^{-2}X^T X+\gamma_{2t}I_p)^{-1}\right]\)

\State \textbf{if} \(c_t\notin(0,1)\) or \(\gamma_{2t}\leq0\), \textbf{return} \(b_\pi\vecone_p\) and stop \labelalgorithmline{algline:bayes-linear-safeguard}

\State \(\Br_1^{t+1}=(\widehat\beta_{2,t}-c_t \Br_2^t)/(1-c_t), \quad \gamma_{1,t+1}=\gamma_{2t}(c_t^{-1}-1)\)

\EndFor \labelalgorithmline{algline:bayes-update-end}

\end{algorithmic}
\end{algorithm}

The denoiser-stage $\tilde{f}_t$ in \cref{algline:bayes-denoiser} of \cref{alg: Bayes VAMP} acts componentwise on the vector $\mathbf r_1^t$. Specifically, 
we can define it through a posterior-mean scalar denoiser:
\begin{equation}
\label{eq:mmse_denoiser_func}
\tilde{f}_t(\Br_1^t, \gamma_{1t})
=
\left(\eta_{\gamma_{1t}}^{\rm B}(r_{1j}^t)\right)_{j\in[p]},
\qquad
\eta_\gamma^{\rm B}(r)
:=
\E[B^\star \mid B^\star + \gamma^{-1/2}Z=r],
\end{equation}
where \(B^\star\sim\pi\) and \(Z\sim N(0,1)\) with \(B^\star \perp Z\). The spectrum stage uses a linear Gaussian
update, $\tilde g_t$ found in \cref{algline:bayes-initial-linear-update,algline:bayes-linear-update} of \cref{alg: Bayes VAMP},
(which we will also call a denoiser here to match the terminology in \textcite{Rangan2018vamp}) given by
\[
\tilde{g}_t(\Br_2^t,\gamma_{2t}) = \left(\sigma^{-2} X^TX + \gamma_{2t}I_p\right)^{-1}
\left(\sigma^{-2} X^T y + \gamma_{2t} \Br_2^t\right).
\]
\cref{alg: Bayes VAMP} alternates these two stages.  Its initial
linear update uses the prior mean and variance and is the only part of the
algorithm that precedes the first denoising step.

The asymptotic behavior of the iterates in \cref{alg: Bayes VAMP} is characterized by the VAMP state evolution, which we will derive as in \textcite[Eq.\ (84)]{Rangan2018vamp}. 
For the state-evolution analysis, it is convenient to express the
recursion in error coordinates through a change of variables.  Write \(X=U\Sigma V^T\), let
\(d_{j,p}=\sqrt{s_{j,p}}\), and define
$\Bp^t:=\Br_1^t-\beta^\star,$ and $\Bv^t:=\Br_2^t-\beta^\star.$
For each nonzero singular value, set
\(\xi_{j,p}=u_{j,p}^T\varepsilon\) where $\varepsilon$ is the noise in the regression model; for a zero singular value, append an
independent \(N(0,\sigma^2)\) variable.  With
\(\Bd=(d_{1,p},\ldots,d_{p,p}) \in \R^p\) and
\(\Bw=(\bm\xi,\Bd)^T \in \R^{2 \times p}\), the entries of \(\bm\xi\) are iid
\(N(0,\sigma^2)\) and $ X^T\varepsilon
    =V\diag(d_{1,p},\ldots,d_{p,p})\bm\xi.$

With $\eta_{\gamma}^{\rm B}$ defined in \cref{eq:mmse_denoiser_func}, 
define
\begin{equation}
\label{eq:denoising-functions}
\begin{aligned}
f(p,\beta,\gamma_1)
    &:=\eta_{\gamma_1}^{\rm B}(p+\beta)-\beta,\\
g(q,w,\gamma_2) = g(q,(\zeta,d),\gamma_2)
    &:=\frac{\sigma^{-2}d\zeta+\gamma_2q}
             {\sigma^{-2}d^2+\gamma_2},\\
C(a)&:=\frac1{1-a},
\qquad
\Gamma(\gamma,a):=\gamma(a^{-1}-1).
\end{aligned}
\end{equation}
Then, after the initialization in \cref{alg: Bayes VAMP}, its
updates are equivalent to
\begin{equation}
\label{eq:alternative-formulation-vamp}
\begin{aligned}
\Bp^t&=V\Bu^t,\\
b_t&=\langle(\eta_{\gamma_{1t}}^{\rm B})'
             (\Bp^t+\beta^\star)\rangle,
&\gamma_{2t}&=\Gamma(\gamma_{1t},b_t),\\
\Bv^t&=C(b_t)\{f(\Bp^t,\beta^\star,\gamma_{1t})-b_t\Bp^t\}, \\
\Bq^t&=V^T\Bv^t,\\
c_t&=\langle g'(\Bq^t,\Bw,\gamma_{2t})\rangle,
&\gamma_{1,t+1}&=\Gamma(\gamma_{2t},c_t),\\
\Bu^{t+1}&=C(c_t)\{g(\Bq^t,\Bw,\gamma_{2t})-c_t\Bq^t\}.
\end{aligned}
\end{equation}
This is the abstract VAMP form used by the state-evolution theorem; see
\textcite[Appendix~G]{Rangan2018vamp} for the full derivation of the equivalence.

Let $\nu_p(\cdot)$ denote the empirical measure of its input. Then, by the conditional Gaussian law of large numbers, the empirical law of the disturbance variables satisfies
$ \nu_p(\Bw)
    \xrightarrow[\mathrm{a.s.}]{W_2}
    \operatorname{Law}(W),$ where
    $W=(\zeta,D),$
with \(\zeta\sim N(0,\sigma^2)\) and \(D^2\sim\mu\), where
\(\zeta\perp D\). Recall the law of $D^2$ is the spectral law defined in \cref{ass1}.

\subsection{Matched state evolution and its scalar reduction}

We now derive the state evolution recursion associated with \cref{alg: Bayes VAMP}.
We first identify the state produced by the spectrum-based initialization, which serves as the initialization of the state evolution. Recall \(b_\pi := \mathbb E_\pi[B^\star]\) and \(V_\pi := \operatorname{Var}_\pi(B^\star)>0\).
Let
\[
    c_{-1}^{\infty}
    :=\E\left[\frac{\sigma^2}{\sigma^2+V_\pi D^2}\right]
    \in(0,1),
\]
where the fact that $c_{-1}^{\infty} < 1$ follows 
from $V_\pi > 0$ and $\mathbb P(D^2 > 0) > 0$.
The initialization error entering the linear stage has asymptotic variance
\(V_\pi\).  A direct conditional-variance calculation therefore gives
\[
    \tau_{10}
    =\frac{V_\pi c_{-1}^{\infty}}{1-c_{-1}^{\infty}}
    =L_\mu(V_\pi),
    \qquad
    \bar\gamma_{10}=\tau_{10}^{-1},
\]
where $L_\mu$ is defined as in \cref{eq:spectrum-mod}.
Thus, in the language of \textcite{Rangan2018vamp}, the effective variance and the denoiser precision are ``matched'' at the
first prior stage, meaning that the reciprocal of the effective
Gaussian error entering the prior stage matches the precision used by the Bayes denoiser.
Under the regularity conditions collected below, the VAMP
state-evolution theorem of \textcite[Theorem~2]{Fletcher2018Plugin}
applied to \cref{eq:alternative-formulation-vamp} yields the following state evolution
\begin{equation}
\label{eq:Bayes VAMP-SE}
\begin{aligned}
\bar b_t
    &=\E[(\eta_{\bar\gamma_{1t}}^{\rm B})'
          (B^\star+P_t)],\\
\tau_{2t}
    &=C(\bar b_t)^2
      \left(\E[f(P_t,B^\star,\bar\gamma_{1t})^2]
                   -\bar b_t^2\tau_{1t}\right),\\
\bar\gamma_{2t}
    &=\Gamma(\bar\gamma_{1t},\bar b_t),\\
\bar c_t
    &=\E[g'(Q_t,W,\bar\gamma_{2t})],\\
\tau_{1,t+1}
    &=C(\bar c_t)^2
      \left(\E[g(Q_t,W,\bar\gamma_{2t})^2]
                   -\bar c_t^2\tau_{2t}\right),\\
\bar\gamma_{1,t+1}
    &=\Gamma(\bar\gamma_{2t},\bar c_t),
\end{aligned}
\end{equation}
where \(P_t\sim N(0,\tau_{1t})\) is independent of \(B^\star\) and
\(Q_t\sim N(0,\tau_{2t})\) independent of \(W\). Derivatives $\tilde f'$ and $g'$ are taken with respect to their first argument. We call the terms $\bgamma_{1t}$ and $\bgamma_{2t}$ the \textit{precision parameters}.

\begin{rem}[Matched initialization]
\label{rem: matched-initialization}
The recursion in \cref{eq:Bayes VAMP-SE} is initialized with $\bgamma_{10} = 1/\tau_{10} = 1 / L_{\mu}(V_\pi)$. 
Because \(\bar\gamma_{10}=\tau_{10}^{-1}\), the matched-state-evolution
identity of \textcite[Theorem~2]{Rangan2018vamp} gives, for every fixed
iteration $t$,
\begin{equation}
\label{eq: bvamp-matched}
    \tau_{1t}=\bar\gamma_{1t}^{-1},
    \qquad
    \tau_{2t}=\bar\gamma_{2t}^{-1}.
\end{equation}
Hence the variance and precision parameters need not be tracked
separately.  In particular, the prior stage observes
\(B^\star+\sqrt{\tau_{1t}}Z\) and applies its posterior mean with the
`correct' noise precision; hence, attaining the MMSE.
Without matched initialization, the identities in
\cref{eq: bvamp-matched} need not hold, so the variances and precisions must
be retained as separate state variables and the Bayes state evolution does
not reduce to the matched scalar recursion, for which we derive the fixed-point equations below.
\end{rem}

The following lemma explains how the state evolution parameters evolve over iterations, and is the key reduction from the full recursion
\cref{eq:Bayes VAMP-SE} to the one-dimensional map used in the main text.

\begin{lemma}[Scalar Bayes-VAMP recursion]
\label{lem: SE-bvamp-tau1}
Let \(\tau_t^B:=\tau_{1t}\)  be the state evolution parameter of Bayes VAMP in \cref{eq:Bayes VAMP-SE}.  Under the matched initialization,
\[
    \tau_{2t}=\cE_\pi(\tau_t^B),
    \qquad
    \tau_{t+1}^B=L_\mu(\tau_{2t})
    =T_{\mu,\pi}(\tau_t^B).
\]
\end{lemma}
\begin{proof}[Proof of \cref{lem: SE-bvamp-tau1}.]
Set \(\tau=\tau_{1t}\) and
\(R=B^\star+\sqrt\tau Z\). Suppose that at iteration $t$, the prior stage receives  these matched scalar quantities with $\gamma = \tau^{-1}$. Using \cref{eq:Bayes VAMP-SE} and result \cref{eq:postderiv} of \cref{lem:mmsecalculus} applied to \(f_{\tau} = \eta_{1/\tau}^{\mathrm B}\), gives
$\bar b_t
    =\E[(\eta_{1/\tau}^{\rm B})'(R)]
    = {m_\pi(\tau)}/{\tau} \in (0, 1).$ Moreover, by \cref{eq:denoising-functions}, we have $\E[f(P_t,B^\star,\gamma)^2] = m_\pi(\tau)$.
Substituting these identities into the prior-stage variance update in
\cref{eq:Bayes VAMP-SE} yields outgoing extrinsic
variance 
\[
    \tau_{2t} = C(\bar b_t)^2
      \left(\E[f(P_t,B^\star,\gamma)^2]
                   -\bar b_t^2\tau\right) = \frac{(m_\pi(\tau)
                   -m^2_\pi(\tau)/\tau)}{(1 -m_\pi(\tau)/\tau)^2} 
    =\frac{m_\pi(\tau)}{1-m_\pi(\tau)/\tau}
    =\cE_\pi(\tau),
\]
and outgoing extrinsic
variance precision $ \bgamma_{2t} = \tau_{2t}^{-1}.$
Now write \(\omega=\tau_{2t}\) and for $S=D^2$, let
\[
    c(\omega)
    :=\E\left[\frac{\sigma^2}{\sigma^2+\omega S}\right] .
\]
Let \(Q\sim N(0,\omega)\) independent of \((\zeta,S)\). Recall that $\zeta \sim  N(0,\sigma^2).$ Then, for the denoiser $g$ defined in \cref{eq:denoising-functions}, we have
\[
\begin{aligned}
 g\left(Q,(\zeta,\sqrt{S}),\omega^{-1}\right) &= \frac{\omega\sqrt S\,\zeta+\sigma^2Q}{\sigma^2+\omega S} \\
 &\qquad \implies \qquad 
    \E\!\left[
      g\left(Q,(\zeta,\sqrt S),\omega^{-1}\right)^2 \Big \lvert S
    \right]
    =\frac{\omega\sigma^2}{\sigma^2+\omega S}.
\end{aligned}
\]
Consequently, we have that
\[
\mathbb E[
  g(Q,(\zeta,\sqrt{S}),\omega^{-1})^2
]
=
\omega c(\omega).
\]
Moreover,
\[\bar c_t
    =\mathbb E[
  g'(Q,(\zeta,\sqrt{S}),\omega^{-1})
] = c(\omega).\]
The spectrum-stage update therefore gives
\[
\begin{aligned}
    \tau_{1,t+1} &= C(\bar c_t)^2
      \left(\E[g(Q_t,W,\omega^{-1})^2]
                   -\bar c_t^2\omega\right) \\
    &= \frac{1}{(1 - c(\omega))^2}
      \left(\omega c(\omega)
                   -\bar \omega c^2(\omega)\right)
    \\
    &=\frac{\omega c(\omega)}{1-c(\omega)}
    =L_\mu(\omega).
\end{aligned}
\]
Combining the two stages proves the claim.    
\end{proof}

\subsection{Greatest-fixed-point selection and Bayes-VAMP risk}\label{sec:bayes-vamp-risk-selection}

In this section, we show that the asymptotic MSE of Bayes VAMP in \cref{alg: Bayes VAMP} converges to the value characterized by the greatest fixed point of the state-evolution equations.

The spectrum-based initialization is
\(\tau_{0,+}^B=L_\mu(V_\pi)\) for  \(V_\pi := \operatorname{Var}_\pi(B^\star)>0\).  Since
\(\cE_\pi(\tau)\leq V_\pi\) by \cref{lem:mmsecalculus}, the initialization lies above the image of the scalar
state-evolution map.  In the next lemma, we show that monotonicity then determines which fixed point is
selected.

\begin{lemma}[Greatest-fixed-point selection]
\label{lem: Bayes-fpt-limit}
Let $\tau_{0,+}^B=L_\mu(V_\pi),$ and 
    \[\tau_{t+1,+}^B=T_{\mu,\pi}(\tau_{t,+}^B).\]
Then
\[
    \tau_{t,+}^B\downarrow\tau_B^+
    =\max\{\tau>0:T_{\mu,\pi}(\tau)=\tau\}.
\]
\end{lemma}

\begin{proof}
By \cref{lem:mmsecalculus}, we have 
\(
    m_\pi(\tau)
    \leq
    \frac{V_\pi\tau}{V_\pi+\tau},
\)
which, by the definition of \(\cE_\pi\) in \cref{eq:extrinsic}, implies
\(\cE_\pi(\tau)\leq V_\pi\). Therefore, since \(L_\mu\) is
nondecreasing, we have
\(
    T_{\mu,\pi}(\tau)
    =
    L_\mu\bigl(\cE_\pi(\tau)\bigr)
    \leq
    L_\mu(V_\pi)
    =
    \tau_{0,+}^B
\)
for every \(\tau>0\).
Thus \(\tau_{1,+}^B\leq\tau_{0,+}^B\).  Since
\(T_{\mu,\pi}\) is nondecreasing, induction shows that
\(\{\tau_{t,+}^B\}\) is nonincreasing.  It is bounded below by
\(L_\mu(0)>0\), so it converges to some \(\tau_\infty>0\).  Continuity 
of \(T_{\mu,\pi}\) by \cref{lem:mmsecalculus,lem:Lcalculus} gives
\(T_{\mu,\pi}(\tau_\infty)=\tau_\infty\).
If \(\tau_\star\) is any fixed point, then
\(\tau_\star=T_{\mu,\pi}(\tau_\star)\leq\tau_{0,+}^B\).
Monotonicity gives \(\tau_\star\leq\tau_{t,+}^B\) for every \(t\); hence, \(\tau_\star\leq\tau_\infty\).  Therefore,
\(\tau_\infty\) is the greatest fixed point that satisfies $T_{\mu, \pi}(\tau) = \tau$, namely \(\tau_B^+\).
\end{proof}

We next state conditions needed to transfer the scalar calculation to the
finite-dimensional iterates.  
In particular, under two additional mild assumptions stated below, the asymptotic MSE of $\widehat\beta_{1,t}$ in \cref{alg: Bayes VAMP} converges to the quantities defined by the state evolution parameters in \cref{eq:Bayes VAMP-SE}. We show this in \cref{prop:bayes-vamp-risk}.

The first additional assumption is a convenient sufficient
condition for the posterior-mean denoiser regularity in the VAMP theorem.

\begin{assumption}[Prior regularity; Assumption 1.3 of
\textcite{LiFanSenWu2024RotInv}]
\label{ass:extra-regularity-prior}
Recall that $B^\star \sim \pi$ with $b_\pi = \E[B^\star]$, and $V_\pi = \VAR(B^\star)$.
Let \(B_0=B^\star-b_\pi\sim\pi_0\), where $\pi_0$ is the centered distribution of $\pi$.  We assume that there exists \(C_\pi>0\) such that
\(V_\pi\leq C_\pi\) and
\[
    \PP(|B_0|>s)
    \leq 2\exp\left(-\frac{s^2}{2C_\pi}\right)
    \qquad\text{for every }s>0.
\]
Moreover, for \(k\in\{1,2\}\), let \(A\in\R^{k\times k}\) be
symmetric with \(A\prec(4C_\pi)^{-1}I_k\) and let \(z\in\R^k\). Define the probability measure
\(\nu_{A,z}\) on \(\mathbb R^k\) by
\[
    d\nu_{A,z}(x)
    :=
    \frac{\exp(x^TAx+x^Tz)\prod_{i=1}^k d\pi_0(x_i)}
         {\int_{\R^k}\exp(u^TAu+u^Tz)
               \prod_{i=1}^k d\pi_0(u_i)}.
\]
We assume that there exists \(K_\pi>0\), depending only on \(C_\pi\), such that for
every unit vector \(v\in\R^k\),
\[
    \operatorname{Var}_{\nu_{A,z}}(v^Tx)\leq K_\pi,
    \qquad
    \operatorname{Var}_{\nu_{A,z}}((v^Tx)^2)
    \leq K_\pi\{1+\E_{\nu_{A,z}}[(v^Tx)^2]\}.
\]
\end{assumption}

The first condition of \cref{ass:extra-regularity-prior} requires the centered prior to have sub-Gaussian
tails, so that large signal values are exponentially unlikely. The second condition requires that, after every admissible quadratic-linear tilt, the variance of \(v^T x\) remains uniformly bounded and the variance of \((v^T x)^2\) remains controlled by its
second moment.

\begin{assumption}[Bayes-VAMP state-evolution regularity]
\label{ass:bayes-denoiser-regularity}
Fix \(T\in\N\).  Assume the following conditions hold along the
state-evolution trajectory through iteration \(T\):
\begin{enumerate}[label=(\roman*)]
\item \(C\) defined in \cref{eq:denoising-functions} is continuous at \(\bar b_t\) for \(t=0,\ldots,T-1\), and
at \(\bar c_t\) for \(t=-1,0,\ldots,T-1\).
\item \(\Gamma\) defined in \cref{eq:denoising-functions} is continuous at
\((\bar\gamma_{1t},\bar b_t)\) for \(t=0,\ldots,T-1\), and at
\((\bar\gamma_{2t},\bar c_t)\) for \(t=-1,0,\ldots,T-1\).
\item For every \(t=0,\ldots,T\), there is a neighborhood
\(I_{1t}\subset(0,\infty)\) of \(\bar\gamma_{1t}\) on which
\((p,\beta)\mapsto f(p,\beta,\gamma_1)\) and
\((p,\beta)\mapsto f'(p,\beta,\gamma_1)\) are jointly Lipschitz in
\((p,\beta)\), with a Lipschitz constant uniform over
\(\gamma_1\in I_{1t}\).
\item For every \(t=-1,0,\ldots,T-1\), there is a neighborhood
\(I_{2t}\subset(0,\infty)\) of \(\bar\gamma_{2t}\) on which, with
\(w\) fixed, \(q\mapsto g(q,w,\gamma_2)\) and
\(q\mapsto g'(q,w,\gamma_2)\) are uniformly Lipschitz in \(q\),
uniformly over \(\gamma_2\in I_{2t}\).
\end{enumerate}
\end{assumption}

Under \cref{ass1,ass2,ass3,ass:extra-regularity-prior},
\cref{ass:bayes-denoiser-regularity} holds for every fixed \(T\).
Indeed, (i)--(ii) follow because
\(\bar b_t,\bar c_t\in(0,1)\) and
\(\bar\gamma_{1t},\bar\gamma_{2t}>0\) so the functions
\(C, \Gamma\) are continuous at the values where
they are evaluated; condition (iii) follows from \cref{ass3}, \cref{ass:extra-regularity-prior}, and
\textcite[Appendix~A]{LiFanSenWu2024RotInv}; and condition (iv), including
the initialization, follows from
\textcite[Lemma~7]{Fletcher2018Plugin}.

\begin{proposition}[Bayes VAMP converges to the greatest-fixed-point risk]
\label{prop:bayes-vamp-risk}
Suppose \cref{ass1,ass2,ass3,ass:bayes-denoiser-regularity} hold for every fixed iteration
count, and suppose \(V_\pi>0\).  Let \(\widehat\beta_{1,t}\) be the
output of \cref{alg: Bayes VAMP} after \(t\) denoising stages,
and define $R_{p,t}^B
    :=\frac1p\|\widehat\beta_{1,t}-\beta^\star\|_2^2.$
For the sequence $\{\tau_{t, +}^B\}$ defined in \cref{lem: Bayes-fpt-limit}, for every fixed integer \(t\geq0\), we have $
    R_{p,t}^B
    \xrightarrow{\PP}
    m_\pi(\tau_{t,+}^B).$
Moreover, for every \(\varepsilon>0\),
\begin{equation}
\label{eq:iterated-bayes-vamp-risk}
    \lim_{t\to\infty}\limsup_{p\to\infty}
    \PP\left(
      \left|R_{p,t}^B-m_\pi(\tau_B^+)\right|>\varepsilon
    \right)
    =0.
\end{equation}
Consequently, there exists a sequence \(t_p\uparrow\infty\) such that $R_{p,t_p}^B\xrightarrow{\PP}m_\pi(\tau_B^+).$
\end{proposition}


\begin{proof}
By \textcite[Theorem~2]{Fletcher2018Plugin}, for every fixed \(t\) and every pseudo-Lipschitz
function \(\psi:\mathbb R^2\to\mathbb R\) of order two, with \(B^\star\sim\pi\) and \(Z\sim N(0,1)\) where
\(B^\star\perp Z\), we have
\[
\frac1p\sum_{j=1}^p
\psi\left(\widehat\beta_{1,t,j},\beta_j^\star\right)
\xrightarrow{\mathbb P}
\mathbb E\left[
  \psi\left(
    \eta_{\bar\gamma_{1t}}^{\rm B}
    \left(B^\star+\sqrt{\tau_{1t}}Z \right),
    B^\star
  \right)
\right].
\]
Because the Bayes denoiser and observation-noise precision are matched, and the initialization satisfies \(\bar\gamma_{10}=\tau_{10}^{-1}\), we may apply Theorem 2 of \textcite{Rangan2018vamp}, which gives that \(\bar\gamma_{1t}=\tau_{1t}^{-1}\) and \(\bar\gamma_{2t}=\tau_{2t}^{-1}\) for every fixed \(t\). \cref{lem: SE-bvamp-tau1} then gives
\(
    \tau_{1,t+1} = T_{\mu, \pi}(\tau_{1t}).
\)
Since \(\tau_{10}=L_\mu(V_\pi)=\tau_{0,+}^B\), it follows by
induction that \(\tau_{1t}=\tau_{t,+}^B\) for every \(t\geq0\). Then, taking 
\(\psi(x,b)=(x-b)^2\) and using the
identity \(\bar\gamma_{1t}=\tau_{1t}^{-1}\) give the first conclusion:
\[
R_{p,t}^B =
\frac1p
\left\|
  \widehat\beta_{1,t}-\beta^\star
\right\|_2^2 \xrightarrow{\mathbb P}
\mathbb E\left[
  \left(
    \eta_{\tau_{1t}^{-1}}^{\rm B}
    (B^\star+\sqrt{\tau_{1t}}Z)-B^\star
  \right)^2
\right] =
m_\pi(\tau_{1t})
=
m_\pi(\tau_{t,+}^B).
\]
We now aim to show \cref{eq:iterated-bayes-vamp-risk}. Fix \(\varepsilon>0\). By the triangle inequality,
\begin{align*}
&\mathbb P\left(
    \left|
        R_{p,t}^B-m_\pi(\tau_B^+)
    \right|>\varepsilon
\right)\leq
\mathbb P\left(
    \left|
        R_{p,t}^B-m_\pi(\tau_{t, +}^B)
    \right|>\frac{\varepsilon}{2}
\right)
+
\mathbf 1\left\{
    \left|
        m_\pi(\tau_{t, +}^B)-m_\pi(\tau_B^+)
    \right|>\frac{\varepsilon}{2}
\right\}.
\end{align*}
For every fixed \(t\), the first term tends to zero as
\(p\to\infty\) since $ R_{p,t}^B
    \xrightarrow{\PP}
    m_\pi(\tau_{t,+}^B)$ as just shown. Therefore,
\begin{align*}
&\limsup_{p\to\infty}
\mathbb P\left(
    \left|
        R_{p,t}^B-m_\pi(\tau_B^+)
    \right|>\varepsilon
\right) \leq
\mathbf 1\left\{
    \left|
        m_\pi(\tau_{t, +}^B)-m_\pi(\tau_B^+)
    \right|>\frac{\varepsilon}{2}
\right\}.
\end{align*}
By \cref{lem: Bayes-fpt-limit} we have that
$\tau_{t, +}^B \xlongrightarrow{t \rightarrow \infty} \tau_B^+.
$ Since the function $m_\pi(\tau)$ is continuous in its
input, $ m_\pi(\tau_{t, +}^B)
    \xlongrightarrow{t \rightarrow \infty}
    m_\pi(\tau_B^+).$
Letting \(t\to\infty\), this proves
\cref{eq:iterated-bayes-vamp-risk}.

We note that \cref{alg: Bayes VAMP} has three stopping rules (\cref{algline:bayes-initial-safeguard,algline:bayes-denoiser-safeguard,algline:bayes-linear-safeguard}). On the event that these stopping rules are not invoked, the algorithm produces exactly the iterates analyzed in \textcite[Theorem~2]{Fletcher2018Plugin}. The state evolution asymptotics then show that the probability of termination tends to zero, as the quantities defining the stopping rules will converge in probability to values within the admissible range. Therefore, the fixed-$t$ limiting risk given by Fletcher et al.'s result also holds for the output of \cref{alg: Bayes VAMP}.
\end{proof}

\begin{rem}[Degenerate prior]
If \(V_\pi=0\), then \(\pi\) is a point mass at \(b_\pi\) and our initialization is not well deined.  The estimator
\(b_\pi\vecone_p\) has zero normalized risk, equal to
\(m_\pi(\tau_B^+)\), and no precision \(V_\pi^{-1}\) is needed.
\end{rem}

\subsection{A fixed-point dictionary for later proofs}
\label{sec:bayes-vamp-fixed-point-dictionary}

The scalar state evolution in \cref{eq:Bayes VAMP-SE} is the most economical description of Bayes VAMP,
but later proofs will compare it with the four-parameter fixed-point equations
for convex VAMP.  We therefore record the exact correspondence here.

Let \((\tau_{1\star},\tau_{2\star},\bar\gamma_{1\star},
\bar\gamma_{2\star})\) be a stationary state of
\cref{eq:Bayes VAMP-SE} under the matched initialization, and define $\eta_\star
    :=\bar\gamma_{1\star}+\bar\gamma_{2\star}.$
In other words, this vector solves
\begin{align}
\bar b_\star
&=\E[(\eta_{\bar\gamma_{1\star}}^{\rm B})'
       (B^\star+P_\star)],
\label{eq:vamp-fp-b}\\
\tau_{2\star}
&=C(\bar b_\star)^2
  \{\E[f(P_\star,B^\star,\bar\gamma_{1\star})^2]
       -\bar b_\star^2\tau_{1\star}\},
\label{eq:vamp-fp-tau2}\\
\bar\gamma_{2\star}
&=\Gamma(\bar\gamma_{1\star},\bar b_\star),
\label{eq:vamp-fp-gamma2}\\
\bar c_\star
&=\E[g'(Q_\star,W,\bar\gamma_{2\star})],
\label{eq:vamp-fp-c}\\
\tau_{1\star}
&=C(\bar c_\star)^2
  \{\E[g(Q_\star,W,\bar\gamma_{2\star})^2]
       -\bar c_\star^2\tau_{2\star}\},
\label{eq:vamp-fp-tau1}\\
\bar\gamma_{1\star}
&=\Gamma(\bar\gamma_{2\star},\bar c_\star),
\label{eq:vamp-fp-gamma1}
\end{align}
where \(P_\star\sim N(0,\tau_{1\star})\) is independent of
\(B^\star\), and \(Q_\star\sim N(0,\tau_{2\star})\) is independent of
\(W\).  The two precision-update equations imply
$ \bar b_\star
    ={\bar\gamma_{1\star}}/{\eta_\star},$ and 
    $\bar c_\star
    ={\bar\gamma_{2\star}}/{\eta_\star}.$

Using \(S=D^2\), theses stationary equations can equivalently be
written as
\begin{align}
\frac{\bar\gamma_{1\star}}{\eta_\star}
&=\E\left[
  \frac{\dd}{\dd y}
  \E[B^\star\mid B^\star+\sqrt{\tau_{1\star}}Z=y]
  \right],
\label{eq: fpt-bvamp-a}\\
\tau_{2\star}
&=\left(\frac{\eta_\star}
              {\eta_\star-\bar\gamma_{1\star}}\right)^2
  \left(
    m_\pi(\tau_{1\star})
    -\left(\frac{\bar\gamma_{1\star}}{\eta_\star}\right)^2
      \tau_{1\star}
  \right),
\label{eq: fpt-bvamp-b}\\
\frac{\bar\gamma_{1\star}}{\eta_\star}
&=1-\frac1{\sigma^2\eta_\star}
G^{-1}\left(\frac1{\sigma^2\eta_\star}\right),
\label{eq: fpt-bvamp-c}\\
\tau_{1\star}
&=\left(\frac{\eta_\star}{\bar\gamma_{1\star}}\right)^2
  \left(
    \E\left[
      \frac{S/\sigma^2+
        (\eta_\star-\bar\gamma_{1\star})^2\tau_{2\star}}
      {(S/\sigma^2+\eta_\star-
        \bar\gamma_{1\star})^2}
    \right]
    -\left(
       \frac{\eta_\star-\bar\gamma_{1\star}}{\eta_\star}
     \right)^2\tau_{2\star}
  \right).
\label{eq: fpt-bvamp-d}
\end{align}
Here \(G\) is the Stieltjes transform from \cref{eq:stieltjes-transform}, with range \((0,\rho_{\max})\), where
\begin{equation}
\label{eq:rho_max}
    \rho_{\max}
    :=\lim_{c\downarrow0}G(c)
    =\E[1/S]\in(0,\infty].
\end{equation}

The correspondence with the scalar fixed point is immediate: if
\(T_{\mu,\pi}(\tau)=\tau\), set
\[
    \tau_{1\star}=\tau,
    \qquad
    \tau_{2\star}=\cE_\pi(\tau),
    \qquad
    \bar\gamma_{1\star}=\tau_{1\star}^{-1},
    \qquad
    \bar\gamma_{2\star}=\tau_{2\star}^{-1},
    \qquad
    \eta_\star=\bar\gamma_{1\star}+\bar\gamma_{2\star}.
\]
Then \cref{eq: fpt-bvamp-a,eq: fpt-bvamp-b,eq: fpt-bvamp-c,eq: fpt-bvamp-d} hold.  Conversely,
any matched solution of these equations satisfies
\(\tau_{2\star}=\cE_\pi(\tau_{1\star})\) and
\(\tau_{1\star}=L_\mu(\tau_{2\star})\), and therefore
\(T_{\mu,\pi}(\tau_{1\star})=\tau_{1\star}\).

\section{\texorpdfstring{Proofs of \cref{thm:perturbed-fixed-point,thm:perturbed-convex-mse}}{Proofs of the perturbed fixed-point lemma and asymptotic-risk theorem}}\label{sec:main_result1_proof}

This appendix proves the two perturbed convex-VAMP results in \cref{thm:perturbed-fixed-point,thm:perturbed-convex-mse}. To begin with, we derive the perturbed convex-VAMP fixed-point system
\cref{eq: fpt-oracle-a,eq: fpt-oracle-b,eq: fpt-oracle-c,eq: fpt-oracle-d}.  We then prove two
logically separate statements.  First, given any solution of this system,
we establish the empirical-law conclusion in \cref{thm:perturbed-convex-mse}.  This step adapts the
fixed-point-initialized VAMP analysis of
\textcite[Appendix~C]{Li2026debiasing}; the precise adaptation is stated
in \cref{lem:li-sur-fixed-time-transfer}.  Second, we prove the existence
claim in \cref{thm:perturbed-fixed-point} for every \(\lambda>0\) by a
two-dimensional Poincar\'e--Miranda argument.

The quadratic perturbation plays two roles throughout.  It makes the
coordinatewise map defined in \cref{eq:perturbed-coordinatewise-prox} a
strict contraction; see \cref{eq:Dlambda-bound}.  It also makes the
objective in \cref{eq:perturbed_beta} \(\lambda\)-strongly convex, which
gives the optimizer approximation in
\cref{eq:oracle-roadmap-optimizer-tracking}.  Neither argument requires
uniqueness of the scalar fixed point.

\subsection{Perturbed convex-VAMP fixed-point equations} \label{sec:perturbed-convex-vamp-fp}  

Convex VAMP is the specialization of VAMP that can be used to find convex estimators of the form in \cref{eq: penalized-OLS}.
The denoising stage applies the proximal map from
\cref{defn:proximal-map}, while the linear stage solves a quadratic
problem involving \(X\) and \(y\). At a fixed point, the effective
variance and precision parameters entering these two stages remain
constant, and the scalar law of the proximal-stage output determines
the limiting risk. We use convex VAMP only as an analysis device: our
results do not require an implementable fixed-point initialization or
assert convergence from an arbitrary initialization. The perturbed
algorithm and its fixed-point equations are given in
\cref{alg:perturbed-convex-vamp,eq: fpt-oracle-a,eq: fpt-oracle-b,eq: fpt-oracle-c,eq: fpt-oracle-d}.

For reference, the stationary convex-VAMP equations for an unperturbed
penalty \(h\) are
\begin{align}
\frac{\bgs}{\eta_\star}
&=\E\left[\prox_{h/\bgs}'
  (B^\star+\sqrt{\tau_{1\star}}Z)\right],
\label{eq: fpt-equation-a}\\
\tau_{2\star}
&=\left(\frac{\eta_\star}{\eta_\star-\bgs}\right)^2
\left\{
\E\left[\left(\prox_{h/\bgs}
  (B^\star+\sqrt{\tau_{1\star}}Z)-B^\star\right)^2\right]
-\left(\frac{\bgs}{\eta_\star}\right)^2\tau_{1\star}
\right\},
\label{eq: fpt-equation-b}\\
\frac{\bgs}{\eta_\star}
&=1-\frac1{\eta_\star}G^{-1}\left(\frac1{\eta_\star}\right),
\label{eq: fpt-equation-c}\\
\tau_{1\star}
&=\left(\frac{\eta_\star}{\bgs}\right)^2
\left\{
\E\left[
\frac{\sigma^2S+(\eta_\star-\bgs)^2\tau_{2\star}}
     {(S+\eta_\star-\bgs)^2}
\right]
-\left(\frac{\eta_\star-\bgs}{\eta_\star}\right)^2\tau_{2\star}
\right\}.
\label{eq: fpt-equation-d}
\end{align}
We now apply \cref{eq: fpt-equation-a}, \cref{eq: fpt-equation-b}, \cref{eq: fpt-equation-c}, and \cref{eq: fpt-equation-d}
to the separable perturbed penalty in \cref{eq:perturbed_beta}.

Fix \(\lambda>0\) and define
\(\phi_{\lambda,b}(x)=h(x)+\lambda(x-b)^2/2\).  Conditional on the
reference vector, the penalty in \cref{eq:perturbed_beta} is
\(\sum_{i=1}^p\phi_{\lambda,b_i}(\beta_i)\).  Thus only the proximal
stage of \cref{eq: fpt-equation-a,eq: fpt-equation-b,eq: fpt-equation-c,eq: fpt-equation-d}
changes.  For arbitrary \(\gamma>0\) and
\(b,r\in\R\), completing the square gives
\begin{equation}
\begin{aligned}
 \prox_{\phi_{\lambda,b}/\gamma}(r)
 &=\argmin_x\left\{
 \frac12(x-r)^2+\frac1\gamma h(x)
 +\frac\lambda{2\gamma}(x-b)^2\right\}\\
 &=\prox_{h/(\gamma+\lambda)}
 \left(\frac{\gamma r+\lambda b}{\gamma+\lambda}\right).
\end{aligned}
 \label{eq:oracle-prox-deriv}
\end{equation}
In particular, at the scalar effective observation
\(r=b^\star+\sqrt\tau Z\), we have
\begin{equation}
 \prox_{\phi_{\lambda,b}/\gamma}(b^\star+\sqrt\tau Z)
 =\prox_{h/(\gamma+\lambda)}
 \left(
 \frac{\gamma b^\star+\lambda b}{\gamma+\lambda}
 +\frac\gamma{\gamma+\lambda}\sqrt\tau Z
 \right).
 \label{eq:oracle-prox-fp-derivation}
\end{equation}
Taking the almost-everywhere derivative with respect to the effective observation gives
\begin{equation}
 \frac{\partial}{\partial r}
 \prox_{\phi_{\lambda,b}/\gamma}(r)
 =\frac\gamma{\gamma+\lambda}
 \prox_{h/(\gamma+\lambda)}'
 \left(\frac{\gamma r+\lambda b}{\gamma+\lambda}\right).
 \label{eq:oracle-prox-derivative-fp}
\end{equation}
We now set \(b^\star=B^\star\) and \(b=B\) with
\(\gamma=\bar\gamma_{1,\lambda}\) and
\(\tau=\tau_{1,\lambda}\) in
\cref{eq:oracle-prox-fp-derivation,eq:oracle-prox-derivative-fp}.
Substituting the resulting identities into
\cref{eq: fpt-equation-a}, and \cref{eq: fpt-equation-b} gives
\cref{eq: fpt-oracle-a,eq: fpt-oracle-b} below. Retaining the unchanged second pair of equations gives \cref{eq: fpt-oracle-c,eq: fpt-oracle-d}. Therefore, we call
\((\bar\gamma_{1,\lambda},\eta_\lambda,
   \tau_{1,\lambda},\tau_{2,\lambda})\)
a perturbed fixed point if
\(0<\bar\gamma_{1,\lambda}<\eta_\lambda\) and
\(\tau_{1,\lambda},\tau_{2,\lambda}\geq0\), and the fixed-point equations hold:
\begin{align}
\frac{\bar\gamma_{1,\lambda}}{\eta_\lambda}
&=
\frac{\bar\gamma_{1,\lambda}}
     {\lambda+\bar\gamma_{1,\lambda}}
\E\Bigg[
\prox_{\frac{h}{\lambda+\bar\gamma_{1,\lambda}}}'
\left(
\frac{\bar\gamma_{1,\lambda}B^\star+\lambda B}
     {\lambda+\bar\gamma_{1,\lambda}}
+
\frac{\bar\gamma_{1,\lambda}}
     {\lambda+\bar\gamma_{1,\lambda}}
\sqrt{\tau_{1,\lambda}}Z
\right)
\Bigg],
\label{eq: fpt-oracle-a}\\
\tau_{2,\lambda}
&=\begin{aligned}[t]
&\Big(\frac{\eta_\lambda}
{\eta_\lambda-\bar\gamma_{1,\lambda}}\Big)^2
\Bigg(
\E\Bigg[
\Bigg(
\prox_{\frac{h}{\lambda+\bar\gamma_{1,\lambda}}}
\left(
\frac{\bar\gamma_{1,\lambda}B^\star+\lambda B}
     {\lambda+\bar\gamma_{1,\lambda}}
+
\frac{\bar\gamma_{1,\lambda}}
     {\lambda+\bar\gamma_{1,\lambda}}
\sqrt{\tau_{1,\lambda}}Z
\right)-B^\star
\Bigg)^2\Bigg]
\\ &-\Big(\frac{\bar\gamma_{1,\lambda}}{\eta_\lambda}\Big)^2
\tau_{1,\lambda}
\Bigg),
\end{aligned}
\label{eq: fpt-oracle-b}\\
\frac{\bar\gamma_{1,\lambda}}{\eta_\lambda}
&=1-\frac1{\eta_\lambda}G^{-1}\left(\frac1{\eta_\lambda}\right),
\label{eq: fpt-oracle-c}\\
\tau_{1,\lambda}
&=
\Big(\frac{\eta_\lambda}{\bar\gamma_{1,\lambda}}\Big)^2
\Bigg(
\E\Bigg[
\frac{\sigma^2S+
(\eta_\lambda-\bar\gamma_{1,\lambda})^2\tau_{2,\lambda}}
{(S+\eta_\lambda-\bar\gamma_{1,\lambda})^2}
\Bigg]
-\Big(\frac{\eta_\lambda-\bar\gamma_{1,\lambda}}
{\eta_\lambda}\Big)^2\tau_{2,\lambda}
\Bigg).
\label{eq: fpt-oracle-d}
\end{align}
The first two equations describe the perturbed proximal stage, and the last two describe the linear stage and its spectral dependence. The existence of a solution for every \(\lambda>0\) is established in \cref{thm:perturbed-fixed-point}, which is proved after \cref{thm:perturbed-convex-mse}.

Fix a solution to \cref{eq: fpt-oracle-a,eq: fpt-oracle-b,eq: fpt-oracle-c,eq: fpt-oracle-d} and abbreviate
$\eta=\eta_\lambda$ and 
 $c=\eta-\gamma.$
For \(b,r\in\R^p\), define \(T_b(r)\in\R^p\) coordinatewise by
\begin{equation}
 [T_b(r)]_i
 :=\prox_{\phi_{\lambda,b_i}/\gamma}(r_i),
 \qquad i\in[p].
 \label{eq:perturbed-coordinatewise-prox}
\end{equation}
For scalar arguments, the same notation denotes the case \(p=1\).
In \cref{alg:perturbed-convex-vamp}, these scalar parameters remain
fixed at their candidate fixed-point values.  The algorithm is used only
as a proof device; it is not an implementable initialization rule. In what follows, we adapt the stationary conditional-Haar argument that underlies \textcite[Appendix~C]{Li2026debiasing}. This result will provide a complete asymptotic analysis of \cref{alg:perturbed-convex-vamp}.

\begin{algorithm}
\caption{Perturbed convex VAMP}
\label{alg:perturbed-convex-vamp}
\begin{algorithmic}[1]
\State \textbf{Initialize: }
\(q^0\sim N(0,\tau_1I_p)\) independently of
\((U,\Sigma,V,\beta^\star,b,\varepsilon)\), and
\(\Br_{1,0}=\beta^\star+q^0\)
\For{\(t=1,2,\ldots\)}
\State
\(\widehat{\beta}_{1,t}^{h,\lambda,b}
      =T_b(\Br_{1,t-1})\)
\State
\(\Br_{2,t}
      =c^{-1}\bigl(
       \eta\widehat{\beta}_{1,t}^{h,\lambda,b}
       -\gamma\Br_{1,t-1}\bigr)\)
\State
\(\widehat{\beta}_{2,t}^{h,\lambda,b}
      =(X^T X+cI_p)^{-1}
       (X^T y+c\Br_{2,t})\)
\State
\(\Br_{1,t}
      =\gamma^{-1}\bigl(
       \eta\widehat{\beta}_{2,t}^{h,\lambda,b}
       -c\Br_{2,t}\bigr)\)
\EndFor
\State \textbf{Output at time \(t\): }
\(\widehat{\beta}_{1,t}^{h,\lambda,b}\).
\end{algorithmic}
\end{algorithm}

\subsection{\texorpdfstring{Proof of \cref{thm:perturbed-convex-mse}}{Proof of the perturbed convex asymptotic-risk theorem}}
\label{sec: oracle-convex-vamp-2}

Fix \(\lambda>0\) and any solution
\((\bar\gamma_{1,\lambda},\eta_\lambda,
\tau_{1,\lambda},\tau_{2,\lambda})\) of
\cref{eq: fpt-oracle-a,eq: fpt-oracle-b,eq: fpt-oracle-c,eq: fpt-oracle-d}.  Throughout this
proof, we will abbreviate the fixed point as \((\gamma,\eta,
\tau_{1},\tau_{2})\) with $c = \eta - \gamma$.

\paragraph{Rewriting the VAMP iteration.}
For scalar \(b,r\), the definition in
\cref{eq:perturbed-coordinatewise-prox} with \cref{eq:oracle-prox-deriv} give
\begin{equation}
 T_b(r)
 =
 \prox_{h/(\gamma+\lambda)}
 \left(\frac{\gamma r+\lambda b}{\gamma+\lambda}\right).
 \label{eq:oracle-prox-proof}
\end{equation}
Consequently, \(T_b\) is nondecreasing and
\(s_\lambda\)-Lipschitz in \(r\), where $s_\lambda:=\frac\gamma{\gamma+\lambda}<1$, by \cref{prop:prox-firm-app}.
This implies the map
\((q,b^\star,b)\mapsto T_b(b^\star+q)-b^\star\) is jointly
Lipschitz.  If $D_\lambda(q,b^\star,b)
 :=\partial_qT_b(b^\star+q),$
where we choose any jointly Borel version of the almost-everywhere
derivative, then this implies
\begin{equation}
 0\le D_\lambda(q,b^\star,b)\le s_\lambda
 \quad\text{for Lebesgue-a.e. }q.
 \label{eq:Dlambda-bound}
\end{equation}
More precisely, whenever \(H:\R^3\to\R\) is jointly Lipschitz,
\(\partial_qH(q,b^\star,b)\) denotes a jointly Borel representative of the one-dimensional almost-everywhere derivative of
\(q\mapsto H(q,b^\star,b)\). For example, we may take
\[
 \partial_qH(q,b^\star,b)
 =
 \begin{cases}
 \displaystyle
 \lim_{m\to\infty}
 m\bigl\{H(q+m^{-1},b^\star,b)-H(q,b^\star,b)\bigr\},
 &\text{if the limit exists},\\[1ex]
 0,&\text{otherwise}.
 \end{cases}
\]
For every fixed \((b^\star,b)\), this equals the weak derivative for
Lebesgue-a.e.\ \(q\). Consequently, it may be evaluated at
\((\sqrt{\tau_1}Z,B^\star,B)\), even when the joint law of
\((B^\star,B)\) is discrete or singular, and one-dimensional Gaussian integration by parts applies after conditioning on \((B^\star,B)\).

The Lipschitz and derivative bounds in
\cref{eq:oracle-prox-proof}, \cref{eq:Dlambda-bound} are the only regularity
properties of the penalty used in
\cref{lem:li-sur-fixed-time-transfer}.

We analyze \cref{alg:perturbed-convex-vamp} with the scalar
parameters fixed at the selected solution of
\cref{eq: fpt-oracle-a,eq: fpt-oracle-b,eq: fpt-oracle-c,eq: fpt-oracle-d}. Throughout the
proof, write $\widehat\beta_{1,t}
    :=
    \widehat\beta_{1,t}^{h,\lambda,b}.$
In what follows, we label \(X=U\Sigma V^T\), \(\xi=U^T\varepsilon\), and define
\begin{align}
 \Lambda
 &:=%
 \frac{\eta c}{\gamma}(\Sigma^T\Sigma+cI_p)^{-1}
 -\frac c\gamma I_p,
 \label{eq:oracle-Lambda}\\
 \widetilde e
 &:=%
 \frac\eta\gamma(\Sigma^T\Sigma+cI_p)^{-1}\Sigma^T\xi,
 \qquad e:=V\widetilde e,
 \label{eq:oracle-e}\\
 F_\lambda(q,b^\star,b)
 &:=%
 \frac\eta c\{T_b(b^\star+q)-b^\star\}
 -\frac\gamma c q.
 \label{eq:oracle-Flambda}
\end{align}
Let
\begin{equation}
\label{eq:xT_defs_and_others}
 x^t=r_{2,t}-\beta^\star,\qquad
 s^t=V^Tx^t,\qquad
 y^t=r_{1,t}-\beta^\star-e.
\end{equation}
Then substituting the updates in
\cref{alg:perturbed-convex-vamp} into these definitions gives
\begin{equation}
 s^t=V^Tx^t,\qquad
 y^t=V\Lambda s^t,\qquad
 x^{t+1}=F_\lambda(y^t+e,\beta^\star,b),
 \label{eq:oracle-Haar-recursion}
\end{equation}
with \(x^1=F_\lambda(q^0,\beta^\star,b)\).  
\cref{eq:oracle-Haar-recursion} has exactly the centered Haar recursion form analyzed in
\textcite[Appendix~C, Eq.~(77)]{Li2026debiasing}, with
\(F_\lambda\) in place of their nonlinear map. We write $\BX_t = [x^1, \ldots, x^t]$, $\BY_t = [y^1, \ldots, y^t]$, and $\BS_t = [s^1, \ldots, s^t]$.

Let \(S\sim\mu\), and define
\begin{equation}
 L
 =
 \frac c\gamma
 \left(\frac{\eta}{S+c}-1\right),
 \qquad
 \kappa=\E[L^2],
 \qquad
 b_\star
 =
 \frac{\eta^2\sigma^2}{\gamma^2}
 \E\left[\frac{S}{(S+c)^2}\right].
 \label{eq:oracle-roadmap-kappa-b}
\end{equation}
The spectrum stage fixed-point equations in \cref{eq: fpt-oracle-c,eq: fpt-oracle-d}
imply
\begin{equation}
 \E[L] =0,\quad \frac1p\operatorname{tr}(\Lambda)
 \xrightarrow{\PP}0,\quad \frac1p\operatorname{tr}(\Lambda^2) \xrightarrow{\PP}\kappa,\quad \frac1p\|e\|_2^2 \xrightarrow{\PP}b_\star,\quad \tau_1 =b_\star+\kappa\tau_2.
 \label{eq:oracle-roadmap-linear-stationarity}
\end{equation}
Indeed, \cref{eq: fpt-oracle-c} gives $G(c)=\E[(S+c)^{-1}]=1/\eta$,
which yields \(\E[L]=0\). The trace limits follow by empirical spectral convergence in \cref{ass1}. Conditional on \((U,\Sigma)\), the quadratic form defining \(\|e\|_2^2\) concentrates around \(b_\star\), and \cref{eq: fpt-oracle-d} then gives
\(\tau_1=b_\star+\kappa\tau_2\).

For the prior stage, let
\(
 Q=\sqrt{\tau_1}Z\), where
\(Z\sim N(0,1)\), and \(Z\perp(B^\star,B).\)
Differentiating \cref{eq:oracle-Flambda} with respect to \(q\) gives
\(
 \partial_qF_\lambda(q,b^\star,b)
 =
 (\eta / c)D_\lambda(q,b^\star,b)-\gamma / c,
\) where we recall $ D_\lambda(q,b^\star,b)
 =\partial_qT_b(b^\star+q).$

Fixed point equations in \cref{eq: fpt-oracle-a,eq: fpt-oracle-b}, together with conditional Gaussian
integration by parts, therefore give
\begin{align}
 \E[\partial_qF_\lambda(Q,B^\star,B)]&=0,
 \label{eq:oracle-roadmap-zero-jacobian}\\
 \E[F_\lambda(Q,B^\star,B)^2]&=\tau_2.
 \label{eq:oracle-roadmap-nonlinear-variance}
\end{align}
\cref{eq:oracle-roadmap-linear-stationarity}, \cref{eq:oracle-roadmap-zero-jacobian}, and
\cref{eq:oracle-roadmap-nonlinear-variance}
are the perturbed counterparts of the stationary identities in \textcite[Appendix~C, Prop.~C.3]{Li2026debiasing}.
The identity \cref{eq:oracle-roadmap-zero-jacobian} makes the cross terms vanish in the conditional-Haar induction, while the variance identities make all
state-evolution marginals stationary.

\paragraph{Applying Li and Sur's state-evolution argument.}
We now record the imported result from  \textcite{Li2026debiasing}.

\begin{lemma}[Fixed-time law of the perturbed VAMP iterates]
\label{lem:li-sur-fixed-time-transfer}
Let \(P_0\sim N(0,\tau_1)\), \(E\sim N(0,b_\star)\), and suppose
that \(P_0,E,(B^\star,B)\) are mutually independent.  Set
\(X_1=F_\lambda(P_0,B^\star,B)\).  Recursively define a centered
Gaussian process \((Y_j)_{j\ge1}\), independent of
\((P_0,E,B^\star,B)\), by
\begin{equation}
 \E[Y_iY_j]=\kappa\Delta_{ij},\qquad
 \Delta_{ij}:=\E[X_iX_j],\qquad
 X_{j+1}=F_\lambda(Y_j+E,B^\star,B).
 \label{eq:oracle-roadmap-limiting-process}
\end{equation}
Then, for every fixed \(t\geq1\), the iterates of
\cref{alg:perturbed-convex-vamp} obey the state evolution in
\cref{eq:oracle-roadmap-limiting-process}.  In
particular,
\begin{equation}
 \nu_p(\widehat\beta_{1,t},r_{1,t-1},\beta^\star,b)
 \xrightarrow{W_2}_{\PP}
 \operatorname{Law}\bigl(
 T_B(B^\star+\sqrt{\tau_1}Z),
 B^\star+\sqrt{\tau_1}Z,B^\star,B
 \bigr),
 \label{eq:li-sur-transfer-law}
\end{equation}
and, for every fixed \(s,t\),
\begin{equation}
 \frac1p\langle x^s,x^t\rangle\xrightarrow{\PP}\Delta_{st},
 \qquad
 \frac1p\|x^t-x^s\|_2^2
 \xrightarrow{\PP} \Delta_{tt}+\Delta_{ss}-2\Delta_{st}.
 \label{eq:li-sur-transfer-overlaps}
\end{equation}
\end{lemma}

\begin{proof}
This follows by a direct adaptation of the proof of
\textcite[Appendix~C, Prop.~C.5 and Cor.~C.7]{Li2026debiasing}.
Their recursion already has the form
\(x^{t+1}=F(y^t+e,\beta^\star)\); in the present setting one simply
replaces the side-information variable \(\beta^\star\) by the pair
\((\beta^\star,b)\) and replaces their nonlinear map by
\(F_\lambda\) in \cref{eq:oracle-Flambda}.
\cref{ass:general-reference-vector} supplies the required joint
empirical \(W_2\) limit of this enlarged side information and its
independence from the Haar matrix.  The matrix \(\Lambda\) and vector
\(e\) in \cref{eq:oracle-Lambda}, and \cref{eq:oracle-e} have the same form as in
their argument, while
\cref{eq:oracle-roadmap-linear-stationarity} verifies the required
zero-trace and limiting second-moment conditions.

The induction in that proof uses the penalty only through continuity,
linear growth, and the average weak derivative of its scalar update.
These properties hold here because
\((q,b^\star,b)\mapsto F_\lambda(q,b^\star,b)\) is jointly
Lipschitz and its Borel weak derivative satisfies
\cref{eq:Dlambda-bound}.  Conditional Gaussian integration by parts
is applied after fixing \((B^\star,B)\), so no density or smoothness
of the side-information law is required.  Thus their induction applies
with the enlarged side-information vector.  The same argument, with
Moore--Penrose inverses (or,
equivalently, an arbitrarily small Gaussian regularization followed by
continuity) when a finite history covariance is singular, gives
\cref{eq:li-sur-transfer-law,eq:li-sur-transfer-overlaps}.
\end{proof}

The identities
\cref{eq:oracle-roadmap-linear-stationarity,eq:oracle-roadmap-nonlinear-variance} imply inductively that
\begin{equation}
 \Delta_{jj}=\E[X_j^2]=\tau_2,\qquad
 Y_j+E\sim N(0,\tau_1),\qquad j\ge1.
 \label{eq:oracle-roadmap-stationary-marginals}
\end{equation}
Thus \cref{eq:li-sur-transfer-law} yields, for every fixed \(t\),
\begin{equation}
 \frac1p\|\widehat\beta_{1,t}-\beta^\star\|_2^2
 \xrightarrow{\PP}
 \E\left[
 \{T_B(B^\star+\sqrt{\tau_1}Z)-B^\star\}^2
 \right]
 =:r_\lambda.
 \label{eq:oracle-fixed-time-risk}
\end{equation}

Let
\begin{equation}
 \mathcal L_{\lambda,p}(\beta)
 =
 \frac12\|y-X\beta\|_2^2
 +\sum_{i=1}^ph(\beta_i)
 +\frac\lambda2\|\beta-b\|_2^2,
 \label{eq:calLdef}
\end{equation}
where \(\widehat\beta_{\cvx}\) is its unique
minimizer. 
The fixed-time law in
\cref{eq:li-sur-transfer-law} does not yet identify asymptotics for \(\widehat\beta_{\cvx}\). To make that identification, we must show that the
VAMP estimate approaches the unique minimizer of
\(\mathcal L_{\lambda,p}\).

For \(t\geq2\), define $u_t
 :=
 X^T(X\widehat\beta_{1,t}-y)
 +\gamma(r_{1,t-1}-\widehat\beta_{1,t})$.
 By the proximal first-order condition in
\cref{lem:prox-optimality-app}, the first-stage update $\widehat{\beta}_{1,t}
      =T_b(\Br_{1,t-1})$ in
\cref{alg:perturbed-convex-vamp} satisfies
$\gamma
    (r_{1,t-1}-\widehat\beta_{1,t})
    \in
    \partial h^{(\lambda,b)}(\widehat\beta_{1,t}).$
Since $\partial\mathcal L_{\lambda,p}(\beta)
    =
    X^T(X\beta-y)+\partial h^{(\lambda,b)}(\beta),$
we obtain
$u_t \in
    \partial\mathcal L_{\lambda,p}(\widehat\beta_{1,t}).$

Combining the equations from the two VAMP stages gives the exact identity
\[
\begin{aligned}
 u_t
 &=
 \frac c\eta
 (X^T X-\gamma I)
 (r_{2,t}-r_{2,t-1}), \\
 &\quad \implies \quad 
\frac1p\|u_t\|_2^2
 \leq
 \left(\frac c\eta\right)^2
 \bigl(\|X\|_{\mathrm{op}}^2+\gamma\bigr)^2
 \frac1p\|r_{2,t}-r_{2,t-1}\|_2^2.
\end{aligned}
\]
 In what follows, we will prove
 \begin{equation}
 \lim_{T\to\infty}
 \sup_{s,t\geq T}
 \limsup_{p\to\infty}
 \PP\left(
 \frac1p\|r_{2,t}-r_{2,s}\|_2^2>\epsilon
 \right)
 =0.
 \label{eq:oracle-r2-cauchy}
\end{equation}
Then, the operator-norm assumption in \cref{ass1} along with
\cref{eq:oracle-r2-cauchy} imply, for every \(\epsilon>0\), $\lim_{t\to\infty}
 \limsup_{p\to\infty}
 \PP\left(
 p^{-1}\|u_t\|_2^2>\epsilon
 \right)
 =0$.
Since \(\mathcal L_{\lambda,p}\) is
\(\lambda\)-strongly convex, strong monotonicity of its
subdifferential gives $\|\widehat\beta_{1,t}
 -\widehat\beta_{\cvx}\|_2
 \leq
 \lambda^{-1}\|u_t\|_2$.
Hence, for every \(\epsilon>0\),
\begin{equation}
 \lim_{t\to\infty}
 \limsup_{p\to\infty}
 \PP\left(
 \frac1p
 \|\widehat\beta_{1,t}
 -\widehat\beta_{\cvx}\|_2^2
 >\epsilon
 \right)
 =0.
 \label{eq:oracle-roadmap-optimizer-tracking}
\end{equation}
This is the same optimizer convergence result as
\textcite[Appendix~C, Prop.~C.10]{Li2026debiasing};
here the required strong convexity is supplied automatically by the
quadratic perturbation.

It remains only to pass from the fixed-time iterate law in \cref{eq:li-sur-transfer-law} to a law for the optimizer without interchanging the \(p\)- and \(t\)-limits. For every fixed \(t\), the coupling that pairs the coordinates of
\((\widehat\beta_{\cvx}^{h,\lambda,b},\beta^\star,b)\) and
\((\widehat\beta_{1,t},\beta^\star,b)\) gives
\[
\begin{aligned}
 &W_2\!\left(
 \nu_p(\widehat\beta_{\cvx},\beta^\star,b),
 \nu_p(\widehat\beta_{1,t},\beta^\star,b)
 \right)\leq
 \frac1{\sqrt p}
\left \|\widehat\beta_{\cvx}-\widehat\beta_{1,t}\right\|_2.
\end{aligned}
\]
For each fixed \(t\), projecting the convergence in
\cref{eq:li-sur-transfer-law} onto its first, third, and fourth
coordinates gives the limit of the second empirical law above, namely, $\nu_p(\widehat\beta_{1,t},\beta^\star,b)$, to \(
 \operatorname{Law}(
 T_B(B^\star+\sqrt{\tau_1}Z),B^\star,B)
\).  Taking
\(\limsup_{p\to\infty}\) first and then \(t\to\infty\), and using
\cref{eq:oracle-roadmap-optimizer-tracking}, the triangle inequality yields
\begin{equation}
 \nu_p(\widehat\beta_{\cvx},\beta^\star,b)
 \xrightarrow{W_2}_{\PP}
 \operatorname{Law}\left(
 T_B(B^\star+\sqrt{\tau_1}Z),B^\star,B
 \right).
 \label{eq:general-joint-law-before-prox-rewrite}
\end{equation}

Finally, \cref{eq:oracle-prox-proof,eq:alpha_and_nu} give
\[
\begin{aligned}
T_B(B^\star+\sqrt{\tau_1}Z)&=
 \prox_{h/(\gamma+\lambda)}
 \left(
 \frac{\gamma B^\star+\lambda B}{\gamma+\lambda}
 +\frac{\gamma}{\gamma+\lambda}\sqrt{\tau_1}Z
 \right)\\
 &=
 \prox_{\alpha_\lambda h}
 \left(
 \frac{\bar\gamma_{1,\lambda}B^\star + \lambda B}
      {\bar\gamma_{1,\lambda}+\lambda}
 +\sqrt{\nu_\lambda}Z
 \right).
\end{aligned}
\]
Substituting this identity into
\cref{eq:general-joint-law-before-prox-rewrite} proves the convergence
to the joint law
\[
    \frac1p\sum_{j=1}^p
    \delta_{(\widehat\beta_{h,j}^{\lambda,b},\,\beta_j^\star,\,b_j)}
    \xrightarrow{W_2}_{\PP}
    \operatorname{Law}\left(
        \prox_{\alpha_\lambda h}
        \left(
            \frac{\bar\gamma_{1,\lambda}}
                 {\bar\gamma_{1,\lambda}+\lambda}B^\star
            +\frac{\lambda}
                  {\bar\gamma_{1,\lambda}+\lambda}B
            +\sqrt{\nu_\lambda}Z
        \right),
        B^\star,
        B
    \right).
\]
Applying the continuous map $(x,b^\star,b)\longmapsto (x-b^\star)^2$
gives \cref{eq:general-signal-risk}, since \(W_2\) convergence includes convergence of second moments.

\paragraph{Correlations between different iterations.}
It remain to prove \cref{eq:oracle-r2-cauchy}.
First, if
\(\tau_2=0\), then
\cref{eq:oracle-roadmap-stationary-marginals} gives
\(\Delta_{tt}=0\) for every \(t\), and Cauchy--Schwarz therefore gives
\(\Delta_{st}=0\) for every \(s,t\). Hence, using the definitions in \cref{eq:xT_defs_and_others}, we see that
\cref{eq:li-sur-transfer-overlaps} implies, for every fixed
\(s,t\),
\[
    \frac1p\|r_{2,t}-r_{2,s}\|_2^2
    =
    \frac1p\|x^t-x^s\|_2^2
    \xrightarrow{\PP}0.
\]
This proves \cref{eq:oracle-r2-cauchy} in the case \(\tau_2=0\), so we now prove  \cref{eq:oracle-r2-cauchy} when 
\(\tau_2>0\).

For \(\Delta\in[0,\tau_2]\), let
\((Q_\Delta',Q_\Delta'')\) be a centered Gaussian pair, independent
of \((B^\star,B)\), with
\(
 \operatorname{Var}(Q_\Delta')
 =
 \operatorname{Var}(Q_\Delta'')
 =
 \tau_1
 \) and
\(
 \operatorname{Cov}(Q_\Delta',Q_\Delta'')
 =
 b_\star+\kappa\delta.
\)
Define
\begin{equation}
 g(\Delta)
 :=\E\left[
 F_\lambda(Q_\Delta',B^\star,B)
 F_\lambda(Q_\Delta'',B^\star,B)
 \right],
 \qquad
 \Delta_{s+1,t+1}=g(\Delta_{st}),
 \label{eq:overlap-covariance-map}
\end{equation}
where the second identity follows from
\cref{eq:oracle-roadmap-limiting-process}.
Moreover, if we define
\[
 m(b^\star,b)
 :=
 \E_Q[F_\lambda(Q,b^\star,b)]\] for \( Q\sim N(0,\tau_1),
\)
then conditional independence gives
\(
 \Delta_{1t}=\E[m(B^\star,B)^2]\geq0,
\)
for \(t\geq2\).

For each fixed \((b^\star,b)\), expand
\(z\mapsto F_\lambda(\sqrt{\tau_1}z,b^\star,b)\) in the Hermite
basis \parencite[Sections~1.4, Proposition~1.4.2]{NourdinPeccati2012}:
\begin{align*}
 F_\lambda(\sqrt{\tau_1}z,b^\star,b)
 &=\sum_{k=0}^\infty a_k(b^\star,b)H_k(z), \qquad a_k(b^\star,b)
 =\frac1{k!}\E\left[
 F_\lambda(\sqrt{\tau_1}Z,b^\star,b)H_k(Z)
 \right].
\end{align*}
The standardized Gaussian pair
\(Z'=Q_\Delta'/\sqrt{\tau_1}\) and 
\(Z''=Q_\Delta''/\sqrt{\tau_1}\) has correlation
\(r(\Delta)=(b_\star+\kappa\Delta)/\tau_1\).  Hence, the Hermite
covariance identity $\E[H_k(Z')H_\ell(Z'')]
 =
 \mathbf 1_{\{k=\ell\}}\,k!\,r(\Delta)^k$, from
\textcite[Proposition~2.2.1, Theorem~2.8.2]{NourdinPeccati2012}, gives
\[
 g(\Delta)=\sum_{k=0}^\infty
 k!\,\E[a_k(B^\star,B)^2]r(\Delta)^k.
\]
All coefficients in this power series are nonnegative, and
\(r(\Delta)\) is nonnegative, affine, and nondecreasing on
\([0,\tau_2]\). Hence \(g\) is nonnegative, nondecreasing, and convex
on this interval. Finally,
\(\tau_1=b_\star+\kappa\tau_2\) implies \(r(\tau_2)=1\), so the
\cref{eq:oracle-roadmap-nonlinear-variance} gives
\(
 g(\tau_2)
 =
 \E\left[
 F_\lambda(\sqrt{\tau_1}Z,B^\star,B)^2
 \right]
 =
 \tau_2.
\)
Consequently, \(0\leq g(\Delta)\leq\tau_2\) for every
\(\Delta\in[0,\tau_2]\). It follows inductively that
\( 0\leq\Delta_{st}\leq\tau_2\) for all \(s,t\geq1\).

Price's identity for one-dimensional weak derivatives
\parencite[Theorem~1 and Corollary~1]{Voigtlaender2021Price} gives,
for \(0\le\Delta<\tau_2\),
\[
 g'(\Delta)=\kappa\E\left[
 \partial_qF_\lambda(Q_\Delta',B^\star,B)
 \partial_qF_\lambda(Q_\Delta'',B^\star,B)
 \right].
\]
Set
\(a=\gamma/\eta
=\E[D_\lambda(\sqrt{\tau_1}Z,B^\star,B)]\) for $D_\lambda(q,b^\star,b)
 =\partial_qT_b(b^\star+q)$.
Taking \(\Delta\uparrow\tau_2\) yields the left limit
\begin{equation}
\begin{aligned}
 g_-'(\tau_2) := \lim_{\Delta \uparrow \tau_2} g'(\Delta)
 &=
 \kappa
 \E\left[
 \partial_qF_\lambda(\sqrt{\tau_1}Z,B^\star,B)^2
 \right]\\
 &=
 \kappa
 \left(\frac\eta c\right)^2
 \left(
 \E[D_\lambda(\sqrt{\tau_1}Z,B^\star,B)^2]-a^2
 \right).
\end{aligned}
\end{equation}
Since \(0\leq D_\lambda\leq s_\lambda<1\), we have
\(
 \E[D_\lambda^2]
 \leq
 s_\lambda\E[D_\lambda]
 =
 s_\lambda a.
\)
Therefore
\(
 \E[D_\lambda^2]-a^2
 \leq s_\lambda a-a^2
 <a-a^2.
\)
On the other hand, if
\(
 J=\E[(S+c)^{-2}],
\)
then
\(
 \kappa
 =
 c^2/ \gamma^2(\eta^2J-1),
\)
and hence $\kappa
 \left(\frac\eta c\right)^2(a-a^2)
 =
 \frac c\gamma(\eta^2J-1)$.
Finally, $\frac1\eta-cJ
 =
 \E[\frac{S}{(S+c)^2}]>0$,
which implies that
\(
(c/ \gamma)(\eta^2J-1)<1.
\)
Combining these bounds gives
\(
 g_-'(\tau_2)<1.
\)

Because \(g\) is convex, \(g'\) is nondecreasing, and therefore $\rho_\lambda
 :=
 \sup_{0\leq\Delta<\tau_2}g'(\Delta)
 =
 g_-'(\tau_2)
 <1$.
Using \(g(\tau_2)=\tau_2\) gives $0\leq\tau_2-g(\Delta)
 \leq
 \rho_\lambda(\tau_2-\Delta)$ for $0\leq\Delta\leq\tau_2$.
Iterating the covariance recursion consequently yields
\begin{equation}
 0\leq\tau_2-\Delta_{st}
 \leq
 \rho_\lambda^{\min(s,t)-1}\tau_2.
 \label{eq:oracle-roadmap-geometric-overlap}
\end{equation}
The derivation of \cref{eq:oracle-roadmap-geometric-overlap} is the perturbed counterpart of the covariance argument  in
\textcite[Appendix~C, Prop.~C.8]{Li2026debiasing}.

For every fixed \(s,t\), \cref{eq:li-sur-transfer-overlaps} gives $\frac1p\|r_{2,t}-r_{2,s}\|_2^2
 =
 \frac1p\|x^t-x^s\|_2^2
 \xrightarrow{\PP}
 2(\tau_2-\Delta_{st})$.
Together with
\cref{eq:oracle-roadmap-geometric-overlap}, this proves \cref{eq:oracle-r2-cauchy} for
every \(\epsilon>0\).
The bound in \cref{eq:Dlambda-bound} and
\(\|(X^TX+cI_p)^{-1}\|_{\mathrm{op}}\le c^{-1}\) give the analogous
Cauchy conclusions for
\(r_{1,t}\), \(\widehat\beta_{1,t}\), and
\(\widehat\beta_{2,t}\), as in
\textcite[Appendix~C, Cor.~C.9]{Li2026debiasing}.

\subsection{\texorpdfstring{Proof of \cref{thm:perturbed-fixed-point}}{Proof of the perturbed fixed-point lemma}}\label{sec:oracle-fp-existence}

Fix \(\lambda>0\).  We will show that for any $\lambda > 0$, there exists a solution for the perturbed fixed-point equations in \cref{eq: fpt-oracle-a,eq: fpt-oracle-b,eq: fpt-oracle-c,eq: fpt-oracle-d}, denoted
$(\bar\gamma_{1\star},\eta_\star,\tau_{1\star},\tau_{2\star}),$ with $0<\bar\gamma_{1\star}<\eta_\star,
    $ and $\tau_{1\star},\tau_{2\star}\ge0.$
To prove this, we reduce
\cref{eq: fpt-oracle-a,eq: fpt-oracle-b,eq: fpt-oracle-c,eq: fpt-oracle-d} to two equations
in two positive variables and find a simultaneous zero by the
Poincar\'e--Miranda theorem.  The spectral identities and continuity
estimates used in the reduction are proved in
\cref{lem:spectral-identities,lem:F-continuity,lem:uniform-risk-estimate}.
This argument is motivated by
\textcite[Prop.~2.11]{Li2026debiasing}, but the modification described
in \cref{rem:fixed-point-boundary-modification} removes the additional
penalty-shape conditions imposed there.

\subsubsection{Reduction from four to two fixed-point equations}
\label{sec:reference-point-reduction}

To condense the four fixed-point equations in \cref{eq: fpt-oracle-a,eq: fpt-oracle-b,eq: fpt-oracle-c,eq: fpt-oracle-d} into two, in Step 1 we will reparameterize the variables $(\bar\gamma_{1\star},\eta_\star,\tau_{1\star},\tau_{2\star})$, then eliminate two of them. This produces two equations in two unknowns $(\rho, \xi)$ of the form:
\begin{equation}
\label{eq:new_FPs}
 F_1(\rho,\xi)=0,\qquad F_2(\rho,\xi)=0.
\end{equation}
In Step 2, we show that the above equation have a
simultaneous zero.  We proceed with the proof here and mention that some technical lemmas, along with their proofs, are given in \cref{sec:technical_lemmas}.

\textbf{Step 1.} We begin with the reparameterization of \cref{eq: fpt-oracle-a,eq: fpt-oracle-b,eq: fpt-oracle-c,eq: fpt-oracle-d} using
\begin{equation}
\label{eq: change-variable-rho}
\theta := {1}/{\bgs} \qquad \text{ and } \qquad \rho := {1}/{\eta_\star}.
\end{equation}
Then the required condition $\eta_\star > \bgs > 0$ is equivalent to $0 < \rho < \theta.$ Introduce an auxiliary parameter \(\xi>0\) through
\(\tau_{1\star}=\theta^2/\xi^2\), and define the spectral
\(R\)-transform using the function $G$ in \cref{eq:stieltjes-transform} by
\begin{equation}
\label{eq:R}
 R(\rho) := G^{-1}(\rho) - \frac{1}{\rho}.
\end{equation}
\cref{eq: fpt-oracle-c}, using the new variables in \cref{eq: change-variable-rho}, becomes
\begin{equation}
\label{eq:Ginv}
\frac{\rho}{\theta} = 1- \rho G^{-1}(\rho) = - \rho R(\rho).
\end{equation}
Hence 
\(
\theta=\theta(\rho):=-1/R(\rho).
\)
Thus, the spectral fixed-point equation determines $\theta$, and hence $\bgs$, as a function of $\rho$. This removes one unknown.  

We mention additionally that \cref{eq: fpt-oracle-b} is explicit in $\tau_{2\star}$, in the sense that it is fully specified once $\rho$ and $\xi$ are specified.

Now, we consider the final reduction. The remaining fixed-point equations are \cref{eq: fpt-oracle-a,eq: fpt-oracle-d}, which we will represent as functions of $(\rho, \xi)$.

We will also rewrite the original denoiser function $\prox_{\frac{h}{\lambda + \bgs }}(\cdot)$ using the following reparameterization: for fixed $\lambda > 0$, define
\begin{equation}\label{eq: definition-w-rho}
w(\rho) = \frac{\theta(\rho)}{1 + \lambda \theta(\rho)} =
    \frac1{\theta(\rho)^{-1}+\lambda}.
\end{equation}
Thus \(w(\rho)=(\bgs+\lambda)^{-1}\) and
\(0<w(\rho)<\lambda^{-1}\).

Let $\xi > 0$ and let $Z\sim N(0,1)$ be independent of $(B^\star,B)$. Then, define
\begin{equation}
U_{\rho, \xi}
:=
\frac{w(\rho)}{\theta(\rho)}B^\star
+
(1-\frac{w(\rho)}{\theta(\rho)})B
+
\frac{w(\rho)}{\xi}Z,
\label{eq:U-rho-xi}
\end{equation}
and $Q_\rho(u):=\prox_{w(\rho)h}(u).$
We define the average derivative and scalar risk as 
\begin{equation}
\label{eq:Aandr}
A(\rho, \xi) = \E\left[Q_{\rho}'(U_{\rho, \xi})\right] \qquad \text{
and } \qquad 
r(\rho, \xi) := \E\left[ \left(Q_{\rho}(U_{\rho, \xi}) - B^\star\right)^2\right].
\end{equation}
The quantities in \cref{eq:Aandr} are exactly the derivative and
signal-MSE terms in \cref{eq: fpt-oracle-a,eq: fpt-oracle-b}.
Since $Q_\rho$ is a one-dimensional proximal map, it is nondecreasing and
$1$-Lipschitz, so $ 0\le Q_\rho'\le1$ almost everywhere.

Using Gaussian integration by parts,
\begin{equation}
\label{eq:A-Stein-no-gap}
    A(\rho,\xi)
    =
    \frac{\xi}{w(\rho)}
    \mathbb E
    \left[
        Q_\rho(U_{\rho,\xi})Z
    \right]
    =
    \frac{\xi}{w(\rho)}
    \mathbb E
    \left[
        \left(
            Q_\rho(U_{\rho,\xi})-B^\star
        \right)Z
    \right].
\end{equation}
This representation of $A$ will be useful for proving its continuity because it avoids having to establish pointwise continuity of $Q'_{\rho}$.

Setting \(\bgs=1/\theta(\rho)\),
\(\eta_\star=1/\rho\), and using
\cref{eq:oracle-prox-fp-derivation}, the perturbed denoiser evaluated at
\(B^\star+\sqrt{\tau_{1\star}}Z\) is
\begin{align*}
\prox_{\phi_{\lambda,B}/\bgs}
\left(B^\star+\sqrt{\tau_{1\star}}Z\right)=
\prox_{\frac{h}{\bgs+\lambda}}
\left(
\frac{\bgs B^\star+\lambda B}{\bgs+\lambda}
+
\frac{\bgs}{\bgs+\lambda}\sqrt{\tau_{1\star}}Z
\right)
=
Q_\rho(U_{\rho,\xi}).
\end{align*}

In our change of variables, \cref{eq: fpt-oracle-a} is therefore equivalent to $\frac{\rho}{\theta(\rho)}
    =
    \frac1{1+\lambda\theta(\rho)}
    A(\rho,\xi)$,
or
\begin{equation}
\label{eq:F1}
    F_1(\rho,\xi)
    :=
    \rho-w(\rho)A(\rho,\xi)
    =
    0.
\end{equation}

Now, we turn our attention to \cref{eq: fpt-oracle-d}. First, notice that the denoising-variance equation in \cref{eq: fpt-oracle-b} can be written in this notation as
\begin{equation}
\label{eq:tau2_new}
    \tau_{2\star} = \frac{\theta(\rho)^2}
         {(\theta(\rho)-\rho)^2} \left(
        r(\rho,\xi)-\frac{\rho^2}{\xi^2}\right),
\end{equation}
for $r(\rho,\xi)$ defined in \cref{eq:Aandr}.
We next eliminate $\tau_{2\star}$ from the linear-variance equation in \cref{eq: fpt-oracle-d}. Recall the definitions of $G$ in \cref{eq:stieltjes-transform} and $\rho_{\max}$ in \cref{eq:rho_max}. For $\rho\in(0,\rho_{\max})$, let
$c=c(\rho)=G^{-1}(\rho).$ We also define $c(\rho_{\max})$ as its continuous limit. Since $ c
    =
    \frac1\rho-\frac1{\theta(\rho)}$ by \cref{eq:Ginv},
we have $ c = \eta_\star - \bgs.$ For every $\rho\in(0,\rho_{\max}]$, define
\begin{equation}
\label{eq:JandM}
    J(\rho)
    := \mathbb E\left[\frac1{(S+c(\rho))^2}\right]
    \qquad \text{ and } \qquad 
    M(\rho) := \mathbb E\left[\frac{S}{(S+c(\rho))^2}\right].
\end{equation}

The linear-variance equation in \cref{eq: fpt-oracle-d} is therefore equivalent to $\tau_{1\star} = (\frac{\eta_\star}{\bgs})^2
    ( \mathbb E
        [ \frac{ \sigma^2 S+c^2\tau_{2\star}
            }{ (S+c)^2 }] - (c\rho)^2\tau_{2\star})$.
Multiplying by $(\rho/\theta)^2$ yields $\frac{\rho^2}{\xi^2}
    =
    \sigma^2M(\rho) + c^2(J(\rho)-\rho^2)\tau_{2\star}$.
Because $c = \frac{\theta-\rho}{\rho\theta}$,
substitution of this and \cref{eq:tau2_new} into the above gives $\frac{\rho^2}{\xi^2} = \sigma^2M(\rho) + \frac{J(\rho)-\rho^2}{\rho^2}
    ( r(\rho,\xi)-\frac{\rho^2}{\xi^2}
    )$.
After rearranging and using part (ii) of \cref{lem:spectral-identities}, we find $\frac1{\xi^2}
    =
    R'(\rho)r(\rho,\xi)
    +
    \sigma^2\frac{M(\rho)}{J(\rho)}$.
Using part (iv) of \cref{lem:spectral-identities}, this becomes
\begin{equation}
\label{eq:F2}
    F_2(\rho,\xi) := \xi^2R'(\rho)r(\rho,\xi)
    + \sigma^2\xi^2K(\rho)-1= 0.
\end{equation}
Equivalently, we have that $F_2(\rho,\xi)
    =
    \xi^2R'(\rho)r(\rho,\xi) +
    \frac{\sigma^2\xi^2}{\theta(\rho)}
   ( 1+\rho\frac{R'(\rho)}{R(\rho)})
    - 1$.

Thus, it suffices to find a simultaneous zero in $(\rho, \xi)$ for $F_1$ and $F_2$ in \cref{eq:F1,eq:F2}. The rest of the proof is dedicated to establishing this. We do
this by a two-dimensional intermediate-value argument.

\textbf{Step 2.} First, we will construct a rectangle
\(\mathcal R=[\rho_-,\rho_+]\times[\xi_-,\xi_+]\) on which the
two reduced equations have the required boundary signs.  

Recall from
\cref{eq:rho_max} that $\rho_{\max}
    = \mathbb E\left[1/S\right] \in(0,\infty].$
We first choose a finite right edge of the rectangle, $\rho_+$, such that \(0<\rho_+<\rho_{\max}\) and
\(\rho_+>w(\rho_+)\).

If $\rho_{\max}=\infty$, choose any finite $\rho_+>1/\lambda.$
Then $0<\rho_+<\rho_{\max}$, and since $w(\rho)<1/\lambda$, we have $ \rho_+>w(\rho_+).$ If $\rho_{\max}<\infty$, then as $\rho\uparrow\rho_{\max}$, or, equivalently, as
$c(\rho)\downarrow0$, we have
$\theta(\rho)\to\rho_{\max},$ and therefore $w(\rho)
    =
    \frac{\theta(\rho)}{1+\lambda\theta(\rho)}
    \to
    \frac{\rho_{\max}}{1+\lambda\rho_{\max}}
    <
    \rho_{\max}$.
Hence, we may choose $\rho_+<\rho_{\max}$
sufficiently close to $\rho_{\max}$ that $ \rho_+>w(\rho_+).$

We next choose \(\xi_+\) large enough that \(F_2>0\) on the top edge of the rectangle and
\(\xi_-\) small enough that \(F_2<0\) whenever  \(F_1 = 0\) along the
bottom edge.

On $[0,\rho_+]$, for the function $K$ defined in \cref{eq:Kdef}, let
\begin{equation}
\label{eq:Kbounds}
    k_-:=\min_{\rho\in[0,\rho_+]}K(\rho)>0,
    \qquad \text{ and } \qquad 
    k_+:=\max_{\rho\in[0,\rho_+]}K(\rho)<\infty.
\end{equation}
Also define
\begin{equation}
\label{eq:qlambda}
    q_\lambda
    :=
    \max_{\rho\in[0,\rho_+]}
    \frac{w(\rho)}{\theta(\rho)}
    =
    \max_{\rho\in[0,\rho_+]}
    \frac1{1+\lambda\theta(\rho)}
    <1.
\end{equation}

Choose $\xi_+>0$ large enough that $\sigma^2 k_-\xi_+^2>1.$
Then, using the definitions of $F_2$ from \cref{eq:F2} and $k_-$ from \cref{eq:Kbounds}, as well as the fact that $R'\ge0 $ by part (ii) of \cref{lem:spectral-identities} and $r\ge0$ by definition \cref{eq:Aandr}, we have that $F_2(\rho,\xi_+) = \xi_+^2R'(\rho)r(\rho,\xi_+)
    + \sigma^2\xi_+^2K(\rho) -
    1 \ge \sigma^2\xi_+^2k_- -1 >0$ for every $\rho\in[0,\rho_+]$.
Next, let $C_R:=\max_{\rho\in[0,\rho_+]}R'(\rho)<\infty$. Choose $0<\xi_-<\xi_+$ sufficiently small that
\begin{equation}
\label{eq:xi-minus-choice-no-gap}
    C_R\varepsilon_h(\xi_-)
    +
    \sigma^2k_+\xi_-^2
    <
    1-q_\lambda
\end{equation}
for the function $\varepsilon_h$ from \cref{lem:uniform-risk-estimate}; the fact that such an $\xi_-$ exists follows from the same lemma, which shows that $ \varepsilon_h(\xi)\to0$ as $\xi\downarrow0.$
If $F_1(\rho,\xi_-)=0$, then \cref{lem:uniform-risk-estimate} further gives
\(
    \xi_-^2r(\rho,\xi_-)
    \le
    w(\rho)\rho+\varepsilon_h(\xi_-).
\)
Therefore, we have
\[
\begin{aligned}
    F_2(\rho,\xi_-)
    &=
    \xi_-^2R'(\rho)r(\rho,\xi_-)
    +
    \sigma^2\xi_-^2K(\rho)
    -
    1\\
    &\le
    R'(\rho)w(\rho)\rho
    + R'(\rho)\varepsilon_h(\xi_-)
    + \sigma^2\xi_-^2k_+
    - 1.
\end{aligned}
\]
Next, using part (iii) of \cref{lem:spectral-identities} and the definition of $q_{\lambda}$ in \cref{eq:qlambda}, we have $R'(\rho)w(\rho)\rho
    =
    \frac{w(\rho)}{\theta(\rho)}
    (
        -\rho\frac{R'(\rho)}{R(\rho)}
    )
    \le
    \frac{w(\rho)}{\theta(\rho)}
    \le
    q_\lambda$,
and therefore $F_2(\rho,\xi_-)
    \le q_\lambda
    +
    C_R\varepsilon_h(\xi_-)
    +
    \sigma^2k_+\xi_-^2-1 <0$,
where the final inequality follows from \cref{eq:xi-minus-choice-no-gap}. Thus,
\begin{equation}
\label{eq:F2-bottom-on-zero-set}
    F_2(\rho,\xi_-)<0
    \qquad
    \text{whenever }F_1(\rho,\xi_-)=0.
\end{equation}

We will now prove that $F_1(\rho_-,\xi)<0$ and $F_1(\rho_+,\xi)>0$
on the vertical edges of $\mathcal R$. To do so, we now choose the left edge of the rectangle. By
\cref{lem:F-continuity}, we have $A(0,\xi)>0$ for every $\xi\in[\xi_-,\xi_+].$
Compactness gives us that $a_0
    :=
    \min_{\xi\in[\xi_-,\xi_+]}A(0,\xi)
    >0.$
Since $w(0)>0$, continuity allows us to choose $0<\rho_-<\rho_+$
small enough that $ w(\rho_-)A(\rho_-,\xi)>\rho_-$ for every $\xi\in[\xi_-,\xi_+].$
Equivalently, using the definition of $F_1$ from \cref{eq:F1}, we have $F_1(\rho_-,\xi) = \rho_--w(\rho_-)A(\rho_-,\xi) <0$ for every $\xi\in[\xi_-,\xi_+]$.
At the right edge, since $A\le1$ and by the way we selected $\rho_+$, for every $\xi\in[\xi_-,\xi_+]$, we have $F_1(\rho_+,\xi)
    =
    \rho_+-w(\rho_+)A(\rho_+,\xi) \ge
    \rho_+-w(\rho_+)
    >0.$
    
We have therefore shown that \(F_1<0\) and \(F_1>0\) on the left
and right edges of \(\mathcal R\), respectively; that \(F_2>0\) on
the top edge; and that \(F_2<0\) on the bottom edge wherever
\(F_1=0\).

The bottom-edge sign in \cref{eq:F2-bottom-on-zero-set} is known only
where \(F_1=0\).  We extend it to the entire edge without changing any
simultaneous zero. To do this, we will use a modified function to properly invoke the Poincar\'e-Miranda theorem. Specifically, we use the  function $\widetilde F_2=F_2+M H(\xi)|F_1|$, which agrees with $F_2$ on $\{F_1=0\}$ but has the desired signs on the entire top and bottom edges. Then, Poincar\'e-Miranda can be applied to $(F_1,\widetilde F_2)$, which gives a point
$(\rho_\star,\xi_\star)$ where $F_1=0$ and $\widetilde F_2=0$. Lastly, since the
modification vanishes on $F_1=0$, this point also satisfies $F_2=0$, yielding
the desired fixed-point solution.

Define
$\mathcal Z_-
    :=
    \left\{
        \rho\in[\rho_-,\rho_+]:
        F_1(\rho,\xi_-)=0
    \right\}.$
The set $\mathcal Z_-$ is nonempty by the intermediate value theorem, since we have shown that $F_1(\rho_-,\xi)<0$ and $F_1(\rho_+,\xi)>0$ for all $\xi \in [\xi_-, \xi_+]$, and it
is compact by continuity. By \cref{eq:F2-bottom-on-zero-set}, we have $ F_2(\rho,\xi_-)<0$ for all $\rho\in\mathcal Z_-.$ Thus, by continuity, there is an open neighborhood $\mathcal U_-\supseteq\mathcal Z_-$
such that $ F_2(\rho,\xi_-)<0$ for all $\rho\in\mathcal U_-.$

If the compact complement
\([\rho_-,\rho_+]\setminus\mathcal U_-\) is nonempty, the continuous
function \(|F_1(\rho,\xi_-)|\) has the positive minimum
\begin{equation}
\label{eq:m_minus}
    m_-:=
    \min_{\rho\in[\rho_-,\rho_+]\setminus\mathcal U_-}
    |F_1(\rho,\xi_-)|>0.
\end{equation}
Writing
\(
B_-:=\max_{\rho\in[\rho_-,\rho_+]}|F_2(\rho,\xi_-)|<\infty,
\)
choose \(M>B_-/m_-\). If the complement is empty, no domination is
needed and we instead set \(M=0\).
Define $H(\xi)
    :=
    \frac{2\xi-\xi_- -\xi_+}{\xi_+ - \xi_-}$
so that $ H(\xi_-)=-1,$ and $H(\xi_+)=1,$
and $\widetilde F_2(\rho,\xi)
    :=
    F_2(\rho,\xi)
    +
    M H(\xi)|F_1(\rho,\xi)|.$
We notice that $ \widetilde F_2$ agrees with $F_2$
on $F_1=0$, and we will now show that it has the desired sign on the entire top and bottom edges.

On the top edge, $\widetilde F_2(\rho,\xi_+)
    =
    F_2(\rho,\xi_+)+M|F_1(\rho,\xi_+)|.$
Since $F_2(\rho,\xi_+)>0$, we immediately have $\widetilde F_2(\rho,\xi_+)>0$ for all $\rho\in[\rho_-,\rho_+]$ as desired.

On the bottom edge, $\widetilde F_2(\rho,\xi_-)
    =
    F_2(\rho,\xi_-)-M|F_1(\rho,\xi_-)|.$
If $\rho\in\mathcal U_-$, then $F_2(\rho,\xi_-)<0$, so $\widetilde F_2(\rho,\xi_-)<0.$
If $\rho\notin\mathcal U_-$, by the definition of $m_-$ in \cref{eq:m_minus}, we have $|F_1(\rho,\xi_-)|\ge m_-,$
and therefore $\widetilde F_2(\rho,\xi_-) = F_2(\rho,\xi_-)
    -
    M |F_1(\rho,\xi_-)|
    \le
    B_- - M m_-<0$.
Hence $ \widetilde F_2(\rho,\xi_-)<0$ for all $\rho\in[\rho_-,\rho_+]$, as desired.

Thus, the continuous map
$(\rho,\xi)
    \mapsto
    (
        F_1(\rho,\xi),
        \widetilde F_2(\rho,\xi)
    )$
satisfies the Poincar\'e--Miranda sign conditions on
$\mathcal R$:
\[
    F_1<0 \text{ on the left edge},
    \qquad
    F_1>0 \text{ on the right edge},
\]
and
\[
    \widetilde F_2<0 \text{ on the bottom edge},
    \qquad
    \widetilde F_2>0 \text{ on the top edge}.
\]
This means that there exists $(\rho_\star,\xi_\star)\in\mathcal R$
such that $ F_1(\rho_\star,\xi_\star)=0,$ and $
    \widetilde F_2(\rho_\star,\xi_\star)=0.$
At a zero of $F_1$, the correction term vanishes. Hence, $ F_2(\rho_\star,\xi_\star)=0$ as well. This proves the existence of a solution to \cref{eq:new_FPs}.

We now reconstruct the original VAMP parameters. Set
\begin{equation}
\label{eq:reference-point-reconstruction}
\begin{aligned}
    \bar\gamma_{1\star}
    &:=
    \frac1{\theta(\rho_\star)},
    &
    \eta_\star
    &:=
    \frac1{\rho_\star},
    &
    \tau_{1\star}
    &:=
    \frac{\theta(\rho_\star)^2}{\xi_\star^2},
    \tau_{2\star}
    &:=
    \frac{\theta(\rho_\star)^2}
         {(\theta(\rho_\star)-\rho_\star)^2}
    \left(
        r(\rho_\star,\xi_\star)
        -
        \frac{\rho_\star^2}{\xi_\star^2}
    \right).
\end{aligned}
\end{equation}

We check admissibility. Since $F_1(\rho_\star,\xi_\star)=0,$
we have $\rho_\star
    =
    w(\rho_\star)A(\rho_\star,\xi_\star)
    \le
    w(\rho_\star)$ because $0 \leq A \leq 1$.
But $ w(\rho_\star)
    =
    \frac{\theta(\rho_\star)}
         {1+\lambda\theta(\rho_\star)}
    <
    \theta(\rho_\star)$.
Therefore $0<\rho_\star<\theta(\rho_\star),$
and hence $0<\bar\gamma_{1\star}<\eta_\star.$

Next, by the Stein identity \cref{eq:A-Stein-no-gap} and
Cauchy--Schwarz,
\[
\Big(
        \frac{w(\rho_\star)}
             {\xi_\star}
        A(\rho_\star,\xi_\star)
    \Big)^2
    =
    \left(
        \mathbb E
        \left[
            \left(
                Q_{\rho_\star}(U_{\rho_\star,\xi_\star})
                -
                B^\star \right)Z \right] \right)^2
\le r(\rho_\star,\xi_\star).
\]
Using \(w(\rho_\star)A(\rho_\star,\xi_\star)=\rho_\star\), we get $r(\rho_\star,\xi_\star)
    \ge
    {\rho_\star^2}/{\xi_\star^2}.$
Thus $\tau_{2\star}\ge0.$

Finally, the construction gives the four perturbed fixed-point equations:
\begin{itemize}
    \item \(F_1(\rho_\star,\xi_\star)=0\) is the denoising-derivative equation in \cref{eq: fpt-oracle-a};
    \item the definition of \(\tau_{2\star}\) is the denoising-variance equation in \cref{eq: fpt-oracle-b};
    \item \(\theta(\rho_\star)=-1/R(\rho_\star)\) is the linear-precision equation in \cref{eq: fpt-oracle-c};
    \item \(F_2(\rho_\star,\xi_\star)=0\) is the linear-variance equation in \cref{eq: fpt-oracle-d}.
\end{itemize}
Therefore $(\bar\gamma_{1\star},\eta_\star,\tau_{1\star},\tau_{2\star})$
solves the perturbed convex-VAMP fixed-point equations for a given
$\lambda>0$,
and every simultaneous zero $(\rho_\star,\xi_\star)$ of
\cref{eq:F1,eq:F2} reconstructs through
\cref{eq:reference-point-reconstruction} a solution of
\cref{eq: fpt-oracle-a,eq: fpt-oracle-b,eq: fpt-oracle-c,eq: fpt-oracle-d}.

\begin{rem}[Why the boundary modification is needed]
\label{rem:fixed-point-boundary-modification}
The two-variable reduction follows \textcite[Section~B.3]{Li2026debiasing}.
Their additional penalty-shape assumptions give the horizontal boundary
signs directly.  For the present penalty class, we obtain the lower sign
only on \(\{F_1=0\}\).  The correction
\(MH(\xi)|F_1|\) extends that sign to the full boundary without moving
any simultaneous zero; this is the only extra topological step.
\end{rem}

\subsection{Auxiliary analytic lemmas}
\label{sec:technical_lemmas}

\begin{lemma}[Spectral identities]
\label{lem:spectral-identities}
Let $0<\rho_+<\rho_{\max}$, where $\rho_{\max}$ is defined in \cref{eq:rho_max}. Let $c(\rho)=G^{-1}(\rho)$ for $G$ defined in \cref{eq:stieltjes-transform}, and recall $J(\rho)$ and $M(\rho)$ defined in \cref{eq:JandM}. Then the following hold on
$(0,\rho_+]$, with continuous extensions to $\rho=0$.

\begin{enumerate}[label=(\roman*)]
    \item $R(\rho)<0$ and
      \[
        \frac1{\theta(\rho)}
        =
        -R(\rho)
        = 
            \mathbb E\left[\dfrac{S}{S+c(\rho)}\right]
        \left(
            \mathbb E\left[\dfrac1{S+c(\rho)}\right]
        \right)^{-1}.
    \]
    In particular,
    \[
        0<
        \frac1{\theta(\rho)}
        \le s_+,
        \qquad
        \theta(\rho)\ge \frac1{s_+}, \text{ where } s_+ = \sup \mathrm{supp}(\mu).
    \]

    \item $R$ is differentiable and
    \(
        R'(\rho) =
        1/(\rho^2)- 1/ J(\rho)
        \geq 0.
    \)

    \item $0 \le -\rho R'(\rho)/R(\rho) <1.$

    \item The quantity
    \begin{equation}
    \label{eq:Kdef}
    K(\rho) :=
        \frac1{\theta(\rho)}
        \left(
            1+\rho\frac{R'(\rho)}{R(\rho)} \right)
    \end{equation}
    satisfies $ K(\rho) = {M(\rho)}/{J(\rho)},$
    and hence $K(\rho)>0$ for every $\rho\in[0,\rho_+]$. Consequently, for some constants $k_-,k_+$, on $[0,\rho_+]$, we have
    $0<k_-\le K(\rho)\le k_+<\infty.$
      \item There exists a finite constant $C_R$, depending only on
    the law of $S$ and $\rho_+$, such that on $[0,\rho_+]$,
    \(
        0\le R'(\rho)\le C_R,
    \)
    where the derivative at $0$ is taken to be the right limit
    \[
    R'(0) := \lim_{\rho \downarrow 0} \frac{R(\rho) - R(0)}{\rho}, \quad \text{ and } \quad R(0) := \lim_{\rho \downarrow 0} R(\rho) \text{ exists }.
    \]
\end{enumerate}
\end{lemma}

\begin{proof}
Some of these results follow from \textcite[Lemma A.8]{Li2026debiasing}, but we include the proof here for completeness.

\textbf{Result (i).} Fix $\rho\in(0,\rho_+]$. Since $c=G^{-1}(\rho),$ we have $\rho = G(c) = \mathbb{E}[1/(c + S)]$. Then, since $S \geq 0$ and $\PP(S > 0) > 0$, we have
\begin{equation}
\label{eq:crho_bound}
    1-c\rho= \mathbb E\left[ 1-\frac{c}{S+c}
    \right] = \mathbb E\left[ \frac{S}{S+c}
    \right]>0. 
\end{equation}
Therefore, by the definition of $R$ from \cref{eq:R}, 
\begin{equation}
\label{eq:negativeR}
    -R(\rho) = \frac1\rho-c = \frac{1-c\rho}{\rho} =  \mathbb E\left[\dfrac{S}{S+c}\right] \left(\mathbb E\left[\dfrac1{S+c}\right] \right)^{-1} >0.
\end{equation}
This is a weighted average of $S$, with nonnegative weights proportional to
$(S+c)^{-1}$. Since $0\le S\le s_+$ and $\mathbb E[S]>0$, the ratio is
strictly positive and at most $s_+$. This proves the bounds on $\theta$ in the first claim.

\textbf{Result (ii).} Since for $J(\rho)$ defined in \cref{eq:JandM},
\(
    G'(c) = -\mathbb E[\frac1{(S+c)^2}]
    = -J(\rho),
\)
the inverse-function theorem gives $c'(\rho)=-1/{J(\rho)}.$
Consequently,
\( R'(\rho) =c'(\rho)+\frac1{\rho^2}=\frac1{\rho^2}-\frac1{J(\rho)}\).
Because $J(\rho)\ge \rho^2$ by Jensen's inequality, $R'(\rho)\ge0$.

\textbf{Result (iii).} Notice that 
\begin{equation}
\label{eq:R_equiv}
-\rho\frac{R'(\rho)}{R(\rho)}=\frac{\rho^2R'(\rho)}{1-c\rho}.
\end{equation}
This is nonnegative using result (ii) and \cref{eq:crho_bound}. For the upper bound, using the expression for $R'$ from result (ii), we have
\(
    1-c\rho-\rho^2R'(\rho)=
    1-c\rho-(
        1-\frac{\rho^2}{J(\rho)}
) = \rho (
        \frac{\rho}{J(\rho)}-c).
\)
But, by \cref{eq:JandM},
\begin{equation}
\label{eq:need_ref}
    \frac{\rho}{J(\rho)}-c=\frac{\rho-cJ(\rho)}{J(\rho)} = 
        \mathbb E\left[ \dfrac{S}{(S+c)^2}
        \right]  \left(
        \mathbb E\left[
            \dfrac1{(S+c)^2}
        \right]\right)^{-1} = \frac{M(\rho)}{J(\rho)} >0.
\end{equation}
Because $\mathbb E[S]>0$ and $c>0$, we have $M(\rho)>0$.
Thus $\rho^2R'(\rho)<1-c\rho$, which, together with \cref{eq:R_equiv}, proves that
\(
    0\le-\rho\frac{R'(\rho)}{R(\rho)}<1.
\)

\textbf{Result (iv).} Since $R=-1/\theta$, we have that $K(\rho) = \frac1{\theta(\rho)}
        (1+\rho\frac{R'(\rho)}{R(\rho)} ) = - R(\rho)-\rho R'(\rho)$.
Now using result (ii) along with \cref{eq:negativeR,eq:need_ref}, we have $K(\rho) =
    \frac1\rho-c
    -
    \rho
    (
        \frac1{\rho^2}-\frac1{J(\rho)}
    ) =
    \frac{\rho}{J(\rho)}-c =
    \frac{M(\rho)}{J(\rho)}$.
This is a weighted average of $S$, with weights proportional to
$(S+c)^{-2}$. Hence, it is strictly positive for every finite $c>0$, because
$\mathbb E[S]>0$.

As \(\rho\downarrow0\), equivalently \(c\to\infty\), we have
\[
\begin{aligned}
    K(\rho)
    &=
        \mathbb E\left[\dfrac{S}{(S+c)^2}\right]
    \left(
        \mathbb E\left[\dfrac1{(S+c)^2}\right] \right)^{-1}\\
    &=
        \mathbb E\left[\dfrac{S}{(1+S/c)^2}\right]
 \left( \mathbb E\left[\dfrac1{(1+S/c)^2}\right] \right)^{-1}
    \\
    &\longrightarrow \mathbb E[S]>0,
\end{aligned}
\]
by dominated convergence. Thus, $K$ extends continuously to $\rho=0$ and is
strictly positive on the compact interval $[0,\rho_+]$. Hence, it has a positive
minimum $k_->0$ and a finite maximum $k_+<\infty$.

\textbf{Result (v).} It remains to show that $R'$ is bounded on $[0,\rho_+]$. Let $c_+:=c(\rho_+)>0.$
For $\rho\in(0,\rho_+]$, we have $c(\rho)\ge c_+$. 
Define $X_c:=\frac1{S+c}$ and notice that $\VAR(X_c) = \mathbb{E}[(X_c)^2] - (\mathbb{E}[X_c])^2 = J(\rho) - G(c)^2 = J(\rho) - \rho^2$.
Then, using result (ii), we have $R'(\rho) = \frac{J(\rho) - \rho^2}{\rho^2J(\rho)}= \frac{\VAR(X_c)}
         {\rho^2J(\rho)}$.
We bound the numerator by Lipschitz continuity of $s\mapsto(s+c)^{-1}$:
\(
    \operatorname{Var}(X_c)
    \le
    {\operatorname{Var}(S)}/{c^4}.
\)
We also have that $\rho
    =\mathbb E[\frac1{S+c}]
    \ge \frac1{s_++c}$ and $J(\rho)
    =
    \mathbb E[\frac1{(S+c)^2}]
    \ge
    \frac1{(s_++c)^2}$.
Therefore, $R'(\rho)
    \le \operatorname{Var}(S)
    (
        \frac{s_++c}{c} )^4
    \le
    \operatorname{Var}(S)
    ( \frac{s_++c_+}{c_+})^4$, and we may conclude that $R'$ is bounded on $(0,\rho_+]$. 

At the remaining endpoint, compact support of $S$ permits the Taylor expansion
\[
 G(z)=z^{-1}-\E[S]z^{-2}+\E[S^2]z^{-3}+O(z^{-4})
 \qquad (z\to\infty).
\]
Inverting this series at \(\rho=0\) gives
\begin{equation}
 R(\rho)=G^{-1}(\rho)-\rho^{-1}
 =-\E[S]+\VAR(S)\rho+O(\rho^2).
 \label{eq: equation-R-polynomial}
\end{equation}
Hence \(R\) and \(R'\) extend continuously to zero, with
\(R(0)=-\E[S]\) and \(R'(0)=\VAR(S)\).  This proves (v).

\end{proof}

\begin{lemma}[Continuity of $F_1$ and $F_2$]
\label{lem:F-continuity}
Fix $0<\rho_+<\rho_{\max}$. On every compact rectangle
\begin{equation}
\label{eq:rectangle}
[0,\rho_+]\times[\xi_-,\xi_+],
    \qquad
    0<\xi_-<\xi_+<\infty,
\end{equation}
the functions $A$ and $r$ defined in \cref{eq:Aandr} are continuous. Consequently, $F_1$ and $F_2$ defined in \cref{eq:F1,eq:F2}
are continuous. Moreover, $A(0,\xi)>0$ for all $\xi>0.$
\end{lemma}

\begin{proof}
On $[0,\rho_+]$, the function $w(\rho)$ in \cref{eq: definition-w-rho} is continuous and bounded:
\[ 0<w_{\min}\le w(\rho)\le w_{\max}<\infty.\]
For $w>0$, we have $Q_w(u):=\prox_{w h}(u).$

We first argue that the map $(w,u)\mapsto Q_w(u)$ is jointly continuous on $(0,\infty)\times\mathbb R$. Indeed, if $w_n\to w>0$ and $u_n\to u$
with $x_n=Q_{w_n}(u_n)$,
then $x_n$ minimizes $x\mapsto \frac12(x-u_n)^2+w_nh(x)$ by definition.
To justify boundedness uniformly in \(n\), fix an affine minorant
\(h(x)\ge s x+t\). Since \(u_n\) is bounded and \(w_n\) stays in a compact
subset of \((0,\infty)\), the lower bounds
\(
    \frac12(x-u_n)^2+w_n(sx+t)
\)
tend to \(+\infty\) as \(|x|\to\infty\), uniformly for all large \(n\).
Comparing the objective at \(x_n\) with its value at any fixed point of
\(\dom h\) therefore shows that \(\{x_n\}\) is bounded. Any subsequential limit
$x$ minimizes $x\mapsto \frac12(x-u)^2+w h(x)$
by lower semicontinuity of \(h\). Since this objective is strongly convex, its
minimizer is unique. Thus $Q_{w_n}(u_n)\to Q_w(u).$

Since every $Q_w$ is $1$-Lipschitz, $|Q_w(u)|
    \le |u|+|Q_w(0)|.$
Moreover, for every fixed \(w>0\),
\cref{lem:finiteness-proximal-0} gives
\(
|Q_{w}(0)| = |\prox_{w h}(0)| < \infty.
\)
The joint continuity proved above makes
\(w\mapsto Q_w(0)\) continuous. Hence
\[
C := \sup_{w \in [w_{\min}, w_{\max}]} |Q_{w}(0)| < \infty,
\]
and therefore $|Q_w(u)|\le |u|+C$ uniformly over
$w\in[w_{\min},w_{\max}]$.

Let $(\rho_n,\xi_n)\to(\rho,\xi)$ inside the compact rectangle given in \cref{eq:rectangle}. Then
\[
\begin{aligned}
U_{\rho_n,\xi_n}
&=
\frac{w(\rho_n)}{\theta(\rho_n)}B^\star
+
\left(1-\frac{w(\rho_n)}{\theta(\rho_n)}\right)B
+
\frac{w(\rho_n)}{\xi_n}Z\\
&\longrightarrow
\frac{w(\rho)}{\theta(\rho)}B^\star
+
\left(1-\frac{w(\rho)}{\theta(\rho)}\right)B
+
\frac{w(\rho)}{\xi}Z
=U_{\rho,\xi}
\end{aligned}
\]
pointwise in $(B^\star,B,Z)$. By joint continuity of $Q_w$, we have
$Q_{\rho_n}(U_{\rho_n,\xi_n})\to Q_\rho(U_{\rho,\xi})$
pointwise. Since $0<w(\rho)/\theta(\rho)<1$, moreover, we have that
\(
|Q_{\rho_n}(U_{\rho_n,\xi_n})-B^\star|^2
\le C'\left(1+(B^\star)^2+B^2+Z^2\right)
\)
for a constant $C'$ independent of $n$. This bound is integrable because
$(B^\star,B)$ has finite second moments. Dominated convergence then gives
$r(\rho_n,\xi_n)\to r(\rho,\xi)$.

For proving continuity of $A$, we use the Stein representation \cref{eq:A-Stein-no-gap}:
\[
    A(\rho,\xi)
    =
    \frac{\xi}{w(\rho)}
    \mathbb E[Q_\rho(U_{\rho,\xi})Z].
\]
The integrand converges pointwise and is bounded in absolute value by
\(
C'(1+|B^\star|+|B|+|Z|)|Z|,
\)
which is integrable by Cauchy--Schwarz. Hence dominated convergence gives
continuity of $A$.

Finally, fix $\xi>0$. Conditional on $(B^\star,B)$, the random variable
\(
U_{0,\xi}
=
\frac{w(0)}{\theta(0)}B^\star
+
(1-\frac{w(0)}{\theta(0)})B
+
\frac{w(0)}{\xi}Z
\)
has a Gaussian density that is strictly positive on all of $\mathbb R$.
Its unconditional density is therefore also strictly positive on all of
$\mathbb R$. Since $h$ satisfies \cref{ass4,ass5}, by
\cref{lem:nonconstant-proximal}, the map
$Q_0=\prox_{w(0)h}$ is a nonconstant, nondecreasing absolutely continuous map.
Hence $Q_0'\ge0$ almost everywhere and $Q_0'>0$ on a set of positive Lebesgue measure. Therefore
$A(0,\xi)=\mathbb E[Q_0'(U_{0,\xi})]>0$.
\end{proof}

\begin{lemma}[Anchored Gaussian inequality]
\label{lem:mod-poincare}
Let $Z\sim N(0,1)$. If $\phi$ is absolutely continuous,
$\phi(0)=0$, and $\mathbb E[\phi'(Z)^2]<\infty,$
then $\mathbb E[\phi(Z)^2] \le \mathbb E[\phi'(Z)^2].$
\end{lemma}

\begin{proof}
For $z>0$, we have $\phi(z)=\int_0^z\phi'(t)\,dt,$
so Cauchy--Schwarz gives
\( \phi(z)^2 \le z\int_0^z\phi'(t)^2\,dt\).
If $\varphi$ denotes the standard Gaussian density, Tonelli's theorem
and the fact that $\int_t^\infty z\varphi(z)\,dz=\varphi(t)$
give
\[
\begin{aligned}
 \int_0^\infty\phi(z)^2\varphi(z)\,dz
    &\le \int_0^\infty
    z \left(\int_0^z\phi'(t)^2\,dt \right)\varphi(z)\,dz\\
    &= \int_0^\infty
    \phi'(t)^2
    \left(\int_t^\infty z\varphi(z)\,dz \right)dt\\
    &= \int_0^\infty \phi'(t)^2\varphi(t)\,dt.
\end{aligned}
\]
The same argument on $(-\infty,0]$ proves the result.
\end{proof}

\begin{lemma}[Uniform small-noise risk bound]
\label{lem:uniform-risk-estimate}
Fix $0<\rho_+<\rho_{\max}$. There exists a function $\varepsilon_h(\xi)\to0$ as $\xi\downarrow0$
such that, for all $\rho\in[0,\rho_+]$ and all $\xi>0$, we have $\xi^2r(\rho,\xi)
    \le
    w(\rho)^2A(\rho,\xi)+\varepsilon_h(\xi)$,
where the functions $A$ and $r$ are defined in \cref{eq:Aandr}.
Consequently, if $F_1(\rho,\xi)=0$, then $\xi^2r(\rho,\xi)
    \le
    w(\rho)\rho+\varepsilon_h(\xi).$
\end{lemma}

\begin{proof}
For fixed $\rho$, condition on $(B^\star,B)=(b^\star,b)$ and define
\[
\begin{aligned}
\phi_{b^\star,b}(z)
:=\xi\Bigg[&
Q_\rho\left(
\frac{w(\rho)}{\theta(\rho)}b^\star
+
\left(1-\frac{w(\rho)}{\theta(\rho)}\right)b
+
\frac{w(\rho)}{\xi}z
\right)\\
&-
Q_\rho\left(
\frac{w(\rho)}{\theta(\rho)}b^\star
+
\left(1-\frac{w(\rho)}{\theta(\rho)}\right)b
\right)
\Bigg].
\end{aligned}
\]
Then $\phi_{b^\star,b}(0)=0$ and, almost everywhere,
\[
\phi_{b^\star,b}'(z)
=
w(\rho)Q_\rho'\left(
\frac{w(\rho)}{\theta(\rho)}b^\star
+
\left(1-\frac{w(\rho)}{\theta(\rho)}\right)b
+
\frac{w(\rho)}{\xi}z
\right).
\]
By \cref{lem:mod-poincare} and the fact that
$0\le Q_\rho'\le1$ because $Q_\rho$ is $1$-Lipschitz, averaging first over $Z$ and then over $(B^\star,B)$ gives
\begin{equation}
\begin{aligned}
&\xi^2\mathbb E\left[
\left(
Q_\rho(U_{\rho,\xi})
-
Q_\rho\left(
\frac{w(\rho)}{\theta(\rho)}B^\star
+
\left(1-\frac{w(\rho)}{\theta(\rho)}\right)B
\right)
\right)^2
\right]\le
w(\rho)^2A(\rho,\xi).
\end{aligned}
\label{eq:increment-bound-no-gap}
\end{equation}

Since $Q_\rho$ is $1$-Lipschitz and $Q_\rho(0)$ is uniformly bounded over
$\rho\in[0,\rho_+]$, and since $w(\rho)/\theta(\rho)\in(0,1)$, there is a finite constant $C$ such that
\[
\left|
Q_\rho\left(
\frac{w(\rho)}{\theta(\rho)}B^\star
+
\left(1-\frac{w(\rho)}{\theta(\rho)}\right)B
\right)-B^\star
\right|
\le C+2|B^\star|+|B|.
\]
Therefore
\[
C_h
:=
\sup_{\rho\in[0,\rho_+]}
\mathbb E\left[
\left(
Q_\rho\left(
\frac{w(\rho)}{\theta(\rho)}B^\star
+
\left(1-\frac{w(\rho)}{\theta(\rho)}\right)B
\right)-B^\star
\right)^2
\right]
<\infty.
\]

Write
\[
\begin{aligned}
X_\rho
&:=
Q_\rho(U_{\rho,\xi})
-
Q_\rho\left(
\frac{w(\rho)}{\theta(\rho)}B^\star
+
\left(1-\frac{w(\rho)}{\theta(\rho)}\right)B
\right),\\
Y_\rho
&:=
Q_\rho\left(
\frac{w(\rho)}{\theta(\rho)}B^\star
+
\left(1-\frac{w(\rho)}{\theta(\rho)}\right)B
\right)-B^\star.
\end{aligned}
\]
Then $Q_\rho(U_{\rho,\xi})-B^\star=X_\rho+Y_\rho$. Let
$w_{\max}:=\max_{\rho\in[0,\rho_+]}w(\rho)<\infty$. Using
\cref{eq:increment-bound-no-gap}, the fact that $0\le A\le1$, and Cauchy--Schwarz,
\[
\begin{aligned}
\xi^2r(\rho,\xi)
=\xi^2\mathbb E[(X_\rho+Y_\rho)^2]&\le
w(\rho)^2A(\rho,\xi)
+2\xi w(\rho)\sqrt{A(\rho,\xi)C_h}
+\xi^2C_h\\
&\le
w(\rho)^2A(\rho,\xi)
+2\xi w_{\max}\sqrt{C_h}
+\xi^2C_h.
\end{aligned}
\]
Thus the result holds with
\(
\varepsilon_h(\xi):=2\xi w_{\max}\sqrt{C_h}+\xi^2C_h.
\)
If $F_1(\rho,\xi)=0$, then $w(\rho)A(\rho,\xi)=\rho$, and hence
$w(\rho)^2A(\rho,\xi)=w(\rho)\rho$. This proves the second claim.
\end{proof}

\begin{corollary}[Oracle specialization of the perturbed results]
\label{cor:oracle-perturbed-mse}
Under \cref{ass1,ass2,ass3,ass4}, suppose that $h$ has at
least two points in its effective domain and set $b=\beta^\star$.
Then \cref{ass5} holds because
\(\operatorname{dom}h\) contains at least two distinct points.
Since \(b=\beta^\star\), \cref{ass:general-reference-vector}
holds with \(B=B^\star\) almost surely. Consequently, for every $\lambda>0$, the
oracle specialization of \cref{eq: fpt-oracle-a,eq: fpt-oracle-b,eq: fpt-oracle-c,eq: fpt-oracle-d}
has a solution. For any such solution,
\begin{equation}
   \frac1p
    \|\widehat\beta_{\cvx}^{h,\lambda}-\beta^\star\|_2^2
    \xrightarrow{\PP}
    \mathbb E
    \left[
        \left(
            \prox_{\alpha_\lambda h}
            \left(B^\star+\sqrt{\nu_\lambda}Z\right)
            -B^\star
        \right)^2
    \right]
    =r^{h,\lambda},
    \label{eq:oracle-corollary-risk}
\end{equation}
where $\alpha_\lambda$ and $\nu_\lambda$ are defined in \cref{eq:alpha_and_nu}.
\end{corollary}

\begin{proof}
Set $b=\beta^\star$ in \cref{thm:perturbed-fixed-point,thm:perturbed-convex-mse}. The joint
empirical convergence required in
\cref{ass:general-reference-vector} follows from
\cref{ass3} by applying
the Lipschitz map $x\mapsto(x,x)$, while the independence requirements
follow from \cref{ass1,ass2}. Since $B=B^\star$
almost surely, the scalar
center in \cref{thm:perturbed-convex-mse} reduces to $B^\star$, giving
\cref{eq:oracle-corollary-risk}.
\end{proof}

\section{\texorpdfstring{Proof of \cref{thm:Bayes VAMP-lower-bound}}{Proof of the convex lower-bound theorem}}
\label{sec:main_result2_proof}

The proof has three steps. We first handle the degenerate and explicitly
solvable cases directly. In the remaining cases, we construct a so-called 
\(h\)-admissible lower-bound pair (see \cref{def:lb-pair}) and use
\cref{prop:perturbed-fp-var} to select an oracle fixed point above its
first coordinate. Finally, \cref{lem:lb-pair-risk} transfers the risk of
that fixed point to the unperturbed convex estimator.

Throughout the proof we use the oracle specialization \(b=\beta^\star\)
of \cref{thm:perturbed-fixed-point,thm:perturbed-convex-mse}.  Thus
\(B=B^\star\) almost surely, and
\(\widehat\beta_{\cvx}^{h,\lambda}\) and \(r^{h,\lambda}\) have the
meanings given in \cref{cor:oracle-perturbed-mse}.  We also abbreviate $T(\tau):=L_\mu(\cE_\pi(\tau)),$ and $V_\pi:=\VAR(B^\star).$ We first notice that if \(V_\pi=0\), then \(m(\tau_B^+)=0\), and the result follows from
nonnegativity of squared error. Assume henceforth that
\(V_\pi>0\). As a first step in the proof, we dispose of some edge cases.

\subsection{Degenerate and explicitly solvable cases} \label{sec:degenerate}

We first handle three cases directly: singleton-domain penalties, affine
penalties, and flat spectral laws. We will repeatedly use the following
deterministic comparison, which says that the oracle quadratic
perturbation can only decrease the estimator's squared distance from the
signal.
\begin{lemma}
\label{lem:oraclemonotone}
Let \(\widehat\beta_{\cvx}^h\) be any minimizer of the unperturbed convex
objective, and let \(\widehat\beta^{h,\lambda}=\widehat\beta_{\cvx}^{h,\lambda}\) minimize the oracle specialization of
\cref{eq:perturbed_beta}. Then, for every realization and every
\(\lambda>0\),
\begin{equation}
 \norm{\widehat\beta_{\cvx}^h-\beta^\star}_2^2
 \geq
 \norm{\widehat\beta^{h,\lambda}-\beta^\star}_2^2. \label{eq:detoracle}
\end{equation}
\end{lemma}

\begin{proof}
Let \(F_p\) be the unperturbed objective. Then $F_p(\widehat\beta^{h,\lambda})
 +\frac\lambda2\norm{\widehat\beta^{h,\lambda}-\beta^\star}_2^2
 \leq
 F_p(\widehat\beta_{\cvx}^h)
 +\frac\lambda2\norm{\widehat\beta_{\cvx}^h-\beta^\star}_2^2$ and optimality gives
$F_p(\widehat\beta_{\cvx}^h)
 \leq F_p(\widehat\beta^{h,\lambda})$.
Subtracting proves \cref{eq:detoracle}.
\end{proof}

\textbf{Singleton-domain penalties.} If \(\operatorname{dom}(h)=\{a\}\), then
\[\widehat\beta_{\cvx}^h=a\mathbf1_p.\]
In this case, Wasserstein--\(2\) convergence of
the empirical prior gives$\frac1p\norm{\widehat\beta_{\cvx}^h-\beta^\star}_2^2
 \xrightarrow{\mathrm{a.s.}}
 \E[(a-B^\star)^2].$
Next, write $b_\pi = \mathbb{E}[B^\star]$ and notice that $\E[(a-B^\star)^2]
 =\VAR(B^\star)+(a-b_\pi)^2\geq\VAR(B^\star)$. Moreover,
the scalar posterior mean cannot have larger risk than the constant
estimator \(b_\pi\), so \(m(\tau_B^+)\leq\VAR(B^\star)\), completing the proof.

\textbf{Affine penalties.}  
Suppose \(h(x)=ux+v\) for some \(u,v\in\R\). We first note that it is easy to show that an affine proximal map has
derivative one. This means that if $I_+ < \infty$ and $p_0 > 0$, an affine penalty does not satisfy
the $q$-bounded width condition in \cref{defn:q-width-bounded-b}. Indeed, in this case, \(q = 1 - p_0 < 1\) with $A_h(\kappa,\nu)=1$; hence, \cref{eq:q-width-condition} fails.

Now we show that if \(p_0=0\) or
\(I_+=\infty\), the conclusion of
\cref{thm:Bayes VAMP-lower-bound} holds.
For each \(\lambda>0\), choose an oracle fixed point supplied by
\cref{thm:perturbed-fixed-point}. For \(\nu_\lambda\) and \(\alpha_\lambda\)  defined in \cref{eq:alpha_and_nu}, denote its scalar risk by \(r^{h,\lambda}:=  \mathbb E
    [
        (
            \prox_{\alpha_\lambda h}
            (B^\star+\sqrt{\nu_\lambda}Z)
            -B^\star
        )^2
    ]\).
At this fixed point,
\begin{equation}\label{eq:lb-risk-nu}
 \prox_{\alpha_\lambda h}(y)=y-\alpha_\lambda u, \qquad \text { and } \qquad  r^{h, \lambda} = \nu_\lambda + \alpha_\lambda^2 u^2.
\end{equation}
Because the proximal derivative is one, the
fixed-point equation in \cref{eq: fpt-oracle-a} gives
\(\eta_\lambda=\bar\gamma_{1,\lambda}+\lambda\). Substitution into \cref{eq: fpt-oracle-c} gives $\eta_\lambda^{-1}
 =\E[\frac1{S+\lambda}]$.
Multiplying \cref{eq: fpt-oracle-d} by
\((\bar\gamma_{1,\lambda}/\eta_\lambda)^2\), and recalling the definition
of \(\nu_\lambda\), gives
\begin{align*}
 \nu_\lambda
 &=
 \E\!\left[
 \frac{\sigma^2S+\lambda^2\tau_{2,\lambda}}
      {(S+\lambda)^2}
 \right]
 -\left(\frac{\lambda}{\eta_\lambda}\right)^2
  \tau_{2,\lambda}\\
  &=
 \sigma^2\E\!\left[\frac{S}{(S+\lambda)^2}\right]
 +\lambda^2\tau_{2,\lambda}
 \left(
 \E\!\left[\frac1{(S+\lambda)^2}\right]
 -\left(\E\!\left[\frac1{S+\lambda}\right]\right)^2
 \right).
\end{align*}
Thus, the exact identity is $\nu_\lambda
 =
 \sigma^2\E[\frac{S}{(S+\lambda)^2}]
 +\lambda^2\tau_{2,\lambda}
 \VAR(\frac1{S+\lambda})$.
The second term is nonnegative, so
\begin{equation}
 \nu_\lambda
 \geq
 \sigma^2\E\!\left[\frac{S}{(S+\lambda)^2}\right]. \label{eq:afflower}
\end{equation}
Next, if \(I_+=\infty\), Fatou's lemma makes the right side of \cref{eq:afflower} diverge, so
\(\liminf_{\lambda\downarrow0} r^{h,\lambda} \geq \liminf_{\lambda\downarrow0} \nu_\lambda = \infty\), where the first inequality holds by \cref{eq:lb-risk-nu}. 
On the other hand, if \(p_0=0\) and \(I_+<\infty\), dominated convergence gives $\E[\frac{S}{(S+\lambda)^2}] \rightarrow I_+$ and consequently $\liminf_{\lambda\downarrow0} r^{h,\lambda}
 \geq\sigma^2I_+$.
It remains to compare this quantity with \(\tau_B^+\). For every
\(\omega\geq0\), rearranging \cref{eq:spectrum-mod} gives
\[
    L_\mu(\omega)+\omega
    =
    \left(
        \E\left[\frac{1}{\omega+\sigma^2/S}\right]
    \right)^{-1}  \leq
    \E\left[\omega+\frac{\sigma^2}{S}\right]
    =
    \omega+\sigma^2I_+, 
\]
where the inequality follows since \(S>0\) almost surely, using the harmonic--arithmetic mean inequality
and \(I_+=\E[S^{-1}]\).
Therefore \(L_\mu(\omega)\leq\sigma^2I_+\). Evaluating this inequality
at \(\omega=\cE_\pi(\tau_B^+)\) and using the Bayes--VAMP fixed-point
equation yields
$\liminf_{\lambda\downarrow0} r^{h,\lambda} \geq \sigma^2I_+
    \geq
    L_\mu\bigl(\cE_\pi(\tau_B^+)\bigr)
    =
    \tau_B^+
    \geq
    m(\tau_B^+).$
This means that for arbitrary \(\epsilon>0\), one can choose a fixed \(\lambda\) small enough so that 
\(r^{h,\lambda}\geq m(\tau_B^+)-\epsilon/2\) in both cases. Then, we have by \cref{lem:oraclemonotone}:
\begin{equation*}
\begin{aligned}
\PP\left(\frac{1}{p}\|\widehat\beta_{\cvx}^{h}-\beta^\star\|_2^2 < m(\tau_B^+) - \epsilon\right) &\leq \PP\left(\frac{1}{p}\|\widehat\beta_{\cvx}^{h}-\beta^\star\|_2^2 < r^{h, \lambda} - \frac{\epsilon}{2}\right) \\
&\leq \PP\left(\frac1p\norm{\widehat\beta^{h,\lambda}-\beta^\star}_2^2
 <r^{h, \lambda} - \frac{\epsilon}{2}
 \right) \rightarrow 0,
\end{aligned}
\end{equation*}
where the last convergence statement holds by \cref{cor:oracle-perturbed-mse}. Letting \(\epsilon\downarrow0\) concludes the proof of \cref{thm:Bayes VAMP-lower-bound} for the affine penalty.

\textbf{Flat spectral distribution.} 
Let \(\mu=\delta_s\) for some \(s>0\). Plugging this into \cref{eq:spectrum-mod}, we have \(L_\mu(\omega)=\sigma^2/s\) for every \(\omega\), as the Bayes--VAMP map
\(
T_{\mu,\pi}(\tau)
=
L_\mu(\cE_\pi(\tau))
\)
is constant, and its unique fixed point is
$\tau_B^+={\sigma^2}/{s}$ for such spectra.  Fix a penalty $h$ whose domain is not a singleton (this case has been taken care of above). For each
\(\lambda>0\), choose an oracle fixed point
$(\bar\gamma_{1,\lambda},\eta_\lambda,
      \tau_{1,\lambda},\tau_{2,\lambda})$
supplied by \cref{thm:perturbed-fixed-point}. For the flat spectral law and the Stieltjes transform \cref{eq:stieltjes-transform}, we have $G(c)=(s+c)^{-1},$ and $G^{-1}(\rho)=\frac1\rho-s.$
The linear-precision equation in \cref{eq: fpt-oracle-c} therefore gives
$\frac{\bar\gamma_{1,\lambda}}{\eta_\lambda}  =
    \frac{s}{\eta_\lambda};$ hence, $\bar\gamma_{1,\lambda}=s.$ Substituting this identity into the linear-variance equation in \cref{eq: fpt-oracle-d} gives
\begin{align*}
    \tau_{1,\lambda}
    &=
    \left(\frac{\eta_\lambda}{s}\right)^2
    \left[
        \frac{
            \sigma^2s+
            (\eta_\lambda-s)^2\tau_{2,\lambda}
        }{\eta_\lambda^2}
        -
        \left(
            \frac{\eta_\lambda-s}{\eta_\lambda}
        \right)^2
        \tau_{2,\lambda}
    \right] =
    \frac{\sigma^2}{s}.
\end{align*}
 For \(\nu_\lambda\) and \(\alpha_\lambda\)  defined in \cref{eq:alpha_and_nu}, let
$r^{h,\lambda}
    :=
    \E[
        (\prox_{\alpha_\lambda h}
            (B^\star+\sqrt{\nu_\lambda}Z)
            -
            B^\star)^2]$
be the asymptotic risk in \cref{cor:oracle-perturbed-mse}. 
Notice that we find
\begin{equation}
\label{eq:flat-effective-noise}
    \nu_\lambda
    =
    \left(
        \frac{\bar\gamma_{1,\lambda}}
             {\bar\gamma_{1,\lambda}+\lambda}
    \right)^2
    \tau_{1,\lambda}
    =
    \frac{\sigma^2s}{(s+\lambda)^2}
    \longrightarrow
    \frac{\sigma^2}{s}
    =
    \tau_B^+.
\end{equation}
Scalar Bayes optimality proved in \cref{lem:scalarcomparison}, gives $r^{h,\lambda}\geq m(\nu_\lambda).$
By \cref{eq:flat-effective-noise} and the continuity of \(m\) from
\cref{lem:mmsecalculus}, we have $\liminf_{\lambda\downarrow0}r^{h,\lambda}
    \geq
    m(\tau_B^+).$
So for any \(\epsilon>0\), we can choose fixed
\(\lambda>0\) sufficiently small such that
$r^{h,\lambda}
    \geq
    m(\tau_B^+)-\frac{\epsilon}{2}.$
\cref{lem:oraclemonotone,cor:oracle-perturbed-mse} then give
\begin{align*}
&\PP\left(
    \frac1p
    \|\widehat\beta_{\cvx}^{h}-\beta^\star\|_2^2
    <
    m(\tau_B^+)-\epsilon
\right)\leq
\PP\left(
    \frac1p
    \|\widehat\beta^{h,\lambda}-\beta^\star\|_2^2
    <
    r^{h,\lambda}-\frac{\epsilon}{2}
\right)
\longrightarrow0.
\end{align*}
Since \(\epsilon>0\) was arbitrary, this proves the result.

\subsection{Fixed-point selection in the remaining cases}
Having treated the degenerate and explicitly solvable cases, assume
that \(V_\pi>0\), that \(\mu\) is not a point mass, and that \(h\) is
non-affine with at least two points in its effective domain.
If \(p_0>0\) and \(I_+<\infty\), fix
\(\epsilon_0,\delta_0>0\) such that the \(q\)-bounded-width
condition for \(h\) is satisfied.

For every \(\tau>0\), take the following everywhere-defined version of
the posterior mean:
\begin{equation}
 f_\tau(y)
 =
 \frac{\displaystyle
       \int_{\R}b\exp\!\left\{-\frac{(y-b)^2}{2\tau}\right\}\pi(db)}
      {\displaystyle
       \int_{\R}\exp\!\left\{-\frac{(y-b)^2}{2\tau}\right\}\pi(db)}.
 \label{eq:canonicalposterior}
\end{equation}
This is the canonical version of the Bayes denoiser introduced in
\cref{eq:mmse_denoiser_func}; namely, \(f_\tau=\eta_{1/\tau}^{\rm B}\).
The denominator is strictly positive, and the Gaussian likelihood factor makes every
posterior moment finite. Thus \(f_\tau\) is a smooth version of
\(\E[B^\star\mid B^\star+\sqrt\tau Z=y]\). Let
\[
    Y_\tau=B^\star+\sqrt\tau Z,\qquad
    m(\tau):=m_\pi(\tau),\qquad
    \cE_\pi(\tau)
    =
    \frac{m(\tau)}{1-m(\tau)/\tau}.
\]
Recall that for a Lipschitz scalar estimator \(g:\mathbb R\to\mathbb R\), write $r_g(\tau)
    :=
    \E[(g(Y_\tau)-B^\star)^2],$ and $
    a_g(\tau)
    :=
    \E[g'(Y_\tau)].$
Whenever \(a_g(\tau)<1\), define
$\omega_g(\tau)
    :=
    \frac{r_g(\tau)-a_g(\tau)^2\tau}
         {(1-a_g(\tau))^2}$.

We now introduce a definition that isolates exactly what the fixed-point construction
needs at its lower horizontal edge.  The first condition says that the
spectral module sends \(u_b\) strictly above \(b\); the second says that
every proximal denoiser generated by \(h\) has extrinsic variance at
least \(u_b\) when its scalar noise variance is \(b\).

\begin{definition}[\(h\)-admissible lower-bound pair]\label{def:lb-pair}
A pair \((b,u_b)\), with \(b>0\) and \(u_b\in[0,\infty)\), is an
\(h\)-admissible lower-bound pair if
\begin{equation}
                       L_\mu(u_b)>b, \label{eq:lb-pair1}
\end{equation}
and if, for every \(\kappa>0\) and whenever \(\E[g'(Y_b)]<1\), the proximal map
\(g=\prox_{\kappa h}\) satisfies
\begin{equation}
 \frac{\E[(g(Y_b)-B^\star)^2]
       -b(\E[g'(Y_b)])^2}
      {(1-\E[g'(Y_b)])^2}
 \geq u_b \label{eq:lb-pair2}.
\end{equation}
\end{definition}

Our major technical result in this section is given in \cref{prop:perturbed-fp-var} below. This result shows that there is a selected oracle fixed point above a prescribed level and its proof can be found in \cref{sec:proof_prop_selected_fp}.

\begin{proposition}
\label{prop:perturbed-fp-var}
Suppose that \(\mu\) is not a point mass. Let \(h\in\cC\) be
non-affine, suppose that \(\operatorname{dom}(h)\) contains at least
two points, and let \((b,u_b)\) be an \(h\)-admissible lower-bound
pair as in \cref{def:lb-pair}.

Assume one of the following:
\begin{enumerate}[label=\textup{(\roman*)},leftmargin=2.2em]
\item \(p_0=0\) or \(I_+=\infty\);
\item \(p_0>0\) and \(I_+<\infty\), and there exist
\(\epsilon_0\in(0,\tau_B^+)\) and \(\delta_0>0\) such that
$ b\in[\tau_B^+-\epsilon_0,\tau_B^+]$
and $\limsup_{\kappa\to\infty}
    \sup_{\nu\in[\tau_B^+-\epsilon_0,\tau_B^+]}
    A_h(\kappa,\nu)
    \leq q-\delta_0.$
\end{enumerate}
Then there exists \(\lambda_b>0\) such that, for every fixed
\(\lambda\in(0,\lambda_b)\), the oracle specialization of the
perturbed fixed-point equations admits a solution whose effective
noise variance satisfies $\nu_\lambda\geq b.$
\end{proposition}

\begin{rem}
The Poincar\'e--Miranda argument in \cref{prop:perturbed-fp-var} does not
prove uniqueness.  It selects one oracle fixed point whose effective
noise satisfies \(\nu_\lambda\geq b\), which is exactly the fixed point
needed to complete the proof.
\end{rem}

The fixed-point selection result \cref{prop:perturbed-fp-var} applies
once we construct an \(h\)-admissible lower-bound pair whose first
coordinate equals or approaches \(\tau_B^+\), so it remains to construct such an \(h\)-admissible lower-bound pair satisfying
the hypotheses of \cref{prop:perturbed-fp-var}. We treat separately the
cases in which \(f_{\tau_B^+}\) is and is not a scalar convex proximal
map. In the exceptional regime \(p_0>0\) and \(I_+<\infty\), the first
coordinate of the pair must also lie in
\([\tau_B^+-\epsilon_0,\tau_B^+]\).

\subsubsection{Case 1: \texorpdfstring{\(f_{\tau_B^+}\)}{f-tau-B-+} is not proximal.} \label{sec:nonproximal-lb-pair}
In this case, we begin by introducing a lemma that gives an exact excess-extrinsic-variance
decomposition.
The lemma shows the Bayes extrinsic variance is a strict universal lower bound for
all convex proximal denoisers and is proved in \cref{sec:Lemma_pythag_proof}.

\begin{lemma}[Extrinsic Pythagorean identity]\label{lem:pythag}
Assume \(V_\pi>0\), fix \(\tau>0\), and suppose that \(f_\tau\) defined in \cref{eq:canonicalposterior} is not
a scalar proximal map. Then there exists
\(\Delta_\tau>0\), depending only on \((\pi,\tau)\), such that
$\omega_g(\tau)
    \geq
    \cE_\pi(\tau)+\Delta_\tau$
for every scalar proximal map \(g\) satisfying
$\E[g'(Y_\tau)]<1.$
\end{lemma}

Set $u_b:=\cE_\pi(\tau_B^+)+\Delta_{\tau_B^+}.$
Because \(\mu\) is not a point mass, \cref{lem:Lcalculus} implies that
\(L_\mu\) is strictly increasing. Hence
$ L_\mu(u_b)
    >
    L_\mu(\cE_\pi(\tau_B^+))
    =
    \tau_B^+,$
which verifies \cref{eq:lb-pair1} with \(b=\tau_B^+\).

Now fix \(\kappa>0\), let \(g=\prox_{\kappa h}\), and suppose that
\(a:=\E[g'(Y_{\tau_B^+})]<1\). By \cref{lem:pythag},
\[
    \frac{
        \E[(g(Y_{\tau_B^+})-B^\star)^2]-\tau_B^+a^2
    }{(1-a)^2}
    =
    \omega_g(\tau_B^+)
    \geq u_b.
\]
Thus \cref{eq:lb-pair2} also holds, and
\((\tau_B^+,u_b)\) is an \(h\)-admissible lower-bound pair. In the
regime \(p_0>0\) and \(I_+<\infty\), its first coordinate automatically
belongs to \([\tau_B^+-\epsilon_0,\tau_B^+]\).

\subsubsection{Case 2: \texorpdfstring{\(f_{\tau_B^+}\)}{f-tau-B-+} is proximal.}
We next show that the Bayes fixed point is crossed strictly from above:
there exists \(\epsilon_{\rm fp}>0\) such that
$T(\tau)>\tau$ for every $\tau\in(\tau_B^+-\epsilon_{\rm fp},\tau_B^+).$
It is enough to prove that \(T'(\tau_B^+)<1\).
To see this, abbreviate
\[
 f=f_{\tau_B^+},\quad m=m(\tau_B^+),\quad
 d=\frac{m}{\tau_B^+}\in(0,1),\quad
 \omega=\cE_\pi(\tau_B^+), \quad x = \frac{\omega S}{\sigma^2 + \omega S}.
\]
\cref{lem:mmsecalculus}, applied at \(\tau_B^+\), gives
\begin{equation}
 \cE_\pi'(\tau_B^+)
 =
 \frac{1}{(1-d)^2} \VAR(f'(Y_{\tau_B^+})). \label{eq:Ederiv}
\end{equation}
At a Bayes fixed point, we have $L_\mu(\omega)+\omega
 =
 (\E[\frac{S}{\sigma^2+\omega S}])^{-1}$.
Since \(\omega=d\tau_B^+/(1-d)\), this identity gives $\E[x]=d$.
By \cref{eq:Lglobalderiv},
we have $L_\mu'(\omega)
    =
    {\VAR(x)}/{\E[x]^2}.$
Since \(\E[x]=d\), this gives
\begin{equation}
    L_\mu'(\omega)
    =
    {\VAR(x)}/{d^2}.
    \label{eq:Lderivx}
    \end{equation}
Because \(f\) is proximal, \(0\leq f'\leq1\), whence $\VAR(f'(Y_{\tau_B^+}))\leq d(1-d)$.
Also \(0\leq x<1\) almost surely and
\(\PP(0<x<1)=\PP(S>0)>0\). Therefore $\VAR(x)<d(1-d)$.
Combining \cref{eq:Ederiv,eq:Lderivx} yields 
 \begin{equation}
    T'(\tau_B^+)=
 \frac{1}{d^2(1-d)^2} \cdot  \VAR(x) \cdot \VAR(f'(Y_{\tau_B^+}))
<1.
    \label{eq:proximal-fixed-point-transversality}
\end{equation}
Since \(T(\tau_B^+)=\tau_B^+\) and
\(T'(\tau_B^+)<1\), there exists
\(\epsilon_{\rm fp}>0\) such that
\(T(\tau)>\tau\) immediately to the left of \(\tau_B^+\).
Choose any sequence $b_j\uparrow\tau_B^+,$ with $b_j\in(\tau_B^+-\epsilon_{\rm fp},\tau_B^+).$
For each such \(b_j\), we have $L_\mu(\cE_\pi(b_j))
    =
    T(b_j)
    >
    b_j,$
so \cref{eq:lb-pair1} holds. Moreover, if
\(g=\prox_{\kappa h}\) and \(\E[g'(Y_{b_j})]<1\), then
\cref{lem:scalarcomparison} gives $\omega_g(b_j)\geq\cE_\pi(b_j),
$
which is exactly \cref{eq:lb-pair2}. Thus
\((b_j,\cE_\pi(b_j))\) is an \(h\)-admissible lower-bound pair.

In the regime $p_0 > 0$ and $I_+ < \infty$, take the sequence sufficiently close to
\(\tau_B^+\) so that we have
\(b_j\in (\tau_B^+-\epsilon_0,\tau_B^+)\) for all $j$; outside that regime there is no restriction on the choice of the sequence $\{b_j\}$.
Each \((b_j,\cE_\pi(b_j))\) is an
\(h\)-admissible lower-bound pair.

\subsection{Transfer from the selected fixed point to the estimator}
\label{sec:proof-main-theorem}

The final step no longer uses fixed-point topology.  The oracle
quadratic perturbation can only decrease squared distance to the true
signal, pathwise (see \cref{lem:oraclemonotone}).  We may therefore fix one sufficiently small
\(\lambda>0\), apply the asymptotic risk formula from
\cref{cor:oracle-perturbed-mse}, and compare with any minimizer of the
unperturbed objective.  In particular, no interchange of the
\(p\to\infty\) and \(\lambda\downarrow0\) limits is required.

To complete the proof, we prove the following:
\begin{lemma}[Risk transfer from a lower-bound pair]
\label{lem:lb-pair-risk}
Under the conditions of \cref{prop:perturbed-fp-var}, every
measurably selected unperturbed minimizer satisfies
$\pliminf_{p\to\infty}
    \frac1p
    \norm{\widehat\beta_{\cvx}^h-\beta^\star}_2^2
    \geq m(b).$
\end{lemma}

\begin{proof}
Choose one sufficiently small fixed \(\lambda>0\) and the
corresponding fixed point from
\cref{prop:perturbed-fp-var}. Its effective noise variance satisfies
\(\nu_\lambda\geq b\). Scalar Bayes optimality in
\cref{lem:scalarcomparison} and the monotonicity of \(m\) give
$r^{h,\lambda}
    \geq m(\nu_\lambda)
    \geq m(b),$
where \(r^{h,\lambda}\) is defined in
\cref{eq: perturbed-risk-version2}.

By \cref{cor:oracle-perturbed-mse}, applied to the fixed point
selected in \cref{prop:perturbed-fp-var},
we have $\frac1p
    \norm{\widehat\beta^{h,\lambda}-\beta^\star}_2^2
    \xrightarrow{\PP}
    r^{h,\lambda}.$
Together with the deterministic oracle comparison
\cref{eq:detoracle}, this gives, for every \(\epsilon>0\),
\begin{align*}
    &\PP\!\left(
        \frac1p
        \norm{\widehat\beta_{\cvx}^h-\beta^\star}_2^2
        <m(b)-\epsilon
    \right)\\
   &\quad\leq
    \PP\!\left(
        \frac1p
        \norm{\widehat\beta^{h,\lambda}-\beta^\star}_2^2
        <m(b)-\epsilon
    \right)\\
  &\quad\leq
    \PP\!\left(
        \frac1p
        \norm{\widehat\beta^{h,\lambda}-\beta^\star}_2^2
        <r^{h,\lambda}-\epsilon
    \right)
    \rightarrow 0.
\end{align*}
This is precisely the asserted probability-\(\liminf\) bound.
\end{proof}

We now complete the proof in the two cases above. If \(f_{\tau_B^+}\) is not proximal, Case~1 provides an admissible
lower-bound pair with first coordinate \(b=\tau_B^+\).
Therefore, \cref{lem:lb-pair-risk} gives
$\pliminf_{p\to\infty}
    \frac1p
    \|\widehat\beta_{\cvx}^h-\beta^\star\|_2^2
    \geq m(\tau_B^+).$ If \(f_{\tau_B^+}\) is proximal, Case~2 provides admissible
lower-bound pairs with first coordinates
\(b_j\uparrow\tau_B^+\). For every fixed \(j\),
\cref{lem:lb-pair-risk} gives
$\pliminf_{p\to\infty}
    \frac1p
    \|\widehat\beta_{\cvx}^h-\beta^\star\|_2^2
    \geq m(b_j).$
Since \(m\) is continuous by \cref{lem:mmsecalculus}, sending
\(j\to\infty\) yields the desired lower bound.

\subsection{Proofs for the fixed-point selection step}

Before beginning the proofs of \cref{prop:perturbed-fp-var,lem:pythag}, we state and prove a few technical results that will be useful.
Throughout the proof, let
\(
Y:=Y_\tau, f:=f_\tau, m:=m_\pi(\tau),
\)
and let \(P_Y\) denote the law of \(Y\).

The oracle perturbation changes the spectral module through two
nonnegative precision parameters: the perturbation strength \(\lambda\)
and a residual linear precision \(c\).  The next map records the
effective noise produced by that module.  Its key property is that
removing the perturbation cannot improve on the unshifted Bayes-VAMP
spectral map \(L_\mu\).

For \(\lambda>0\), \(c\geq0\), and \(\omega\geq0\), define
\begin{equation}
 L_{\lambda,c}(\omega)
 =
 \frac{\displaystyle
 \E\!\left[\frac{\sigma^2S}{(S+c+\lambda)^2}\right]
 +c^2\omega\,
 \VAR\!\left(\frac1{S+c+\lambda}\right)}
 {\displaystyle
 \left(
 1-c\E\!\left[\frac1{S+c+\lambda}\right]
 \right)^2}. \label{eq:Lshift}
\end{equation}
For \(c=\infty\), its continuous extension is $L_{\lambda,\infty}(\omega)
 =
 \frac{\sigma^2\E[S]+\omega\VAR(S)}{(\E[S]+\lambda)^2}$.

\begin{lemma}\label{lem:linearcomparison}
For every \(c\in(0,\infty]\) and \(\omega\geq0\),
$L_{0,c}(\omega)\geq L_\mu(\omega).$
If the conditional law of \(S\) given \(S>0\) is not a point mass,
equality holds exactly when
\(c=\sigma^2/\omega\), with the convention
\(\sigma^2/0=\infty\).  If that conditional law is a point mass, then
equality holds for every \(c\in(0,\infty]\).  Finally, if \(p_0=0\) and \(I_+<\infty\), then 
$ L_{0,0}(\omega)
    =
    \sigma^2\E[S^{-1}]
    =
    \sigma^2 I_+,$
and
$L_{0,0}(\omega)\geq L_\mu(\omega).$
\end{lemma}

\begin{proof}
For \(0<c<\infty\), put
\[
 a_c=\E\!\left[\frac{S}{S+c}\right],\quad
 A_c=\E\!\left[\frac{\sigma^2S}{(S+c)^2}\right],\quad
 \Delta_c=\VAR\!\left(\frac{S}{S+c}\right).
\]
Then \(L_{0,c}(\omega)=(A_c+\omega\Delta_c)/a_c^2\), and $L_{0,c}(\omega)+\omega
 =
 \frac1{a_c^2}{
 \E[\frac{S(\sigma^2+\omega S)}{(S+c)^2}]}$.
Cauchy--Schwarz gives $a_c^2
 \leq
 \E[\frac{S(\sigma^2+\omega S)}{(S+c)^2}]
 \E[\frac{S}{\sigma^2+\omega S}]$, and consequently, $L_{0,c}(\omega)
 \geq
 (\E[\frac{S}{\sigma^2+\omega S}])^{-1}
 -\omega
 =L_\mu(\omega)$.
Equality in Cauchy--Schwarz holds precisely when
\((\sigma^2+\omega S)/(S+c)\) is constant on the positive support of
\(\mu\).  If this support contains at least two points, this is
equivalent to \(c=\sigma^2/\omega\).  If it contains only one point,
the ratio is automatically constant and equality holds for every
\(c\in(0,\infty)\).

Next, the endpoint \(c=\infty\) is a \(0/0\) limit, so we make the cancellation
explicit. With \(v=1/c\),
\[
 L_{0,1/v}(\omega)
 =
 \frac{\displaystyle
 \E\!\left[\frac{\sigma^2S}{(1+vS)^2}\right]
 +\omega\VAR\!\left(\frac{S}{1+vS}\right)}
 {\displaystyle
 \left(\E\!\left[\frac{S}{1+vS}\right]\right)^2}.
\]
Compact support and \(\E[S]>0\) allow dominated convergence as
\(v\downarrow0\), giving $L_{0,\infty}(\omega)
 =
 \frac{\sigma^2\E[S]+\omega\VAR(S)}{(\E[S])^2}$.
Passing the already-proved finite-\(c\) inequality to this limit proves $L_{0,c}(\omega)\geq L_\mu(\omega)$ at \(c=\infty\). 

Finally, if \(p_0=0\) and
\(I_+<\infty\), then $c\,\E[\frac1{S+c}]\rightarrow0$ and $c^2\VAR(\frac1{S+c})
 \leq
 \E[(\frac{c}{S+c})^2]\rightarrow0$
by dominated convergence. Hence the denominator in
\cref{eq:Lshift} tends to one. Moreover,
\(S/(S+c)^2\leq1/S\), so dominated convergence gives $\E[\frac{S}{(S+c)^2}]\rightarrow I_+$.
Consequently \(L_{0,c}(\omega)\to\sigma^2I_+\), proving
\(L_{0,0}(\omega)=\sigma^2I_+\) and the other endpoint.
\end{proof}

We now reduce the four oracle equations in \cref{eq: fpt-oracle-a},\cref{eq: fpt-oracle-b},\cref{eq: fpt-oracle-c}, and \cref{eq: fpt-oracle-d}
to two equations.  The variable \(\rho\) parametrizes the linear
precision through the Stieltjes transform, while \(t\) is the effective
scalar noise variance.  The resulting residuals \(F_{1,\lambda}\) and
\(F_{2,\lambda}\) will be used in the Poincar\'e--Miranda argument in
\cref{prop:perturbed-fp-var}.

For the Stieltjes transform $G(d)=\E[\frac1{S+d}]$ defined for $d > 0$ and with \(1/0:=+\infty\) in the extended integral, we have $\rho_{\max}=\lim_{d\downarrow0}G(d)=\E[1/S]\in(0,\infty]$.
For \(\rho\in(0,\rho_{\max})\), let
\[
 d(\rho)=G^{-1}(\rho),\qquad
 R(\rho)=d(\rho)-\frac1\rho,\qquad
 \theta(\rho)=-\frac1{R(\rho)}.
\]
Since $-R(\rho)
 =
 \frac{\E[S/(S+d(\rho))]}{\E[1/(S+d(\rho))]}>0$, we have $\theta(\rho)>0$. Set
\begin{equation}\label{eq:definition-J-MS-K}
 J(\rho)=\E\!\left[\frac1{(S+d(\rho))^2}\right],
 \quad
 M_S(\rho)=\E\!\left[\frac{S}{(S+d(\rho))^2}\right],
\quad
 K(\rho)=\frac{M_S(\rho)}{J(\rho)}.
\end{equation}
Differentiation gives
\begin{equation}
 R'(\rho)=\frac1{\rho^2}-\frac1{J(\rho)}\geq0,\qquad
 0\leq-\rho\frac{R'(\rho)}{R(\rho)}<1. \label{eq:Rfacts}
\end{equation}
Indeed, the first identity follows from
\(d'(\rho)=-1/J(\rho)\). For the strict upper bound, a direct
calculation gives $1-d\rho-\rho^2R'(\rho)
 =
 \rho(\frac{\rho}{J(\rho)}-d)
 =
 \rho\,\frac{M_S(\rho)}{J(\rho)}>0$.
The functions above have continuous extensions at \(\rho=0\). To make
the endpoint explicit, write \(\mu_j=\E[S^j]\). Uniformly for bounded
\(S\), expansion in \(d^{-1}\) as \(d\to\infty\) gives $\rho=d^{-1}-\mu_1d^{-2}+\mu_2d^{-3}+O(d^{-4})$ and $J=d^{-2}-2\mu_1d^{-3}+3\mu_2d^{-4}+O(d^{-5})$.
It then follows that
\[
 R(0)=-\E[S],\qquad
 \theta(0)=\frac1{\E[S]},\qquad
 R'(0)=\VAR(S),\qquad
 K(0)=\E[S].
\]
Consequently, on every
compact interval \([0,\rho_+]\subset[0,\rho_{\max})\),
\[
 0<\inf\theta,\quad
 \sup R'<\infty,\quad
 0<\inf K\leq\sup K<\infty. 
\]
The strict lower bounds follow from \(\E[S]>0\), continuity, and
\(M_S(\rho)>0\).

For fixed \(\lambda>0\), define $w_\lambda(\rho)
 =
 \frac{\theta(\rho)}{1+\lambda\theta(\rho)}
 =
 \frac1{\theta(\rho)^{-1}+\lambda}$.
For an effective scalar-noise variable \(t>0\), set $Q_{\lambda,\rho}(y)=\prox_{w_\lambda(\rho)h}(y)$ and
\begin{align}
 A_\lambda(\rho,t)
 &=
 \E[Q_{\lambda,\rho}'(B^\star+\sqrt t Z)],\\
 \mathfrak r_\lambda(\rho,t)
 &=
 \E[(Q_{\lambda,\rho}(B^\star+\sqrt t Z)-B^\star)^2].
\end{align}
Finally, define
\begin{align}
 F_{1,\lambda}(\rho,t)
 &:=
 \rho-w_\lambda(\rho)A_\lambda(\rho,t), \label{eq:F1red}\\
 F_{2,\lambda}(\rho,t)
 &:=
 \frac{w_\lambda(\rho)^2}{t}
 \left(
 R'(\rho)\mathfrak r_\lambda(\rho,t)+\sigma^2K(\rho)
 \right)-1. \label{eq:F2red}
\end{align}
These are the functions \(F_1,F_2\) from the fixed-point proof in
\cref{eq:F1,eq:F2}, after replacing the inverse-noise variable \(\xi\) by
\(
                     \xi={w_\lambda(\rho)}/{\sqrt t}.
\)
By \cref{lem:F-continuity}, we have tht $F_{1, \lambda}$ and $F_{2, \lambda}$ are continuous on compact subsets of
\([0,\rho_{\max})\times(0,\infty)\).

\begin{lemma}[Spectral interpretation of the second residual]
\label{lem:zero-set-sign-match}
Suppose \(h\) is non-affine, that its effective domain contains at least two
points, and that $ F_{1,\lambda}(\rho,t)=0$.
Set
\[
 a=A_\lambda(\rho,t)=\frac{\rho}{w_\lambda(\rho)},\qquad
 c=\frac1\rho-\frac1{w_\lambda(\rho)},       \qquad
 \omega
 =
 \frac{\mathfrak r_\lambda(\rho,t)-a^2t}{(1-a)^2}.      
\]
Then \(0<a<1\), \(c\geq0\), \(d(\rho)=c+\lambda\), and
\begin{equation}
 F_{2,\lambda}(\rho,t)
 =
 \frac{\rho^2}{J(\rho)t}
 \left(L_{\lambda,c}(\omega)-t\right). \label{eq:sign-match}
\end{equation}
In particular, on the zero set of \(F_{1,\lambda}\), \(F_{2,\lambda}\) has the same sign as \(L_{\lambda,c}(\omega)-t\).
\end{lemma}

\begin{proof}
The equality \(a=\rho/w_\lambda(\rho)\) follows since $F_{1,\lambda}(\rho,t)=0$. \cref{lem:nonconstant-proximal} shows that the proximal
map is nonconstant. A nonconstant monotone
one-Lipschitz map has derivative strictly positive on a set of positive
Lebesgue measure. The input law has a strictly positive density, and
therefore \(a>0\).
If \(a=1\), then \(Q_{\lambda,\rho}'=1\) almost everywhere because the
law of \(B^\star+\sqrt tZ\) has a strictly positive density. Hence
\(Q_{\lambda,\rho}(y)=y+b\). The proximal optimality condition \cref{lem:prox-optimality-app} would
then imply that \(h\) is affine, contrary to assumption. Thus \(a<1\);
and
\(
 c=\frac1\rho-\frac1{w_\lambda}
   =\frac{1-a}{aw_\lambda}\geq0.
\)
By \cref{lem:scalarcomparison}, applied at scalar noise \(t\),
we have \(\omega\geq\cE_\pi(t)\geq0\).  Together with \(\lambda>0\) and
\(c\geq0\), this verifies all three domain restrictions in the
definition \cref{eq:Lshift} of \(L_{\lambda,c}(\omega)\).

From \(R=d-\frac{1}{\rho}\) and \(\frac{1}{w_\lambda}=-R+\lambda\), we have 
\(
 \frac1{w_\lambda}
 =
 \frac1\rho-d+\lambda,
\)
which proves \(d=c+\lambda\). Moreover
\[
 1-c\rho=a,\qquad
 c=\frac{1-a}{aw_\lambda},\qquad
 \rho=aw_\lambda.                                      
\]
For this paragraph, abbreviate $(J(\rho), M_S(\rho), w_\lambda(\rho), \mathfrak r_\lambda(\rho,t))$ as $(J, M_S, 
 w, r)$.
Since \(S+c+\lambda=S+d(\rho)\),
\[
 \VAR\!\left(\frac1{S+c+\lambda}\right)=J-\rho^2,
 \qquad
 L_{\lambda,c}(\omega)
 =\frac{\sigma^2M_S+c^2\omega(J-\rho^2)}{a^2}.
\]
Moreover,
\[
 w^2c^2\omega
 =
 \frac{r-a^2t}{a^2},
\qquad
 \rho=aw,
\qquad
 R'(\rho)=\frac{J-\rho^2}{\rho^2J},
\qquad
 K(\rho)=\frac{M_S}{J}.
\]
Substitution now gives,
\begin{align*}
 \frac{\rho^2}{Jt}(L_{\lambda,c}(\omega)-t)
 &=
 \frac{\sigma^2w^2M_S}{Jt}
 +\frac{J-\rho^2}{Ja^2t}(r-a^2t)
 -\frac{\rho^2}{J}\\
 &=
 \frac{\sigma^2w^2K(\rho)}{t}
 +\frac{w^2R'(\rho)r}{t}
 -\frac{J-\rho^2}{J}-\frac{\rho^2}{J}\\
 &=
 \frac{\sigma^2w^2K(\rho)}{t}
 +\frac{w^2R'(\rho)r}{t}-1=F_{2,\lambda}(\rho,t).
\end{align*}
This shows \cref{eq:sign-match}.
\end{proof}


\begin{lemma}[Persistence of the lower-edge sign]\label{lem:lb-pair-sign}
Let \(h\) be non-affine and have at least two points in its effective
domain, and let \((b,u_b)\) be an \(h\)-admissible lower-bound pair.
If \(p_0>0\) and \(I_+<\infty\), assume in addition that \(h\)
satisfies \cref{eq:q-width-condition} and that
\(b\in[\tau_B^+-\epsilon_0,\tau_B^+]\). If instead \(p_0=0\) or
\(I_+=\infty\), impose no restriction on \(b>0\) and no width
condition on \(h\). Then there is
\(\lambda_b>0\) such that, for every
\(\lambda\in(0,\lambda_b)\) and every
\(\rho\in(0,\rho_{\max})\) satisfying $ F_{1,\lambda}(\rho,b)=0$, one has
\begin{equation}
                  F_{2,\lambda}(\rho,b)>0. \label{eq:loweredgesign}
\end{equation}
\end{lemma}

\begin{proof}
Let \(\Delta_b:=L_\mu(u_b)-b>0\).  At a point satisfying
\(F_{1,\lambda}(\rho,b)=0\), define \(a,c,\omega\) as in
\cref{lem:zero-set-sign-match}.  Since the denoiser is proximal and
\(a<1\), \cref{eq:lb-pair2} gives \(\omega\geq u_b\).  By
\cref{eq:sign-match}, it is therefore enough to prove
\(L_{\lambda,c}(\omega)>b\).

Suppose, to the contrary, that this inequality is not uniform for small
\(\lambda\). Then there are
\(\lambda_k\downarrow0\) and \(\rho_k\) on the indicated zero set such
that
\begin{equation}
 L_{\lambda_k,c_k}(\omega_k)\leq b, \label{eq:badresid}
\end{equation}
where \(c_k\geq0\). After passing to a subsequence, either \(c_k\)
stays bounded away from zero or \(c_k\to0\).

\noindent\textbf{Case 1: the residual precision stays positive.}
For every \(\delta>0\) and fixed finite \(u\),
\begin{equation}
 \sup_{c\in[\delta,\infty]}
 |L_{\lambda,c}(u)-L_{0,c}(u)|\longrightarrow0. \label{eq:awayzero}
\end{equation}
To verify this, set \(v=1/c\). After cancellation of \(v^2\),
\[
 L_{\lambda,1/v}(u)
 =
 \frac{\displaystyle
 \E\!\left[
 \frac{\sigma^2S}{(1+v(S+\lambda))^2}\right]
 +u\,\VAR\!\left(
 \frac{S+\lambda}{1+v(S+\lambda)}
 \right)}
 {\displaystyle
 \left(
 \E\!\left[
 \frac{S+\lambda}{1+v(S+\lambda)}
 \right]\right)^2}.
\]
This is uniformly continuous on the compact set
\((\lambda,v)\in[0,1]\times[0,\delta^{-1}]\); its denominator is bounded
away from zero because \(\E[S]>0\). This proves \cref{eq:awayzero}.
By monotonicity in the last argument, \(\omega_k\geq u_b\),
\cref{eq:awayzero}, and \cref{lem:linearcomparison},
\[
 L_{\lambda_k,c_k}(\omega_k)
 \geq L_{\lambda_k,c_k}(u_b)
 \geq L_{0,c_k}(u_b)-o(1)
 \geq L_\mu(u_b)-o(1)
 =b+\Delta_b-o(1).
\]
For large \(k\), the right side is larger than \(b\), contradicting
\cref{eq:badresid}.

\noindent\textbf{Case 2: the residual precision collapses.}
Set \(d_k:=c_k+\lambda_k\downarrow0\).  The conclusion now depends on
the behavior of the spectrum at zero.

\noindent\textbf{Case 2(a): \(I_+=\infty\).}
The denominator in
\cref{eq:Lshift}, evaluated at a point with \(F_1=0\), equals \(a_k^2\leq1\).
The variance term is nonnegative. Hence $L_{\lambda_k,c_k}(\omega_k)
 \geq
 \sigma^2\E[\frac{S}{(S+d_k)^2}]$.
Fatou's lemma makes the right-hand side tend to \(+\infty\), because $\liminf_{k\to\infty}\frac{S}{(S+d_k)^2}
 =\frac{1}{S}\mathbf 1_{\{S>0\}}$.
This diverging lower bound contradicts \cref{eq:badresid}, which bounds
the same quantity above by the finite number \(b\).

\noindent\textbf{Case 2(b): \(p_0=0\) and \(I_+<\infty\).}
For fixed \(u_b\), dominated convergence
gives $L_{\lambda_k,c_k}(u_b)\longrightarrow\sigma^2I_+$.
Indeed, both denominators tend to one, the observation-noise term tends
to \(\sigma^2I_+\), and $c_k^2\VAR(\frac1{S+d_k})
 \leq
 \E[(\frac{c_k}{S+d_k})^2]
 \rightarrow0$.
The endpoint part of \cref{lem:linearcomparison} gives
\(\sigma^2I_+\geq L_\mu(u_b)=b+\Delta_b\), again contradicting
\cref{eq:badresid}.

\noindent\textbf{Case 2(c): \(p_0>0\) and \(I_+<\infty\).}
This is the only regime in which the spectral endpoint argument alone
does not exclude \(c_k\to0\).
Here \(d_k\to0\), and we have 
\begin{equation}
    \rho_k = G(d_k),
            \qquad
    a_k     = \frac{\rho_k}{w_{\lambda_k}(\rho_k)}
            = 1-c_k\rho_k. 
\label{eq:signaturea}
\end{equation}
Writing $G_+(d)=\E[\frac{\mathbf 1_{\{S>0\}}}{S+d}]$, then
\cref{eq:signaturea} yields $a_k
 =1-p_0\frac{c_k}{d_k}-c_kG_+(d_k)\geq q-c_kG_+(d_k)$.
 
The last term tends to zero by dominated convergence, since its integrand
is bounded by one and converges pointwise to zero on \(\{S>0\}\).
Consequently,
\begin{equation}
    \liminf_{k\to\infty}a_k\geq q. 
    \label{eq:alimq2}
\end{equation}
Also \(\rho_k\geq p_0/d_k\to\infty\), and since \(a_k\leq1\), we have $w_{\lambda_k}(\rho_k)=\frac{\rho_k}{a_k}\geq\rho_k
 \rightarrow\infty$.
On the other hand, $a_k=A_h(w_{\lambda_k}(\rho_k),b)$,
and \(b\) lies in $[\tau_B^+ - \epsilon_0, \tau_B^+]$. The condition in
\cref{eq:q-width-condition} therefore
implies \(a_k\leq q-\delta_0/2\) for all large \(k\).  This contradicts
\cref{eq:alimq2}, which says that the same sequence has limit inferior
at least \(q\).  This is precisely where \(q\)-bounded width enters the
proof.

All possible subsequential regimes contradict \cref{eq:badresid}. Therefore
\(L_{\lambda,c}(\omega)>b\) on the \(F_1\)-zero set for all sufficiently
small \(\lambda\). The positive prefactor in
\cref{eq:sign-match} proves \cref{eq:loweredgesign}.
\end{proof}

We now combine the inequality in \cref{eq:loweredgesign} with an
upper-edge estimate and two vertical signs.  The signs of \(F_2\) are
initially known only where \(F_1=0\), so we modify \(F_2\) away from
that zero set before applying the Poincar\'e--Miranda theorem.  The
modification does not create or remove any common zero.

\begin{lemma}\label{lem:largenoise}
Fix \(\lambda>0\) and
\(0<\rho_+<\rho_{\max}\). There is \(M_0<\infty\) such that, for every
\(M\geq M_0\) and every \(\rho\in[0,\rho_+]\),
\begin{equation}
 F_{1,\lambda}(\rho,M)=0
 \quad\Longrightarrow\quad
 F_{2,\lambda}(\rho,M)<0.                          
\end{equation}
\end{lemma}

\begin{proof}
Set
\(
    \xi_{\rho,M}
    :=
    {w_\lambda(\rho)}/{\sqrt M}.
\)
Under the change of variables
\(
    w(\rho)=w_\lambda(\rho)\) and \(
    \xi=\xi_{\rho,M},
\)
the quantities in \cref{lem:uniform-risk-estimate} satisfy
\(
    A(\rho,\xi)=A_\lambda(\rho,M)\) and \(
    r(\rho,\xi)=\mathfrak r_\lambda(\rho,M).
\)
Therefore, \cref{lem:uniform-risk-estimate} gives
\[
    \frac{w_\lambda(\rho)^2}{M}\,
    \mathfrak r_\lambda(\rho,M)
    \leq
    w_\lambda(\rho)^2A_\lambda(\rho,M)
    +
    \varepsilon_h(\xi_{\rho,M}),
\]
uniformly for \(\rho\in[0,\rho_+]\), where
\(
    \sup_{\rho\in[0,\rho_+]}
    \varepsilon_h(\xi_{\rho,M}) \to 0
\)
as \(M\to\infty\).

If \(F_{1,\lambda}(\rho,M)=0\), then
\(w_\lambda(\rho)A_\lambda(\rho,M)=\rho\). Hence 
$$F_{2,\lambda}(\rho,M)
\leq
R'(\rho)w_\lambda(\rho)\rho
+
R'(\rho)\varepsilon_h(\xi_{\rho,M})
+
\frac{1}{M} \sigma^2w_\lambda(\rho)^2K(\rho)
-1.$$
By \cref{eq:Rfacts}, we have $R'(\rho)w_\lambda(\rho)\rho
    \leq
    {w_\lambda(\rho)}/{\theta(\rho)}$.
Moreover,
\[
\sup_{\rho\in[0,\rho_+]}
    \frac{w_\lambda(\rho)}{\theta(\rho)}
    =
    \sup_{\rho\in[0,\rho_+]}
    \frac1{1+\lambda\theta(\rho)}
    <1.
\]
The remaining terms converge to zero uniformly in \(\rho\). Therefore
\(F_{2,\lambda}(\rho,M)<0\) for every sufficiently large \(M\).
\end{proof}

\subsubsection{Proof of \texorpdfstring{\cref{prop:perturbed-fp-var}}{lem-selected}} \label{sec:proof_prop_selected_fp}

\begin{proof}
Now, under the stated assumptions, \cref{lem:lb-pair-sign} supplies
\(\lambda_b>0\) such that the inequality in \cref{eq:loweredgesign} holds for every
\(\lambda\in(0,\lambda_b)\). Fix any such \(\lambda\). We construct a
rectangle in the \((\rho,t)\)-plane and apply the
Poincar\'e--Miranda theorem.

\noindent\textbf{Right vertical edge.}
Choose a finite \(0<\rho_+<\rho_{\max}\) such that
\begin{equation}
                         \rho_+>w_\lambda(\rho_+). \label{eq:rightchoice}
\end{equation}
This is always possible. If \(\rho_{\max}=\infty\), take
\(\rho_+>1/\lambda\), since \(w_\lambda<1/\lambda\). If
\(\rho_{\max}<\infty\), then \(d(\rho)\downarrow0\) and
\(\theta(\rho)\to\rho_{\max}\) as
\(\rho\uparrow\rho_{\max}\), whereas $w_\lambda(\rho)\to
 \frac{\rho_{\max}}{1+\lambda\rho_{\max}}
 <\rho_{\max}$.
Thus \cref{eq:rightchoice} holds near the endpoint. Since
\(A_\lambda\leq1\),
\begin{equation}
 F_{1,\lambda}(\rho_+,t)
 \geq\rho_+-w_\lambda(\rho_+)>0
 \quad\text{for every }t>0. \label{eq:rightsign}
\end{equation}

\noindent\textbf{Upper effective-noise edge.}
By \cref{lem:largenoise}, choose \(M>b\) so large that
\begin{equation}
 F_{2,\lambda}(\rho,M)<0
 \quad\text{whenever}\quad
 F_{1,\lambda}(\rho,M)=0,\quad \rho\in[0,\rho_+]. \label{eq:topsignonzero}
\end{equation}

\noindent\textbf{Left vertical edge.}
At \(\rho=0\), the proximal scale \(w_\lambda(0)\) is positive.
Because \(h\) has at least two points in its effective domain,
\cref{lem:nonconstant-proximal} shows that
\(\prox_{w_\lambda(0)h}\) is nonconstant. The law of
\(B^\star+\sqrt tZ\) has a strictly positive density; therefore $A_\lambda(0,t)>0$ for all $t > 0$. 
By continuity and compactness of \(t\in[b,M]\), choose
\(0<\rho_-<\rho_+\) so small that
\begin{equation}
 F_{1,\lambda}(\rho_-,t)<0
 \quad\text{for every }t\in[b,M]. \label{eq:leftsign}
\end{equation}

Let $\mathcal R=[\rho_-,\rho_+]\times[b,M]$.
\cref{eq:leftsign,eq:rightsign} give the opposite
vertical signs. By \cref{lem:lb-pair-sign}
\begin{equation}
 F_{2,\lambda}(\rho,b)>0
 \quad\text{whenever }F_{1,\lambda}(\rho,b)=0, \label{eq:bottomzero}
\end{equation}
and \cref{eq:topsignonzero} gives the opposite sign at \(t=M\).

To apply Poincar\'e--Miranda, we extend the two horizontal signs from
the \(F_1\)-zero sets to the whole horizontal edges without changing
any simultaneous zero. Define $Z_b=\{\rho\in[\rho_-,\rho_+]:F_{1,\lambda}(\rho,b)=0\}$ and $Z_M=\{\rho\in[\rho_-,\rho_+]:F_{1,\lambda}(\rho,M)=0\}$.
Both sets are nonempty by the vertical signs and compact by continuity.
The inequalities in \cref{eq:bottomzero,eq:topsignonzero} are strict. Hence there are open neighborhoods
\(U_b\supset Z_b\) and \(U_M\supset Z_M\) on which the corresponding
strict signs of \(F_2\) continue to hold.

On each nonempty compact complement of these neighborhoods, \(|F_1|\)
has a positive minimum; if a complement is empty, no domination is needed
on that edge. Since \(F_2\) is bounded on each horizontal edge, there is
a constant \(C>0\) large enough that the continuous modification
\begin{equation}
 \widetilde F_{2,\lambda}(\rho,t)
 =
 F_{2,\lambda}(\rho,t)
 +C\,H(t)|F_{1,\lambda}(\rho,t)|,\qquad
 H(t)=\frac{M+b-2t}{M-b}, \label{eq:correction}
\end{equation}
satisfies $\widetilde F_{2,\lambda}(\rho,b)>0$ and $\widetilde F_{2,\lambda}(\rho,M)<0
 \quad(\rho\in[\rho_-,\rho_+])$.
Indeed, \(H(b)=1\), \(H(M)=-1\); on \(U_b,U_M\) the original signs already
hold, and on their complements the \(C|F_1|\) term dominates the bounded
value of \(F_2\).

After an affine rescaling of \(\mathcal R\) to a square and, if
necessary, replacing the first component by its negative, the
continuous map $(\rho,t)\longmapsto
(F_{1,\lambda}(\rho,t),\widetilde F_{2,\lambda}(\rho,t))$ now satisfies the Poincar\'e--Miranda sign conditions on \(\mathcal R\).
Therefore it has a zero \((\rho_\star,t_\star)\in\mathcal R\). At a zero
of \(F_1\), the correction term in \cref{eq:correction} vanishes, so $F_{1,\lambda}(\rho_\star,t_\star)
 =
 F_{2,\lambda}(\rho_\star,t_\star)=0$.

 It remains to translate this common zero back to the four oracle
fixed-point parameters. Define
\begin{equation}
\begin{aligned}
 & \xi_\star=\frac{w_\lambda(\rho_\star)}{\sqrt{t_\star}},
 \quad
 \bar\gamma_{1,\lambda}=\frac1{\theta(\rho_\star)},
 \\
 &\eta_\lambda=\frac1{\rho_\star}, \quad
 \tau_{1,\lambda}
 =
 \frac{\theta(\rho_\star)^2}{\xi_\star^2},
\\
 &\tau_{2,\lambda}
 =
 \frac{\theta(\rho_\star)^2}
 {(\theta(\rho_\star)-\rho_\star)^2}
 \left(
 \mathfrak r_\lambda(\rho_\star,t_\star)
 -\frac{\rho_\star^2}{\xi_\star^2}
 \right). \label{eq:reconstruct}
\end{aligned}
\end{equation}
From \(F_{1, \lambda}=0\), we have 
\(\rho_\star=w_\lambda A_\lambda\leq w_\lambda<\theta\), so $0<\bar\gamma_{1,\lambda}<\eta_\lambda$.
\cref{lem:scalarcomparison}, applied at noise variance
\(t_\star\), gives the sharper inequality $\mathfrak r_\lambda(\rho_\star,t_\star)
 -A_\lambda^2t_\star
 \geq
 (1-A_\lambda)^2\cE_\pi(t_\star)>0$.
Here \(A_\lambda<1\) by \cref{lem:zero-set-sign-match}, and
\(\cE_\pi(t_\star)>0\) because \(\VAR(B^\star)>0\).
Because
\(\rho_\star^2/\xi_\star^2=A_\lambda^2t_\star\),
the last parameter in \cref{eq:reconstruct} is strictly positive.
We verify the four equations in order. Since $$\frac{\bar\gamma_{1,\lambda}}{\eta_\lambda}
 =\frac{\rho_\star}{\theta(\rho_\star)}
 =\frac{w_\lambda(\rho_\star)}{\theta(\rho_\star)}
  A_\lambda(\rho_\star,t_\star)
 =\frac{\bar\gamma_{1,\lambda}}
        {\bar\gamma_{1,\lambda}+\lambda}
  A_\lambda(\rho_\star,t_\star),$$ we have that \(F_{1, \lambda}=0\) is exactly the denoising-divergence equation in \cref{eq: fpt-oracle-a}. Moreover, $\frac{\eta_\lambda}{\eta_\lambda-\bar\gamma_{1,\lambda}}
 =\frac{\theta(\rho_\star)}
        {\theta(\rho_\star)-\rho_\star}$ and $(\frac{\bar\gamma_{1,\lambda}}{\eta_\lambda})^2
 \tau_{1,\lambda}
 =\frac{\rho_\star^2}{\xi_\star^2}$.
Therefore the displayed definition of \(\tau_{2,\lambda}\) in
\cref{eq:reconstruct} is precisely the denoising-variance equation in \cref{eq: fpt-oracle-b}.
Next, because \(R(\rho)=d(\rho)-1/\rho\) and
\(\theta=-1/R\), we have $1-\rho_\star d(\rho_\star)
 =-\rho_\star R(\rho_\star)
 ={\rho_\star}/{\theta(\rho_\star)}$,
which is the linear-precision equation in \cref{eq: fpt-oracle-c}, because
\(G^{-1}(\rho_\star)=d(\rho_\star)\).

The only remaining equation to verify is \cref{eq: fpt-oracle-d}. Evaluate all scalar functions at
\(\rho_\star\), and abbreviate
\[
 d=d(\rho_\star)=\eta_\lambda-\bar\gamma_{1,\lambda},
 \quad J=J(\rho_\star),\quad M_S=M_S(\rho_\star),
 \quad A=A_\lambda(\rho_\star,t_\star),\quad
 r=\mathfrak r_\lambda(\rho_\star,t_\star).
\]
After multiplying \cref{eq: fpt-oracle-d} by
\((\bar\gamma_{1,\lambda}/\eta_\lambda)^2\), its left side is
\(A^2t_\star\). \cref{eq: fpt-oracle-b} gives $d^2\tau_{2,\lambda}
=\eta_\lambda^2(r-A^2t_\star)
={(r-A^2t_\star)}/{\rho_\star^2}$, and \cref{eq: fpt-oracle-d} is thus equivalent to $A^2t_\star
=
\sigma^2M_S
+\frac{r-A^2t_\star}{\rho_\star^2}
(J-\rho_\star^2)$.
Moving the term containing \(A^2t_\star\) to the left and using
\(\rho_\star=A w_\lambda(\rho_\star)\) gives $t_\star
 =
 w_\lambda(\rho_\star)^2
 (
 \frac{J-\rho_\star^2}{\rho_\star^2J}\,r
 +\sigma^2\frac{M_S}{J}
 )$.
The right side is
\(w_\lambda^2(R'(\rho_\star)r+\sigma^2K(\rho_\star))\);
hence the last display is exactly
\(F_{2,\lambda}(\rho_\star,t_\star)=0\). This proves
\cref{eq: fpt-oracle-d}. Therefore \cref{eq:reconstruct} gives an
admissible oracle-specialized fixed point satisfying all four equations.

Finally, its scalar effective noise is $\nu_\lambda
 =
 (
 \frac{\bar\gamma_{1,\lambda}}
 {\bar\gamma_{1,\lambda}+\lambda}
 )^2\tau_{1,\lambda}
 =
 (\frac{w_\lambda}{\theta})^2
 \frac{\theta^2t_\star}{w_\lambda^2}
 =t_\star\in[b,M]$.

For later reference, denote the scalar risk of this selected solution by
\begin{equation}\label{eq: perturbed-risk-version2}
 r^{h,\lambda}
 :=
 \mathfrak r_\lambda(\rho_\star,t_\star).
\end{equation}
This agrees with the notation of
\cref{cor:oracle-perturbed-mse}.
This proves the result.
\end{proof}

\subsubsection{Proof of \texorpdfstring{\cref{lem:pythag}}{lem-pythag}} \label{sec:Lemma_pythag_proof}

\begin{proof}
Write
$d=\frac{1}{\tau}m(\tau),$ and $e_B(y)=\frac{1}{1-d}(f_\tau(y)-dy).$
For a scalar proximal map \(g\), let
$a=\E[g'(Y_\tau)],$ and $e_g(y)=\frac{1}{1-a}(g(y)-ay).$
We first establish the Pythagorean decomposition
$\omega_g(\tau)-\cE_\pi(\tau)
    =
    \|e_g-e_B\|_{L^2(P_{Y_\tau})}^2.$

Define on
\(\mathcal H=L^2(P_Y)\) the bounded linear functional $\ell(e)
 :=
 \frac1\tau\E[e(Y)(Y-f(Y))]$.
It is bounded because \(\E[(Y-f(Y))^2]=\tau-m<\infty\).
For the functions used below, \cref{lem:posterioribp} gives
\(\ell(e)=\E[e'(Y)]\). Indeed, \(e_g\) is Lipschitz, and we have $e_B'=\frac{f'-d}{1-d}$ and $\E|e_B'(Y)|\leq\frac{2d}{1-d}<\infty$;
both functions are square-integrable with respect to the law of $Y$.
In particular, \(\ell(e_g)=\ell(e_B)=0\). Furthermore,
\(
 f-e_B=\frac{d}{1-d}(Y-f),
\)
so \(f-e_B\) is a scalar multiple of the Riesz representer of \(\ell\).
Thus \(e_B\) is the orthogonal projection of \(f\) onto the closed
hyperplane \(\ker\ell\).

To make the risk identity explicit, observe that $ (1-a)(e_g(Y)-B^\star)
 =g(Y)-B^\star-a\sqrt\tau Z$,
and, conditionally on \(B^\star\), Gaussian integration by parts gives $\E[Z(g(Y)-B^\star)]
 =\sqrt\tau\,\E[g'(Y)]
 =\sqrt\tau\,a$.
Expanding the preceding square therefore gives $\E[(e_g(Y)-B^\star)^2]
 =\frac{r-a^2\tau}{(1-a)^2}
 =\omega_g(\tau)$.
A direct calculation gives $\norm{f-e_B}_{\mathcal H}^2
 =\frac{d^2}{(1-d)^2}\E[(Y-f(Y))^2]
 =\frac{\tau d^2}{1-d}$.
Because conditional expectation is an \(L^2\)-projection, we have that $\E[(e(Y)-B^\star)^2]
 =m+\norm{e-f}_{\mathcal H}^2$, and taking \(e=e_B\) gives $\E[(e_B(Y)-B^\star)^2]
 =\tau d+\frac{\tau d^2}{1-d}
 =\frac{\tau d}{1-d}
 =\cE_\pi(\tau)$.
Applying the Hilbert-space Pythagorean theorem inside \(\ker\ell\) proves $\omega_g(\tau)-\cE_\pi(\tau)
    =
    \|e_g-e_B\|_{L^2(P_{Y_\tau})}^2.$

It remains to prove uniform strictness: the uniform gap over all scalar proximal maps. By
\cref{eq:canonicalposterior}, \(f\) is smooth, and the density of \(Y\)
is smooth and strictly positive on \(\R\). Consider the closed set
\[
 \mathcal K
 =
 \left\{
 e\in\mathcal H:\ell(e)=0,\quad De\leq\mathcal L^1
 \text{ in the sense of distributions}
 \right\}.
\]
To see that \(\mathcal K\) is closed, suppose \(e_n\to e\) in
\(\mathcal H\). On every compact interval, the density of \(Y\) is bounded
below by a positive number, so \(e_n\to e\) in \(L^1_{\mathrm{loc}}\).
For every nonnegative \(\varphi\in C_c^\infty(\R)\), the inequality $-\int e_n\varphi'\leq\int\varphi$ passes to the limit. The condition \(\ell(e)=0\) also passes to the limit
because \(\ell\) is bounded. Hence \(e\in\mathcal K\).

If \(g\) is proximal, then \(0\leq g'\leq1\) almost everywhere and $ e_g'=\frac{g'-a}{1-a}\leq1$.
Thus \(e_g\in\mathcal K\). On the other hand, $ e_B'=\frac{f'-d}{1-d}$.
The posterior mean is nondecreasing because
\(f'=\VAR(B^\star\mid Y=\cdot)/\tau\geq0\). By \cref{prop:scalar-prox-characterization-app}, a map
\(T:\mathbb R\to\mathbb R\) is the proximal map of some proper, lsc convex function if and only if it is nondecreasing
and one-Lipschitz. Since \(f=f_\tau\) is nondecreasing, the assumption
that \(f\) is not a scalar proximal map implies that
\(f\) is not one-Lipschitz. If \(f'\leq1\) almost everywhere, then local absolute continuity implies that $f$ is one-Lipschitz. Because $f$ is locally absolutely continuous but not one-Lipschitz, \(f'>1\) on a set of positive Lebesgue measure. 
Consequently, \(e_B'>1\) on a set of positive Lebesgue measure, so
\(e_B\notin\mathcal K\). Since \(\mathcal K\) is closed,
$\Delta_\tau
    :=
    \operatorname{dist}_{\mathcal H}(e_B,\mathcal K)^2
    >0.$
Finally, since \(e_g\in\mathcal K\), the Pythagorean identity gives
$ \omega_g(\tau)-\cE_\pi(\tau)
    =
    \|e_g-e_B\|_{\mathcal H}^2
    \geq
    \operatorname{dist}_{\mathcal H}(e_B,\mathcal K)^2
    =
    \Delta_\tau.$
Therefore,
$\omega_g(\tau)
    \geq
    \cE_\pi(\tau)+\Delta_\tau,$
which proves the result.
\end{proof}

\section{Technical lemmas}\label{sec:technical-lemmas}

For every \(\tau>0\), recall the  everywhere-defined version of
the posterior mean $f_{\tau}$ defined in \cref{eq:canonicalposterior}.
Let
\[
    Y_\tau=B^\star+\sqrt\tau Z,\qquad
    m(\tau):=m_\pi(\tau),\qquad
    \cE_\pi(\tau)
    =
    \frac{m(\tau)}{1-m(\tau)/\tau}.
\]
Recall that for a Lipschitz scalar estimator \(g:\mathbb R\to\mathbb R\), write $r_g(\tau)
    :=
    \E[(g(Y_\tau)-B^\star)^2],$ and $
    a_g(\tau)
    :=
    \E[g'(Y_\tau)].$
Whenever \(a_g(\tau)<1\), define
$\omega_g(\tau)
    :=
    \frac{r_g(\tau)-a_g(\tau)^2\tau}
         {(1-a_g(\tau))^2}$.

\begin{lemma}
\label{lem:mmsecalculus}
Assume \(0<V_\pi<\infty\). For every \(\tau>0\),
\begin{align}
 f_\tau'(y)
 &=
 \frac{\VAR(B^\star\mid Y_\tau=y)}{\tau}
 \text{for Lebesgue-a.e. }y,
 \label{eq:postderiv} \qquad  d_\tau
 :=
 \E[f_\tau'(Y_\tau)]
 =
 \frac{m(\tau)}{\tau}
 \in(0,1), \\
 m'(\tau)
 &=
 \E[f_\tau'(Y_\tau)^2],
 \label{eq:mmseprime} \qquad 
 \cE_\pi'(\tau)
 =
 \frac{\VAR(f_\tau'(Y_\tau))}
      {(1-d_\tau)^2}
 \geq 0.
\end{align}
Consequently, \(m\) and \(\cE_\pi\) are continuous and
nondecreasing on \((0,\infty)\).  Moreover, for every \(\tau>0\),
\(0<\cE_\pi(\tau)\leq V_\pi\), and for every \(\tau\geq0\),
\begin{equation}
 m(\tau)
 \leq
 \frac{V_\pi\tau}{V_\pi+\tau} \leq V_\pi,
 \label{eq:bestlinear}
\end{equation}
where the second inequality is strict for every finite \(\tau\).
The function \(m\) extends continuously to \(\tau=0\) by setting
\(m(0)=0\). Finally,
\begin{equation}
 \lim_{\tau\to\infty}m(\tau)
 =
 \lim_{\tau\to\infty}\cE_\pi(\tau)
 =
 V_\pi.
 \label{eq:mmse-large-noise-limit}
\end{equation}
\end{lemma}

\begin{proof}
Let \(p_\tau\) denote the everywhere-positive smooth density of
\(Y_\tau\), and let \(s_\tau=(\log p_\tau)'\) be its score.
Differentiating the Gaussian convolution under the integral gives
Tweedie's identities
$f_\tau(y)
 =
 y+\tau s_\tau(y),$ and $
 f_\tau'(y)
 =
 1+\tau s_\tau'(y)
 =
 \frac{1}{\tau}\VAR(B^\star\mid Y_\tau=y).$
These identities hold pointwise for the canonical smooth version of
\(f_\tau\), and hence also \(P_{Y_\tau}\)-almost surely. Averaging the
second identity and using the conditional-variance representation of
the scalar MMSE gives
$\E[f_\tau'(Y_\tau)]
 =
 \frac{1}{\tau}\E[\VAR(B^\star\mid Y_\tau)]
 =
 \frac{1}{\tau}m(\tau).$

We next establish the bounds needed to show that this quantity lies in
\((0,1)\). The linear estimator
$\widehat B_{\mathrm{lin}}
 =
 b_\pi+
 \frac{V_\pi}{V_\pi+\tau}(Y_\tau-b_\pi)$
has mean-squared error
$\E[(B^\star-\widehat B_{\mathrm{lin}})^2]
 =
 \frac{V_\pi\tau}{V_\pi+\tau}.$
Because the conditional expectation minimizes mean-squared error, $m(\tau)
 \leq
 \frac{V_\pi\tau}{V_\pi+\tau}.$
Furthermore, the Gaussian likelihood is strictly positive, so the
posterior distribution is equivalent to \(\pi\). Since \(V_\pi>0\),
the posterior is nondegenerate, and therefore \(m(\tau)>0\).
The linear bound is strictly smaller than \(\tau\), so
$0<d_\tau=\frac{1}{\tau}m(\tau)<1.$
Since \(x\mapsto x/(1-x/\tau)\) is increasing on \([0,\tau)\),
the same bound gives
\[
 0<
 \cE_\pi(\tau)
 =
 \frac{m(\tau)}{1-m(\tau)/\tau}
 \leq V_\pi,
\]
proving \cref{eq:bestlinear}. The bound on \(m(\tau)\) also implies
\(m(\tau)\to0=m(0)\) as \(\tau\downarrow0\).

We now use the positive-SNR MMSE differentiation theorem
\parencite[Proposition~9]{Guo2011MMSE}. In the present
noise-variance parametrization, it states that, for every
\(\tau>0\), we have
\begin{equation}
 m'(\tau)
 =
 \frac{1}{\tau^2}
 \E\!\left[
   \VAR(B^\star\mid Y_\tau)^2
 \right].
 \label{eq:mmsetheorem}
\end{equation}
Indeed, observing \(Y_\tau\) is equivalent, after multiplication by
\(\tau^{-1/2}\), to observing
\(\sqrt{\gamma}B^\star+Z\) with \(\gamma=1/\tau\).
No fourth-moment assumption on the prior is needed: conditional
Cauchy--Schwarz gives
\[
 \VAR(B^\star\mid Y_\tau)^2
 \leq
 \E[(B^\star-Y_\tau)^2\mid Y_\tau]^2
 \leq
 \E[(B^\star-Y_\tau)^4\mid Y_\tau],
\]
and the expectation of the right side is
\(\E[(\sqrt{\tau}Z)^4]=3\tau^2\).
Combining \cref{eq:mmsetheorem} with \cref{eq:postderiv} proves
\cref{eq:mmseprime}.

Differentiating
$\cE_\pi(\tau)
 =
 \frac{m(\tau)}{1-m(\tau)/\tau}
 =
 \frac{\tau d_\tau}{1-d_\tau}$
and using \cref{eq:mmseprime} yields
\[
 \cE_\pi'(\tau)
 =
 \frac{
   \E[f_\tau'(Y_\tau)^2]-d_\tau^2
 }{(1-d_\tau)^2}
 =
 \frac{
   \VAR(f_\tau'(Y_\tau))
 }{(1-d_\tau)^2}.
\]
Thus \(m\) and \(\cE_\pi\) are nondecreasing and continuous on
\((0,\infty)\).

It remains to prove the limits in
\cref{eq:mmse-large-noise-limit}. Set
$W_\tau
 :=
 \frac{1}{\sqrt{\tau}} Y_\tau
 =
 Z+\frac{1}{\sqrt{\tau}}B^\star.$
Since \(W_\tau\) and \(Y_\tau\) generate the same sigma-field, the
orthogonality property of conditional expectation gives
\begin{equation}
 m(\tau)
 =
 V_\pi-
 \left\|
   \E[B^\star\mid W_\tau]-b_\pi
 \right\|_{L^2}^2.
 \label{eq:mmse-total-variance}
\end{equation}
We claim that
$\E[B^\star\mid W_\tau]
 \rightarrow b_\pi$ in $L^2.$
To prove this, let
$B_M^\star
 :=
 B^\star\mathbf 1\{|B^\star|\leq M\},$ and $
 b_M:=\E[B_M^\star],$
and let \(\psi\) denote the standard Gaussian density. Bayes' formula
gives
\[
 h_{\tau,M}(w)
 :=
 \E[B_M^\star\mid W_\tau=w]
 =
 \frac{
   \E\!\left[
     B_M^\star\psi(w-B^\star/\sqrt{\tau})
   \right]
 }{
   \E\!\left[
     \psi(w-B^\star/\sqrt{\tau})
   \right]
 }.
\]
Because \(\psi\) is globally Lipschitz and \(B^\star\in L^1\), for
every compact \(K\subset\R\), we have 
\[
\begin{aligned}
&\sup_{w\in K}
 |
   \E[\psi(w-B^\star/\sqrt{\tau})]-\psi(w)
 |
 \rightarrow0\\
&\text{ and } \quad  \sup_{w\in K}
 |
   \E[B_M^\star\psi(w-B^\star/\sqrt{\tau})]
   -b_M\psi(w)
 |
 \rightarrow0.
\end{aligned}
\]
Since \(\inf_{w\in K}\psi(w)>0\), it follows that
\(h_{\tau,M}\to b_M\) uniformly on compact sets. Moreover,
\(|h_{\tau,M}|\leq M\), and \(W_\tau\to Z\) in probability, so the
family \(\{W_\tau:\tau\geq1\}\) is tight. Therefore,
$\|
   \E[B_M^\star\mid W_\tau]-b_M
 \|_{L^2}
 \rightarrow0$ for every fixed $M.$

Finally, conditional expectation is an \(L^2\)-contraction, and hence
\begin{align*}
 \left\|
   \E[B^\star\mid W_\tau]-b_\pi
 \right\|_{L^2}
 &\leq
 \|B^\star-B_M^\star\|_{L^2}
 +
 \left\|
   \E[B_M^\star\mid W_\tau]-b_M
 \right\|_{L^2}
 +
 |b_M-b_\pi|.
\end{align*}
First letting \(\tau\to\infty\) and then \(M\to\infty\) proves the
claimed \(L^2\) convergence. \cref{eq:mmse-total-variance} therefore gives
\(m(\tau)\to V_\pi\). Finally, since
\(m(\tau)\leq V_\pi\), we have \(m(\tau)/\tau\to0\), and consequently
$\cE_\pi(\tau)
 =
 \frac{m(\tau)}{1-m(\tau)/\tau}
 \rightarrow V_\pi.$
\end{proof}

\begin{lemma}[Posterior integration by parts]\label{lem:posterioribp}
Fix \(\tau>0\). If \(e\) is locally absolutely continuous,
\(e(Y_\tau)\in L^2\), and
\(\E|e'(Y_\tau)|<\infty\), then
$\E[e'(Y_\tau)]
 =\frac1\tau
  \E[e(Y_\tau)(Y_\tau-f_\tau(Y_\tau))].$
\end{lemma}

\begin{proof}
Tweedie's identity gives
\(
 \frac{p_\tau'(y)}{p_\tau(y)}
 =\frac{1}{\tau}(f_\tau(y)-y).
\)
Let \(\chi_R\) be a smooth cutoff which equals one on \([-R,R]\), is
supported on \([-2R,2R]\), and satisfies
\(\|\chi_R'\|_\infty\leq C/R\). Ordinary compactly supported integration
by parts gives
\[
 -\int e\chi_R p_\tau'
 =\int e'\chi_Rp_\tau+\int e\chi_R'p_\tau.
\]
The first term on the right converges by dominated convergence. The
absolute value of the second is at most
\((C/R)\E|e(Y_\tau)|\) and hence tends to zero. The left side converges to
\(\tau^{-1}\E[e(Y_\tau)(Y_\tau-f_\tau(Y_\tau))]\);
this is integrable by Cauchy--Schwarz because
\(\E[(Y_\tau-f_\tau(Y_\tau))^2]\leq\tau\).
\end{proof}

\begin{lemma}[Bayes risk and extrinsic-variance comparison]
\label{lem:scalarcomparison}
Fix \(\tau>0\), and let \(g:\mathbb R\to\mathbb R\) be Lipschitz.
Then
\begin{equation}
    r_g(\tau)
    =
    m(\tau)
    +
    \E[(g(Y_\tau)-f_\tau(Y_\tau))^2]
    \geq m(\tau).
    \label{eq:scalar-bayes-comparison}
\end{equation}
Equality holds if and only if
\(g(Y_\tau)=f_\tau(Y_\tau)\) almost surely. If \(a_g(\tau)<1\), then $\omega_g(\tau)\geq\cE_\pi(\tau).$
\end{lemma}

\begin{proof}
Fix \(\tau>0\), and abbreviate $(Y_\tau, f_\tau, m(\tau), a_g(\tau), r_g(\tau))$ as $(Y, f, m, a, r)$.
Set $ u(Y):=g(Y)-f(Y).$ Because \(f(Y)=\E[B^\star\mid Y]\), the conditional-expectation
projection identity gives
\begin{equation}
    r
    =
    m+\E[u(Y)^2].
    \label{eq:scalar-projection-identity}
\end{equation}
Indeed, the cross term in $ g(Y)-B^\star
    =
    (g(Y)-f(Y))
    +
    (f(Y)-B^\star)$
vanishes after conditioning on \(Y\). This proves
\(r\geq m\), with equality if and only if
\(g(Y)=f(Y)\) almost surely.

It remains to prove the extrinsic-variance comparison. Define
$d:={m}/{\tau}.$
By \cref{lem:mmsecalculus}, we have $\E[f'(Y)]=d.$
Consequently, $\E[u'(Y)]=a-d.$
Applying the posterior integration-by-parts identity
\cref{lem:posterioribp} to \(u\) gives
\begin{equation}
    a-d
    =
    \frac1\tau
    \E\bigl[u(Y)(Y-f(Y))\bigr].
    \label{eq:scalar-comparison-ibp}
\end{equation}
The hypotheses of \cref{lem:posterioribp} hold because \(g\) is
Lipschitz, \(u(Y)\in L^2\), and
$ \E|u'(Y)|
    \leq \operatorname{Lip}(g)+\E[f'(Y)]
    =\operatorname{Lip}(g)+d<\infty.$

We also have
that $\E[(Y-f(Y))^2]=\tau-m=\tau(1-d).$
To see this, write
$ Y-B^\star
    =
    \bigl(Y-f(Y)\bigr)
    +
    \bigl(f(Y)-B^\star\bigr).$
The two terms on the right are orthogonal because
$\E[f(Y)-B^\star\mid Y]=0.$
Taking squared expectations and using
\(\E[(Y-B^\star)^2]=\tau\) therefore gives
the result.

Now Cauchy--Schwarz in \cref{eq:scalar-comparison-ibp} yields
\[
    (a-d)^2
    \leq
    \frac1{\tau^2}
    \E[u(Y)^2]\,
    \E[(Y-f(Y))^2]
    =
    \frac{1-d}{\tau}\E[u(Y)^2].
\]
Since \(d<1\), it follows that
$\E[u(Y)^2]
    \geq
    \frac{\tau(a-d)^2}{1-d}.$
Combining this with \cref{eq:scalar-projection-identity}, and using \(m=\tau d\), gives
\begin{align*}
    r-a^2\tau
    &=
    \tau d+\E[u(Y)^2]-a^2\tau
    \geq
    \tau\left(
        d-a^2+\frac{(a-d)^2}{1-d}
    \right)
   =
    \frac{\tau d(1-a)^2}{1-d}.
\end{align*}
If \(a<1\), division by \((1-a)^2\) therefore gives
\[
    \omega_g(\tau)
    =
    \frac{r-a^2\tau}{(1-a)^2}
    \geq
    \frac{\tau d}{1-d}
    =
    \frac{m(\tau)}{1-m(\tau)/\tau}
    =
    \cE_\pi(\tau).
\]
This proves the result.
\end{proof}

\begin{lemma}[Spectral-map calculus]\label{lem:Lcalculus}
The map \(L_\mu:[0,\infty)\to(0,\infty)\), with
\(L_\mu(0)=\sigma^2/\E[S]\), is continuous and nondecreasing. For
\(\omega>0\), if
\[
 q_\mu(\omega)=\E\!\left[\frac{\sigma^2}{\sigma^2+\omega S}\right],
 \qquad
 x_\omega(S)=\frac{\omega S}{\sigma^2+\omega S},
\]
then
\begin{equation}
 L_\mu(\omega)=\frac{\omega q_\mu(\omega)}{1-q_\mu(\omega)},
 \qquad
 L_\mu'(\omega)
 =\frac{\VAR(x_\omega(S))}{(1-q_\mu(\omega))^2}\geq0. \label{eq:Lglobalderiv}
\end{equation}
It is strictly increasing unless \(\mu\) is a point mass.
Consequently \(T=L_\mu\circ\cE_\pi\) is continuous and
nondecreasing.
\end{lemma}

\begin{proof}
The first representation follows directly from \cref{eq:spectrum-mod}. Since $\omega q_\mu'(\omega)
 =-\E[
 \frac{\sigma^2\omega S}{(\sigma^2+\omega S)^2}
 ]$, differentiating \(\omega c/(1-c)\) and simplifying gives
\cref{eq:Lglobalderiv}; the numerator is unchanged if \(x_\omega\) is
replaced by \(1-x_\omega\). It vanishes exactly when \(S\) is almost
surely constant. Dominated convergence gives continuity for
\(\omega>0\), and expansion at zero gives
\(L_\mu(\omega)\to\sigma^2/\E[S]\). The last assertion follows from
\cref{lem:mmsecalculus}.
\end{proof}

\section{Attainment of the lower bound and the strict gap}
\label{sec: proof-strict-equality}

This section proves \cref{prop: strict-equality,cor:strict-gap-criterion}.
We first construct the scalar penalty in the log-concave case and
establish a uniform scalar risk gap otherwise. To prove attainment,
we analyze a quadratically perturbed estimator and then remove the
perturbation, taking \(p\to\infty\) before taking the strength of the perturbation to zero.
The final argument translates log-concavity at \(\tau_B^+\) into a
criterion involving the prior and spectrum.

\subsection{Constructing a convex penalty from the posterior mean}
\label{sec:strict-gap-technical-lemmas}

For \(\tau>0\), write $Y_\tau=B^\star+\sqrt\tau Z,$ and $f_\tau(y)=\E[B^\star\mid Y_\tau=y],$ and let \(p_\tau\) denote the everywhere-positive density of \(Y_\tau\).  Define the Gaussian-tilted log-partition function and its
convex conjugate by
\begin{equation}
\begin{aligned}
    \Lambda_\tau(s)
    &:=%
    \log\int_{\R}
    \exp\left\{
        sb-\frac{b^2}{2\tau}
    \right\}\pi(\rd b),                                             \quad \text{ and } \quad 
    \Lambda_\tau^\star(x)
    :=%
    \sup_{s\in\R}\{sx-\Lambda_\tau(s)\}.
\end{aligned}
\label{eq:Gaussian-tilted-log-partition}
\end{equation}
For \(\gamma>0\), also define the extended-real function
\begin{equation}
    h_{\tau,\gamma}(x)
    :=
    \gamma\left\{
        \tau\Lambda_\tau^\star(x)-\frac{x^2}{2}
    \right\}.
    \label{eq:Bayes-matched-penalty}
\end{equation}
The definitions in \cref{eq:main-Bayes-log-partition,eq:main-explicit-Bayes-penalty} are recovered by taking
\(\tau=\tau_B^+\) and \(\gamma=\gamma_B=\sigma^2/\tau_B^+\).
In particular, \(\Lambda_{\tau_B^+}=\Lambda_B\) and
\(h_{\tau_B^+,\gamma_B}=h_B^\star\).
\cref{lem:posterior-prox-characterization} verifies when the candidate
in \cref{eq:Bayes-matched-penalty} is convex and has proximal map
\(f_\tau\), its additive normalization does not affect the estimator.

\begin{lemma}
\label{lem:posterior-prox-characterization}
Fix \(\tau>0\). The posterior mean \(f_\tau\) is the proximal map
of a proper, lsc, convex function if and only if
\(p_\tau\) is log-concave.

Suppose that \(p_\tau\) is log-concave, and fix \(\gamma>0\).
Then \(h_{\tau,\gamma}\), defined in
\cref{eq:Bayes-matched-penalty}, is proper, lsc,
and convex, with
$\prox_{h_{\tau,\gamma}/\gamma}(y)=f_\tau(y)$ for all $ y\in\R.$
It is the unique proper, lsc, convex function
with this proximal map, up to an additive constant.

If, in addition, \(V_\pi>0\), then \(f_\tau'(y)>0\) for every
\(y\in\R\). Its range \(I_\tau=f_\tau(\R)\) is an open interval,
and \(h_{\tau,\gamma}\) is twice continuously differentiable on
\(I_\tau\), with
\begin{align}
    h_{\tau,\gamma}'(x)
    &=
    \gamma\bigl(f_\tau^{-1}(x)-x\bigr),
    \qquad x\in I_\tau,
    \label{eq:proximal-inversion}\\
    h_{\tau,\gamma}''(f_\tau(y))
    &=
    \gamma\left(\frac{1}{f_\tau'(y)}-1\right),
    \qquad y\in\R.
    \label{eq:explicit-penalty-curvature}
\end{align}
\end{lemma}

\begin{proof}
We first relate log-concavity of the observation density to the
derivative of the posterior mean. We then verify that the explicit
penalty has the required proximal map and regularity.

\noindent\emph{Step 1: relate log-concavity to the posterior mean.}
Completing the square gives
\begin{equation}
    p_\tau(y)
    =
    \frac{1}{\sqrt{2\pi\tau}}
    \exp\left\{
        -\frac{y^2}{2\tau}
        +\Lambda_\tau\left(\frac{y}{\tau}\right)
    \right\}.
    \label{eq:smoothed-density-log-partition}
\end{equation}
The Gaussian factor in the definition of \(\Lambda_\tau\)
ensures that it is finite and smooth on \(\R\): on every compact
\(s\)-interval, each derivative of
\(\exp\{sb-b^2/(2\tau)\}\) is bounded uniformly in \(b\).
Differentiation under the integral is therefore valid.

The tilted law at \(s=y/\tau\) is the posterior law of
\(B^\star\) given \(Y_\tau=y\). Consequently,
\begin{equation}
    \Lambda_\tau'\left(\frac{y}{\tau}\right)=f_\tau(y),
    \qquad
    \Lambda_\tau''\left(\frac{y}{\tau}\right)
    =
    \VAR(B^\star\mid Y_\tau=y).
    \label{eq:posterior-mean-log-partition}
\end{equation}
Differentiating \cref{eq:smoothed-density-log-partition} gives
\begin{equation}
    f_\tau'(y)
    =
    \frac{1}{\tau}
    \Lambda_\tau''\left(\frac{y}{\tau}\right)\geq0,
    \qquad
    (\log p_\tau)''(y)
    =
    \frac{f_\tau'(y)-1}{\tau}.
    \label{eq:log-concavity-prox-derivative}
\end{equation}
Thus \(p_\tau\) is log-concave exactly when
\(0\leq f_\tau'\leq1\).

In particular, if \(f_\tau\) is the proximal map of a proper,
lsc, convex function, then it is nondecreasing
and one-Lipschitz. Hence \(0\leq f_\tau'\leq1\), proving the
necessity of log-concavity.

\medskip
\noindent\emph{Step 2: verify convexity and the proximal identity.}
Suppose that \(p_\tau\) is log-concave.
By \cref{eq:log-concavity-prox-derivative},
we have $0\leq\Lambda_\tau''(s)\leq\tau$ for all $s\in\R.$
Thus \(\Lambda_\tau'\) is \(\tau\)-Lipschitz.
Smoothness--strong-convexity duality implies that
\(\Lambda_\tau^\star\) is \(1/\tau\)-strongly convex.
Equivalently, $ \psi_\tau(x)
    :=
    \tau\Lambda_\tau^\star(x)-{x^2}/{2}$
is convex. It is also proper and lsc, since
\(\Lambda_\tau^\star\) has these properties and the subtracted
quadratic is finite and continuous. These conclusions include
the degenerate case: if \(\pi=\delta_b\), then
\(\psi_\tau\) is zero at \(b\) and \(+\infty\) elsewhere.

Since \(h_{\tau,\gamma}=\gamma\psi_\tau\), the explicit penalty
is proper, lsc, and convex. Moreover, the
subdifferential sum rule gives
$\partial\psi_\tau(x)
    =
    \tau\partial\Lambda_\tau^\star(x)-x.$
The proximal optimality condition and Fenchel duality now yield
\begin{align*}
    x=\prox_{\psi_\tau}(y)
    &\quad\Longleftrightarrow\quad
    y-x\in\partial\psi_\tau(x) \\
    &\quad\Longleftrightarrow\quad
    y/\tau\in\partial\Lambda_\tau^\star(x) \\
    &\quad\Longleftrightarrow\quad
    x=\Lambda_\tau'(y/\tau)=f_\tau(y).
\end{align*}
This proves that $\prox_{h_{\tau,\gamma}/\gamma}(y)=f_\tau(y)$ for all $ y\in\R$ and the sufficiency
of log-concavity.

The proximal map also determines the subdifferential, because
$u\in\partial(h/\gamma)(x)
    \Longleftrightarrow
    x=\prox_{h/\gamma}(x+u).$
Therefore any other proper, lsc, convex
penalty with the same proximal map has the same subdifferential
as \(h_{\tau,\gamma}\). Such functions differ only by an
additive constant, proving uniqueness.

\medskip
\noindent\emph{Step 3: obtain the derivative and curvature formulas.}
Assume \(V_\pi>0\). The Gaussian likelihood is everywhere
positive, so the posterior law is equivalent to the
nondegenerate prior. Its variance is therefore strictly positive.
By \cref{eq:posterior-mean-log-partition,eq:log-concavity-prox-derivative}, we have
\(f_\tau'(y)>0\) for every \(y\).
Consequently, \(f_\tau\) is strictly increasing, its range
\(I_\tau\) is an open interval, and its inverse is continuously
differentiable there.

For \(x\in I_\tau\), Fenchel duality gives
$ (\Lambda_\tau^\star)'(x)
    =
    \frac{1}{\tau}f_\tau^{-1}(x).$
Differentiating \cref{eq:Bayes-matched-penalty} therefore proves
\cref{eq:proximal-inversion}. Differentiating once more gives
\[
    h_{\tau,\gamma}''(x)
    =
    \gamma\left(
        \frac{1}{f_\tau'(f_\tau^{-1}(x))}-1
    \right),
    \qquad x\in I_\tau.
\]
This expression is continuous and proves
\cref{eq:explicit-penalty-curvature}.
\end{proof}

\subsection{Obtaining a uniform scalar risk gap}
\label{sec:appendix-E-scalar-gap}

For a strict risk gap, we must rule out arbitrarily close
approximation of the posterior mean by proximal maps.
\cref{lem:local-scalar-prox-gap} does so uniformly over nearby
variances, as needed when the effective noise changes under perturbation.

\begin{lemma}
\label{lem:local-scalar-prox-gap}
Fix \(\tau_0>0\).  If \(\pi*N(0,\tau_0)\) is not log-concave, then
there exist \(\epsilon\in(0,\tau_0)\) and
\(\Delta_{\rm sc}>0\) such that, for every
\(\tau\in[\tau_0-\epsilon,\tau_0+\epsilon]\) and every scalar
proximal map \(g\), we have 
$\E[(g(Y_\tau)-B^\star)^2]
    \geq
    m_\pi(\tau)+\Delta_{\rm sc}.$
\end{lemma}

\begin{proof}
The assumption forces \(V_\pi>0\), because convolution of a point mass
with a Gaussian is log-concave.  By
\cref{lem:posterior-prox-characterization}, we know \(f_{\tau_0}\) is not a
proximal map.  Equivalently, \(p_{\tau_0}\) is not log-concave, so
\cref{eq:log-concavity-prox-derivative} implies that there is a point
\(y_0\in\R\) such that \(f_{\tau_0}'(y_0)>1\).

 We first prove that the map \((\tau,y)\mapsto f_\tau'(y)\) is jointly continuous on \((0,\infty)\times\R\), which is needed to make this violation uniform on the interval.
Define the moments
$M_j(\tau,y)
    =
    \int_{\R}x^j
    \exp\{-\frac{1}{2\tau}(y-x)^2\}\pi(\rd x)$ for \(j=0,1,2\).
On every compact subset of \((0,\infty)\times\R\), the integrands have
a common bounded envelope.  Indeed, after restricting \(\tau\) away
from zero and \(y\) to a bounded interval, a polynomial of degree at
most two multiplied by the Gaussian factor is bounded uniformly in
\(x\).  Dominated convergence therefore makes each \(M_j\) jointly
continuous.  Since \(M_0>0\), the posterior first and second moments,
and hence
$f_\tau'(y)
    =\frac{1}{\tau}\VAR(B^\star\mid Y_\tau=y),$
are jointly continuous in \((\tau,y)\).

Choose \(\xi>0\), \(a>0\), and \(\epsilon\in(0,\tau_0)\) so that $f_\tau'(y)\geq1+2\xi$ whenever $|\tau-\tau_0|\leq\epsilon$ and $|y-y_0|\leq2a$.
Let \(g\) be any scalar proximal map.  Since
\(g\) is one-Lipschitz, there are two cases. 
If \(g(y_0)\leq f_\tau(y_0)\), then we have $f_\tau(y)-g(y)
 \geq(1+2\xi)(y-y_0)-(y-y_0)
 \geq2\xi a$ for every \(y\in[y_0+a,y_0+2a]\).
If \(g(y_0)>f_\tau(y_0)\), the same argument  to the left of \(y_0\)
gives
\(g(y)-f_\tau(y)\geq2\xi a\) for \([y_0-2a,y_0-a]\).

The map \((\tau,y)\mapsto p_\tau(y)\) is continuous and strictly
positive.  As a consequence, we have that
$c
    :=
    \min_{\substack{|\tau-\tau_0|\leq\epsilon\\
                     |y-y_0|\leq2a}}
    p_\tau(y)
    >0.$
In either of the two cases above,
\begin{equation}
    \E[(g(Y_\tau)-f_\tau(Y_\tau))^2]
    \geq4c\xi^2a^3
    =:\Delta_{\rm sc}>0.
    \label{eq:L2-prox-gap}
\end{equation}
Finally, because \(f_\tau(Y_\tau)=\E[B^\star\mid Y_\tau]\), the
conditional-expectation projection identity gives $\E[(g(Y_\tau)-B^\star)^2]
    =
    m_\pi(\tau)
    +\E[(g(Y_\tau)-f_\tau(Y_\tau))^2].$
Combining this identity with \cref{eq:L2-prox-gap} proves the result.
\end{proof}

\subsection{Finding fixed points after adding a quadratic penalty}
\label{sec:appendix-E-perturbed-fixed-points}

To apply \cref{thm:perturbed-convex-mse}, we add
\(\epsilon x^2/2\) to \(h_B^\star\).
\cref{lem:bayes-l2-continuation} selects fixed points approaching the
Bayes solution, so their risks approach the required bound.

\begin{lemma}
\label{lem:bayes-l2-continuation}
Assume \(V_\pi>0\) and suppose that
\(p_{\tau_B^+}=\pi*N(0,\tau_B^+)\) is log-concave.  Define
\[
    \gamma_B:={\sigma^2}/{\tau_B^+},
    \qquad
    \eta_B:={\sigma^2}/{m_\pi(\tau_B^+)},
    \qquad
    h_B^\star:=h_{\tau_B^+,\gamma_B}.
\]
Equivalently,
$h_B^\star(x)
    =
    \sigma^2\Lambda_{\tau_B^+}^\star(x)
    -\frac{\sigma^2}{2\tau_B^+}x^2.$
Then $(
        \frac{\gamma_B}{\sigma^2},
        \frac{\eta_B}{\sigma^2},
        \tau_B^+,
        \cE_\pi(\tau_B^+))$
is the corresponding matched solution of the Bayes-VAMP fixed-point
equations in \cref{eq: fpt-bvamp-a,eq: fpt-bvamp-b,eq: fpt-bvamp-c,eq: fpt-bvamp-d}.  The
parameters \(\gamma_B\) and \(\eta_B\) use the precision normalization
of the canonical convex-VAMP equations in \cref{eq: fpt-equation-a,eq: fpt-equation-b,eq: fpt-equation-c,eq: fpt-equation-d}.

The function \(h_B^\star\) is proper, lsc, and convex, its effective
domain contains at least two points, and
$\prox_{h_B^\star/\gamma_B}=f_{\tau_B^+}.$
There exists \(\epsilon_0>0\) such that for every
\(\epsilon\in(0,\epsilon_0)\) the canonical convex-VAMP equations for
$h_{B,\epsilon}(x)
    :=
    h_B^\star(x)+\frac{\epsilon}{2}x^2$
admit a solution
\((\gamma_\epsilon,\eta_\epsilon,
  \tau_{1,\epsilon},\tau_{2,\epsilon})\).  Moreover,
\begin{equation}
    (\gamma_\epsilon,\eta_\epsilon,
      \tau_{1,\epsilon},\tau_{2,\epsilon})
    \longrightarrow
    (\gamma_B,\eta_B,\tau_B^+,\cE_\pi(\tau_B^+))
    \qquad\text{as }\epsilon\downarrow0,
    \label{eq:Bayes-ridge-parameters}
\end{equation}
and
\begin{equation}
\begin{aligned}
    &\E\left[
        \left(
            \prox_{h_{B,\epsilon}/\gamma_\epsilon}
            (B^\star+\sqrt{\tau_{1,\epsilon}}Z)
            -B^\star
        \right)^2
    \right] \\
    &\quad\longrightarrow
    \E\left[
        \left(
            \E[B^\star\mid B^\star+\sqrt{\tau_B^+}Z]
            -B^\star
        \right)^2
    \right]
    =m_\pi(\tau_B^+).
    \label{eq:Bayes-ridge-risk}
\end{aligned}
\end{equation}
\end{lemma}

\begin{proof}
We first account for the factor \(\sigma^2\) separating the
Bayes-VAMP precisions from the canonical convex-VAMP precisions. Abbreviate $ m_B=m_\pi(\tau_B^+)$ and $d_B={m_B}/{\tau_B^+}$, which we recall from \cref{eq:postderiv} satisfies $d_B \in (0, 1)$, as well as $ \omega_B=\cE_\pi(\tau_B^+)={m_B}/{(1-d_B)}$. Define $c_B={\sigma^2}/{\omega_B},$ and $
    \rho_B=G_\mu(c_B).$
The fixed-point identity \(\tau_B^+=L_\mu(\omega_B)\), together with
\cref{eq:stieltjes-L-preview}, gives
\[
    d_B=1-c_B\rho_B,
    \qquad
    \tau_B^+
    =\frac{\sigma^2\rho_B}{1-c_B\rho_B},
    \qquad
    m_B=\tau_B^+d_B=\sigma^2\rho_B.
\]
Indeed, substitute \(\sigma^2=c_B\omega_B\) into
\cref{eq:stieltjes-L-preview} and use
\(\omega_B=\tau_B^+d_B/(1-d_B)\).  The last identity implies
$\rho_B^{-1}
    =\frac{1}{m_B} \sigma^2
    =\eta_B.$
Moreover, $ \eta_B-c_B
    =\sigma^2(\frac1{m_B}-\frac1{\omega_B})
    ={\sigma^2}/{\tau_B^+}
    =\gamma_B.$
Set \(\theta_B=\gamma_B^{-1}\).  We have therefore proved
\begin{equation}
    d_B
    =1-c_B\rho_B
    =\frac{\gamma_B}{\eta_B}
    =\frac{\rho_B}{\theta_B}
    <1.
    \label{eq:Bayes-precision-identities}
\end{equation}
Thus \((\gamma_B/\sigma^2,\eta_B/\sigma^2,
\tau_B^+,\omega_B)\) has the Bayes-VAMP precision normalization,
whereas \((\gamma_B,\eta_B,\tau_B^+,\omega_B)\) has the canonical
convex-VAMP normalization. Also,
$R(\rho_B)=c_B-\rho_B^{-1}=-\gamma_B,$ and
    $\theta(\rho_B)=\theta_B.$

Set
\[
    h^\star:=\frac1{\gamma_B}h_B^\star
    =\tau_B^+\Lambda_{\tau_B^+}^\star-\frac12|\cdot|^2,
    \qquad
    h_{B,\epsilon}(x)
    =h_B^\star(x)+\frac{\epsilon}{2}x^2.
\]
By \cref{lem:posterior-prox-characterization}, we have
\(
    \prox_{h^\star}=f_{\tau_B^+}.
\)
\cref{lem:posterior-prox-characterization} and nondegeneracy of \(\pi\) imply that
\(f_{\tau_B^+}\) is strictly increasing and that its nontrivial range
is contained in \(\dom(h^\star)\).  Hence
\(\dom(h_B^\star)\) contains at least two points.

\medskip
\noindent\emph{Step 1: reduce the fixed-point problem to two equations.}
For \(\epsilon\geq0\), \(\rho\in(0,\rho_{\max})\), and \(\tau>0\),
define
\begin{equation}
\begin{aligned}
    Q_{\epsilon,\rho}(y)
    &:=%
    \prox_{\theta(\rho)h_{B,\epsilon}}(y),                           \\
    A_\epsilon(\rho,\tau)
    &:=%
    \E\left[
        Q_{\epsilon,\rho}'(B^\star+\sqrt\tau Z)
    \right],                                                        \\
    r_\epsilon(\rho,\tau)
    &:=%
    \E\left[
        \left(
            Q_{\epsilon,\rho}(B^\star+\sqrt\tau Z)-B^\star
        \right)^2
    \right].
\end{aligned}
\label{eq:ridge-Q-A-r}
\end{equation}
Let
\begin{align}
    F_{1,\epsilon}(\rho,\tau)
    &:=%
    \rho-\theta(\rho)A_\epsilon(\rho,\tau),
    \label{eq:ridge-F1}                                             \\
    F_{2,\epsilon}(\rho,\tau)
    &:=%
    \frac{\theta(\rho)^2}{\tau}
    \left(
        R'(\rho)r_\epsilon(\rho,\tau)+\sigma^2K(\rho)
    \right)-1.
    \label{eq:ridge-F2}
\end{align}

These are the two residuals obtained from the reference-point
reduction in \cref{sec:oracle-fp-existence}.  To verify this directly,
specialize that reduction to
$h=h_B^\star,$ and
    $\lambda=\epsilon,$ with
    $B=0,$
and set $\xi={\theta(\rho)}/{\sqrt\tau},$ and $w_\epsilon(\rho)
    ={\theta(\rho)}/{(1+\epsilon\theta(\rho))}.$
The scalar input in \cref{eq:U-rho-xi} then becomes
\((B^\star+\sqrt\tau Z)/(1+\epsilon\theta(\rho))\), and completing
the square gives
\begin{equation}
    Q_{\epsilon,\rho}(y)
    =
    \prox_{w_\epsilon(\rho)h_B^\star}
    \left(
        \frac{y}{1+\epsilon\theta(\rho)}
    \right).
    \label{eq:ridge-prox-identity}
\end{equation}
Together with \(\xi^2=\theta(\rho)^2/\tau\), this transforms
\cref{eq:F1,eq:F2} into
\cref{eq:ridge-F1,eq:ridge-F2}.  The same algebra remains
valid at \(\epsilon=0\).

For \(\epsilon>0\), every common zero \((\rho,\tau)\) of
\((F_{1,\epsilon},F_{2,\epsilon})\) reconstructs a canonical
convex-VAMP fixed point for \(h_{B,\epsilon}\) through
\begin{equation}
\begin{aligned}
    \gamma
    &=\theta(\rho)^{-1},
    &\qquad
    \eta
    &=\rho^{-1},
    &\qquad
    \tau_1
    &=\tau,   &\qquad \tau_2
    &=
    \frac{
        r_\epsilon(\rho,\tau)
        -\left(\rho/\theta(\rho)\right)^2\tau
    }{
        \left(1-\rho/\theta(\rho)\right)^2
    }.
\end{aligned}
\label{eq:ridge-reconstruction}
\end{equation}
This is the reconstruction in \cref{sec:reference-point-reduction},
expressed in the variables \((\rho,\tau)\).  It remains to construct
common zeros converging to \((\rho_B,\tau_B^+)\).

\medskip
\noindent\emph{Step 2: verify the unperturbed equations and an
invertible Jacobian.}
At \(\epsilon=0\),
\begin{equation}
    Q_{0,\rho_B}=f_{\tau_B^+},
    \qquad
    A_0(\rho_B,\tau_B^+)=d_B,
    \qquad
    r_0(\rho_B,\tau_B^+)=m_B.
    \label{eq:strict-gap-change-variable}
\end{equation}
Indeed, $Q_{0,\rho_B}
    =\prox_{\theta_Bh_B^\star}
    =\prox_{h^\star}
    =f_{\tau_B^+}.$
\cref{eq:Bayes-precision-identities,eq:strict-gap-change-variable} immediately give $F_{1,0}(\rho_B,\tau_B^+)=0.$

We now verify the second equation.  At this zero of \(F_{1,0}\), the
variables used in the \cref{lem:zero-set-sign-match} proof are
\[
    a=d_B,
    \qquad
    c=\frac1{\rho_B}-\frac1{\theta_B}=c_B,
    \qquad
    \omega
    =\frac{m_B-d_B^2\tau_B^+}{(1-d_B)^2}
    =\omega_B.
\]
The function \(h_B^\star\) is not affine: otherwise
\(\prox_{h^\star}\) would be a translation of the identity and would
have derivative one everywhere, contrary to
\(\E[f_{\tau_B^+}'(Y_{\tau_B^+})]=d_B<1\).  Since \(c_B>0\), the
algebra in the proof of \cref{lem:zero-set-sign-match} remains valid
at \(\lambda=0\) and yields
$ F_{2,0}(\rho_B,\tau_B^+)
    =
    \frac{\rho_B^2}{J(\rho_B)\tau_B^+}
    (
        L_{0,c_B}(\omega_B)-\tau_B^+
    ).$
By \cref{lem:linearcomparison}, the choice
\(c_B=\sigma^2/\omega_B\) gives
$ L_{0,c_B}(\omega_B)
    =L_\mu(\omega_B)
    =\tau_B^+.$
Thus \(F_{2,0}(\rho_B,\tau_B^+)=0\).

We now show that the first residual determines \(\rho\) locally
as a function of \(\tau\), and the second has a nonzero derivative
along that curve. Write
\[
    Y_\tau=B^\star+\sqrt\tau Z,
    \qquad
    q_s=\prox_{s h^\star},
    \qquad
    a(s,\tau)=\E[q_s'(Y_\tau)].
\]
Let \(I=f_{\tau_B^+}(\R)\).  By
\cref{eq:proximal-inversion,eq:explicit-penalty-curvature}, applied
with \(\gamma=1\), the function \(h^\star\) is \(C^2\) on \(I\) and
satisfies, for \(x\in I\),
\[
    (h^\star)'(x)
    =f_{\tau_B^+}^{-1}(x)-x,
    \qquad
    (h^\star)''(x)
    =
    \frac1{f_{\tau_B^+}'(f_{\tau_B^+}^{-1}(x))}-1
    \geq0.
\]

For \(s>0\) and \(x=f_{\tau_B^+}(y)\),
we have $ x+s(h^\star)'(x)
    =sy+(1-s)f_{\tau_B^+}(y).$
This expression is strictly increasing in \(y\) and tends to
\(\pm\infty\) as \(y\to\pm\infty\).  To see the monotonicity, use
\(0\leq f_{\tau_B^+}'\leq1\); for the endpoint behavior, use
monotonicity of \(f_{\tau_B^+}\) when \(s\leq1\) and monotonicity of
\(y-f_{\tau_B^+}(y)\) when \(s\geq1\).  Hence
\(x\mapsto x+s(h^\star)'(x)\) maps \(I\) increasingly onto \(\R\),
so \(q_s(y)\in I\) for every \(y\).  Implicit differentiation gives
\begin{equation}
    q_s'(y)
    =
    \frac1{1+s(h^\star)''(q_s(y))}
    \in(0,1],
    \qquad
    \partial_sq_s(y)
    =
    -q_s'(y)\frac{y-q_s(y)}s.
    \label{eq:prox-scale-derivative-general}
\end{equation}
On compact positive \(s\)-intervals, firm nonexpansiveness and
\cref{eq:prox-scale-derivative-general} imply $|q_s(y)|+|\partial_sq_s(y)|
    \leq C(1+|y|).$
Consequently, all differentiations below are justified by dominated
convergence when \(\tau\) ranges in a compact subset of
\((0,\infty)\).  If $r(s,\tau)=\E[(q_s(Y_\tau)-B^\star)^2],$
then
\begin{align*}
    \partial_sa(s,\tau)
    &=
    \frac1{\sqrt\tau}\E[Z\,\partial_sq_s(Y_\tau)],                 \\
    \partial_\tau a(s,\tau)
    &=
    -\frac1{2\tau^{3/2}}\E[Zq_s(Y_\tau)]
    +\frac1{2\tau}\E[Z^2q_s'(Y_\tau)],                            \\
    \partial_sr(s,\tau)
    &=
    2\E[(q_s(Y_\tau)-B^\star)\partial_sq_s(Y_\tau)],             \\
    \partial_\tau r(s,\tau)
    &=
    \frac1{\sqrt\tau}
    \E[(q_s(Y_\tau)-B^\star)q_s'(Y_\tau)Z].
\end{align*}
These derivatives are jointly continuous, so \(a\) and \(r\) are
\(C^1\) near \((1,\tau_B^+)\).

At \(s=1\), we find 
$ \dot q_1(y)
    :=
    \partial_sq_s(y)|_{s=1}
    =
    -f_{\tau_B^+}'(y)
    (y-f_{\tau_B^+}(y))$ via implicit differentiation of
\(y=q_s(y)+s(h^\star)'(q_s(y))\).
Using the Stein representation of \(a\) and
$\E[Z\mid Y_{\tau_B^+}]
    =
  (Y_{\tau_B^+}-f_{\tau_B^+}(Y_{\tau_B^+}))/
         {\sqrt{\tau_B^+}},$
we obtain
\begin{equation}
\begin{aligned}
    \partial_sa(1,\tau_B^+)
    &=
    \frac1{\sqrt{\tau_B^+}}
    \E\left[Z\dot q_1(Y_{\tau_B^+})\right]              =-\frac1{\tau_B^+}
    \E\left[
        f_{\tau_B^+}'(Y_{\tau_B^+})
        \left(
            Y_{\tau_B^+}-f_{\tau_B^+}(Y_{\tau_B^+})
        \right)^2
    \right]
    \leq0.
\end{aligned}
\label{eq:ridge-a-scale-sign}
\end{equation}
Since $A_0(\rho,\tau)
    =a(\gamma_B\theta(\rho),\tau),$ along with
    $\gamma_B\theta_B=1,$
and \cref{eq:Rfacts} gives
\[
0\leq
    {\rho\theta'(\rho)}/{\theta(\rho)}
    =-\rho R'(\rho)/R(\rho)
    <1,
\]
differentiating \cref{eq:ridge-F1} yields
\begin{equation}
\begin{aligned}
    \partial_\rho F_{1,0}(\rho_B,\tau_B^+)
    &=
    1-\theta'(\rho_B)d_B
      -\theta_B\partial_sa(1,\tau_B^+)\gamma_B\theta'(\rho_B)\\
    &=
    1-\frac{\rho_B\theta'(\rho_B)}{\theta_B}
      -\theta'(\rho_B)\partial_sa(1,\tau_B^+)
    >0.
\end{aligned}
\label{eq:F1-rho-positive}
\end{equation}
The same conclusion holds for a point-mass spectrum, for which
\(\theta'=0\).

The functions \(R'\), \(K\), and \(\theta\) are smooth at the
interior point \(\rho_B\).  Hence the implicit-function theorem gives
a \(C^1\) curve \(\tau\mapsto\rho_0(\tau)\), defined near
\(\tau_B^+\), such that
$\rho_0(\tau_B^+)=\rho_B,$ and $
    F_{1,0}(\rho_0(\tau),\tau)=0.$
Along this curve, define
\[
    g_\tau=Q_{0,\rho_0(\tau)},
    \qquad
    a_\tau=A_0(\rho_0(\tau),\tau)
    =\frac{\rho_0(\tau)}{\theta(\rho_0(\tau))},
\]
\[
    \omega(\tau)
    =
    \frac{
        r_0(\rho_0(\tau),\tau)-a_\tau^2\tau
    }{(1-a_\tau)^2},
    \qquad
    c(\tau)
    =
    \frac1{\rho_0(\tau)}
    -\frac1{\theta(\rho_0(\tau))}.
\]
After shrinking the neighborhood, we have
\(0<a_\tau<1\) and \(c(\tau)>0\).  The scalar comparison in
\cref{lem:scalarcomparison} gives
$\omega(\tau)\geq\cE_\pi(\tau).$
At \(\tau=\tau_B^+\), equality holds because
\(g_{\tau_B^+}=f_{\tau_B^+}\).  Both sides are differentiable, so the
nonnegative difference has derivative zero at \(\tau_B^+\), i.e.\ $\omega'(\tau_B^+)=\cE_\pi'(\tau_B^+).$

The linear comparison is also an equality at the Bayes solution,
so its nonnegative difference has zero gradient there.
Define
$D(c,\omega)
    =L_{0,c}(\omega)-L_\mu(\omega),$ for all $c,\omega>0.$
By \cref{lem:linearcomparison}, \(D\geq0\), and
\(D(c_B,\omega_B)=0\).  Since this is an interior minimum of a
continuously differentiable function,
$\partial_cD(c_B,\omega_B)
    =
    \partial_\omega D(c_B,\omega_B)
    =0.$
Combining this with $\omega'(\tau_B^+)=\cE_\pi'(\tau_B^+)$ gives
\[
    \left.
    \frac{\rd}{\rd\tau}
    L_{0,c(\tau)}(\omega(\tau))
    \right|_{\tau=\tau_B^+}
    =
    L_\mu'(\omega_B)\cE_\pi'(\tau_B^+)
    =T_{\mu,\pi}'(\tau_B^+).
\]
\cref{eq:proximal-fixed-point-transversality} gives
\(T_{\mu,\pi}'(\tau_B^+)<1\).  On the zero set of \(F_{1,0}\), the
same algebra as in \cref{eq:sign-match} gives
$F_{2,0}(\rho,\tau)
    =
    \frac{\rho^2}{J(\rho)\tau}
    \left(
        L_{0,c}(\omega)-\tau
    \right).$
Since the bracket vanishes at the Bayes point,
\begin{equation}
    \left.
    \frac{\rd}{\rd\tau}
    F_{2,0}(\rho_0(\tau),\tau)
    \right|_{\tau=\tau_B^+}
    =
    \frac{\rho_B^2}{J(\rho_B)\tau_B^+}
    \left(T_{\mu,\pi}'(\tau_B^+)-1\right)
    <0.
    \label{eq:F2-transverse}
\end{equation}
Because
\(\rho_0'=-F_{1,\tau}/F_{1,\rho}\),
\cref{eq:F1-rho-positive,eq:F2-transverse} imply
\begin{equation}
    \det
    D_{(\rho,\tau)}(F_{1,0},F_{2,0})(\rho_B,\tau_B^+)
    <0.
    \label{eq:Bayes-reduced-Jacobian}
\end{equation}
Thus \((\rho_B,\tau_B^+)\) is an isolated zero with nonzero local
Brouwer degree.

\medskip
\noindent\emph{Step 3: obtain nearby solutions for small perturbations.}
The invertible Jacobian gives a nonzero local degree. Uniform
convergence of the residuals preserves that degree without requiring
differentiability in \(\epsilon\).
Let \(K\) be a compact neighborhood of
\((\rho_B,\tau_B^+)\) on which \(\theta\) is bounded above and away
from zero.  By \cref{eq:ridge-prox-identity}, the scale and input of
the proximal map vary continuously with
\((\epsilon,\rho,\tau)\).  The dominated-convergence argument in the
proof of \cref{lem:F-continuity}, now allowing \(\epsilon\) to vary in
a compact interval, shows that \(r_\epsilon\) is jointly continuous on
\([0,\epsilon_1]\times K\), with a common integrable quadratic
envelope.  For the derivative term, use Stein's identity:
$A_\epsilon(\rho,\tau)
    = 
    \E[
        ZQ_{\epsilon,\rho}(B^\star+\sqrt\tau Z)
    ]/{\sqrt\tau}.$
The same argument gives joint continuity of \(A_\epsilon\).  Uniform
continuity on the compact set therefore yields
\begin{equation}
    \sup_{(\rho,\tau)\in K}
    \left\|
        (F_{1,\epsilon},F_{2,\epsilon})(\rho,\tau)
        -(F_{1,0},F_{2,0})(\rho,\tau)
    \right\|
    \longrightarrow0.
    \label{eq:ridge-F-uniform}
\end{equation}

Choose a bounded open neighborhood \(U\) with
\(\overline U\subset K\) such that \((\rho_B,\tau_B^+)\) is the only
zero of \(F_0:=(F_{1,0},F_{2,0})\) in \(\overline U\).  Then
$\delta_U
    :=
    \min_{x\in\partial U}\|F_0(x)\|
    >0.$
By \cref{eq:Bayes-reduced-Jacobian}, the degree of \(F_0\) on \(U\)
is nonzero.  For all sufficiently small \(\epsilon>0\),
\cref{eq:ridge-F-uniform} makes
\(\sup_{\overline U}\|F_\epsilon-F_0\|<\delta_U\).  The straight-line
homotopy from \(F_0\) to
\(F_\epsilon:=(F_{1,\epsilon},F_{2,\epsilon})\) has no zero on
\(\partial U\).  Homotopy invariance of Brouwer degree therefore gives
a zero
$(\rho_\epsilon,\tau_\epsilon)\in U$
of \(F_\epsilon\) for every sufficiently small \(\epsilon>0\); see
\textcite[Lemma~6.6 and Property~7.4]{Vandervorst2008}.

These zeros satisfy
\begin{equation}
    (\rho_\epsilon,\tau_\epsilon)
    \longrightarrow
    (\rho_B,\tau_B^+).
    \label{eq:ridge-zero-convergence}
\end{equation}
Indeed, every sequence \(\epsilon_n\downarrow0\) has a subsequence for
which \((\rho_{\epsilon_n},\tau_{\epsilon_n})\) converges in
\(\overline U\).  Uniform convergence in \cref{eq:ridge-F-uniform}
forces every such limit to be a zero of \(F_0\) in \(\overline U\),
and the choice of \(U\) forces that limit to be
\((\rho_B,\tau_B^+)\).

\medskip
\noindent\emph{Step 4: recover the four parameters and their risk.}
For every sufficiently small \(\epsilon>0\), reconstruct
\begin{equation}
\begin{aligned}
    \gamma_\epsilon
    &=\theta(\rho_\epsilon)^{-1},
    &\qquad
    \eta_\epsilon
    &=\rho_\epsilon^{-1},
    &\qquad
    \tau_{1,\epsilon}
    &=\tau_\epsilon,  & \qquad \tau_{2,\epsilon}
    &=
    \frac{
        r_\epsilon(\rho_\epsilon,\tau_\epsilon)
        -\left(
            \rho_\epsilon/\theta(\rho_\epsilon)
        \right)^2\tau_\epsilon
    }{
        \left(
            1-\rho_\epsilon/\theta(\rho_\epsilon)
        \right)^2
    }.
\end{aligned}
\label{eq:ridge-selected-reconstruction}
\end{equation}
By the reduction and reconstruction in
\cref{sec:reference-point-reduction}, this tuple solves the four
canonical convex-VAMP fixed-point equations for \(h_{B,\epsilon}\).
Joint continuity and
\cref{eq:strict-gap-change-variable}, \cref{eq:ridge-zero-convergence}, and \cref{eq:Bayes-precision-identities}
give
${\rho_\epsilon}/{\theta(\rho_\epsilon)}
    \rightarrow d_B,$ and $r_\epsilon(\rho_\epsilon,\tau_\epsilon)
    \rightarrow m_B.$
Therefore,
\[
    \tau_{2,\epsilon}
    \longrightarrow
    \frac{m_B-d_B^2\tau_B^+}{(1-d_B)^2}
    =\frac{m_B}{1-d_B}
    =\omega_B>0.
\]
After decreasing \(\epsilon_0\), if necessary, the reconstructed
solutions have \(\tau_{2,\epsilon}>0\), and
\cref{eq:Bayes-ridge-parameters} follows.

Finally, by definition of \(r_\epsilon\),
we have $\E[
      (
            \prox_{h_{B,\epsilon}/\gamma_\epsilon}
            (B^\star+\sqrt{\tau_{1,\epsilon}}Z)-B^\star
        )^2
    ]
    =r_\epsilon(\rho_\epsilon,\tau_\epsilon)
    \rightarrow m_B.$
Since
\(\prox_{h_B^\star/\gamma_B}=\prox_{h^\star}=f_{\tau_B^+}\), this is
exactly \cref{eq:Bayes-ridge-risk}.
\end{proof}

\subsection{Controlling directions with little penalty curvature}
\label{sec:appendix-E-design-curvature}

To compare perturbed and unperturbed minimizers, we need curvature
of the full objective even when \(X\) has a null space and the
penalty is not globally strongly convex.
\cref{lem:rri-coordinate-curvature} ensures that directions
concentrated on a small coordinate set are controlled by the design.
Its uniformity allows that set to depend on the data.

\begin{lemma}
\label{lem:rri-coordinate-curvature}
Under \cref{ass1}, there exist constants
\(\alpha\in(0,1/2)\) and \(a>0\) such that, with probability tending
to one, $\|Xv\|_2^2+\|v_{J^c}\|_2^2
    \geq a\|v\|_2^2$ holds simultaneously for every
\(v\in\R^p\) and every \(J\subseteq[p]\) satisfying
\(|J|\leq\alpha p\).
Here \(v_J\) denotes the vector obtained from \(v\) by setting all
coordinates outside \(J\) equal to zero.
\end{lemma}

\begin{proof}
We use rotational invariance and finite nets to control every vector
supported on few coordinates, then decompose an arbitrary vector
over \(J\) and \(J^c\).

Because \(\E_\mu[S]>0\), choose \(s_0>0\) and \(r\in(0,1)\) such that
$\mu((s_0,\infty))>r.$
Let \(m_p=\lfloor rp\rfloor\), and let \(\Pi_p\) be the orthogonal
projection onto the span of the right singular vectors corresponding
to the \(m_p\) largest eigenvalues of \(X^TX\).  \cref{ass1}
and the Portmanteau theorem imply that, with probability tending to
one, at least \(m_p\) eigenvalues exceed \(s_0\).  On this event,
$\|Xv\|_2^2 = \sum_i s_{i, p} \langle v, v_i \rangle^2 
    \geq s_0\|\Pi_pv\|_2^2$ for every $v\in\R^p.$
Conditional on the singular values, \(\Pi_p\) is a uniformly
distributed rank-\(m_p\) orthogonal projection.

Fix a deterministic unit vector \(u\).  Conditional on the singular
values, $\|\Pi_pu\|_2^2
    \stackrel{d}{=}
    \frac{A}{A+B},$
where \(A\sim\chi^2_{m_p}\) and
\(B\sim\chi^2_{p-m_p}\) are independent.  For all sufficiently large
\(p\), we have \(m_p\geq rp/2\).  On the event
\(A\geq m_p/2\) and \(A+B\leq2p\), we therefore have 
$\|\Pi_pu\|_2^2
    \geq {m_p}/{(4p)}
    \geq r/8.$
The chi-square moment-generating function gives constants
\(c_r>0\), independent of \(p\) and \(u\), such that, with
\(b_0=\sqrt{r/8}\),
\begin{equation}
    \PP\left(
        \|\Pi_pu\|_2<b_0
        \mid \Sigma
    \right)
    \leq2e^{-c_rp}.
    \label{eq:rri-fixed-vector-tail}
\end{equation}
For example, one may use
\(\PP(A<m_p/2)\leq\exp\{-m_p(\log2-1/2)/2\}\) and
\(\PP(A+B>2p)\leq\exp\{-p(1-\log2)/2\}\).

Set \(\eta=b_0/2\), and write
\(H(x)=-x\log x-(1-x)\log(1-x)\).  Choose
\(\alpha\in(0,1/2)\) so small that
$H(\alpha)+\alpha\log(1+2/\eta)<c_r.$
Let \(k_p=\lfloor\alpha p\rfloor\).  For every coordinate set of size
\(k_p\), take an \(\eta\)-net of its Euclidean unit sphere with at most
\((1+2/\eta)^{k_p}\) points.  Since
\[
    \binom{p}{k_p}(1+2/\eta)^{k_p}
    \leq
    \exp\left
    \{p[H(\alpha)+\alpha\log(1+2/\eta)]\right\},
\]
\cref{eq:rri-fixed-vector-tail} and a union bound show that, with
probability tending to one, \(\|\Pi_pu\|_2\geq b_0\) at every net
point.  Since \(\|\Pi_p\|_{\rm op}\leq1\), approximation by the nets
gives
$\|\Pi_pv\|_2
    \geq\eta\|v\|_2$
for every vector supported on at most \(k_p\) coordinates.  Vectors
with smaller support are covered by enlarging their support to a set of
size \(k_p\).

Both bounds hold simultaneously with probability tending to one.
Work on this event, and fix \(v\in\R^p\) and
\(J\subseteq[p]\) with \(|J|\leq\alpha p\).
Since \(v_J\) is supported on at most \(k_p\) coordinates,
the sparse-vector bound and the triangle inequality give
$\eta\|v_J\|_2
    \leq \|\Pi_pv_J\|_2
    \leq \|\Pi_pv\|_2+\|\Pi_pv_{J^c}\|_2
    \leq \|\Pi_pv\|_2+\|v_{J^c}\|_2.$
Squaring and using
\(\|v\|_2^2=\|v_J\|_2^2+\|v_{J^c}\|_2^2\), we obtain
\[
\begin{aligned}
    \|v\|_2^2
    &\leq
    2\eta^{-2}\|\Pi_pv\|_2^2
    +(1+2\eta^{-2})\|v_{J^c}\|_2^2 \leq
    \frac{2}{s_0\eta^2}\|Xv\|_2^2
    +(1+2\eta^{-2})\|v_{J^c}\|_2^2,
\end{aligned}
\]
where the second inequality uses the bound
\(\|Xv\|_2^2\geq s_0\|\Pi_pv\|_2^2\)
established at the start of the proof.
Taking
$a=
(
        \max\{
            \frac{2}{s_0\eta^2},
            1+2\eta^{-2}
       \}
   )^{-1}$
proves the claim.
\end{proof}

\subsection{Comparing minimizers as the quadratic term vanishes}
\label{sec:appendix-E-optimizer-stability}

The empirical law places most coordinates of the perturbed
minimizer where the penalty has positive curvature.
\cref{lem:rri-coordinate-curvature} controls the remaining directions.
\cref{lem:ridge-optimizer-continuity} combines these facts to compare
the minimizers, first taking \(p\to\infty\) at fixed \(\epsilon\).

\begin{lemma}
\label{lem:ridge-optimizer-continuity}
Suppose \cref{ass1} holds.  Let \(h\) be a proper, lsc,
convex function, and suppose that
$\mathcal L_p(z)
    :=
    \frac12\|y-Xz\|_2^2+\sum_{j=1}^ph(z_j)$
attains its minimum almost surely for every \(p\).  Let
\(\widehat z_{0,p}\) be any measurable minimizer, and for
\(\epsilon>0\) let \(\widehat z_{\epsilon,p}\) be the unique minimizer
of
$\mathcal L_p(z)+\frac\epsilon2\|z\|_2^2.$
Suppose \(h\) is twice continuously differentiable on an open interval
\(I\subseteq\operatorname{dom}(h)\).  Suppose further that there are
deterministic laws \(\nu_\epsilon,\nu_0\in\mathcal P_2(\R)\) such
that, for every sufficiently small fixed \(\epsilon>0\),
\[
    \frac1p\sum_{j=1}^p
    \delta_{\widehat z_{\epsilon,p,j}}
    \xrightarrow{W_2}_{\PP}\nu_\epsilon,
    \qquad
    \nu_\epsilon\xrightarrow{W_2}\nu_0
    \quad\text{as }\epsilon\downarrow0,
\]
and $\nu_0(\{x\in I:h''(x)>0\})=1.$
Then, for every \(\delta>0\),
\[
    \lim_{\epsilon\downarrow0}
    \limsup_{p\to\infty}
    \PP\left(
        \frac1p
        \|\widehat z_{\epsilon,p}-\widehat z_{0,p}\|_2^2
        >\delta
    \right)
    =0.
\]
The conclusion holds uniformly over the choice of the measurable
minimizer \(\widehat z_{0,p}\).
\end{lemma}

\begin{proof}
\emph{Step 1: locate most coordinates with positive curvature.}
Let \(\alpha,a\) be supplied by
\cref{lem:rri-coordinate-curvature}, and set
$U=\{x\in I:h''(x)>0\}.$
Since \(\nu_0(U)=1\), choose a compact set \(K\subset U\) such that
\(\nu_0(K)>1-\alpha/8\).  By continuity of \(h''\), there are
\(d,\kappa>0\) such that the closed \(2d\)-neighborhood of \(K\) is
contained in \(I\) and
$h''(x)\geq\kappa$
throughout that neighborhood.  Let
$O=\{x:\operatorname{dist}(x,K)<d\}.$

Choose a continuous function \(\psi:\R\to[0,1]\) that equals one on
\(K\) and vanishes outside \(O\).  The convergence
\(\nu_\epsilon\to\nu_0\) implies that $\int\psi\,\rd\nu_\epsilon>1-\alpha/4$ for all sufficiently small
\(\epsilon>0\).
For each such fixed \(\epsilon\), empirical convergence then gives
$ |
        \{j:\widehat z_{\epsilon,p,j}\notin O\}
    |
    \leq{\alpha p}/{2}$ with probability tending to one.
Moreover, \(W_2\)-convergence of \(\nu_\epsilon\) to \(\nu_0\)
bounds the second moments of \(\nu_\epsilon\) uniformly for all
sufficiently small \(\epsilon\).  Hence there is a deterministic
\(M<\infty\), independent of such \(\epsilon\), for which
$\PP(
        \|\widehat z_{\epsilon,p}\|_2\leq M\sqrt p
    )
    \rightarrow1$
for every fixed sufficiently small \(\epsilon>0\).

\noindent\emph{Step 2: establish local growth of the unperturbed
objective.}
Fix such an \(\epsilon\), and work on the intersection of these two events, and the high-probability event in
\cref{lem:rri-coordinate-curvature}.  Write
\(z=\widehat z_{\epsilon,p}\).  The first-order condition and the
subdifferential sum rule give a vector \(s\), with
\(s_j\in\partial h(z_j)\), such that
$X^T(Xz-y)+s=-\epsilon z.$
Set $ R=d\sqrt{\alpha/2}.$
For any \(v\in\R^p\) satisfying \(\|v\|_2\leq R\sqrt p\), define
$J
    =
    \{j:z_j\notin O\}
    \cup
    \{j:|v_j|>d\}.$
Then
$|J|
    \leq
    {\alpha p}/{2}
    +{\|v\|_2^2}/{d^2}
    \leq\alpha p.$
If \(j\notin J\), the entire segment from \(z_j\) to \(z_j+v_j\)
lies in the closed \(2d\)-neighborhood of \(K\).  Therefore,
$ h(z_j+v_j)-h(z_j)-s_jv_j
    \geq \kappa v_j^2/2.$
For \(j\in J\), the same left-hand side is nonnegative by convexity.
Using $X^T(Xz-y)+s=-\epsilon z$, expanding the quadratic
loss, and applying \cref{lem:rri-coordinate-curvature}, we obtain
$\mathcal L_p(z+v)-\mathcal L_p(z)
    \geq
    -\epsilon\langle z,v\rangle
    +\frac12\|Xv\|_2^2
    +\frac\kappa2\|v_{J^c}\|_2^2                               \geq
    -\epsilon\langle z,v\rangle
    +c\|v\|_2^2,$
where
$c= a \min\{1,\kappa\}/2>0.$
If \(\mathcal L_p(z+v)=+\infty\),
this inequality is understood as automatic.

\medskip
\noindent\emph{Step 3: locate every unperturbed minimizer and bound
its distance.}
All constants above are independent of sufficiently small
\(\epsilon\).  Reduce the upper bound on \(\epsilon\), if necessary,
so that $\epsilon M<cR.$
On the sphere \(\|v\|_2=R\sqrt p\), we therefore have
$\mathcal L_p(z+v)-\mathcal L_p(z)
    \geq
    p(cR^2-\epsilon MR)>0.$
Every minimizer of \(\mathcal L_p\) must therefore lie strictly inside
this sphere.  Otherwise, convexity along the segment from \(z\) to a
minimizer on or outside the sphere would produce a point on the sphere
with objective value at most \(\mathcal L_p(z)\), a contradiction.

We may now apply the same local inequality on $L_p(z+v)-\mathcal L_p(z)$ to
\(v=\widehat z_{0,p}-z\).  Since
\(\mathcal L_p(\widehat z_{0,p})\leq\mathcal L_p(z)\), we have 
$c\|\widehat z_{0,p}-z\|_2^2
    \leq
    \epsilon
    \langle z,\widehat z_{0,p}-z\rangle
    \leq
    \epsilon\|z\|_2
    \|\widehat z_{0,p}-z\|_2.$
Hence
$ \|\widehat z_{0,p}-\widehat z_{\epsilon,p}\|_2/\sqrt{p}
    \leq {\epsilon M}/{c}.$
For every sufficiently small fixed \(\epsilon>0\), the event on which
this inequality holds has probability tending
to one.  Since \(M/c\) is independent of \(\epsilon\), taking first
\(p\to\infty\) and then \(\epsilon\downarrow0\) proves the claim.  No
step depended on which minimizer of \(\mathcal L_p\) was selected.
\end{proof}


\subsection{\texorpdfstring{Proof of \cref{prop: strict-equality}}{Proof of the attainment and strict-gap theorem}}

\begin{proof}[Proof of \cref{prop: strict-equality}]   
For an admissible penalty \(h\), write
\[
\mathcal R(h)
    :=
    \overset{\PP}{\liminf_{p\to\infty}}
    \frac1p
    \|\widehat\beta_{\cvx}^h-\beta^\star\|_2^2,
\]
and $ R_B^+:=m_\pi(\tau_B^+).$
The lower bound $\mathcal R(h)\geq R_B^+$
for every \(h\in\cC_{\rm LB}(\mu,\pi,\sigma^2)\) is
\cref{thm:Bayes VAMP-lower-bound}.  We prove first that this lower bound
is attained when the Gaussian-smoothed prior is log-concave, and then
that a uniform positive gap remains when it is not.

\textbf{Equality of the infimal risk in the log-concave case.}
Suppose first that \(\pi*N(0,\tau_B^+)\) is log-concave and define
\(h_B^\star\) by \cref{eq:main-explicit-Bayes-penalty}. If
\(V_\pi=0\), write \(B^\star=b_\pi\) almost surely.  In this case,
\(
    \Lambda_{\tau_B^+}(s)
    =sb_\pi-b_\pi^2/(2\tau_B^+),
\)
so \cref{eq:main-explicit-Bayes-penalty} reduces exactly to the
indicator penalty of the singleton \(\{b_\pi\}\).  This penalty is
admissible by \cref{prop:q-bounded-minimizer-set}, its estimator is
\(b_\pi\mathbf 1_p\), and \cref{ass3} gives
$\frac1p
    \|b_\pi\mathbf1_p-\beta^\star\|_2^2
    \rightarrow{\mathrm{a.s.}}0
    =R_B^+.$
Hence assume for the rest of the forward direction that \(V_\pi>0\).

\noindent\emph{Step 1: verify that the fixed penalty is admissible.}
We first verify admissibility of the unperturbed penalty.
The density \(p_{\tau_B^+}\) is positive, real analytic,
log-concave, and tends to zero at infinity. It therefore has
a unique maximizer \(x_B\): two distinct maximizers would,
by log-concavity, force a constant interval, and analyticity
would then force the density to be constant everywhere.

By \cref{lem:posterior-prox-characterization} and
Tweedie's identity,
$x\in\argmin h_B^\star
   \Longleftrightarrow
    f_{\tau_B^+}(x)=x
     \Longleftrightarrow
    p_{\tau_B^+}'(x)=0.$
Log-concavity therefore gives
\(\argmin h_B^\star=\{x_B\}\).
A proper, lsc, convex function on \(\R\)
with a unique minimizer is coercive, so the corresponding
regression objective attains its minimum.
Moreover, \cref{prop:q-bounded-minimizer-set} verifies
the width condition whenever required. Hence
\(h_B^\star\in\cC_{\rm LB}(\mu,\pi,\sigma^2)\).

\noindent\emph{Step 2: verify positive curvature at the limiting scalar
output.}
\cref{lem:ridge-optimizer-continuity} requires positive curvature
almost surely under the limiting output law, rather than a global
uniform curvature bound. The posterior law given any observation is equivalent
to the nondegenerate prior \(\pi\), because the Gaussian likelihood is
strictly positive.  Hence
$f_{\tau_B^+}'(y)
    =
    \VAR(B^\star\mid Y_{\tau_B^+}=y)/{\tau_B^+}
    >0$ for every $y.$
Proximality gives \(f_{\tau_B^+}'\leq1\).  Moreover,
\(f_{\tau_B^+}\) is real analytic, and
\cref{eq:postderiv,eq:bestlinear} give
$\E[f_{\tau_B^+}'(Y_{\tau_B^+})]
    ={m_\pi(\tau_B^+)}/{\tau_B^+}
    <1.$
Thus the analytic function \(1-f_{\tau_B^+}'\) is not identically
zero.  Its zeros are isolated, and because \(Y_{\tau_B^+}\) has a
density,
$\PP(
        f_{\tau_B^+}'(Y_{\tau_B^+})<1
    )=1.$
Let \(I=f_{\tau_B^+}(\R)\), an open interval.  By
\cref{eq:proximal-inversion,eq:explicit-penalty-curvature}, the explicit
penalty \(h_B^\star\) is \(C^2\) on \(I\) and
$(h_B^\star)''(f_{\tau_B^+}(y))
    =
    \gamma_B
    (
        1/{f_{\tau_B^+}'(y)}-1
    ).$
Thus, 
\begin{equation}
    \PP\left(
        f_{\tau_B^+}(Y_{\tau_B^+})\in I,
        \ (h_B^\star)''(f_{\tau_B^+}(Y_{\tau_B^+}))>0
    \right)=1.
    \label{eq:Bayes-output-positive-curvature}
\end{equation}

\noindent\emph{Step 3: determine the risk for each fixed positive
perturbation.}
Let \(\epsilon_0\) and
\((\gamma_\epsilon,\eta_\epsilon,
\tau_{1,\epsilon},\tau_{2,\epsilon})\) be supplied by
\cref{lem:bayes-l2-continuation}.
For \(\epsilon\in(0,\epsilon_0)\), let
$h_{B,\epsilon}(x)
    =h_B^\star(x)+\frac\epsilon2x^2,$ and $
    \widehat\beta_\epsilon
    =\widehat\beta_{\cvx}^{h_{B,\epsilon}}.$
The estimator \(\widehat\beta_\epsilon\) is exactly the perturbed
estimator \(\widehat\beta_{\cvx}^{h_B^\star,\epsilon,0}\).  The
reference vector \(b=0\) satisfies
\cref{ass:general-reference-vector}, with
\((B^\star,B)=(B^\star,0)\), and
\cref{ass:general-reference-vector} together with the non-singleton
effective domain verifies \cref{ass5,ass4}.
Completing the square gives, for every \(r\in\R\),
\begin{equation}
    \prox_{h_{B,\epsilon}/\gamma_\epsilon}(r)
    =
    \prox_{h_B^\star/(\gamma_\epsilon+\epsilon)}
    \left(
        \frac{\gamma_\epsilon}{\gamma_\epsilon+\epsilon}r
    \right).
    \label{eq:Bayes-ridge-prox-rewrite}
\end{equation}
Thus the canonical fixed point for \(h_{B,\epsilon}\) supplied by
\cref{lem:bayes-l2-continuation} is the corresponding perturbed fixed
point for the base penalty \(h_B^\star\), perturbation strength
\(\epsilon\), and reference zero.  Applying
\cref{thm:perturbed-convex-mse} and using
\cref{eq:Bayes-ridge-prox-rewrite} gives
\begin{equation}
    \frac1p
    \|\widehat\beta_\epsilon-\beta^\star\|_2^2
    \xrightarrow{\PP}
    r_\epsilon, \qquad \text{ where } \qquad  r_\epsilon
    :=
    \E\left[
        \left(
            \prox_{h_{B,\epsilon}/\gamma_\epsilon}
            (B^\star+\sqrt{\tau_{1,\epsilon}}Z)
            -B^\star
        \right)^2
    \right].
    \label{eq:Bayes-ridge-estimator-risk}
\end{equation}
By \cref{eq:Bayes-ridge-risk}, we have $r_\epsilon\rightarrow R_B^+.$
The perturbed risks approach \(R_B^+\); attainment by the fixed
penalty still requires comparison with the unperturbed minimizer.

\medskip
\noindent\emph{Step 4: compare the perturbed and unperturbed minimizers.}
We now verify the law-convergence hypotheses of
\cref{lem:ridge-optimizer-continuity}. Define $Q_\epsilon(y)
    =\prox_{h_{B,\epsilon}/\gamma_\epsilon}(y),$ and $
    U_\epsilon
    =Q_\epsilon(B^\star+\sqrt{\tau_{1,\epsilon}}Z),$
with $ U_0
    =f_{\tau_B^+}(B^\star+\sqrt{\tau_B^+}Z).$
The joint-law conclusion
\cref{eq:general-joint-law-before-prox-rewrite}, together with
\cref{eq:Bayes-ridge-prox-rewrite}, gives
$\frac{1}{p}\sum_{j=1}^p\delta_{\widehat\beta_{\epsilon,j}}
    \rightarrow_{\PP}
    \operatorname{Law}(U_\epsilon)$ in $W_2$
for every sufficiently small fixed \(\epsilon>0\).
By \cref{eq:Bayes-ridge-prox-rewrite},
\cref{eq:Bayes-ridge-parameters}, and continuity of proximal maps in
their positive scale parameter,
$Q_\epsilon(y)\rightarrow f_{\tau_B^+}(y)$ for every $y\in\R$.
The maps \(Q_\epsilon\) are one-Lipschitz, and
\(Q_\epsilon(0)\) remains bounded for small \(\epsilon\).  Coupling
all scalar variables through the same \((B^\star,Z)\), the convergence
\(\tau_{1,\epsilon}\to\tau_B^+\), the finite second moments, and
dominated convergence therefore give
$\E[(U_\epsilon-U_0)^2]\rightarrow0.$
In particular,
\(\operatorname{Law}(U_\epsilon)\to\operatorname{Law}(U_0)\) in
\(W_2\).

Let \(\widehat\beta_0\) be any measurable minimizer of the
unperturbed objective with penalty \(h_B^\star\).  The empirical-law
convergence of $\widehat\beta_{\epsilon}$, the law convergence $\E[(U_\epsilon-U_0)^2]\rightarrow0$, and the curvature condition
\cref{eq:Bayes-output-positive-curvature} verify every hypothesis of
\cref{lem:ridge-optimizer-continuity}.  Hence, for every \(\delta>0\),
\begin{equation}
    \lim_{\epsilon\downarrow0}
    \limsup_{p\to\infty}
    \PP\left(
        \frac1p
        \|\widehat\beta_\epsilon-\widehat\beta_0\|_2^2
        >\delta
    \right)=0.
    \label{eq:Bayes-ridge-estimator-continuity}
\end{equation}
\medskip
\noindent\emph{Step 5: transfer the risk limit to the fixed penalty.}
For every \(p\) and \(\epsilon>0\),
\[
    \left|
        \frac1{\sqrt p}\|\widehat\beta_0-\beta^\star\|_2
        -
        \frac1{\sqrt p}\|\widehat\beta_\epsilon-\beta^\star\|_2
    \right|
    \leq
    \frac1{\sqrt p}
    \|\widehat\beta_0-\widehat\beta_\epsilon\|_2.
\]
Taking first \(p\to\infty\) in
\cref{eq:Bayes-ridge-estimator-risk}, then
\(\epsilon\downarrow0\), and using $r_\epsilon\rightarrow R_B^+$ and 
\cref{eq:Bayes-ridge-estimator-continuity}
gives $ \frac1p
    \|\widehat\beta_0-\beta^\star\|_2^2
    \rightarrow{\PP}R_B^+.$
Thus the lower bound is achieved by the fixed admissible penalty
\(h_B^\star\).  Together with \cref{thm:Bayes VAMP-lower-bound}, this proves \cref{eq: strict-equality} in the log-concave case.

\textbf{A uniform gap in the non-log-concave case.}
Suppose now that \(\pi*N(0,\tau_B^+)\) is not log-concave.  Then
\(V_\pi>0\).  Apply \cref{lem:local-scalar-prox-gap} at \(\tau_B^+\), and
shrink its \(\epsilon>0\), if necessary, so that continuity of \(m_\pi\)
also gives
\(
  |m_\pi(\tau)-m_\pi(\tau_B^+)|
  =
  |m_\pi(\tau)-R_B^+|
  <
  \frac{1}{2}\Delta_{\rm sc}\)
whenever
\(|\tau-\tau_B^+|\le\epsilon.
\) In other words,
\begin{equation}
 m_\pi(\tau)
 \geq R_B^+-\frac{1}{2}\Delta_{\rm sc}
 \quad\text{for every }\tau\in
 [\tau_B^+-\epsilon,\tau_B^++\epsilon].
 \label{eq:mmse-local-lower}
\end{equation}
Because the prior is nondegenerate, \cref{eq:mmseprime} implies that
\(m_\pi\) is strictly increasing. This follows since  the
Gaussian likelihood is strictly positive everywhere, and therefore the posterior law
of \(B^\star\) given \(Y_\tau=y\) is equivalent to \(\pi\); hence, is
nondegenerate for every \(y\).  Therefore, we use that 
\(
  f_\tau'(y)
  =
  \frac{1}{\tau} \operatorname{Var}(B^\star\mid Y_\tau=y)
  >0
\) in \cref{eq:mmseprime}. Consequently,
$\Delta_{\rm far}
 :=m_\pi(\tau_B^++\epsilon)-R_B^+>0.$
Also, since \(V_\pi,\tau_B^+>0\), the best-linear bound  \cref{eq:bestlinear}  also gives
$R_B^+<\min\{V_\pi,\tau_B^+\}.$
Define
\begin{equation}
 \Delta_{\rm cvx}
 :=\min\left\{
  \frac{\Delta_{\rm sc}}2,
  \Delta_{\rm far},
  \frac{\tau_B^+-R_B^+}{2},
  V_\pi-R_B^+
 \right\}>0.
 \label{eq:uniform-convex-gap}
\end{equation}
We prove that this same constant works for every admissible penalty \(h \in \cC_{\rm LB}\).  We
separate the four possible cases so that the role of each earlier result is
explicit.

First note the following step we have used repeatedly.  If a fixed
\(\lambda>0\) and a selected oracle fixed point (whose existence follows from \cref{thm:perturbed-fixed-point} with reference vector \(\beta^\star\)) have scalar risk
\(r^{h,\lambda}\), then
\cref{lem:oraclemonotone,cor:oracle-perturbed-mse} imply, for every
\(\delta>0\),
\begin{align}
\label{eq:transfer_step}
 &\PP\left(
  \frac1p\|\widehat\beta_{\cvx}^h-\beta^\star\|_2^2
  <r^{h,\lambda}-\delta
 \right) \leq
 \PP\left(
  \frac1p\|\widehat\beta^{h,\lambda}-\beta^\star\|_2^2
  <r^{h,\lambda}-\delta
 \right)
 \longrightarrow0.
\end{align}
Consequently, \(\mathcal R(h)\geq r^{h,\lambda}\).  It therefore suffices
in each nondegenerate case to select one fixed \(\lambda\) whose scalar
risk is at least \(R_B^++\Delta_{\rm cvx}\).

\emph{Case 1: singleton effective domain.}
If \(\operatorname{dom}(h)=\{a\}\), then
the singleton-domain case of \cref{sec:degenerate} gives
$\mathcal R(h)
 =\E[(a-B^\star)^2]
 \geq V_\pi
 \geq R_B^++\Delta_{\rm cvx}.$

\emph{Case 2: flat spectrum and nonsingleton domain.}
Suppose \(\mu=\delta_s\) for \(s>0\).  By
\cref{eq:flat-effective-noise}, every selected oracle fixed point has
$\nu_\lambda
 ={\sigma^2s}/{(s+\lambda)^2}
 \rightarrow {\sigma^2}/{s}
 =\tau_B^+.$
Choose one sufficiently small, fixed \(\lambda>0\) such that
\(\nu_\lambda\in[\tau_B^+-\epsilon,\tau_B^+]\).  Its scalar denoiser is a
convex proximal map, and hence
\cref{lem:local-scalar-prox-gap,eq:mmse-local-lower} yield
$
 r^{h,\lambda}
 \geq m_\pi(\nu_\lambda)+\Delta_{\rm sc}
 \geq R_B^++\frac{1}2\Delta_{\rm sc}
 \geq R_B^++\Delta_{\rm cvx}.$
The argument in \cref{eq:transfer_step} gives the same bound for \(\mathcal R(h)\).

\emph{Case 3: nonflat spectrum and affine penalty.}
Let \(h(x)=ux+v\).  Since \(h\) is nonconstant, \(u\ne0\).  In the
regime \(p_0>0\) and \(I_+<\infty\), such a penalty fails the width
condition and is therefore not in \(\cC_{\rm LB}\).  In the remaining
spectral regimes, $p_0 = 0$ or $I_+ = \infty$, the calculation for the affine case of \cref{sec:degenerate} gives
$r^{h,\lambda}
 \geq\nu_\lambda
 \geq\sigma^2\E[{S}/{(S+\lambda)^2}].$
If \(I_+=\infty\), the last expression diverges as
\(\lambda\downarrow0\).  If \(p_0=0\) and \(I_+<\infty\), it converges
to \(\sigma^2I_+\geq\tau_B^+\), where the inequality follows from
\cref{lem:linearcomparison} at \(c=0\) and the Bayes fixed-point identity  $ \tau_B^+ = L_\mu(\cE_\pi(\tau_B^+)$, which follows from \cref{eq:Bayes-VAMP-recursion,eq:tau-B-definition}.
In either case, one may choose a sufficiently small, fixed \(\lambda\) so
that
$r^{h,\lambda}
 \geq R_B^++\frac{1}{2}(\tau_B^+-R_B^+)
 \geq R_B^++\Delta_{\rm cvx}.$
Again, the argument in \cref{eq:transfer_step} gives the conclusion for \(\mathcal R(h)\).

\emph{Case 4: nonflat spectrum and non-affine penalty with nonsingleton
domain.}

By \cref{lem:posterior-prox-characterization}, the posterior mean
\(f_{\tau_B^+}\) is not proximal.  Therefore
\cref{lem:pythag}(i) and $\tau_B^+
    =L_\mu(\cE_\pi(\tau_B^+))$ show that $(
  \tau_B^+,
  \cE_\pi(\tau_B^+)+\Delta_{\tau_B^+})$
is an \(h\)-admissible lower-bound pair.  \cref{prop:perturbed-fp-var} then supplies one sufficiently small, fixed
\(\lambda>0\) and an oracle fixed point with
\(\nu_\lambda\geq\tau_B^+\).
If \(\nu_\lambda\leq\tau_B^++\epsilon\), then
\cref{lem:local-scalar-prox-gap} gives
$r^{h,\lambda}
 \geq m_\pi(\nu_\lambda)+\Delta_{\rm sc}
 \geq R_B^++\Delta_{\rm sc}.$
If instead \(\nu_\lambda\geq\tau_B^++\epsilon\), scalar Bayes
optimality and monotonicity of \(m_\pi\) give
$r^{h,\lambda}
 \geq m_\pi(\nu_\lambda)
 \geq m_\pi(\tau_B^++\epsilon)
 =R_B^++\Delta_{\rm far}.$
Thus \(r^{h,\lambda}\geq R_B^++\Delta_{\rm cvx}\) in both subcases, and
the common transfer step gives
\(\mathcal R(h)\geq R_B^++\Delta_{\rm cvx}\).

These four cases cover every admissible penalty.  Although the fixed
\(\lambda\) selected above may depend on \(h\), the constant
\(\Delta_{\rm cvx}\) in \cref{eq:uniform-convex-gap} does not.  Taking
the infimum over \(h\in\cC_{\rm LB}\) proves the uniform strict gap and,
together with the first part of the proof, the ``only if'' direction of
\cref{eq: strict-equality}.
\end{proof}


\subsection{\texorpdfstring{Proof of \cref{cor:strict-gap-criterion}}{Proof of the smoothing-threshold corollary}}
\label{sec:strict-gap-criterion}

\begin{proof}[Proof of \cref{cor:strict-gap-criterion}]
We show that log-concavity at \(\tau_B^+\) can be tested by evaluating
the recursion at \(\tau_{\rm lc}\), without solving for its greatest
fixed point.

\noindent\emph{Step 1: the log-concavity threshold is attained.}
Define
\[
A_\pi
    :=
    \{
        \tau\geq0:
        \pi*N(0,\tau)\text{ is log-concave}
    \}.
\]
The set \(A_\pi\) is upward closed.  Indeed, if \(s\in A_\pi\) and
\(t>s\), then
$\pi*N(0,t)
    =
    (\pi*N(0,s))*N(0,t-s),$
and convolution preserves log-concavity
\parencite[Theorem~7]{Prekopa1973}.

Write \(\tau_{\rm lc}=\tau_{\rm lc}(\pi)\).  By assumption,
\(\tau_{\rm lc}\in(0,\infty)\).  For every \(t>\tau_{\rm lc}\), the
definition of the infimum gives some \(s\in A_\pi\) with \(s<t\), and
upward closure then gives \(t\in A_\pi\).  In particular,
$\tau_m=\tau_{\rm lc}+\frac1m
    \in A_\pi,$ with $\tau_m\downarrow\tau_{\rm lc}.$
On a common probability space,
$B^\star+\sqrt{\tau_m}Z
    \rightarrow
    B^\star+\sqrt{\tau_{\rm lc}}Z$ almost surely.
Hence $\pi*N(0,\tau_m)
    \xrightarrow{w}
    \pi*N(0,\tau_{\rm lc}).$
Log-concavity of probability measures is preserved under weak limits
\parencite[Proposition~3.6]{SaumardWellner2014}, so
\(\tau_{\rm lc}\in A_\pi\).  Therefore,
$A_\pi=[\tau_{\rm lc},\infty).$
Combining this with
\cref{prop: strict-equality} proves that equality in
\cref{eq: strict-equality} holds if and only if $\tau_B^+\geq\tau_{\rm lc}.$

\noindent\emph{Step 2: prove monotonicity of the normalized recursion.}
For \(\tau>0\), abbreviate $m(\tau)=m_\pi(\tau),$ and $
    \omega(\tau)=\cE_\pi(\tau),$ with
    $c(\tau)={\sigma^2}/{\omega(\tau)},$
and recall $T_{\mu,\pi}(\tau)
    =L_\mu(\omega(\tau)).$
Because \(\tau_{\rm lc}>0\), the prior is nondegenerate.  Thus
\(m(\tau_{\rm lc})>0\), while \cref{eq:bestlinear} gives
$m(\tau_{\rm lc})<\tau_{\rm lc}.$
It follows that
\(0<\omega(\tau_{\rm lc})<\infty\), so
\(c_{\rm lc}:=c(\tau_{\rm lc})\) is well defined and positive.

The Stieltjes representation of \(L_\mu\) gives
\begin{equation}
\begin{aligned}
    \frac{T_{\mu,\pi}(\tau)}{\tau}
    &=
    \left(\frac{\omega(\tau)}{\tau}\right)
    \left(\frac{
        c(\tau)G_\mu(c(\tau))
    }{
        1-c(\tau)G_\mu(c(\tau))
    }\right)                                          =
    \left(\frac{m(\tau)}{\tau-m(\tau)}\right)
    \left(\frac{
        c(\tau)G_\mu(c(\tau))
    }{
        1-c(\tau)G_\mu(c(\tau))
    }\right).
\end{aligned}
\label{eq:normalized-T-factorization}
\end{equation}
We show that both factors in \cref{eq:normalized-T-factorization} are nonincreasing on
\([\tau_{\rm lc},\infty)\).

For \(\tau\geq\tau_{\rm lc}\), we have that $A_\pi=[\tau_{\rm lc},\infty)$ and
\cref{lem:posterior-prox-characterization} give $0\leq f_\tau'(y)\leq1.$
Using \cref{lem:mmsecalculus},
\[
\begin{aligned}
    \frac{\rd}{\rd\tau}
    \left(\frac{m(\tau)}{\tau}\right)
    &=
    \frac1\tau
    \left(
        m'(\tau)-\frac{m(\tau)}\tau
    \right)  =
    \frac1\tau
    \left(
        \E[f_\tau'(Y_\tau)^2]
        -\E[f_\tau'(Y_\tau)]
    \right)
    \leq0.
\end{aligned}
\]
Since \(x\mapsto x/(1-x)\) is increasing on \([0,1)\), the first
factor
${m(\tau)}/{(\tau-m(\tau))}$
is nonincreasing.

For the second factor, \cref{lem:mmsecalculus} shows that
\(\omega(\tau)\) is nondecreasing, so \(c(\tau)\) is nonincreasing.
Moreover,
\[
    \frac{\rd}{\rd c}
    \bigl(cG_\mu(c)\bigr)
    =
    \E\left[
        \frac{S}{(c+S)^2}
    \right]
    >0.
\]
Thus \(c(\tau)G_\mu(c(\tau))\) is nonincreasing.  Since
\(cG_\mu(c)\in[0,1)\) and \(x\mapsto x/(1-x)\) is increasing, the
second factor in \cref{eq:normalized-T-factorization} is also
nonincreasing.  Because both factors in the last line of
\cref{eq:normalized-T-factorization} are nonnegative and nonincreasing,
\begin{equation}
    \tau\longmapsto
    {T_{\mu,\pi}(\tau)}/{\tau}
    \quad\text{is nonincreasing on }
    [\tau_{\rm lc},\infty).
    \label{eq:normalized-T-monotonicity}
\end{equation}

\noindent\emph{Step 3: testing the greatest fixed point at
\(\tau_{\rm lc}\).}
We claim that
$\tau_B^+<\tau_{\rm lc}
    \Longleftrightarrow
    T_{\mu,\pi}(\tau_{\rm lc})<\tau_{\rm lc}.$
If
\(T_{\mu,\pi}(\tau_{\rm lc})<\tau_{\rm lc}\), then
\cref{eq:normalized-T-monotonicity} gives
${T_{\mu,\pi}(\tau)}/\tau
    \leq
    {T_{\mu,\pi}(\tau_{\rm lc})}/{\tau_{\rm lc}}
    <1$ for every $\tau\geq\tau_{\rm lc}.$
There is therefore no fixed point at or above \(\tau_{\rm lc}\), and
the greatest fixed point satisfies \(\tau_B^+<\tau_{\rm lc}\).

Conversely, suppose
\(T_{\mu,\pi}(\tau_{\rm lc})\geq\tau_{\rm lc}\).  If equality holds,
then \(\tau_{\rm lc}\) itself is a fixed point.  If the inequality is
strict, use \(\cE_\pi(\tau)\leq V_\pi\) and monotonicity of \(L_\mu\)
to obtain
$T_{\mu,\pi}(\tau)
    \leq L_\mu(V_\pi).$
Choose $\overline\tau
    >
    \max\{\tau_{\rm lc},L_\mu(V_\pi)\}.$
Then $ T_{\mu,\pi}(\tau_{\rm lc})-\tau_{\rm lc}>0,$ and $T_{\mu,\pi}(\overline\tau)-\overline\tau<0.$
Continuity of \(T_{\mu,\pi}\) gives a fixed point in
\((\tau_{\rm lc},\overline\tau)\).  In both cases,
\(\tau_B^+\geq\tau_{\rm lc}\), proving
the desired equivalence.

\medskip
\noindent\emph{Step 4: express the threshold test through the prior
and spectrum.}
Evaluate the Stieltjes representation at \(\tau_{\rm lc}\).  With
\(c_{\rm lc}=c(\tau_{\rm lc})\) and
\(\omega_{\rm lc}=\omega(\tau_{\rm lc})\),
we have $ T_{\mu,\pi}(\tau_{\rm lc})
    =
    \omega_{\rm lc}
    \frac{
        c_{\rm lc}G_\mu(c_{\rm lc})
    }{
        1-c_{\rm lc}G_\mu(c_{\rm lc})
    }.$
Since \(c_{\rm lc}G_\mu(c_{\rm lc})\in[0,1)\),
\[
\begin{aligned}
    T_{\mu,\pi}(\tau_{\rm lc})<\tau_{\rm lc}
    \quad\Longleftrightarrow\quad
    c_{\rm lc}G_\mu(c_{\rm lc})
    &<
    \frac{\tau_{\rm lc}}
         {\tau_{\rm lc}+\omega_{\rm lc}}                         =
    1-\frac{m_\pi(\tau_{\rm lc})}{\tau_{\rm lc}},
\end{aligned}
\]
where the final identity uses
$\omega_{\rm lc}
    =
    {m_\pi(\tau_{\rm lc})}/
         {(1-m_\pi(\tau_{\rm lc})/\tau_{\rm lc})}.$
Together with the Step 3 result just above, this gives
$\tau_B^+<\tau_{\rm lc}
    \Longleftrightarrow 
    c_{\rm lc}G_\mu(c_{\rm lc})
    <
    1-{m_\pi(\tau_{\rm lc})}/{\tau_{\rm lc}}.$
Taking complements and using
$\tau_B^+\geq\tau_{\rm lc}$ proves that equality in
\cref{eq: strict-equality} holds if and only if
$c_{\rm lc}G_\mu(c_{\rm lc})
    \geq
    1-{m_\pi(\tau_{\rm lc})}/{\tau_{\rm lc}},$
as claimed.
\end{proof}

\section{Spectral result proofs}
\label{sec:spectral_proofs}

\subsection{\texorpdfstring{Spectral results in \cref{sec:main-results}}{Spectral results from the main results section}}

We prove the spectral results presented in \cref{sec:main-results}.

\begin{proof}[Proof of \cref{prop:spectral-monotonicity}]
Set \(R_\omega(S):=\frac{\sigma^2}{\sigma^2+\omega S}\). Then 
\(
    L_\mu(\omega) = \frac{\omega\E[R_\omega(S)]}{1-\E[R_\omega(S)]}.
\)
Differentiation gives
\(
    L_\mu'(\omega) = \frac{\VAR(R_\omega(S))}{(1-\E[R_\omega(S)])^2} \ge0,
\)
which proves (i). For (ii), use the Stieltjes representation \cref{eq:stieltjes-L-preview}: for fixed
\(c>0\), the map \(x\mapsto\sigma^2x/(1-cx)\) is strictly increasing on
\((0,1/c)\). Part (iii) follows because \(s\mapsto1/(s+c)\) is convex.
Note that the comparison in part (ii) also holds at \(\omega=0\), by taking \(\omega\downarrow0\) and using the continuous extension \(L_\mu(0)=\sigma^2/\E_\mu[S]\). This case is only relevant for the degenerate prior.

Finally, for part (iv), since the spectral laws in part (ii) share a prior-stage variance \(V_\pi\), we have $\tau_{\mu_2,+}^{0}
=L_{\mu_2}(V_\pi)
\geq
L_{\mu_1}(V_\pi)
=\tau_{\mu_1,+}^{0}$.
Suppose inductively that
\(\tau_{\mu_2,+}^{k}\geq\tau_{\mu_1,+}^{k}\). Since
\(\mathcal E_\pi\) and \(L_{\mu_2}\) are nondecreasing, and
\(L_{\mu_2}(\omega)\geq L_{\mu_1}(\omega)\) by part (ii),
\(
\tau_{\mu_2,+}^{k+1}
=L_{\mu_2}\!\left(\mathcal E_\pi(\tau_{\mu_2,+}^{k})\right)
\geq
L_{\mu_1}\!\left(\mathcal E_\pi(\tau_{\mu_1,+}^{k})\right)
=\tau_{\mu_1,+}^{k+1}.
\)
The comparison therefore holds at every iterate.
Since \(T_{\mu,\pi}(\tau)\leq L_\mu(\VAR(B^\star))=\tau_{0,+}^B\),
the Bayes VAMP state evolution, $\tau_{t+1,+}^B
    :=
    T_{\mu,\pi}(\tau_{t,+}^B)$, initialized at $\tau_{0,+}^B$ is nonincreasing.
Monotonicity of \(T_{\mu,\pi}\) shows that it remains above every fixed
point, and continuity then gives
$\tau_{t,+}^B\downarrow\tau_B^+.$
Using this result, the two limits are respectively
the greatest fixed points \(\tau_B^+(\mu_2)\) and
\(\tau_B^+(\mu_1)\). The MSE comparison follows because \(m_\pi\) is
nondecreasing.
\end{proof}

\begin{proof}[Proof of \cref{cor:flat-spectrum-best}]
Fix \(c>0\), and write \(\phi_c(s)=(s+c)^{-1}\). Since \(\phi_c\) is convex, Jensen's inequality gives \(\phi_c(\bar s)\leq\E_\mu[\phi_c(S)]\). Moreover, for every \(s\in[a,b]\), convexity gives
\(\phi_c(s)\leq\frac{b-s}{b-a}\phi_c(a)+\frac{s-a}{b-a}\phi_c(b)\).
Taking expectations and using \(\E_\mu[S]=\bar s\), we obtain
\(G_{\mu_{\bar s}}(c)\leq G_\mu(c)\leq G_{\mu_{a,b,\bar s}}(c)\).
Since this holds for every \(c>0\), \cref{prop:spectral-monotonicity} gives the asserted ordering of the spectrum-stage maps, greatest Bayes-VAMP fixed points, and associated lower bounds.
\end{proof}

\begin{proof}[Proof of \cref{cor:spectral-attainment-risk}]
By \cref{prop:spectral-monotonicity}, we have \(\tau_2\geq\tau_1\) since \(\mu_2\succeq_{\rm St}\mu_1\).
If the variances are equal, log-concavity at \(\tau_2\) follows
from the assumption of log-concavity at \(\tau_1\). Otherwise,
$\pi*N(0,\tau_2)
 =(\pi*N(0,\tau_1))*N(0,\tau_2-\tau_1)$
is log-concave because convolution preserves log-concavity.
Apply \cref{prop: strict-equality} separately to the two models.
It identifies each infimum with \(m_\pi(\tau_i)\) and supplies
the stated attaining penalty for each model. Monotonicity of
\(m_\pi\) proves \cref{eq:spectral-attainment-risk-order}.

For strictness, assume \(V_\pi>0\). Then
\(0<\cE_\pi(\tau_1)<\infty\), so \(c_1\) is positive and
finite. The Stieltjes representation
\cref{eq:stieltjes-L-preview} and the strict increase of
\(u\mapsto\sigma^2u/(1-c_1u)\) for \(0<u<1/c_1\) give
$ T_{\mu_2,\pi}(\tau_1)
 >T_{\mu_1,\pi}(\tau_1)=\tau_1.$
On the other hand,
\(T_{\mu_2,\pi}(\tau)\leq L_{\mu_2}(V_\pi)\) for every
\(\tau>0\). Choose
\(\overline\tau>\max\{\tau_1,L_{\mu_2}(V_\pi)\}\).
Continuity gives a fixed point of \(T_{\mu_2,\pi}\) in
\((\tau_1,\overline\tau)\); hence, its greatest fixed point
satisfies \(\tau_2>\tau_1\). For a nondegenerate prior the
posterior variance in every positive-variance Gaussian channel
is strictly positive at every observation. The derivative
identity in \cref{lem:mmsecalculus} therefore gives
\(m_\pi'(\tau)>0\) for every \(\tau>0\), proving
\(\mathcal R_2>\mathcal R_1\).
\end{proof}

\textbf{Interaction between the spectrum and the incoming uncertainty $\omega$.} We conduct a rigorous analysis of the spectral impact on the lower bound under different regimes of the incoming uncertainty $\omega$.

In summary, for $S \sim \mu$, as \(\omega\downarrow0\), the leading coefficients in the asymptotic
expansion of \(L_\mu(\omega)\) are determined by the first few moments of
the spectral law. As \(\omega\to\infty\), the behavior is instead governed
by the behavior of the spectrum near zero.

\begin{proposition}
\label{prop:moment-expansion}
Assume \(m_1:=\E[S]>0\) and \(m_2:=\E[S^2]<\infty\). Then, as
\(\omega\downarrow0\), we have 
$L_\mu(\omega)
    =
    \frac{\sigma^2}{m_1}
    +
    \frac{\VAR(S)}{m_1^2}\,\omega
    +o(\omega).$
In particular,
$ L_\mu'(0+)
    = {\VAR(S)}/{\E[S]^2}.$
\end{proposition}

\begin{proof}
As \(c=\sigma^2/\omega\to\infty\), we have, using a finite geometric expansion and dominated convergence, $G_\mu(c)
    =
    \frac1c-\frac{m_1}{c^2}+\frac{m_2}{c^3}+o(c^{-3})$.
Substitution into the Stieltjes representation of the spectrum-stage map in \cref{eq:stieltjes-L-preview} and expansion of the
quotient give the desired result.
\end{proof}

Thus the term \(\sigma^2/\E[S]\)  using only the mean is the zeroth-order term, and
spectral variance is the first correction. Hence, we see that for small uncertainty $\omega$, and therefore at large $c$, the first few moments of the spectrum are what $L_\mu(\omega)$ ``sees,'' and therefore are the features that matter.

\begin{proposition}
\label{prop:hard-edge}
We start with the mass-at-zero case, then discuss how density distributed around zero affects the spectrum-stage map $L_\mu(\omega)$.

\begin{enumerate}[label=\textup{(\roman*)},leftmargin=2.2em]

\item Suppose there exists mass at zero, meaning $\mathbb P(S = 0) = 1 - q$ for $q < 1$. We represent $\mu$ as a mixture of the Dirac measure at zero $\delta_0$ and the measure $\nu_+$ of the remaining, strictly positive spectral values:
$\mu=(1-q)\delta_0+q\nu_+$ for $q\in(0,1),$ and $\nu_+(\{0\})=0.$
Let \(G_+(c)=\E_{\nu_+}[1/(S+c)]\) be the Stieltjes transform for the strictly positive law $\nu_+$. Then
\[
    L_\mu(\omega)
    =
    \frac{(1-q)\omega+q\sigma^2G_+(c)}
         {q(1-cG_+(c))},
    \qquad c=\frac{\sigma^2}{\omega}.
\]
Specifically, having mass at zero means that \(q<1\), and we would have
\begin{equation}
    L_\mu(\omega)
    \sim
   \left(\frac{1-q}{q}\right)\omega \qquad \text{as } \omega \rightarrow \infty.
    \label{eq:zero-atom-linear}
\end{equation}

\item Now suppose that \(\PP(S=0)=0\). If \(\E[S^{-1}]<\infty\), then
$L_\mu(\omega)
    \rightarrow
    \sigma^2\E[S^{-1}]$ as $\omega \rightarrow \infty.$

\item Suppose that \(\PP(S=0)=0\) and that \(\mu\) has density
$f(s)=Cs^\alpha+o(s^\alpha)$ as $s\downarrow 0,$ for $C>0,$ and $\alpha>-1.$
This rate representation allows us to separate the edge-at-0 into three subcases. If \(\alpha=0\), then
$L_\mu(\omega)
    \sim
    \sigma^2C\log(\omega/\sigma^2).$
If \(-1<\alpha<0\), we can use Karamata's theorem, presented in Theorem A.7 of \textcite{LangerWoracek2024Karamata}, to obtain
\(
    L_\mu(\omega)
    \sim
    \sigma^2C
    \left({\sigma^2}/{\omega}\right)^\alpha
    {\pi}/{\sin(\pi(\alpha+1))}.
\)
For \(\alpha>0\), we have $\mathbb E[S^{-1}] < \infty$, which is covered by case $\text{(ii)}$.

\end{enumerate}
\end{proposition}

\begin{proof}
The proof of part (i) uses $G_\mu(c)=\frac{1}{c}(1-q)+qG_+(c)$
in \cref{eq:stieltjes-L-preview}. Since \(cG_+(c)\to0\), the mass at
zero produces the linear term in \cref{eq:zero-atom-linear}.
Part (ii) follows from dominated convergence in
\cref{eq:stieltjes-L-preview}. For part (iii), split the integral
near zero and substitute \(s=cu\); the integral
\(\int_0^\infty u^\alpha/(1+u)\,\rd u\) equals
\(\pi/\sin(\pi(\alpha+1))\) for \(-1<\alpha<0\).  
\end{proof}

The mass-at-zero and continuous edge-at-zero cases are therefore qualitatively
different. When \(q<1\), \cref{prop:hard-edge} gives asymptotically linear growth, namely \(L_\mu(\omega)/\omega\to(1-q)/q\) as \(\omega\to\infty\). By contrast, the density regimes considered in \cref{prop:hard-edge} give a finite limit, logarithmic growth, or power-law growth according to the behavior of the density near zero.

\subsection{Proofs of spectral example results}
\label{sec:spectral-examples}

\begin{proof}[Proof of \cref{prop: product-gaussian}]
For every \(j\geq1\), write $X_p^{(j)}
    :=
    G_jG_{j-1}\cdots G_1$ and $A_{p,j}
    :=
    (X_p^{(j)})^T X_p^{(j)}
    \in\R^{p\times p}$,
and define the empirical spectral measure $\widehat\mu_{p,j}
    :=
    \frac1p\sum_{i=1}^p
    \delta_{\lambda_i(A_{p,j})}$.

For the base case, define \(X_p^{(0)}=A_{p,0}:=I_p\), which trivially has limiting law \(\mu_0=\delta_1\), the Dirac measure at 1. We first prove by induction that, for every fixed \(j\), \(\widehat\mu_{p,j}\) converges weakly almost surely to a deterministic law \(\mu_j\).

For the inductive step, since $A_{p, j} = (X_p^{(j)})^\top X_p^{(j)}$, we have the polar decomposition \(X_p^{(j)}=U_{p,j}A_{p,j}^{1/2}\), where \(U_{p,j}^TU_{p,j}=P_{p,j}\) and \(P_{p,j}\) is the orthogonal projection onto the range of \(A_{p,j}\). Let us now define \(W_{p,j+1}\in\R^{n_{j+1}\times p}\), which has iid standard Gaussian entries that are independent of all the original variables, and set
\(H_{p,j+1}:=\sqrt{n_{j+1}}\,G_{j+1}U_{p,j}+W_{p,j+1}(I_p-P_{p,j})\).
Conditional on \(X_p^{(j)}\), the rows of \(H_{p,j+1}\) are independent centered Gaussian vectors with covariance \(P_{p,j}+(I_p-P_{p,j})=I_p\). Thus \(H_{p,j+1}\) has iid standard Gaussian entries and is independent of \(X_p^{(j)}\). Also,
\(H_{p,j+1}A_{p,j}^{1/2}=\sqrt{n_{j+1}}\,G_{j+1}X_p^{(j)}\),
so
\(A_{p,j+1}=n_{j+1}^{-1}A_{p,j}^{1/2}H_{p,j+1}^TH_{p,j+1}A_{p,j}^{1/2}\).
Since \(n_{j+1}/p\to\delta_{j+1}\in(0,\infty)\), we may apply the induction argument of \textcite[Theorem~1.1]{Silverstein1995} to obtain the conclusion: the asserted almost-sure weak convergence at every fixed depth $j$.

We first establish the operator-norm and moment bounds needed to verify the spectral requirements of \cref{ass1}, which we will use to establish the lower bound monotonicity. For an \(m\times d\) matrix \(G\) with iid \(N(0,1/m)\) entries, Gaussian tail bounds applied to \(1/4\)-nets of the two unit spheres give
\(\PP(\|G\|_{\mathrm{op}}>t)\leq2\,9^{m+d}\exp(-mt^2/8)\).
Because we defined the dimensions of all layers to be proportional to \(p\), this implies, for each fixed \(\ell\), that there are constants \(C_\ell,c_\ell>0\) such that
\(\PP(\|G_\ell\|_{\mathrm{op}}>C_\ell)\leq2e^{-c_\ell p}\)
for all sufficiently large \(p\), and that \(\sup_p\E[\|G_\ell\|_{\mathrm{op}}^4]<\infty\). Since
\(\|A_{p,j}\|_{\mathrm{op}}\leq\prod_{\ell=1}^j\|G_\ell\|_{\mathrm{op}}^2\),
the summability of these tail probabilities gives an almost-sure eventual deterministic bound on \(\|A_{p,j}\|_{\mathrm{op}}\). Consequently, \(\mu_j\) is compactly supported, and the almost-sure weak convergence strengthens to
\(\widehat\mu_{p,j}\xrightarrow{W_2}_{\mathrm{a.s.}}\mu_j\).

Independence of the layers and the fourth-moment bounds also give
\(\sup_p\E[\|A_{p,j}\|_{\mathrm{op}}^2]<\infty\).
Hence \(p^{-1}\tr A_{p,j}\) is uniformly integrable. The spectral convergence implies
\(p^{-1}\tr A_{p,j}\to\int s\,\mu_j(\rd s)\) almost surely, and uniform integrability gives convergence in \(L^1\). Finally, \(\E[G_\ell^TG_\ell]=I_{n_{\ell-1}}\), so repeated conditional expectation yields \(\E[p^{-1}\tr A_{p,j}]=1\). Therefore \(\int s\,\mu_j(\rd s)=1\).

This verifies the spectral requirements of \cref{ass1}. The design law is also invariant under right multiplication by any deterministic orthogonal matrix, since the rightmost factor \(G_1\) has this invariance. 

We now continue with the proof of the Stieltjes transform ordering. Fix \(k\geq1\) and set
\(
y:=\delta_{k+1}^{-1},
\)
and, for \(c>0\), define
\(
r_j(c):=cG_{\mu_j}(c).
\)
Because \(\mu_j\) has mean one,
\(
\mu_j(\{0\})<cG_{\mu_j}(c)<1,
\)
and differentiation under the integral gives
\(
r_j'(c)
    =
    \int_{[0,\infty)}
        {s}/{(c+s)^2}\,
        \mu_j(\rd s)
    >
    0.
\)
Thus \(r_j\) is strictly increasing.

The Stieltjes equation (1.4) supplied by
\textcite{Silverstein1995}, evaluated at \(z=-c<0\), gives
\[
G_{\mu_{k+1}}(c)
    =
    \int_{[0,\infty)}
    \frac{1}{
        c+s\left(
            1-y+y\,cG_{\mu_{k+1}}(c)
        \right)
    }\,
    \mu_k(\rd s).
\]
Define
\(
a_c
    :=
    1-y+y\,r_{k+1}(c).
\)
We claim that \(0<a_c<1\). The upper bound follows from
\(r_{k+1}(c)<1\). If \(y\leq1\), the lower bound follows directly from
\(r_{k+1}(c)>0\). If \(y>1\), then
\(\text{rank}(A_{p,k+1})\leq n_{k+1}\), so weak convergence gives
\(
\mu_{k+1}(\{0\})\geq1-y^{-1}.
\)
Consequently,
\(
r_{k+1}(c)>\mu_{k+1}(\{0\})\geq1-y^{-1},
\)
which again gives \(a_c>0\).

The preceding Stieltjes equation can therefore be rewritten as
\(
G_{\mu_{k+1}}(c)
    =
    \frac1{a_c}
    G_{\mu_k}\!\left({c}/{a_c}\right),
\)
or equivalently,
\(
r_{k+1}(c)
    =
    r_k\!\left({c}/{a_c}\right).
\)
Since \(0<a_c<1\), we have \(c/a_c>c\). The strict monotonicity of
\(r_k\) therefore gives
\(
r_{k+1}(c)>r_k(c),
\)
and hence
\(
G_{\mu_{k+1}}(c)>G_{\mu_k}(c)
\)
for every \(c>0\).

We now proceed to showing the strict inequality of the spectrum-stage map and subsequently the lower bound. We first calculate the quantity $\mu_j\{0\}$ needed below. Write \(z_j:=\mu_j(\{0\})=\lim_{c\downarrow0}r_j(c)\). The inequality \(r_{k+1}(c)>r_k(c)\) and the rank bound give
\(z_{k+1}\geq\max\{z_k,1-\delta_{k+1}\}\).
If this inequality were strict, then
\(a_c=1-\delta_{k+1}^{-1}+\delta_{k+1}^{-1}r_{k+1}(c)\)
would converge to a strictly positive number as \(c\downarrow0\). Consequently, the identity \(r_{k+1}(c)=r_k(c/a_c)\) would give \(z_{k+1}=z_k\), a contradiction. Therefore
\(z_{k+1}=\max\{z_k,1-\delta_{k+1}\}\).
The same calculation applies to the first layer with \(\mu_0=\delta_1\) and \(z_0=0\). Induction yields
\(\mu_j(\{0\})=1-\min\{1,\delta_1,\ldots,\delta_j\}\).
In particular, when \(\delta_1,\delta_2\geq\bar\delta=\min\{1,\delta\}\), the three-layer law has mass exactly \(1-\bar\delta\) at zero. This conclusion also applies when \(\delta=1\).

Fix \(\omega>0\) and put \(c=\sigma^2/\omega\). In
\cref{eq:stieltjes-L-preview}, the map
\(
x\mapsto\sigma^2x/(1-cx)
\)
is strictly increasing on the relevant interval. Hence
\(
L_{\mu_{k+1}}(\omega)
    >
    L_{\mu_k}(\omega)
\)
for every \(\omega>0\).

For \(j\in\{k,k+1\}\), write
\(
T_j(\tau)
    :=
    L_{\mu_j}(\cE_\pi(\tau))
\)
and
\(
\tau_j^+
    :=
    \tau_B^+(\mu_j).
\)
Suppose first that \(V_\pi>0\). By
\cref{lem:mmsecalculus},
\(
\cE_\pi(\tau_k^+)>0,
\)
and therefore
\(
T_{k+1}(\tau_k^+)
    >
    T_k(\tau_k^+)
    =
    \tau_k^+.
\)
Set
\(
M_{k+1}
    :=
    L_{\mu_{k+1}}(V_\pi).
\)
Since
\(
\cE_\pi(\tau)\leq V_\pi
\)
and \(L_{\mu_{k+1}}\) is nondecreasing, we have
\(
T_{k+1}(M_{k+1})\leq M_{k+1}
\)
and
\(
M_{k+1}
    \geq
    T_{k+1}(\tau_k^+)
    >
    \tau_k^+.
\)
Continuity therefore gives a fixed point of \(T_{k+1}\) strictly
larger than \(\tau_k^+\), and hence
\(
\tau_{k+1}^+>\tau_k^+.
\)

Moreover, \cref{lem:mmsecalculus} gives
\(
m_\pi'(\tau)
    =
    \E[f_\tau'(Y_\tau)^2]
    >
    0
\)
for every \(\tau>0\). Consequently,
\(
m_\pi(\tau_{k+1}^+)
    >
    m_\pi(\tau_k^+).
\)

If \(V_\pi=0\), then
\(
m_\pi(\tau)=\cE_\pi(\tau)=0
\)
for every \(\tau>0\). Since both limiting laws have mean one,
\(
L_{\mu_k}(0)=L_{\mu_{k+1}}(0)=\sigma^2.
\)
Thus both state-evolution maps are the constant map
\(
\tau\mapsto\sigma^2,
\)
their greatest fixed points are equal to \(\sigma^2\), and both
associated MSE lower bounds are zero. This proves the proposition.
\end{proof}

\begin{proof}[Proof of \cref{prop:lnn-fix-three-layers}]
For any fixed $\delta \neq 1$ and
\(
    (\delta_1,\delta_2)\in\mathcal G_{C,\delta},
\)
define
\[
    r_{\delta_1,\delta_2,\delta}(c) := cG_{\mu_{\delta_1,\delta_2,\delta}}(c).
\]
Recall that $\bar\delta := \min \{1, \delta\}$. The limiting spectral law has mean one and an atom of mass \(1-\bar\delta\) at zero. Moreover,
\(
    r_{\delta_1,\delta_2,\delta}'(c) = \int s(c+s)^{-2}\mu_{\delta_1,\delta_2,\delta}(\rd s)>0.
\)
Thus
\(
    r_{\delta_1,\delta_2,\delta}
\) 
is strictly increasing from \((0,\infty)\) onto \((1-\bar\delta,1)\). Let \(c_{\delta_1,\delta_2,\delta}(r)\) denote its inverse.

The Silverstein equation used in the proof of \cref{prop: product-gaussian} implies that appending a Gaussian layer with limiting ratio \(n_\ell/p =: \eta\) multiplies the inverse $c_\ell(r)$ by
\(
    1-(1-r)/\eta.
\)
Starting from the depth-zero Gram law, which has inverse \(r/(1-r)\), and applying this relation successively to the three layers gives
\[
    c_{\delta_1,\delta_2,\delta}(r) = \frac{r}{1-r}\left(1-\frac{1-r}{\delta_1}\right) \left(1-\frac{1-r}{\delta_2}\right) \left(1-\frac{1-r}{\delta}\right),
    \qquad
    1-\bar\delta<r<1.
\]
This expression is symmetric in \(\delta_1\) and \(\delta_2\), and therefore
\(
    \mu_{\delta_1,\delta_2,\delta} = \mu_{\delta_2,\delta_1,\delta}.
\)
Since \(0<1-r<\bar\delta\), every factor in the inverse function $c_{\delta_1, \delta_2, \delta}(r)$ is positive. Increasing either hidden width strictly increases the corresponding factor
\(
    1-(1-r)/\delta_\ell
\)
and hence strictly increases the inverse. Inverting back the strictly increasing functions therefore shows that, with all other aspect ratios fixed, increasing either hidden width strictly decreases
\(
    G_{\mu_{\delta_1,\delta_2,\delta}}(c)
\) 
for every \(c>0\).

Now suppose that either \(\delta>1\) and \(\delta_1<\delta_2\), or \(0<\delta<1\) and \(\delta_1>\delta_2\). In either case,
\[
    D := C_\delta(\delta_1,\delta_2) - C_\delta(\delta_2,\delta_1)
    =
    (1-\delta)(\delta_1-\delta_2)
    > 0.
\]
Thus the reversed profile \((\delta_2,\delta_1)\) induces the same limiting spectral law but uses a normalized width ``budget'' smaller by \(D\). Define
\(
    t := {D}/({1+\delta_1})>0,
\) 
and
\(
    (\widetilde\delta_1,\widetilde\delta_2) := (\delta_2+t,\delta_1).
\)
Then
\[
\begin{aligned}
C_\delta(\widetilde\delta_1,\widetilde\delta_2) =
C_\delta(\delta_2,\delta_1)+t(1+\delta_1)=
C_\delta(\delta_2,\delta_1)+D=
C_\delta(\delta_1,\delta_2)
=
C.
\end{aligned}
\]
Moreover,
\(
\widetilde\delta_1,\widetilde\delta_2\geq\bar\delta,
\)
so
\(
(\widetilde\delta_1,\widetilde\delta_2) \in\mathcal G_{C,\delta}.
\)
Since the first hidden width has been strictly increased relative to the reversed profile, we have
\(
G_{\mu_{\widetilde\delta_1,\widetilde\delta_2,\delta}}(c)
    <
G_{\mu_{\delta_2,\delta_1,\delta}}(c)
    =
G_{\mu_{\delta_1,\delta_2,\delta}}(c)
\)
for every \(c>0\). This proves both parts of the proposition. The corresponding lower-bound comparison follows from \cref{prop:spectral-monotonicity}; when \(V_\pi>0\), strictness follows from the strict fixed-point argument in the proof of \cref{prop: product-gaussian}.

Fix \(C\geq3\), and let
\(
    (\delta_1,\delta_2)\in\mathcal G_{C,1}.
\)
For \(r\in(0,1)\), write \(d:=1-r\). The inverse-map calculation for $c_{\delta_1, \delta_2, 1}$ follows from the proof of \cref{prop:lnn-fix-three-layers}, with \(\delta=1\), gives
\[
    c_{\delta_1,\delta_2,1}(r) = \frac{r^2}{1-r} \left(1-\frac{d}{\delta_1}\right) \left(1-\frac{d}{\delta_2}\right).
\]
Define
\(
    s:=\delta_1+\delta_2.
\)
The fixed-budget constraint gives
\(
\delta_1\delta_2+s=C,
\)
and
\(
(\delta_1+1)(\delta_2+1)=C+1.
\)
Consequently, the part of the inverse map $c_{\delta_1, \delta_2, 1}$ depending on the hidden
widths can be written as
\[
\begin{aligned}
H_s(d)
&:=
\left(1-\frac{d}{\delta_1}\right)
\left(1-\frac{d}{\delta_2}\right)=
1-d\frac{\delta_1+\delta_2}{\delta_1\delta_2}
    +\frac{d^2}{\delta_1\delta_2}=
1-d\frac{s-d}{C-s}.
\end{aligned}
\]
For every \(d\in(0,1)\), we hav
\(
\frac{\partial}{\partial s}H_s(d)
    =
    -\frac{d(C-d)}{(C-s)^2}
    <
    0.
\)
Thus increasing \(s=\delta_1+\delta_2\) strictly decreases the inverse. Since each inverse mapping is strictly increasing in \(r\), inverting back shows that increasing \(s\) strictly increases the Stieltjes transform for every \(c>0\).

On the other hand, using the arithmetic--geometric mean inequality $(a + b) \geq 2 \sqrt{ab}$ and the definition of $C$ gives
\[
\begin{aligned}
s+2 =
(\delta_1+1)+(\delta_2+1)\geq
2\sqrt{(\delta_1+1)(\delta_2+1)}
=
2\sqrt{C+1}.
\end{aligned}
\]
Equality holds if and only if
\(
\delta_1+1=\delta_2+1,
\)
or equivalently,
\(
\delta_1=\delta_2=\sqrt{C+1}-1.
\)
Since \(C\geq3\), this profile belongs to
\(\mathcal G_{C,1}\). It is therefore the unique profile minimizing
\(s\), and hence the unique profile maximizing the inverse map at every \(r\in(0,1)\). Consequently, for every
\(
    (\delta_1,\delta_2)\in\mathcal G_{C,1}
\)
distinct from
\(
    (\sqrt{C+1}-1,\sqrt{C+1}-1),
\)
we have
\(
G_{\mu_C}(c)
    <
G_{\mu_{\delta_1,\delta_2,1}}(c)
\)
for every \(c>0\). This proves the claim. Any corresponding Bayes--VAMP lower-bound comparison follows from \cref{prop:spectral-monotonicity}; when \(V_\pi>0\), its strictness follows from the strict fixed-point argument in the proof of \cref{prop: product-gaussian}.
\end{proof}

\section{Relation of \texorpdfstring{$q$}{q}-bounded width condition to the AMP \texorpdfstring{\(\delta\)}{delta}-bounded-width condition}\label{sec: comparison-celentano-delta}

The localized $q$-bounded width condition used in this work is the VAMP analogue of the
\(\delta\)-bounded-width condition in
\textcite[Definition~2.1]{Celentano2022barrier}. To make the comparison
explicit, define the global \(q\)-bounded-width class \(\cC_{\rm GB}(q)\) by
\begin{equation}
    h\in\cC_{\rm GB}(q)
    \quad\Longleftrightarrow\quad
    \text{for every compact \(K\subset(0,\infty)\),}\quad
    \limsup_{\kappa\to\infty}
    \sup_{\nu\in K}A_h(\kappa,\nu)<q.
    \label{eq:global-q-width}
\end{equation}
Because \(\cC_{\rm WB}(q;\tau_B^+)\) requires this strict gap only on a
one-sided neighborhood of \(\tau_B^+\), clearly $\cC_{\rm GB}(q)
    \subseteq
    \cC_{\rm WB}(q;\tau_B^+)$; hence, the lower bound in \cref{thm:Bayes VAMP-lower-bound} still holds if we instead take the infimum over all $h \in \cC_{\rm GB}$.

    The setting in \textcite{Celentano2022barrier} considers the iid \(N(0,1/n)\) design with \(n/p\to\delta\). Its limiting spectral law $\mu$ is the Marchenko--Pastur law such that
\(
q = \PP(S > 0) = \min\{\delta, 1\}\), where \(S \sim \mu\). The global version of the \cref{eq:global-q-width} condition therefore becomes
\[
    \text{for every compact \(K\subset(0,\infty)\),}\qquad
    \limsup_{\kappa\to\infty}
    \sup_{\nu\in K}A_h(\kappa,\nu)
    <\min\{\delta,1\}.
\]
To compare the above conidiotn with \textcite[Definition~2.1]{Celentano2022barrier}, consider the separable penalty
\(\rho_p(x)=p^{-1}\sum_{i=1}^p h(\sqrt p\,x_i)\).
Let \(\beta_{0,p,i}=B_i^\star/\sqrt p\) and \(z_{p,i}=Z_i/\sqrt p\), where the pairs \((B_i^\star,Z_i)\) are iid with law \(\pi\otimes N
(0,1)\). For \(\kappa,t>0\), separation of the proximal minimization gives \([\prox_{\kappa\rho_p}(\beta_{0,p}+t z_p)]_i
=p^{-1/2}\prox_{\kappa h}(B_i^\star+tZ_i)\).
Consequently, Gaussian integration by parts gives
\(t^{-1}\E\langle z_p,\prox_{\kappa\rho_p}(\beta_{0,p}+t z_p)\rangle=A_h(\kappa,t^2)\).
The high-dimensional expression is therefore independent of \(p\), and the scalar separable specialization of the cited condition is \(\limsup_{\kappa\to\infty}\sup_{\nu\in K}A_h(\kappa,\nu)<\delta\)
for every compact \(K\subset(0,\infty)\).
Thus \cref{eq:global-q-width}, with \(q=\min\{\delta,1\}\), agrees with this specialization when \(0<\delta\leq1\), but is strictly stronger when \(\delta>1\). Indeed, \(0\leq A_h(\kappa,\nu)\leq1<\delta\) makes the cited condition automatic when \(\delta>1\), whereas \cref{eq:global-q-width} still requires a strict gap below one. For example, \(h(x)=x\) gives \(\prox_{\kappa h}(y)=y-\kappa\) and \(A_h(\kappa,\nu)=1\).
This comparison concerns the global width conditions. Our actual admissible class is defined by \cref{eq:admissible-penalty-class}: for the Marchenko--Pastur law with \(\delta\geq1\), we have \(p_0=0\), so no width condition is imposed. In particular, at \(\delta
=1\), we require no width condition, while \textcite[Definition~2.1]{Celentano2022barrier} still requires the penalty $h$ satisfies $A_h(\kappa, \nu) < 1$. Hence, agreement between the two global conditions at $\delta = 1$ does not imply agreement between the admissible penalty classes.

\section{Numerical experiment details}
\label{sec:empirical-mse-setup}

This section describes the setup for \cref{fig:quantiles_barriers}, which visualizes the convex barriers for several examples of design matrices.

We remark that the codebook contains additional numerical experiments: most notably, algorithmic validation experiments to demonstrate that the state evolution predictions accurately describe the iterates given by the implementable VAMP algorithms.
    
\subsection{\texorpdfstring{Matrix generation \& spectral visualization for \cref{fig:quantiles_barriers}}{Matrix generation and spectral visualization}}
\label{sec:li-sur-matrices}

We use four design matrices defined in \textcite{Li2026debiasing}. They are generated as follows:

\begin{itemize}
    \item \textbf{MatrixNormal.} 
    Let \(G\in\mathbb{R}^{n\times p}\) have iid \(N(0,1)\) entries, and let 
    $(\Sigma_r)_{ij}:=0.5^{|i-j|}$ 
    and $\Sigma_c\sim\operatorname{InvWishart}_p(1.1p, I_p)$
    be independent. The ``MatrixNormal'' design is then given by 
    ${X} = \Sigma_r^{1/2} G\Sigma_c^{1/2}$.
    
    \item \textbf{LNN.}
    Generate $G_1\in\mathbb R^{n\times p}$ and
    $G_2,G_3,G_4\in\mathbb R^{n\times n}$ with iid $N(0,1)$ entries, and set
    \(
        X=n^{-2}G_4G_3G_2G_1.
    \)

    \item \textbf{Multi-t.} Independently draw $g_i\sim N(0, I_p)$ and $u_i\sim\chi^2_3$ for $i \in [n]$, and define the rows by $X_{i\mathbin{:}} =\frac{g_i^T}{\sqrt{u_i/3}}$.
    Thus, before the common normalization, the rows are independent multivariate \(t\) vectors with identity scale and three degrees of freedom.

    \item \textbf{Spiked.} Let \(r=50\), let
    \(U\in\mathbb{R}^{n\times r}\) and
    \(V\in\mathbb{R}^{p\times r}\) be independent matrices with orthogonal columns, drawn from the rotationally invariant distributions on their respective spaces, and let
    \(G\in\mathbb{R}^{n\times p}\) have iid \(N(0,1)\) entries, independently of \(U\) and \(V\). Set $X
        =10U V^T+\frac{1}{n}G$.

\end{itemize}

To visualize the quantiles of limiting spectra of the examples given by \textcite{Li2026debiasing}, we formed pooled finite-dimensional approximations from \(12\) independent \(400\times400\) matrix realizations. Each realized design was then rescaled so that all four designs have the same average measurement strength---mean Gram eigenvalue $\frac{1}{p} \Vert X \Vert_F^2 = 1$. Since we plot the regime where the convex barrier coincides with the Bayes VAMP MSE curve, we further rescale to ensure that the convolved Rademacher prior is log-concave as defined in \cref{prop: strict-equality}. The Gram eigenvalues are therefore rescaled further to have common mean $\bar{s}=0.019$ so that the Bayes VAMP MSE coincides with the convex barrier over all noise variances $\sigma^2$ plotted.

For each design, we treat the pooled finite-dimensional empirical Gram spectrum as the limiting spectral law for an RRI matrix sequence. The finite-dimensional spectra satisfy the compact-support and uniform spectral-bound conditions required by our theorem, so---subject to the remaining assumptions---the corresponding Bayes--VAMP curve is a valid theoretical lower bound for that constructed sequence. This interpretation does not claim that the original Li--Sur random-matrix ensemble has the same natural asymptotic spectral law or itself satisfies those boundedness conditions.

\section{Proximal operator identities}
\label{app:proximal-identities}

This appendix collects facts about proximal operators used throughout the paper. We state the basic results in finite-dimensional Euclidean space, although most applications in the main text are one-dimensional. Let $d\geq 1$ and let $h:\R^d\rightarrow \R\cup\{+\infty\}$ be proper, lsc, and convex. For \(a>0\), we write
\begin{equation}
    P_a(y)
    :=
    \prox_{a h}(y)
    :=
    \argmin_{x\in\R^d}
    \left\{
        \frac12\norm{x-y}_2^2+a h(x)
    \right\}.
    \label{eq:prox-definition-app}
\end{equation}
The minimizer in \cref{eq:prox-definition-app} exists and is unique because
the objective is coercive and strongly convex. We adopt the convention
\(P_0=\Id\). When \(d=1\), the almost-everywhere derivative of the scalar proximal map is denoted $P'_a$. When \(d>1\), its Jacobian and divergence at points of differentiability are denoted \(D P_a\) and \(\dive P_a\), respectively.

The first-order optimality condition defining the minimization problem can be written in terms of the convex subdifferential as follows.

\begin{lemma}
\label{lem:prox-optimality-app}
For every \(y\in\R^d\), we have
\(
    y-P_a(y)\in a\,\partial h(P_a(y)).
\)
Equivalently, $P_a=(\Id+a\,\partial h)^{-1}$.
Moreover, the following are equivalent: (i) $P_a(y)=y$, (ii) $0\in\partial h(y)$, and (iii) $y\in\argmin h$.
\end{lemma}
\begin{proof}
See \textcite[Eq.\ (O.1)--(O.2)]{Celentano2022barrier}.
\end{proof}

\begin{lemma}\label{lem:finiteness-proximal-0}
Let \(h:\R^d\to\R\cup\{+\infty\}\) be proper, lsc, and
convex. For every \(a>0\), the problem defining \(\prox_{a h}(0)\) has a
unique minimizer in \(\dom\,h\). In particular, $\|\prox_{a h}(0)\|_2<\infty$.
\end{lemma}

\begin{proof}
We gave
\(
\prox_{ah}(0)
=
\argmin_x\{
h(x)+\frac{1}{2a}\|x\|_2^2
\}
\) by \cref{eq:prox-definition-app}.
By \textcite[Theorems~1.25 and~2.26(a)]{RockafellarWets1998},
this proximal point exists and is unique for every \(a>0\).
Properness of \(h\) implies that the minimizer belongs to \(\dom\,h\);
hence its Euclidean norm is finite.
\end{proof}

We say that \(h\) is \(\kappa\)-strongly convex, for \(\kappa\geq0\), when
\(x\mapsto h(x)-\frac\kappa2\norm{x}_2^2\) is convex. The case
\(\kappa=0\) includes every convex penalty.

\begin{proposition}[Firm nonexpansiveness]
\label{prop:prox-firm-app}
If \(h\) is \(\kappa\)-strongly convex, then for all
\(y,\widetilde y\in\R^d\),
we have $\inner{y-\widetilde y}
    {P_a(y)-P_a(\widetilde y)}
    \geq
    (1+a\kappa)
    \norm{P_a(y)-P_a(\widetilde y)}_2^2.$
Consequently, $\norm{P_a(y)-P_a(\widetilde y)}_2
    \leq
    \norm{y-\widetilde y}_2/(1+a\kappa).$
In particular, every proximal map of a convex function is nonexpansive.
\end{proposition}

\begin{proof}
See \textcite[Eq.\ (O.3)--(O.4)]{Celentano2022barrier}.
\end{proof}

\begin{corollary}
\label{cor:prox-jacobian-app}
The map \(P_a\) is differentiable almost everywhere. At every point of
differentiability,
$0\preceq D P_a(y)
    \preceq
    \frac{1}{1+a\kappa}I_d.$
Therefore, for almost every $y$, 
we have $0\leq \dive P_a(y)
    \leq
    \frac{d}{1+a\kappa}.$ In the scalar case, for almost every $y$, we have \(0 \leq P_a'(y) \leq \frac{1}{1 + a\kappa}\).
\end{corollary}

\begin{proof}
See \textcite[Eq.\ (O.7)--(O.10)]{Celentano2022barrier}.
\end{proof}

\begin{lemma}[Nonconstant proximal maps]
\label{lem:nonconstant-proximal}
Suppose that \(\dom\,h\) contains at least two distinct points. Then, for
every \(a>0\), the map \(P_a=\prox_{a h}\) is nonconstant. If \(d=1\), then letting  \(\mathcal L^1\) denote the Lebesgue measure, we have $\mathcal L^1(\{y\in\R:P_a'(y)>0\})>0$.
\end{lemma}

\begin{proof}
By \textcite[Theorem~23.4]{Rockafellar1970},
\(
    \operatorname{ri}(\dom\,h)
    \subseteq \dom(\partial h),
\) where $\operatorname{ri}(C)$, for a convex set $C$, is its interior taken within its affine hull. Because \(\dom\,h\) is convex and contains at least two
distinct points, so does \(\operatorname{ri}(\dom\,h)\).
On the other hand, the resolvent characterization of the proximal map in \textcite[Theorem~12.15 and Proposition~12.19]{RockafellarWets1998} gives
\(
    P_a=(\Id+a\partial h)^{-1}\),
and
\(
    \operatorname{range}(P_a)=\dom(\partial h).
\)
Consequently, the range of \(P_a\) contains at least two points, so
\(P_a\) is nonconstant.

Now suppose \(d=1\). By \cref{prop:prox-firm-app}, \(P_a\) is nondecreasing and
one-Lipschitz, and hence absolutely continuous on every bounded
interval. Since \(P_a\) is nonconstant, there exist \(x<y\) such that
\(P_a(y)>P_a(x)\). Therefore,
\[
    0<P_a(y)-P_a(x)=\int_x^y P_a'(t)\,dt.
\]
Moreover, \(P_a'(t)\geq 0\) almost everywhere because \(P_a\) is
nondecreasing. It follows that
\(
    \mathcal L^1(\{t\in[x,y]:P_a'(t)>0\})>0,
\)
and hence
\(
    \mathcal L^1(\{t\in\mathbb R:P_a'(t)>0\})>0 .
\)
\end{proof}

\begin{lemma}\label{lem:proxfacts}
Let \(h:\R\to\R\cup\{+\infty\}\) be proper, lsc, and
convex. The map
\(
 (w,y)\longmapsto P_{w}(y) := \prox_{wh}(y)
\)
is continuous on \((0,\infty)\times\R\).
\end{lemma}

\begin{proof}
This follows from
\cite[Theorem~2.26(a)]{RockafellarWets1998}, which states that
\(
    (\lambda,y)\longmapsto \prox_{\lambda h}(y)
\)
is continuous at every point with \(\lambda>0\).
\end{proof}

The next bound is useful when the effective regularization parameter varies
along a fixed-point sequence.

\begin{lemma}
\label{lem:prox-scale-continuity-app}
For \(a, \widetilde a\geq0\), and \(y\in\R^d\),
we have $\norm{P_a(y)-P_{\widetilde a}(y)}_2
    \leq
    \norm{y-P_a(y)}_2
    \abs{\frac{\widetilde a}{a}-1}.$
In particular, on any compact interval
\([a_{\min},a_{\max}]\subset(0,\infty)\), the map
\(a\mapsto P_a(y)\) is locally Lipschitz whenever
\(\norm{y-P_a(y)}_2\) is locally bounded.
\end{lemma}

\begin{proof}
See \textcite[Eq.\ (O.5)]{Celentano2022barrier}.
\end{proof}

The localized width condition concerns the behavior of
\(\prox_{a h}\) when \(a\) is large. The following result identifies the
limit exactly whenever the minimizer set is nonempty. We denote by $\Pi_C(y)$ the projection of $y$ onto $C$ for any set $C \subset \R^d$, and $\operatorname{zer}(\partial h)$ the zero set of the subdifferential of $h$.

\begin{proposition}[Convergence to the minimizer projection]
\label{prop:prox-large-scale-app}
Assume
\[
M:=\argmin h\neq\emptyset.
\]
Then \(M\) is closed and convex, and for every \(y\in\R^d\), we have $\prox_{a h}(y) \longrightarrow \Pi_M(y)$ as $a\to\infty$.
\end{proposition}

\begin{proof}
Since \(h\) is proper, lsc, and convex,
\(\partial h\) is maximally monotone
\cite[Theorem~20.25]{BauschkeCombettes2017}. Moreover,
Fermat's rule and the proximal--resolvent identity give
\(
    \operatorname{zer}(\partial h)
      = \argmin h=M\), and
\(
    \prox_{ah}
      = J_{a\partial h}
      :=(\Id + a\partial h)^{-1};
\)
see \textcite[Theorem~16.3 and Example~23.3]{BauschkeCombettes2017}.
Therefore,
\(
    \prox_{ah}(y)
      =J_{a\partial h}(y)
      \rightarrow
      \Pi_{\operatorname{zer}(\partial h)}(y)
      =\Pi_M(y)
\)
as \(a\to\infty\), by
\textcite[Theorem~23.44(i)]{BauschkeCombettes2017}.
Finally, \(M\) is closed and convex because it is the nonempty
minimizer set of a lower-semicontinuous convex function.
\end{proof}

The scalar version gives a convenient characterization of the limiting
average derivative. Recall that the minimizer set of a convex function on
\(\R\), when nonempty, is a closed interval, possibly a singleton or an
unbounded interval.

\begin{corollary}
\label{cor:prox-large-scale-derivative-app}
Let \(d=1\) and assume \(M=\argmin h\neq\emptyset\). Let
\(B\in L^2\) and let \(Z\sim N(0,1)\), independent of \(B\). For every
\(\nu>0\), we have $\lim_{a\to\infty}
    \E[
        \prox_{a h}'(B+\sqrt\nu Z)
    ]
    =
    \PP(B+\sqrt\nu Z\in\operatorname{int}(M)).$
Because \(B+\sqrt\nu Z\) has a continuous density, the right-hand side is
unchanged if \(\operatorname{int}(M)\) is replaced by \(M\).
Moreover, for every nonempty compact \(K\subset(0,\infty)\),
\begin{equation}
    \sup_{\nu\in K}
    \left|
        \E\left[\prox_{a h}'(B+\sqrt\nu Z)\right]
        -\PP\left(B+\sqrt\nu Z\in\operatorname{int}(M)\right)
    \right|
    \longrightarrow0.
    \label{eq:prox-large-scale-derivative-uniform-app}
\end{equation}
\end{corollary}

\begin{proof}
Write \(P_a:=\prox_{a h}\). By Stein's lemma,
\[
\E[\prox_{a h}'(B+\sqrt\nu Z)]
    =
    \frac1{\sqrt\nu}
    \E[Z\prox_{a h}(B+\sqrt\nu Z)].
\]
Fix \(m\in M\). Since \(P_a(m)=m\), nonexpansiveness gives $\abs{P_a(y)} \leq \abs{m}+\abs{y-m}$, which supplies an integrable dominating function after multiplication by
\(\abs Z\). \cref{prop:prox-large-scale-app} and dominated
convergence therefore imply $\lim_{a\to\infty}
    \E[\prox_{a h}'(B+\sqrt\nu Z)]
    =
    \frac1{\sqrt\nu}
    \E[Z\Pi_M(B+\sqrt\nu Z)]$. A final application of Stein's identity identifies the last expression with
\(\E[\Pi_M'(B+\sqrt\nu Z)]\). The derivative of the projection onto an
interval is one on its interior and zero outside it, almost everywhere.

It remains to prove uniformity. Let $a_k\to\infty$ and let $\nu_k$ lie in a
fixed compact $K\subset(0,\infty)$. Along any subsequence for which $\nu_k\to\nu\in K$, couple $Y_k:=B+\sqrt{\nu_k}Z$, and $Y:=B+\sqrt\nu Z$ using the same pair $(B,Z)$. Nonexpansiveness gives $|P_{a_k}(Y_k)-\Pi_M(Y)| \leq |Y_k-Y|+|P_{a_k}(Y)-\Pi_M(Y)|\rightarrow0$ almost surely.
For a fixed $m\in M$, $|P_{a_k}(Y_k)|
    \le |m|+|Y_k-m|
    \le 2|m|+|B|+\sqrt{\max K}\,|Z|$.
After multiplication by $|Z|$, the right-hand side is integrable because
$B\in L^2$. Dominated convergence and Stein's lemma therefore
show $\E[P_{a_k}'(Y_k)]
    \rightarrow
    \frac1{\sqrt\nu}\E[Z\Pi_M(Y)]
    =\PP(Y\in\operatorname{int}(M))$.
The same representation and domination show that the last probability is a
continuous function of $\nu>0$. The sequential characterization of uniform
convergence on the compact set $K$ now proves
\cref{eq:prox-large-scale-derivative-uniform-app}.
\end{proof}

The strict-gap argument uses the fact that the class of scalar proximal maps
has a simple intrinsic description.

\begin{proposition}[One-dimensional characterization]
\label{prop:scalar-prox-characterization-app}
Let \(g:\R\to\R\). The following are equivalent.

\begin{enumerate}[label=(\roman*)]
\item There exists a proper, lsc, convex function
\(h:\R\to\R\cup\{+\infty\}\) such that \(g=\prox_h\).

\item The map \(g\) is nondecreasing and one-Lipschitz.
\end{enumerate}

Equivalently, if \(g\) is locally absolutely continuous, then it is a scalar proximal map if and only if $0\leq g'(y)\leq1$ for almost every $y$. 
\end{proposition}

\begin{proof}
See \textcite[Proposition~2.4]{CombettesPesquet2007}.
\end{proof}

\section{Useful tools}\label{sec:useful-tools}

In this section, we present some useful results that are used throughout the paper. All the results are well known, so we omit their proofs and instead provide citations for each.

\begin{theorem}[Poincar\'e--Miranda Theorem in $\R^2$, Eq.\ (1) in \textcite{Frankowska2018PM}]
For a fixed $L$, let $I = [-L, L] \times [-L, L] \subset \R^2$. Let $f = (f_1, f_2): I \rightarrow \R^2$ be a continuous function such that $f_i(x) \geq 0$ for all $x \in \{(x_1, x_2) \in I: x_i = -L\}$
and $f_i(x) \leq 0$ for all $x \in \{(x_1, x_2) \in I: x_i = L\}$ for both $i = 1, 2$.
Then there exists some $\bar{x} \in I$ such that $f(\bar{x}) = 0$. We call this point the equilibrium point of $f$ in $I$.
\end{theorem}

\begin{proposition}[Tweedie's formula, Eq.\ (2.8) in \textcite{Efron2011Tweedie}]
Let $Y = B+\sqrt{\tau}Z$, where $Z\sim N(0,1)$. Assume that $\E[|B|] < \infty$. Then Tweedie's formula gives $\mathbb{E}[B\mid Y=y]
=
y+\tau\,\frac{\mathrm d}{\mathrm dy}\log p_Y(y),
$
where \(p_Y\) denotes the marginal density of \(Y\).
\end{proposition}

\begin{lemma}[Stein's lemma, \parencite{Stein1981}]
Let \(Z\sim N(0,1)\). For any absolutely continuous function
\(f:\R \to \R\) such that the expectations below are finite,
Stein's lemma gives $\E\!\left[Zf(Z)\right]
=
\E\!\left[f'(Z)\right]$.
\end{lemma}

\section{Proof of the minimizer-set characterization}\label{sec:minimizer-set-proof}
\begin{proof}[Proof of \cref{prop:q-bounded-minimizer-set}]
The uniform convergence \cref{eq:Ah-minimizer-limit} is
\cref{cor:prox-large-scale-derivative-app}. Since \(M\) is a closed
convex subset of \(\R\), its interior is an interval, possibly empty or
unbounded. The law of \(B^\star+\sqrt\nu Z\) has a continuous density
for every \(\nu>0\), so it assigns zero mass to the boundary of this
interval. Because the effective Gaussian observation model $Y_\nu = B^\star + \sqrt{\nu}Z$ is weakly continuous in \(\nu\), we have that
$\phi_M(\nu)
    :=
    \PP(B^\star+\sqrt\nu Z\in\operatorname{int}(M))$
is continuous on \((0,\infty)\).

For a compact \(K\), uniform convergence gives
\[
    \lim_{\kappa\to\infty}
    \sup_{\nu\in K}A_h(\kappa,\nu)
    =
    \max_{\nu\in K}\phi_M(\nu).
\]
Thus, the $q$-bounded width condition is equivalent to \(\phi_M(\tau_B^+)<q\), and continuity gives a one-sided interval
\([\tau_B^+-\epsilon_0,\tau_B^+]\) on which
\(\phi_M\le q-\delta_0\). Uniform convergence gives \cref{eq:q-width-condition}. The converse follows by evaluating that condition at \(\nu=\tau_B^+\). This proves the desired result.
\end{proof}

\printbibliography

@string{jan = {January}}

@string{apr = {April}}

@article{Rangan2018vamp,
  author  = {Rangan, Sundeep and Schniter, Philip and Fletcher, Alyson K.},
  title   = {Vector Approximate Message Passing},
  journal = {IEEE Transactions on Information Theory},
  year    = {2019},
  volume  = {65},
  number  = {10},
  pages   = {6664--6684}
}

@article{Li2026debiasing,
  author  = {Li, Yufan and Sur, Pragya},
  title   = {Spectrum-Aware Debiasing: A Modern Inference Framework with Applications to Principal Components Regression},
  journal = {The Annals of Statistics},
  year    = {2026},
  volume  = {54},
  number  = {2},
  pages   = {745--770}
}

@inproceedings{Gerbelot2020vamp,
  author    = {Gerbelot, C{\'e}dric and Abbara, Alia and Krzakala, Florent},
  title     = {Asymptotic Errors for High-Dimensional Convex Penalized Linear Regression beyond {Gaussian} Matrices},
  booktitle = {Proceedings of the Thirty-Third Conference on Learning Theory},
  series    = {Proceedings of Machine Learning Research},
  year      = {2020},
  volume    = {125},
  pages     = {1682--1713},
  publisher = {PMLR}
}

@article{Gerbelot2023glm,
  author  = {Gerbelot, C{\'e}dric and Abbara, Alia and Krzakala, Florent},
  title   = {Asymptotic Errors for Teacher--Student Convex Generalized Linear Models (or: How to Prove {Kabashima's} Replica Formula)},
  journal = {IEEE Transactions on Information Theory},
  year    = {2023},
  volume  = {69},
  number  = {3},
  pages   = {1824--1852}
}

@article{Celentano2022barrier,
  author  = {Celentano, Michael and Montanari, Andrea},
  title   = {Fundamental Barriers to High-Dimensional Regression with Convex Penalties},
  journal = {The Annals of Statistics},
  year    = {2022},
  volume  = {50},
  number  = {1},
  pages   = {170--196}
}

@article{Zhang2020identical,
  author  = {Zhang, Haochuan},
  title   = {Identical Fixed Points in State Evolutions of {AMP} and {VAMP}},
  journal = {Signal Processing},
  year    = {2020},
  volume  = {173},
  pages   = {107601}
}

@article{Bean2013OptimalMEstimation,
  author  = {Bean, Derek and Bickel, Peter J. and El Karoui, Noureddine and Yu, Bin},
  title   = {Optimal {$M$}-Estimation in High-Dimensional Regression},
  journal = {Proceedings of the National Academy of Sciences},
  year    = {2013},
  volume  = {110},
  number  = {36},
  pages   = {14563--14568}
}

@article{Frankowska2018PM,
  author  = {Frankowska, H{\'e}l{\`e}ne},
  title   = {The {Poincar\'e--Miranda} Theorem and Viability Condition},
  journal = {Journal of Mathematical Analysis and Applications},
  year    = {2018},
  volume  = {463},
  number  = {2},
  pages   = {832--837}
}

@article{Stein1981,
  author  = {Stein, Charles M.},
  title   = {Estimation of the Mean of a Multivariate Normal Distribution},
  journal = {The Annals of Statistics},
  year    = {1981},
  volume  = {9},
  number  = {6},
  pages   = {1135--1151}
}

@article{Efron2011Tweedie,
  author  = {Efron, Bradley},
  title   = {Tweedie's Formula and Selection Bias},
  journal = {Journal of the American Statistical Association},
  year    = {2011},
  volume  = {106},
  number  = {496},
  pages   = {1602--1614}
}

@article{Alenazi2026Stieltjes,
  author  = {Alenazi, Abdulaziz and Mehrez, Khaled},
  title   = {The Stieltjes Transform Order and Related Ratio Order},
  journal = {Statistics \& Probability Letters},
  year    = {2026},
  volume  = {233},
  pages   = {110663}
}

@article{Bayati2011AMP,
  author  = {Bayati, Mohsen and Montanari, Andrea},
  title   = {The Dynamics of Message Passing on Dense Graphs, with Applications to Compressed Sensing},
  journal = {IEEE Transactions on Information Theory},
  year    = {2011},
  volume  = {57},
  number  = {2},
  pages   = {764--785}
}

@article{MarchenkoPastur,
  author  = {Marchenko, Vladimir A. and Pastur, Leonid A.},
  title   = {Distribution of Eigenvalues for Some Sets of Random Matrices},
  journal = {Mathematics of the USSR-Sbornik},
  year    = {1967},
  volume  = {1},
  number  = {4},
  pages   = {457--483}
}

@incollection{tightframe,
  author    = {Casazza, Peter G. and Kutyniok, Gitta and Philipp, Friedrich},
  title     = {Introduction to Finite Frame Theory},
  booktitle = {Finite Frames: Theory and Applications},
  editor    = {Casazza, Peter G. and Kutyniok, Gitta},
  series    = {Applied and Numerical Harmonic Analysis},
  year      = {2013},
  pages     = {1--53},
  publisher = {Birkh{\"a}user},
  location  = {Boston}
}

@book{convex_order,
  author    = {Shaked, Moshe and Shanthikumar, J. George},
  title     = {Stochastic Orders},
  year      = {2007},
  publisher = {Springer},
  location  = {New York},
  address   = {New York}
}

@article{Hoerl1970ridge,
  author  = {Hoerl, Arthur E. and Kennard, Robert W.},
  title   = {Ridge Regression: Biased Estimation for Nonorthogonal Problems},
  journal = {Technometrics},
  year    = {1970},
  volume  = {12},
  number  = {1},
  pages   = {55--67}
}

@article{Tibshirani1996lasso,
  author  = {Tibshirani, Robert},
  title   = {Regression Shrinkage and Selection via the {Lasso}},
  journal = {Journal of the Royal Statistical Society: Series B (Methodological)},
  year    = {1996},
  volume  = {58},
  number  = {1},
  pages   = {267--288}
}

@article{Zou2005elastic,
  author  = {Zou, Hui and Hastie, Trevor},
  title   = {Regularization and Variable Selection via the Elastic Net},
  journal = {Journal of the Royal Statistical Society: Series B (Statistical Methodology)},
  year    = {2005},
  volume  = {67},
  number  = {2},
  pages   = {301--320}
}

@article{Donoho2009MessagePassing,
  author  = {Donoho, David L. and Maleki, Arian and Montanari, Andrea},
  title   = {Message-Passing Algorithms for Compressed Sensing},
  journal = {Proceedings of the National Academy of Sciences},
  year    = {2009},
  volume  = {106},
  number  = {45},
  pages   = {18914--18919}
}

@article{Bayati2012LassoRisk,
  author  = {Bayati, Mohsen and Montanari, Andrea},
  title   = {The LASSO Risk for Gaussian Matrices},
  journal = {IEEE Transactions on Information Theory},
  year    = {2012},
  volume  = {58},
  number  = {4},
  pages   = {1997--2017}
}

@article{Thrampoulidis2018Precise,
  author  = {Thrampoulidis, Christos and Abbasi, Ehsan and Hassibi, Babak},
  title   = {Precise Error Analysis of Regularized {$M$}-Estimators in High Dimensions},
  journal = {IEEE Transactions on Information Theory},
  year    = {2018},
  volume  = {64},
  number  = {8},
  pages   = {5592--5628}
}

@article{Miolane2021Lasso,
  author  = {Miolane, L{\'e}o and Montanari, Andrea},
  title   = {The Distribution of the {Lasso}: Uniform Control over Sparse Balls and Adaptive Parameter Tuning},
  journal = {The Annals of Statistics},
  year    = {2021},
  volume  = {49},
  number  = {4},
  pages   = {2313--2335}
}

@article{Ma2017OAMP,
  author  = {Ma, Junjie and Ping, Li},
  title   = {Orthogonal {AMP}},
  journal = {IEEE Access},
  year    = {2017},
  volume  = {5},
  pages   = {2020--2033}
}

@article{Takeuchi2020EP,
  author  = {Takeuchi, Keigo},
  title   = {Rigorous Dynamics of Expectation-Propagation-Based Signal Recovery from Unitarily Invariant Measurements},
  journal = {IEEE Transactions on Information Theory},
  year    = {2020},
  volume  = {66},
  number  = {1},
  pages   = {368--386}
}

@inproceedings{Fletcher2018Plugin,
  author    = {Fletcher, Alyson K. and Pandit, Parthe and Rangan, Sundeep and Sarkar, Subrata and Schniter, Philip},
  title     = {Plug-In Estimation in High-Dimensional Linear Inverse Problems: A Rigorous Analysis},
  booktitle = {Advances in Neural Information Processing Systems},
  year      = {2018},
  volume    = {31},
  pages     = {7451--7460}
}

@article{Fan2022RotInvAMP,
  author  = {Fan, Zhou},
  title   = {Approximate Message Passing Algorithms for Rotationally Invariant Matrices},
  journal = {The Annals of Statistics},
  year    = {2022},
  volume  = {50},
  number  = {1},
  pages   = {197--224}
}

@article{Wang2024AMPUniversality,
  author  = {Wang, Tianhao and Zhong, Xinyi and Fan, Zhou},
  title   = {Universality of Approximate Message Passing Algorithms and Tensor Networks},
  journal = {The Annals of Applied Probability},
  year    = {2024},
  volume  = {34},
  number  = {4},
  pages   = {3943--3994}
}

@inproceedings{BarbierEtAl2018MutualInformation,
  author    = {Barbier, Jean and Macris, Nicolas and Maillard, Antoine and Krzakala, Florent},
  title     = {The Mutual Information in Random Linear Estimation beyond {i.i.d.} Matrices},
  booktitle = {2018 IEEE International Symposium on Information Theory},
  year      = {2018},
  pages     = {1390--1394},
  publisher = {IEEE}
}

@article{LiFanSenWu2024RotInv,
  author  = {Li, Yufan and Fan, Zhou and Sen, Subhabrata and Wu, Yihong},
  title   = {Random Linear Estimation with Rotationally Invariant Designs: Asymptotics at High Temperature},
  journal = {IEEE Transactions on Information Theory},
  year    = {2024},
  volume  = {70},
  number  = {3},
  pages   = {2118--2153}
}

@article{DobribanWager2018Ridge,
  author  = {Dobriban, Edgar and Wager, Stefan},
  title   = {High-Dimensional Asymptotics of Prediction: Ridge Regression and Classification},
  journal = {The Annals of Statistics},
  year    = {2018},
  volume  = {46},
  number  = {1},
  pages   = {247--279}
}

@article{HastieEtAl2022Ridgeless,
  author  = {Hastie, Trevor and Montanari, Andrea and Rosset, Saharon and Tibshirani, Ryan J.},
  title   = {Surprises in High-Dimensional Ridgeless Least Squares Interpolation},
  journal = {The Annals of Statistics},
  year    = {2022},
  volume  = {50},
  number  = {2},
  pages   = {949--986}
}

@article{Prekopa1973,
  author  = {Pr{\'e}kopa, Andr{\'a}s},
  title   = {On Logarithmic Concave Measures and Functions},
  journal = {Acta Scientiarum Mathematicarum},
  year    = {1973},
  volume  = {34},
  number  = {1--4},
  pages   = {335--343}
}

@article{LangerWoracek2024Karamata,
   author       = {Langer, Matthias and Woracek, Harald},
   title        = {Karamata’s theorem for regularized Cauchy                   transforms},
   journal      = {Proceedings of the Royal Society of                         Edinburgh: Section A Mathematics},
   volume       = {155},
   number       = {4},
   publisher    = {Cambridge University Press (CUP)},
   year         = {2024},
   month        = Jan, 
   pages        = {1431--1491} 
}

@article{Feng2021Tutorial,
    author = {Feng, Oliver Y. and Venkataramanan, Ramji and Rush, Cynthia and Samworth, Richard J.},
    title = {A Unifying Tutorial on Approximate Message Passing},
    journal = {Foundations and Trends in Machine Learning},
    volume = {15},
    number = {4},
    pages = {335-536},
    year = {2022},
    month = {05}
}

@INPROCEEDINGS{SchniterRF16a,
  author={Schniter, Philip and Rangan, Sundeep and Fletcher, Alyson K.},
  title={Vector approximate message passing for the generalized linear model}, 
  booktitle={2016 50th Asilomar Conference on Signals, Systems and Computers}, 
  year={2016},
  volume={},
  number={},
  pages={1525-1529}
}

@article{Guo2011MMSE,
  author  = {Guo, Dongning and Wu, Yihong and Shamai, Shlomo and Verd{\'u}, Sergio},
  title   = {Estimation in Gaussian Noise: Properties of the Minimum Mean-Square Error},
  journal = {IEEE Transactions on Information Theory},
  volume  = {57},
  number  = {4},
  pages   = {2371--2385},
  year    = {2011},
  month   = apr
}

@article{CombettesPesquet2007,
  author  = {Combettes, Patrick L. and Pesquet, Jean-Christophe},
  title   = {Proximal Thresholding Algorithm for Minimization over
             Orthonormal Bases},
  journal = {SIAM Journal on Optimization},
  volume  = {18},
  number  = {4},
  pages   = {1351--1376},
  year    = {2008}
}

@article{GribonvalNikolova2021,
  author  = {Gribonval, R{\'e}mi and Nikolova, Mila},
  title   = {On Bayesian Estimation and Proximity Operators},
  journal = {Applied and Computational Harmonic Analysis},
  volume  = {50},
  pages   = {49--72},
  year    = {2021}
}

@misc{Vandervorst2008,
  author = {R. C. A. M. Vandervorst},
  title  = {Topological Methods for Nonlinear Differential Equations:
            From Degree Theory to Floer Homology},
  year   = {2008},
  note   = {Lecture notes, Version 2.0, April 25, 2008}
}

@misc{Zhang2026optimality,
      author={Zhang, Yihan and Ji, Hong Chang and Venkataramanan, Ramji and Mondelli,  Marco},
      title={Optimal Estimation in Orthogonally Invariant Generalized Linear Models: Spectral Initialization and Approximate Message Passing}, 
      year={2026},
      archivePrefix={arXiv},
      primaryClass={math.ST} 
}

@ARTICLE{Dudeja2022SpectralUniversality,
  author={Dudeja, Rishabh and Sen, Subhabrata and Lu, Yue M.},
  title={Spectral Universality in Regularized Linear Regression With Nearly Deterministic Sensing Matrices}, 
  journal={IEEE Transactions on Information Theory}, 
  year={2024},
  volume={70},
  number={11},
  pages={7923-7951}
}

@article{SaumardWellner2014,
  author  = {Saumard, Adrien and Wellner, Jon A.},
  title   = {Log-Concavity and Strong Log-Concavity: A Review},
  journal = {Statistics Surveys},
  year    = {2014},
  volume  = {8},
  pages   = {45--114}
}

@article{Voigtlaender2021Price,
  author  = {Voigtlaender, Felix},
  title   = {A General Version of {Price's} Theorem: A Tool for Bounding
             the Expectation of Nonlinear Functions of Gaussian Random
             Vectors},
  journal = {Journal of Theoretical Probability},
  year    = {2021},
  volume  = {34},
  pages   = {1474--1485}
}

@article{sur2019modern,
  author={Sur, Pragya and Cand{\`e}s, Emmanuel J},
  title={A modern maximum-likelihood theory for high-dimensional logistic regression},
  journal={Proceedings of the National Academy of Sciences},
  volume={116},
  number={29},
  pages={14516--14525},
  year={2019},
  publisher={National Academy of Sciences}
}

@article{bu2020algorithmic,
  author={Bu, Zhiqi and Klusowski, Jason M and Rush, Cynthia and Su, Weijie J},
  title={Algorithmic analysis and statistical estimation of SLOPE via approximate message passing},
  journal={IEEE Transactions on Information Theory},
  volume={67},
  number={1},
  pages={506--537},
  year={2020},
  publisher={IEEE}
  }

@article{donoho2016high,
  author={Donoho, David and Montanari, Andrea},
  title={High dimensional robust m-estimation: Asymptotic variance via approximate message passing},
  journal={Probability Theory and Related Fields},
  volume={166},
  pages={935--969},
  year={2016},
  publisher={Springer}
}

@book{NourdinPeccati2012,
  author    = {Nourdin, Ivan and Peccati, Giovanni},
  title     = {Normal Approximations with Malliavin Calculus:
               From Stein's Method to Universality},
  series    = {Cambridge Tracts in Mathematics},
  volume    = {192},
  publisher = {Cambridge University Press},
  year      = {2012}
}

@book{RockafellarWets1998,
  author    = {Rockafellar, R. Tyrrell and Wets, Roger J.-B.},
  title     = {Variational Analysis},
  series    = {Grundlehren der Mathematischen Wissenschaften},
  volume    = {317},
  publisher = {Springer},
  location  = {Berlin},
  address   = {Berlin},
  year      = {1998}
}

@book{Rockafellar1970,
  author    = {Rockafellar, R. Tyrrell},
  title     = {Convex Analysis},
  series    = {Princeton Mathematical Series},
  volume    = {28},
  publisher = {Princeton University Press},
  location  = {Princeton, NJ},
  address   = {Princeton, NJ},
  year      = {1970},
  pages     = {1--451}
}

@book{BauschkeCombettes2017,
  author    = {Bauschke, Heinz H. and Combettes, Patrick L.},
  title     = {Convex Analysis and Monotone Operator Theory in
               {Hilbert} Spaces},
  edition   = {2},
  series    = {CMS Books in Mathematics},
  publisher = {Springer},
  address   = {Cham},
  year      = {2017}
}

@article{Ma2023Spikiness,
author = {Ma, Junjie and Xu, Ji and Maleki, Arian},
year = {2023},
month = {01},
pages = {482-508},
title = {Toward Designing Optimal Sensing Matrices for Generalized Linear Inverse Problems},
volume = {70},
journal = {IEEE Transactions on Information Theory}
}

@misc{Liu2024UnifyingRRI,
      title={Unifying AMP Algorithms for Rotationally-Invariant Models}, 
      author={Songbin Liu and Junjie Ma},
      year={2024},
      eprint={2412.01574},
      archivePrefix={arXiv},
      primaryClass={math.ST},
}

@article{Silverstein1995,
  author  = {Silverstein, Jack W.},
  title   = {Strong Convergence of the Empirical Distribution of
             Eigenvalues of Large Dimensional Random Matrices},
  journal = {Journal of Multivariate Analysis},
  volume  = {55},
  number  = {2},
  pages   = {331--339},
  year    = {1995}
}
\end{document}